\documentclass[11pt]{amsart}
\usepackage[foot]{amsaddr}
\usepackage[a4paper,margin=2.65cm]{geometry}
\usepackage[english]{babel}
\usepackage{csquotes}
\usepackage[T1]{fontenc}
\usepackage{textcomp}
\usepackage{amsmath,amssymb,amsthm,mathtools,mathrsfs}
\usepackage{newtxtext}
\usepackage{newtxmath}
\usepackage[bb=boondox,cal=boondoxo,scr=boondoxo]{mathalfa}
\usepackage{enumitem}
\usepackage{tikz}
\usetikzlibrary{arrows.meta,positioning}
\usepackage{placeins}
\usepackage{needspace}
\definecolor{mixedJacobi}{HTML}{5B4B8A}
\definecolor{mixedLI}{HTML}{2563A6}
\definecolor{mixedLII}{HTML}{C56A18}
\definecolor{mixedHermite}{HTML}{74558D}
\usepackage[hidelinks]{hyperref}
\hypersetup{
	pdftitle={An Askey Scheme for Jacobi Systems of Mixed Type: Laguerre Limits of the First and Second Kinds and Hermite Limits},
	pdfauthor={Manuel Ma\~nas},
	pdfsubject={A confluence scheme for Jacobi mixed-type systems with direct Jacobi-to-Hermite and Laguerre-to-Hermite limits},
	pdfkeywords={mixed-type multiple orthogonal polynomials; Askey scheme; Jacobi systems; Laguerre systems of the first and second kinds; Hermite systems; Cauchy-Vandermonde determinants; confluent limits}
}
\usepackage{microtype}
\usepackage[
	backend=biber,
	style=numeric,
	sorting=nyt,
	giveninits=true,
	maxnames=99,
	doi=true,
	url=true,
	isbn=false
]{biblatex}

\DeclareFieldFormat[article]{title}{\mkbibemph{#1}}

\DeclareFieldFormat[article]{journaltitle}{#1}

\DeclareFieldFormat[article]{volume}{\textbf{#1}}

\DeclareFieldFormat{pages}{#1}
\DeclareFieldFormat{eid}{#1}

\renewbibmacro{in:}{%
	\ifentrytype{article}
	{}
	{\printtext{\bibstring{in}\intitlepunct}}%
}

\renewbibmacro*{journal+issuetitle}{%
	\usebibmacro{journal}%
	\setunit*{\addspace}%
	\printfield{volume}%
	\setunit{\addspace}%
	\printtext[parens]{\printfield{year}}%
	\setunit{\addspace}%
	\iffieldundef{pages}
		{\printfield{eid}}
		{\printfield{pages}}%
	\newunit
}

\renewbibmacro*{note+pages}{%
	\ifentrytype{article}
	{}
	{\printfield{note}%
		\setunit{\bibpagespunct}%
		\printfield{pages}}%
}
\usepackage{bigints}
\usepackage[renew-dots,renew-matrix]{nicematrix}
\mathtoolsset{showonlyrefs}
\allowdisplaybreaks

\newcommand{\dx}{\,\mathrm{d}x}

\newcommand{\e}{\mathrm{e}}
\newcommand{\M}{\mathcal{M}}
\newcommand{\one}{\mathbf{1}}
\newcommand{\pFq}[5]{\;{}_{#1}F_{#2}\left(\begin{matrix}#3\\#4\end{matrix};#5\right)}
\newcommand{\Jmat}{\boldsymbol{\mu}^{\,\mathrm J}}
\newcommand{\Lagmat}{\boldsymbol{\mu}^{\,\mathrm L}}
\newcommand{\LIImat}{\boldsymbol{\mu}^{\,\mathrm{LII}}}
\newcommand{\Lap}{\mathcal L}
\newcommand{\thetaop}{\mathscr D}

\renewcommand{\top}{\mathsf T}
\theoremstyle{plain}
\newtheorem{theorem}{Theorem}[section]
\newtheorem{proposition}[theorem]{Proposition}
\newtheorem{lemma}[theorem]{Lemma}
\newtheorem{corollary}[theorem]{Corollary}

\newtheoremstyle{definitionstyle}
{6pt}
{12pt}
{\normalfont}
{}
{\bfseries}
{.}
{0.5em}
{}

\theoremstyle{definitionstyle}
\newtheorem{definition}[theorem]{Definition}
\newtheorem{example}[theorem]{Example}

\newtheoremstyle{remarkstyle}
{6pt}
{12pt}
{\normalfont}
{}
{\itshape}
{.}
{0.5em}
{}

\theoremstyle{remarkstyle}
\newtheorem{remark}[theorem]{Remark}

\title[An Askey scheme for Jacobi systems of mixed type]
{An Askey Scheme for Jacobi Systems of Mixed Type:\\
Laguerre Limits of the First and Second Kinds and Hermite Limits}
\author{Manuel Ma\~nas}
\address{Department of Theoretical Physics, Faculty of Physical Sciences,
	Complutense University of Madrid, 28040 Madrid, Spain}
\email{manuel.manas@ucm.es}
\date{September 18, 2026}

\begin{document}
	
\begin{abstract}
	An Askey-type confluence scheme is developed for Jacobi multiple
	orthogonal systems of mixed type, with \(q\) row weights and \(p\)
	column weights. Rescaling the Jacobi variable near \(0\) and \(1\)
	produces Laguerre systems of the first and second kinds, respectively.
	A Hermite system with one Gaussian--gamma row and \(q-1\) exponential
	convolutions is obtained from Laguerre of the first kind, directly from
	the second kind, and directly from Jacobi.
	The construction provides explicit limiting
	weights and orthogonal forms, together with convergence of their
	mixed moments. For each fixed index, the recurrence coefficients
	and the entries of the bidiagonal factors also converge in the
	limit to Laguerre of the first kind under the stated nonvanishing conditions.
	
	The Laguerre family of the second kind is described by additive gamma convolutions.
	A closed product formula for its moment determinants is obtained for
	arbitrary positive real shape parameters and near-diagonal row
	multi-indices, yielding explicit normality conditions. Its orthogonal
	forms admit contour and Rodrigues representations, and their polynomial
	components have terminating hypergeometric expressions: multiple
	Kamp\'e de F\'eriet series for \(A\), and at most \(q+1\)
	Srivastava--Daoust series for each \(B\)-component when \(q\ge2\).
	
	For this Hermite family, normality holds at every Laguerre-admissible
	near-diagonal index under the stated separation conditions; these indices
	form an infinite set. The polynomial components have finite Hermite
	expansions and terminating Srivastava--Daoust representations. The
	three routes give the same limiting moments and normalized components,
	with a distinction between the direct Jacobi limit and the two
	iterated limits through Laguerre. An appendix treats limits with
	several derivative rows, including their exact loss of rank and
	divergence of components. The one-row reductions recover the classical
	multiple Laguerre system of the second kind and the classical multiple
	Hermite systems.
\end{abstract}


\maketitle

\tableofcontents

\section{Introduction}
\label{sec:introduction}

Multiple orthogonal polynomials extend orthogonality in one variable
from a single measure to a finite system of measures. Their origins in
simultaneous rational approximation and Hermite--Pad\'e approximation
connect them with questions of convergence and irrationality; they also
occur in random matrix models and models of non-intersecting paths
\cite[Chapter~23]{Ismail2005} and~\cite{MartinezVanAssche2016}.

Classical multiple families provide explicit examples in which the
weights, polynomial formulas, and recurrence relations can be studied
together. They include Jacobi--Pi\~neiro, the two kinds of multiple
Laguerre polynomials, multiple Hermite, and the discrete Hahn, Meixner,
Kravchuk, and Charlier families
\cite{VanAsscheCoussement2001,BranquinhoDiazFoulquieManas2025Classical,
BranquinhoDiazFoulquieManasWolfs2024Discrete}. Their explicit structure
makes them useful both for applications and for understanding how
classical orthogonality changes when several weights are present.

Mixed-type orthogonality places polynomial components on both sides of
a rectangular matrix of moment functionals. For a rank-one matrix
weight, polynomial combinations of one family of weights are tested
against a second family, with independently prescribed degree bounds.
Finding explicit classical families in this setting requires more than
combining two ordinary multiple systems: the mixed moment matrices must
be nonsingular for the chosen pairs of multi-indices, and formulas for
both polynomial vectors must satisfy all the coupled conditions.
These questions have direct applications. Mixed Hermite systems describe
non-intersecting Brownian paths with several initial and final
positions~\cite{DaemsKuijlaars2007}; more generally, mixed orthogonality
is related to multicomponent Toda hierarchies
\cite{AlvarezFidalgoManas2011} and provides spectral representations
of banded matrices with positive bidiagonal factorizations
\cite{BranquinhoFoulquieManas2023Spectral}.

Recent constructions of mixed Bessel, Jacobi, and Pi\~neiro systems
provide explicit families for which these questions can be resolved.
The Bessel construction in~\cite{Manas2026BesselMixed} gives a matrix
weight on the unit circle of generically maximal rank, terminating
hypergeometric formulas for both polynomial vectors, and explicit
normality and recurrence results in the stated index ranges.
The Jacobi construction in~\cite{Manas2026HypergeometricMixed} combines
Mellin-convolution row weights with independent power column weights;
it gives hypergeometric representations, recurrence coefficients, and
bidiagonal factorizations, together with a related Laguerre family.
For the mixed Pi\~neiro system, contour and hypergeometric formulas and
explicit recurrence and factorization coefficients are obtained
in~\cite{PineiroMixed2026}. These results make the mixed families
available for a detailed study of limits, including the behaviour of
their polynomial components and recurrence matrices.

In the classical theory, the Askey scheme organizes orthogonal
polynomials of hypergeometric type through their limiting relations
\cite{KoekoekLeskySwarttouw2010}. Each vertex represents a family and
each arrow a specified limit of parameters and, when required, of the
variable and normalization. The diagram relates continuous and discrete
orthogonality and explains how families with fewer parameters are obtained
by scaling and parameter limits. Its continuous part includes the familiar
Jacobi-to-Laguerre and Laguerre-to-Hermite limits. In the multiple case,
such limits can produce distinct families from the same scalar
specialization: multiple Laguerre polynomials of the first kind vary
the powers of the variable and share an exponential factor, whereas
those of the second kind share a power and vary the exponential factors.

Parts of a multiple Askey scheme were developed for the classical
continuous families in~\cite{VanAsscheCoussement2001} and for multiple
Wilson and Jacobi--Pi\~neiro polynomials
in~\cite{BeckermannCoussementVanAssche2005}. Explicit type-I formulas
and limits for Hahn and its limiting families were derived
in~\cite{BranquinhoDiazFoulquieManas2023HahnTypeI}, with the corresponding
bidiagonal factorizations studied
in~\cite{BranquinhoDiazFoulquieManas2025HahnBidiagonal}.
Contour and hypergeometric representations for Jacobi--Pi\~neiro,
Laguerre of the first kind, and Hahn with an arbitrary number of weights are obtained
in~\cite{BranquinhoDiazFoulquieManasWolfs2025Integral}.
The continuous and discrete families with an arbitrary number of weights
are treated in~\cite{BranquinhoDiazFoulquieManas2025Classical,
BranquinhoDiazFoulquieManasWolfs2024Discrete}. In particular, the latter
work gives the limits of the weights, polynomials, and recurrence
coefficients along the discrete part of the multiple scheme.

The present article begins an extension of this programme to
multiple orthogonality of mixed type by constructing the continuous
Jacobi--Laguerre--Hermite part of the scheme. The starting point is the
Jacobi matrix of~\cite{Manas2026HypergeometricMixed}, formed from the
\(q\) Mellin-convolution weights \(w_j^{\mathrm J}\) of
Wolfs~\cite{Wolfs2024} and \(p\) independent powers:
\[
	\mathrm d\Jmat(x)
	=[w_j^{\mathrm J}(x)x^{\alpha_i}]_{1\le j\le q,\,1\le i\le p}\,\mathrm dx,
	\qquad 0<x<1.
\]
Its mixed-type forms, denoted by \(A\) and \(B\), combine the column
and row weights, respectively, with polynomial coefficients.
For \(p=1\), this gives Wolfs's Jacobi-like family; for \(q=1\), it
gives the Jacobi--Pi\~neiro system. Ordinary Jacobi orthogonality is
recovered when \(p=q=1\). When the Jacobi parameter vectors
\(\boldsymbol b\) and \(\boldsymbol a\) coincide, the row weights
reduce to powers of \(x\) and the mixed Pi\~neiro family is obtained.
This specialization is described in
Section~\ref{subsec:Jacobi-source-data}.

The Laguerre family of~\cite{Manas2026HypergeometricMixed} uses Mellin
convolutions of beta and gamma weights and successive applications of
the Euler operator \(-x\,\mathrm d/\mathrm dx\). It is called
Laguerre-like there and a \emph{mixed-type Laguerre system of the first kind} here, since
its one-row case is the classical multiple Laguerre system of the first
kind. The required results for Jacobi and Laguerre of the first kind are recalled below with
references to their proofs in
\cite[Sections~3--7]{Manas2026HypergeometricMixed}.
Rescaling the Jacobi variable near \(x=0\) and \(x=1\) gives
Laguerre systems of the first and second kinds, respectively. Centering
and rescaling the variable in the system of the first kind then gives a
Hermite system on the real line. A direct limit from the system of the
second kind gives the same Hermite rows, together with
convergence of their polynomial components at the admissible normal indices.
These limits are studied at the level of the mixed matrix measures,
moments, and normalized forms. For the limit from Jacobi to Laguerre of the first kind,
the recurrence and factorization coefficients also converge whenever
the required limiting minors are nonzero.
A direct rescaling of the Jacobi weights near \(x=1/2\) also gives the
Hermite family with \(q-1\) exponential convolutions of one base row,
without an intermediate Laguerre limit;
see Theorem~\ref{thm:Jacobi-Hermite-direct}.
The five limits are shown in Figure~\ref{fig:Askey-confluence-diagram}.

In the one-row reduction \(q=1\), the underlying confluence from multiple
Laguerre polynomials of the second kind to multiple Hermite polynomials
appears at the level of the weights and associated random-matrix processes
in Adler, van Moerbeke, and Wang
\cite[preprint version, equation~(217)]{AdlerVanMoerbekeWang2013}.
The mixed-type extension is established here in
Section~\ref{subsec:LII-Hermite-direct}, including convergence of the
individual polynomial components under the stated normality assumptions.
No previous formulation of this mixed-type confluence has been found
in the literature consulted.

\begin{figure}[!ht]
\centering
\begin{tikzpicture}[
	>={Stealth[length=2.4mm,width=1.6mm]},
	box/.style={rounded corners=2.2pt,align=center,font=\small,
		text width=32mm,minimum height=10mm,inner sep=2.5pt,line width=.9pt},
	J/.style={box,draw=mixedJacobi,fill=mixedJacobi!7},
	LI/.style={box,draw=mixedLI,fill=mixedLI!7},
	LII/.style={box,draw=mixedLII,fill=mixedLII!7},
	H/.style={box,draw=mixedHermite,fill=mixedHermite!8},
	limit/.style={->,line width=.9pt,draw=black!58},
	annotation/.style={font=\scriptsize,fill=white,inner sep=1.5pt}
]
	\node[J] (J) at (0,5.00) {Jacobi\\of mixed type};
	\node[LI] (LI) at (-3.10,2.50) {Laguerre\\of the first kind\\of mixed type};
	\node[LII] (LII) at (3.10,2.50) {Laguerre\\of the second kind\\of mixed type};
	\node[H] (H) at (0,0) {Hermite\\of mixed type};

	\draw[limit] (J) -- node[annotation,midway,below,sloped]
		{near \(x=0\)} node[annotation,midway,above,sloped]
		{\(t\to+\infty\)} (LI);
	\draw[limit] (J) -- node[annotation,midway,below,sloped]
		{near \(x=1\)} node[annotation,midway,above,sloped]
		{\(\tau\to+\infty\)} (LII);
	\draw[limit] (LI) -- node[annotation,midway,above,sloped]
		{\(\tau\to+\infty\)} (H);
	\draw[limit] (LII) -- node[annotation,midway,above,sloped]
		{\(\tau\to+\infty\)} (H);
	\draw[limit] (J.south) -- node[annotation,pos=.5,right,align=left]
		{\(\tau\to+\infty\)} (H.north);
\end{tikzpicture}
\caption{The Askey scheme for the Jacobi system of mixed type.
Rescaling near \(x=0\), together with a triangular change of row weights,
gives Laguerre of the first kind; rescaling near \(x=1\) gives Laguerre
of the second kind. Centering and rescaling the variable in the system
of the first kind gives Hermite. The direct limit from the second kind
gives the same Hermite system, by
Theorem~\ref{thm:LII-Hermite-direct-matrix} and
Corollary~\ref{cor:LII-Hermite-direct-components}. The vertical
Jacobi--Hermite arrow is the direct limit of
Theorem~\ref{thm:Jacobi-Hermite-direct}. The Hermite node has one
Gaussian--gamma row and \(q-1\) exponential convolutions; the first-kind
route uses the Laguerre family with one gamma factor.
The agreement of the three routes to Hermite, with the two Laguerre
routes interpreted as iterated limits, is stated in
Corollary~\ref{cor:Hermite-three-route-agreement}.
All four systems have
continuous measures. The admissibility and separation conditions at the
Hermite node are stated in Section~\ref{subsec:Laguerre-to-Hermite-confluence}.}
\label{fig:Askey-confluence-diagram}
\end{figure}
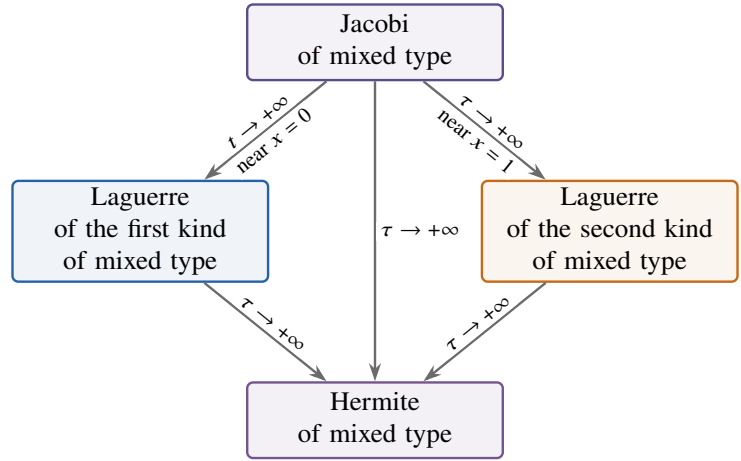

Throughout, Jacobi, Laguerre of the first and second kinds, and Hermite
denote the systems of mixed type. The suffix ``-like'' is reserved for the
corresponding one-column reductions \(p=1\).

For the limit near \(x=0\), the selected denominator parameters in the
Mellin transforms tend to infinity simultaneously. Some rescaled row
weights then have the same leading term. Divided differences of those
weights give a triangular change of basis whose limits contain the
successive Euler derivatives required for Laguerre of the first kind. At \(p=1\),
Wolfs obtained the weight-vector limit by letting the parameters tend to
infinity one at a time~\cite[Lemma~3.1]{Wolfs2024}. The simultaneous
limit considered here also gives convergence of the mixed moments and
normalized forms. For each fixed finite part of the recurrence matrix,
its coefficients and bidiagonal factors converge whenever the limiting
shifted moment minors used to define them are nonzero.

The limit near \(x=1\) gives the Laguerre system of the second kind. For positive
\(\rho_h\) and \(d_h\), and column parameters \(\xi_i\) satisfying
\(\rho_h+\xi_i>0\), set
\[
	a_h(\tau)+1=\tau\rho_h,
	\qquad
	b_h(\tau)-a_h(\tau)=d_h>0,
	\qquad
	\alpha_i(\tau)=\tau\xi_i,
\]
and use \(x=1-y/\tau\), with \(\tau\to+\infty\). After normalization, the
beta weights tend to gamma weights, and their Mellin convolutions tend
to ordinary additive convolutions. The limiting row weights \(F_j\)
have Laplace transforms
\[
	\Lap[F_j](z)
	=
	(z+\rho_j)^{-1}\prod_{h=1}^{q}(z+\rho_h)^{-d_h}.
\]
The column powers tend to
\((\e^{-\xi_1y},\ldots,\e^{-\xi_py})\). For \(q=1\), writing
\(d=d_1\) and \(\rho=\rho_1\), the entries are
constant multiples of \(y^d\e^{-(\rho+\xi_i)y}\), the classical multiple
Laguerre weights of the second kind
\cite[Section~3.3]{VanAsscheCoussement2001}
\cite[Equations~(19) and~(21)]{BranquinhoDiazFoulquieManas2025Classical}.
Classical product formulas for Cauchy--Vandermonde determinants with repeated
nodes and prescribed poles were developed in
\cite{GascaMartinezMuehlbach1989Rational,CarstensenMuehlbach1992NevilleAitken,
Muehlbach1993CauchyVandermonde,Muehlbach2000Interpolation}. In the present
mixed-type moment problem, division by a common analytic factor and a
triangular change of polynomial basis reduce the determinant to that
classical form. This gives a closed product formula for arbitrary
positive real \(d_h\) and every near-diagonal row multi-index, whose
entries differ by at most one. The determinant is nonzero when the
\(\rho_h\)'s with positive row indices are pairwise distinct and the
\(\xi_i\)'s with positive column indices are pairwise distinct. The
\(A\)-components are terminating multiple Kamp\'e de F\'eriet
polynomials. For \(q\ge2\), each \(B\)-component is a linear combination
of at most \(q+1\) terminating Srivastava--Daoust generalized Lauricella
functions, with explicit parameters and polynomial prefactors. The number
of terms is independent of the multi-indices. These are formulas for the
individual polynomial components; the complete \(B\)-form also has a
finite confluent Lauricella representation.
The relation of the row weights to the derivative-type families
of Wolfs~\cite[Proposition~2.4]{Wolfs2025DerivativeType} is also described.

The Hermite family has a Gaussian--gamma convolution \(F\) as its
last row and \(q-1\) exponential convolutions \(K_{\rho_h,1}F\) as
the remaining rows. Its connection determinant is nonzero at every
Laguerre-admissible index under the stated separation conditions. These
indices form an infinite near-diagonal class. The \(A\)-problem is
strongly normal, and the \(B\)-problem has a unique nonzero normalized
form. Each \(A\)-component is one terminating Srivastava--Daoust function,
equivalently a finite Hermite sum. Each \(B\)-component is a combination
of a number of terminating functions bounded independently of the
degrees. Both its finite-pole and polynomial contributions are explicit.

The choice of one final row is important. Retaining several Euler rows
in the first-kind Laguerre source gives \(F,F',\ldots\) in the
limit. Pearson relations then restrict nonzero normalized \(B\)-forms
to finitely many admissible indices. Appendix~\ref{app:Hermite-derivative-limits}
preserves the general confluence and determines its rank loss, polynomial
relations, and component divergence. It also compares the
derivative-only rows with coinciding-centre limits of the Gaussian
systems of Daems and Kuijlaars~\cite{DaemsKuijlaars2007}. In the main
Hermite family, the extra convolution rows give a different system.

The paper is organized as follows. Section~\ref{sec:source-families} fixes
the notation and recalls the weights and forms of Jacobi and Laguerre of the first kind.
Section~\ref{sec:left-endpoint-confluence} obtains Laguerre of the first kind by rescaling
the Jacobi variable near \(x=0\).
Section~\ref{sec:right-LII-confluence} obtains Laguerre of the second kind by rescaling
near \(x=1\) and proves the product formula for its moment determinants.
Section~\ref{subsec:Laguerre-to-Hermite-confluence} establishes the Hermite
family, its explicit components and normality at admissible indices,
the classical multiple-Hermite reductions, and the three confluence
routes. Appendix~\ref{app:Hermite-derivative-limits} treats the general
first-kind limits with several derivative rows and their restrictions.

\section{Mixed-type forms for Jacobi and Laguerre of the first kind}
\label{sec:source-families}

The Jacobi systems and the Laguerre systems of the first kind used in the
confluences are recalled here, together with the normalization of their mixed-type forms. The
notation for the special functions is fixed first. The recalled moment
determinant for Laguerre of the first kind gives the precise normality conditions required in
the subsequent limits.

\subsection{Basic notation, special functions, and transforms}
\label{subsec:hypergeometric-notation}

Throughout, \(\mathbb N\coloneq\{1,2,\ldots\}\),
\(\mathbb N_0\coloneq\{0,1,2,\ldots\}\), and
\(\mathbb R_+\coloneq(0,\infty)\).
The vector \(\one_d\coloneq(1,\ldots,1)\) has \(d\) entries,
\(\boldsymbol e_j^{(d)}\) is the \(j\)-th coordinate vector, and
\(I_d\) is the \(d\times d\) identity matrix; dimension superscripts on
coordinate vectors are omitted when clear from the context. The symbol
\(I\) denotes the identity operator when used as an operator.
The Kronecker symbol \(\delta_{i,j}\) is one for \(i=j\) and zero
otherwise, and \(\boldsymbol 1_E(x)\) is one for \(x\in E\) and zero
otherwise. For a multi-index \(\boldsymbol\nu\in\mathbb N_0^d\), put
\(|\boldsymbol\nu|\coloneq\sum_{j=1}^d\nu_j\).
Empty sums are zero and empty products are one.

For any variable \(u\), the differentiation operators are denoted by
\[
 \partial_u\coloneq\frac{\partial}{\partial u},\qquad
 \partial_u^k\coloneq\frac{\partial^k}{\partial u^k}\quad(k\in\mathbb N),
 \qquad \partial_u^0\coloneq I.
\]
All other variables and parameters are held fixed. For a function of one
variable, these are ordinary derivatives; differentiation in a complex
variable means holomorphic differentiation.

For a fixed compact set \(K\) and a positive scale \(r_\tau\), the notation
\(E_\tau(z)=\mathrm O_K(r_\tau)\) as \(\tau\to+\infty\) means that
\[
 \sup_{z\in K}|E_\tau(z)|\le C_Kr_\tau,\qquad \tau\ge \tau_K,
\]
where \(C_K\) and \(\tau_K\) are independent of \(\tau\) and \(z\), but may
depend on \(K\) and the fixed parameters. Other subscripts on
\(\mathrm O\) likewise indicate permitted dependence of the constant.
For polynomials, a coefficientwise estimate bounds each coefficient in
the stated basis; for finite matrices or vectors, it bounds each entry.
The degree bounds and dimensions are fixed in these estimates.

The gamma function is given by
\[
\Gamma(z)\coloneq\int_0^\infty t^{z-1}\e^{-t}\,\mathrm dt,
\qquad \operatorname{Re}z>0,
\]
and elsewhere by its meromorphic continuation. For \(x>0\), complex
powers mean \(x^z\coloneq\exp(z\log x)\), with the real logarithm.

The Pochhammer symbol is denoted by
\[
(a)_0\coloneq1,
\qquad
(a)_k\coloneq\prod_{j=0}^{k-1}(a+j),
\quad k\ge1.
\]
When \(a\notin\mathbb Z_{\le0}\), this finite product also equals
\(\Gamma(a+k)/\Gamma(a)\). Thus the product definition remains meaningful
at the nonpositive integers, where the quotient of Gamma functions need not
be. The falling factorial and generalized binomial coefficient are
\[
(a)_{\underline k}\coloneq\prod_{j=0}^{k-1}(a-j),
\qquad \binom ak\coloneq\frac{(a)_{\underline k}}{k!},
\qquad k\in\mathbb N_0.
\]
Thus \((a)_{\underline0}=1\). For \(\alpha_1,\ldots,\alpha_d\in\mathbb N_0\)
with \(\alpha_1+\cdots+\alpha_d=k\), the multinomial coefficient is
\(\binom{k}{\alpha_1,\ldots,\alpha_d}\coloneq
k!/(\alpha_1!\cdots\alpha_d!)\).
For a parameter vector
\[
\boldsymbol a=(a_1,\ldots,a_r),
\]
the shorthand notation
\[
\Gamma(\boldsymbol a)\coloneq\prod_{\rho=1}^r\Gamma(a_\rho),
\qquad (\boldsymbol a)_k\coloneq\prod_{\rho=1}^r(a_\rho)_k
\]
is used. In particular,
\[
\Gamma(z\one_r+\boldsymbol a)
=
\prod_{\rho=1}^{r}\Gamma(z+a_\rho),
\qquad
(z\one_r+\boldsymbol a)_k
=
\prod_{\rho=1}^{r}(z+a_\rho)_k.
\]
More generally, if
\[
\boldsymbol c=(c_1,\ldots,c_r),
\qquad
\boldsymbol \nu=(\nu_1,\ldots,\nu_r)\in\mathbb N_0^r,
\]
the notation
\[
(\boldsymbol c)_{\boldsymbol \nu}
\coloneq
\prod_{\rho=1}^{r}(c_\rho)_{\nu_\rho}
\]
is used.
Parameters \(a_1,\ldots,a_r\) are called pairwise nonresonant when
\(a_i-a_j\notin\mathbb Z\) for every \(i\ne j\).

For parameter vectors
\[
\boldsymbol a=(a_1,\ldots,a_p),
\qquad
\boldsymbol b=(b_1,\ldots,b_q),
\]
the generalized hypergeometric function is
\begin{equation}
	\label{eq:generalized-hypergeometric-definition}
	\pFq{p}{q}
	{\boldsymbol a}
	{\boldsymbol b}
	{z}
	\coloneq
	\sum_{k=0}^{\infty}
	\frac{(\boldsymbol a)_k}{(\boldsymbol b)_k}
	\frac{z^k}{k!},
\end{equation}
where the series is understood in its domain of convergence, or by analytic
continuation when appropriate.  The hypergeometric series occurring as
polynomial components below are terminating series.

The classical Kamp\'e de F\'eriet double hypergeometric series is also used;
see \cite[Chapter~1]{SrivastavaKarlsson1985}. If the parameter strings
\(\boldsymbol a,\boldsymbol b,\boldsymbol c,\boldsymbol d,\boldsymbol e,\boldsymbol f\) have respective lengths
\(p,q,k,\ell,m,n\), then
\begin{equation}
	\label{eq:standard-KdF-definition}
	F_{\ell:m;n}^{p:q;k}
	\left(
	\begin{matrix}
		\boldsymbol a:\boldsymbol b;\boldsymbol c\\
		\boldsymbol d:\boldsymbol e;\boldsymbol f
	\end{matrix}
	\middle|x,y
	\right)
	\coloneq
	\sum_{r,\lambda\ge0}
	\frac{
		(\boldsymbol a)_{r+\lambda}
		(\boldsymbol b)_r
		(\boldsymbol c)_\lambda
	}{
		(\boldsymbol d)_{r+\lambda}
		(\boldsymbol e)_r
		(\boldsymbol f)_\lambda
	}
	\frac{x^r}{r!}\frac{y^\lambda}{\lambda!}.
\end{equation}
Thus the first parameter pair is coupled to \(r+\lambda\), while the second
and third pairs depend only on \(r\) and \(\lambda\), respectively.  Empty
parameter strings are allowed.

The following multivariable extension will also be used. If the coupled parameter
strings are \(\boldsymbol a,\boldsymbol d\), and the strings attached to the
\(v\)-th summation variable are
\(\boldsymbol b^{(v)},\boldsymbol e^{(v)}\), with respective lengths
\(p,\ell\) and \(q_v,m_v\), then
\begin{equation}
	\label{eq:multiple-KdF-definition}
	F_{\ell:m_1;\ldots;m_R}^{p:q_1;\ldots;q_R}
	\left[
	\begin{array}{c}
		\boldsymbol a:
		\boldsymbol b^{(1)};\ldots;\boldsymbol b^{(R)}\\
		\boldsymbol d:
		\boldsymbol e^{(1)};\ldots;\boldsymbol e^{(R)}
	\end{array}
	\middle|x_1,\ldots,x_R
	\right]
	\coloneq
	\sum_{\nu_1,\ldots,\nu_R\ge0}
	\frac{(\boldsymbol a)_{|\boldsymbol\nu|}}
	{(\boldsymbol d)_{|\boldsymbol\nu|}}
	\prod_{v=1}^{R}
	\frac{(\boldsymbol b^{(v)})_{\nu_v}}
	{(\boldsymbol e^{(v)})_{\nu_v}}
	\frac{x_v^{\nu_v}}{\nu_v!}.
\end{equation}
For \(R=2\), this is \eqref{eq:standard-KdF-definition}. Empty individual
strings are allowed. A negative integer in the coupled upper string makes
the series a polynomial whenever no lower parameter vanishes; exceptional
terminating cases are understood by polynomial continuation.
Here polynomial continuation means coefficientwise continuation of the
complete finite expression, including its prefactor, after cancellation
of removable singularities in the parameters.

The Srivastava--Daoust generalized Lauricella series also allows
Pochhammer indices that are sums of selected summation indices
\cite[p.~284, equation~(5)]{SaigoSrivastava1994}.
The specialization needed here is defined as follows. Let
\(\boldsymbol\theta_\mu,\boldsymbol\psi_\nu\in\mathbb N_0^R\),
\(\mu\in\{1,\ldots,A\}\), \(\nu\in\{1,\ldots,C\}\), and let
\(\boldsymbol b^{(v)},\boldsymbol e^{(v)}\) have lengths \(B_v,D_v\).
With \(\boldsymbol\theta\cdot\boldsymbol k
\coloneq\sum_{v=1}^R\theta_vk_v\), put
\begin{equation}
 F_{C:D_1;\ldots;D_R}^{A:B_1;\ldots;B_R}
 \left[\begin{array}{c}
 (a_\mu:\boldsymbol\theta_\mu)_{\mu=1}^A:
 \boldsymbol b^{(1)};\ldots;\boldsymbol b^{(R)}\\
 (c_\nu:\boldsymbol\psi_\nu)_{\nu=1}^C:
 \boldsymbol e^{(1)};\ldots;\boldsymbol e^{(R)}
 \end{array}\,\middle|\,\boldsymbol z\right]
 \coloneq
 \sum_{\boldsymbol k\in\mathbb N_0^R}
 \frac{\prod_{\mu=1}^A(a_\mu)_{\boldsymbol\theta_\mu\cdot\boldsymbol k}}
      {\prod_{\nu=1}^C(c_\nu)_{\boldsymbol\psi_\nu\cdot\boldsymbol k}}
 \prod_{v=1}^R
 \frac{(\boldsymbol b^{(v)})_{k_v}}{(\boldsymbol e^{(v)})_{k_v}}
 \frac{z_v^{k_v}}{k_v!}.
 \label{eq:Srivastava-Daoust-definition}
\end{equation}
Taking every coefficient vector equal to \(\one_R\) recovers
\eqref{eq:multiple-KdF-definition}. Below, all coefficient vectors have
entries \(0\), \(1\), or \(2\). A coupled upper parameter
\((-N:\boldsymbol e)\), with positive integer entries in
\(\boldsymbol e\), restricts the sum to
\(\boldsymbol e\cdot\boldsymbol k\le N\). Every specialization below
is a finite sum. When a common parameter is cancelled before evaluation,
the corresponding polynomial interpretation is stated explicitly.

Several component formulas use the following polynomial form of a
terminating multiple Kamp\'e de F\'eriet series. For \(m\in\mathbb N_0\),
\(\boldsymbol b,\boldsymbol x\in\mathbb C^R\), and \(v\in\mathbb C\), define
\begin{equation}
 \mathfrak C_m(\boldsymbol b,\boldsymbol x;v)\coloneq
 \sum_{|\boldsymbol\alpha|\le m}
 \frac{v^{m-|\boldsymbol\alpha|}}{(m-|\boldsymbol\alpha|)!}
 \prod_{h=1}^R\frac{(b_h)_{\alpha_h}x_h^{\alpha_h}}{\alpha_h!}.
 \label{eq:finite-exponential-KdF}
\end{equation}
For \(v\ne0\), its identification with \eqref{eq:multiple-KdF-definition} is
\begin{equation}
 \mathfrak C_m(\boldsymbol b,\boldsymbol x;v)
 =\frac{v^m}{m!}
 F_{0:0;\ldots;0}^{1:1;\ldots;1}
 \left[\begin{array}{c}-m:b_1;\ldots;b_R\\-:-;\ldots;-\end{array}
 \,\middle|-\frac{x_1}{v},\ldots,-\frac{x_R}{v}\right].
 \label{eq:finite-exponential-KdF-identification}
\end{equation}
The finite sum defines its value also at \(v=0\). Empty vectors give
\(\mathfrak C_m(\varnothing,\varnothing;v)=v^m/m!\). For an additional
parameter \(w\), define the polynomial
\begin{equation}
 \mathfrak E_m(\boldsymbol b,\boldsymbol x;v,w)\coloneq
 \sum_{|\boldsymbol\alpha|+2a\le m}
 \frac{v^{m-|\boldsymbol\alpha|-2a}w^a}
      {(m-|\boldsymbol\alpha|-2a)!a!}
 \prod_{h=1}^R\frac{(b_h)_{\alpha_h}x_h^{\alpha_h}}{\alpha_h!}.
 \label{eq:finite-quadratic-KdF}
\end{equation}
For \(v\ne0\), this is one terminating Srivastava--Daoust function
of \(R+1\) arguments:
\begin{equation}
 \mathfrak E_m(\boldsymbol b,\boldsymbol x;v,w)
 =\frac{v^m}{m!}F_{0:0;\ldots;0}^{1:1;\ldots;1;0}
 \left[\begin{array}{c}
 (-m:(\one_R,2)):b_1;\ldots;b_R;-\\-:-;\ldots;-;-
 \end{array}\,\middle|-\frac{x_1}{v},\ldots,-\frac{x_R}{v},\frac{w}{v^2}\right].
 \label{eq:finite-quadratic-SD-identification}
\end{equation}
The finite polynomial in \eqref{eq:finite-quadratic-KdF} supplies its value
at \(v=0\). The entry \(2\) records the quadratic term of the exponential.
The coefficient identities explaining these definitions are
\[
 [u^m]\e^{vu}\prod_h(1-x_hu)^{-b_h}
 =\mathfrak C_m(\boldsymbol b,\boldsymbol x;v),\qquad
 [u^m]\e^{vu+wu^2}\prod_h(1-x_hu)^{-b_h}
 =\mathfrak E_m(\boldsymbol b,\boldsymbol x;v,w).
\]
Both follow by multiplying the exponential and binomial series at \(u=0\).

The Lauricella function \(F_D\), used below only in terminating form, is the
special case
\begin{equation}
	\label{eq:Lauricella-FD-definition}
	F_D^{(R)}
	\left(a;b_1,\ldots,b_R;c;x_1,\ldots,x_R\right)
	\coloneq
	\sum_{\nu_1,\ldots,\nu_R\ge0}
	\frac{(a)_{|\boldsymbol\nu|}}
	{(c)_{|\boldsymbol\nu|}}
	\prod_{v=1}^R\frac{(b_v)_{\nu_v}}{\nu_v!}x_v^{\nu_v}.
\end{equation}

The Mellin transform of a function \(f\) on \(\mathbb R_+\) is denoted by
\begin{equation}
	\label{eq:Mellin-transform-notation}
	\mathcal M[f](s)
	\coloneq
	\int_0^\infty x^{s-1}f(x)\dx,
	\qquad
	s\in\mathcal S_f,
\end{equation}
where \(\mathcal S_f\) denotes the fundamental strip of \(f\), that is, the
vertical strip of values of \(s\) for which the integral converges.  If \(f\)
is supported on \((0,1)\), the same notation is used with the integral
restricted to \((0,1)\). The elementary shift rule
\[
\mathcal M[x^\ell f](s)=\mathcal M[f](s+\ell),
\qquad
\ell\in\mathbb N_0,
\]
is used whenever both sides are defined.

For functions on \(\mathbb R_+\), the Mellin convolution is
\[
(f*g)(x)
\coloneq
\int_0^\infty
f(t)g\left(\frac{x}{t}\right)\frac{\mathrm{d}t}{t}.
\]
Whenever the integrals are convergent, the Mellin transform turns Mellin
convolution into multiplication:
\begin{equation}
	\label{eq:Mellin-convolution-product}
	\mathcal M[f*g](s)
	=
	\mathcal M[f](s)\mathcal M[g](s).
\end{equation}
These identities will also be used in the standard meromorphic-continuation
sense when both sides have meromorphic continuations.

Additive convolution is distinguished from Mellin convolution by the
subscript \(+\):
\[
(f*_+g)(y)\coloneq\int_{\mathbb R}f(u)g(y-u)\,\mathrm du.
\]
If both functions vanish on \((-\infty,0]\), this reduces to
\[
(f*_+g)(y)=\int_0^y f(u)g(y-u)\,\mathrm du,
\qquad y>0.
\]
Each convolution is used when its integral converges absolutely.
Iterated convolutions are understood associatively under the same
integrability condition.

For \(0\le m\le q\) and \(0\le n\le p\), the Meijer \(G\)-function
is defined by the Mellin--Barnes integral
\begin{equation}
	\label{eq:Meijer-G-definition}
	G_{p,q}^{m,n}
	\left(
	x
	\left|
	\begin{matrix}
		a_1,\ldots,a_p\\
		b_1,\ldots,b_q
	\end{matrix}
	\right.
	\right)
	\coloneq
	\frac{1}{2\pi\mathrm{i}}
	\bigintsss_{L}
	\frac{
		\prod_{j=1}^{m}\Gamma(b_j+s)
		\prod_{j=1}^{n}\Gamma(1-a_j-s)
	}{
		\prod_{j=m+1}^{q}\Gamma(1-b_j-s)
		\prod_{j=n+1}^{p}\Gamma(a_j+s)
	}
	x^{-s}\,\mathrm{d}s,
\end{equation}
where the contour \(L\) separates the poles of the gamma factors
\(\Gamma(b_j+s)\), \(j\in\{1,\ldots,m\}\), from those of
\(\Gamma(1-a_j-s)\), \(j\in\{1,\ldots,n\}\).
With this convention,
\begin{equation}
	\label{eq:Meijer-G-Mellin-transform}
	\int_0^\infty
	x^{z-1}
	G_{p,q}^{m,n}
	\left(
	x
	\left|
	\begin{matrix}
		a_1,\ldots,a_p\\
		b_1,\ldots,b_q
	\end{matrix}
	\right.
	\right)
	\dx
	=
	\frac{
		\prod_{j=1}^{m}\Gamma(b_j+z)
		\prod_{j=1}^{n}\Gamma(1-a_j-z)
	}{
		\prod_{j=m+1}^{q}\Gamma(1-b_j-z)
		\prod_{j=n+1}^{p}\Gamma(a_j+z)
	},
\end{equation}
in the corresponding fundamental strip. Standard contour and convergence
conditions are given in the NIST Digital Library of Mathematical
Functions~\cite[\S16.17, Eq.~(16.17.1)]{DLMFMeijerG}.

The space \(\mathcal D'(\mathbb R)\) consists of the continuous linear
functionals on \(C_c^\infty(\mathbb R)\), the space of smooth compactly
supported test functions with its usual test-function topology.
Convergence of distributions means convergence against each such test
function: \(\langle T_\tau,\varphi\rangle
\xrightarrow[\tau\to+\infty]{}\langle T,\varphi\rangle\) for every
\(\varphi\in C_c^\infty(\mathbb R)\).
Locally integrable functions are identified with distributions by
\(\langle f,\varphi\rangle\coloneq
\int_{\mathbb R}f(t)\varphi(t)\,\mathrm dt\), and distributional
derivatives satisfy \(\langle T',\varphi\rangle\coloneq
-\langle T,\varphi'\rangle\).
For finite positive measures, weak convergence means convergence of
their integrals against every bounded continuous function. It is denoted
by an arrow labelled \(\mathrm w\). Convergence of densities in this
sense always refers to their associated measures.

\subsection{Mixed-type forms}

Let \(\mathbb R[x]\) denote the ring of polynomials in one variable with
real coefficients, and put
\[
	\mathbb P_N\coloneq\{P\in\mathbb R[x]:\deg P\le N\},
	\qquad N\in\mathbb N_0,
	\qquad \mathbb P_{-1}\coloneq\{0\}.
\]
The polynomial variable is indicated when needed; complex coefficients are
used for contour calculations, and \(\deg0\coloneq-\infty\).
For a polynomial \(P\), the notation
\([x^d]P\) denotes its coefficient of \(x^d\).
If \(P(x)=\sum_{k=0}^d p_kx^k\), then
\(P(\partial_u)\coloneq\sum_{k=0}^d p_k\partial_u^k\) is the
corresponding constant-coefficient differential operator in the variable
\(u\).

Let \(p,q\in\mathbb N\), let \(I\subseteq\mathbb R\) be an interval, and let
\[
	\mathrm d\boldsymbol\mu(x)
	=
	[U_j(x)V_i(x)]_{1\le j\le q,\,1\le i\le p}\,\mathrm dx
\]
be a rank-one \(q\times p\) matrix of measures on \(I\). For
\(\boldsymbol n\in\mathbb N_0^p\) and
\(\boldsymbol m\in\mathbb N_0^q\), the mixed-type \(A\)-form is
\[
	\mathcal A_{\boldsymbol n,\boldsymbol m}(x)
	=
	\sum_{i=1}^p A_i(x)V_i(x),
	\qquad
	\deg A_i\le n_i-1,
\]
and, in the balance \(|\boldsymbol n|=|\boldsymbol m|+1\), it satisfies
\[
	\int_I x^kU_j(x)\mathcal A_{\boldsymbol n,\boldsymbol m}(x)\,\mathrm dx=0,
	\qquad 0\le k<m_j.
\]
The dual mixed-type \(B\)-form is
\[
	\mathcal B_{\boldsymbol n,\boldsymbol m}(x)
	=
	\sum_{j=1}^q B_j(x)U_j(x),
	\qquad
	\deg B_j\le m_j-1,
\]
and, in the balance \(|\boldsymbol m|=|\boldsymbol n|+1\), it satisfies
\[
	\int_I x^kV_i(x)\mathcal B_{\boldsymbol n,\boldsymbol m}(x)\,\mathrm dx=0,
	\qquad 0\le k<n_i.
\]
An index equal to zero prescribes a zero polynomial component. In either
balance, the problem is called \emph{weakly normal} if the space of
polynomial component vectors satisfying the homogeneous moment conditions
is one-dimensional. It is called \emph{strongly normal} if, in addition,
every component with a positive index attains its prescribed maximal
degree. Unless strong normality is expressly stated, normality below means
weak normality. An invertible square moment matrix obtained by adjoining
one normalization condition proves weak normality and fixes a unique
representative; maximal component degrees require a separate argument.
This definition concerns component vectors. For a polynomially dependent
row family, a nonzero vector may represent the zero \(B\)-form; a
normalization by a nonzero moment then requires a separate existence
statement. This distinction occurs in
Corollary~\ref{cor:Hermite-solution-dimensions}.

The square and rectangular moment matrices refer to the same entry
formula. For row index \(\boldsymbol m\) and column index
\(\boldsymbol n\), the entry indexed by \((j,k)\) and \((i,\ell)\) is
\[
\int_I x^{k+\ell}U_j(x)V_i(x)\,\mathrm dx,
\qquad 0\le k<m_j,\quad 0\le\ell<n_i.
\]
Their sizes are \(|\boldsymbol m|\times|\boldsymbol n|\); the ordering of
the index pairs is specified for each determinant. An adjacent moment
is the first additional moment in one family, obtained by increasing
one multi-index component by one.

An ordered family of continuous real functions \(f_1,\ldots,f_N\) on
an interval is a Chebyshev system if every linear combination with coefficients not all zero has
at most \(N-1\) distinct zeros. For sufficiently differentiable
functions, it is an extended Chebyshev system if zeros are counted with
multiplicity, and it is complete if the same property holds for every
initial subfamily. A family of weights \(U_1,\ldots,U_q\) is an
AT (algebraic Chebyshev) system for \(\boldsymbol m\) if the functions
\(x^kU_j(x)\), \(1\le j\le q\), \(0\le k<m_j\), form a Chebyshev
system. It is AT for all multi-indices if this holds for every
\(\boldsymbol m\); a multiple orthogonality system is called perfect
when all its multi-indices are normal.

\begin{definition}[Near-diagonal multi-index]
	\label{def:near-diagonal}
	A multi-index \(\boldsymbol m=(m_1,\ldots,m_q)\in\mathbb N_0^q\) is said to be
	near the diagonal if
	\[
		|m_i-m_j|\le1,
		\qquad i,j\in\{1,\ldots,q\}.
\]
\end{definition}
Equivalently, \(\max_i m_i-\min_i m_i\le1\). In particular, zero components
are allowed; if one component vanishes, all components belong to
\(\{0,1\}\).

For \(d\in\mathbb N\), the ordered \(d\)-step-line is the sequence
\[
	\boldsymbol\sigma_d(N)
	\coloneq
	(\underbrace{\rho_d(N)+1,\ldots,\rho_d(N)+1}_{\ell_d(N)},
	 \underbrace{\rho_d(N),\ldots,\rho_d(N)}_{d-\ell_d(N)}),
	\qquad
	N=d\rho_d(N)+\ell_d(N),
\]
where \(N\in\mathbb N_0\), \(0\le\ell_d(N)<d\). Equivalently,
\(\rho_d(N)=\lfloor N/d\rfloor\) and
\(\ell_d(N)=N-d\rho_d(N)\). In the coordinate-vector notation above,
\[
	\boldsymbol\sigma_d(N+1)-\boldsymbol\sigma_d(N)
	=\boldsymbol e_{\ell_d(N)+1}^{(d)}.
\]
For a vector \(\boldsymbol c=(c_1,\ldots,c_d)\), the notation
\(\boldsymbol c^{*j}\) means that its \(j\)-th component has been deleted.

\subsection{Jacobi weights and mixed-type forms}
\label{subsec:Jacobi-source-data}

For \(-1<a<b\), define
\begin{equation}
	\label{eq:beta-weight-definition}
	\mathcal B_{a,b}(x)
	\coloneq
	\frac{x^a(1-x)^{b-a-1}}{\Gamma(b-a)}
	\boldsymbol 1_{(0,1)}(x),
	\qquad
	\M[\mathcal B_{a,b}](z)=\frac{\Gamma(z+a)}{\Gamma(z+b)}.
\end{equation}
For
\(\boldsymbol a=(a_1,\ldots,a_q)\) and
\(\boldsymbol b=(b_1,\ldots,b_q)\), assume
\[
	-1<a_j<b_j,\qquad 1\le j\le q,
\]
and put
\begin{equation}
	w_0^{\mathrm J}
	\coloneq
	\mathcal B_{a_1,b_1}*\cdots*\mathcal B_{a_q,b_q},
	\label{eq:Jacobi-source-base}\qquad
	w_j^{\mathrm J}(x)
	\coloneq
	w_0^{\mathrm J}(x;\boldsymbol a,\boldsymbol b+\boldsymbol e_j),
	\qquad 1\le j\le q.
\end{equation}
Thus
\[
	\M[w_j^{\mathrm J}](z)
	=
	\frac{\Gamma(z\one_q+\boldsymbol a)}
	{\Gamma(z\one_q+\boldsymbol b+\boldsymbol e_j)}.
\]
Let \(\boldsymbol\alpha\in\mathbb R^p\) satisfy
\[
	\alpha_i+a_j>-1,
	\qquad 1\le i\le p,\quad 1\le j\le q.
\]
With \(V_i(x)=x^{\alpha_i}\), the Jacobi matrix is
\begin{equation}
	\label{eq:mixed-Jacobi-like-matrix}
	\mathrm d\Jmat(x)
	=
	[w_j^{\mathrm J}(x)x^{\alpha_i}]_{j,i}\,\mathrm dx,
	\qquad 0<x<1.
\end{equation}

The mixed Pi\~neiro specialization is obtained when
\(b_h\downarrow a_h\), \(h\in\{1,\ldots,q\}\). Indeed,
\begin{equation*}
	\M[w_j^{\mathrm J}](z)
	=\frac{\Gamma(z\one_q+\boldsymbol a)}
	{(z+b_j)\Gamma(z\one_q+\boldsymbol b)}
	\xrightarrow[\boldsymbol b\to\boldsymbol a]{}
	\frac{1}{z+a_j},
	\qquad \operatorname{Re}z>-\min_h a_h.
\end{equation*}
The last expression is the Mellin transform of
\(x^{a_j}\boldsymbol 1_{(0,1)}(x)\). Hence every fixed mixed moment
converges to that of \([x^{a_j+\alpha_i}]_{j,i}\,\mathrm dx\) on
\((0,1)\), which is the mixed Pi\~neiro matrix of
\cite[Section~5]{PineiroMixed2026}, with row parameters
\(\beta_j=a_j\) in the notation of that paper.

For multi-indices \(\boldsymbol n\in\mathbb N_0^p\) and
\(\boldsymbol m\in\mathbb N_0^q\), define
\[
	D_{\boldsymbol\alpha,\boldsymbol n}(t)
	\coloneq
	((t+1)\one_p-\boldsymbol\alpha-\boldsymbol n)_{\boldsymbol n},
	\qquad
	P_{\boldsymbol b,\boldsymbol m}(t)
	\coloneq
	((t+1)\one_q+\boldsymbol b)_{\boldsymbol m}.
\]
The two transform formulas needed below are recalled from
\cite[Theorem~3.10]{Manas2026HypergeometricMixed}, where they are proved.

\begin{theorem}[Normalized Jacobi representatives]
	\label{thm:mixed-Jacobi-like-forms}
	Assume the parameter hypotheses above and that the row multi-index
	\(\boldsymbol m\) is near the diagonal. If
	\(|\boldsymbol n|=|\boldsymbol m|+1\) and the nodes
	\[
		\{\alpha_i+k:1\le i\le p,\ 0\le k<n_i\}
\]
	are pairwise distinct, the contour representative is
	\begin{equation}
		\label{eq:A-contour}
		\mathcal A_{\boldsymbol n,\boldsymbol m}^{\mathrm J}(x)
		=
		\frac{(-1)^{|\boldsymbol n|}}{2\pi\mathrm i}
		\int_\Sigma
		\frac{P_{\boldsymbol b,\boldsymbol m}(t)}
		{D_{\boldsymbol\alpha,\boldsymbol n}(t)}
		\frac{\Gamma(t\one_q+\boldsymbol b+\one_q)}
		{\Gamma(t\one_q+\boldsymbol a+\one_q)}x^t\,\mathrm dt,
	\end{equation}
	where \(\Sigma\) is a positively oriented finite union of simple closed
	contours enclosing precisely those nodes and no other pole of the
	integrand. If
	\(|\boldsymbol m|=|\boldsymbol n|+1\), the Mellin representative is fixed by
	\begin{equation}
		\label{eq:B-Mellin}
		\M[\mathcal B_{\boldsymbol n,\boldsymbol m}^{\mathrm J}](z)
		=
		-
		\frac{\Gamma(z\one_q+\boldsymbol a)}
		{\Gamma(z\one_q+\boldsymbol b+\boldsymbol m)}
		(\boldsymbol\alpha+(1-z)\one_p)_{\boldsymbol n}.
	\end{equation}
	Whenever the corresponding moment determinant is nonzero, these formulas
	are the unique mixed-type representatives with the stated normalization.
\end{theorem}

\subsection{Weights and mixed-type forms for Laguerre of the first kind}

Write \(q=r+s\), with \(r\ge0\) and \(s\ge1\), and set
\[
	\mathcal G_a(x)=x^a\e^{-x}\boldsymbol 1_{(0,\infty)}(x),
	\qquad \M[\mathcal G_a](z)=\Gamma(z+a).
\]
Take
\[
	\boldsymbol a\in(-1,\infty)^q,\qquad
	\boldsymbol b\in(-1,\infty)^r,\qquad
	a_h<b_h\quad(1\le h\le r).
\]
The base density and its Mellin transform are
\begin{align}
	w_0^{\mathrm L}(x;\boldsymbol a,\boldsymbol b)
	&\coloneq
	\mathcal B_{a_1,b_1}*\cdots*\mathcal B_{a_r,b_r}
	*\mathcal G_{a_{r+1}}*\cdots*\mathcal G_{a_q}(x),
	\label{eq:w0-explicit-convolution}
	\\
	\M[w_0^{\mathrm L}](z)
	&=
	\frac{\Gamma(z\one_q+\boldsymbol a)}
	{\Gamma(z\one_r+\boldsymbol b)}.
	\label{eq:w0-Laguerre-Mellin}
\end{align}
Define
\begin{equation}
	\label{eq:wh-def}
	w_h^{\mathrm L}(x)
	\coloneq
	w_0^{\mathrm L}(x;\boldsymbol a,\boldsymbol b+\boldsymbol e_h),
	\qquad 1\le h\le r,
\end{equation}
and, with \(\thetaop=-x\mathrm d/\mathrm dx\),
\[
	v_\ell^{\mathrm L}(x)
	\coloneq
	\thetaop^{\ell-1}w_0^{\mathrm L}(x),
	\qquad 1\le\ell\le s.
\]
The row vector is
\(\mathbf U^{\mathrm L}=(w_1^{\mathrm L},\ldots,w_r^{\mathrm L},
v_1^{\mathrm L},\ldots,v_s^{\mathrm L})\), and the mixed-type matrix is
\begin{equation}
	\label{eq:mixed-Laguerre-matrix}
	\mathrm d\Lagmat(x)
	=
	[U_j^{\mathrm L}(x)x^{\beta_i}]_{j,i}\,\mathrm dx,
	\qquad x>0,
\end{equation}
where \(\boldsymbol\beta\in\mathbb R^p\) is required to satisfy
\[
	\beta_i+a_j>-1,
	\qquad 1\le i\le p,\quad 1\le j\le q.
\]

There is one notational reversal to keep explicit. In the generic convention
above the power index is written first, whereas the explicit Laguerre
notation writes the row index first:
\[
	\mathcal A^{\mathrm L}_{\boldsymbol\lambda,\boldsymbol m}
	=\mathcal A_{\boldsymbol m,\boldsymbol\lambda},
	\qquad
	\mathcal B^{\mathrm L}_{\boldsymbol\lambda,\boldsymbol m}
	=\mathcal B_{\boldsymbol m,\boldsymbol\lambda}.
\]

For
\(\boldsymbol\lambda=(\boldsymbol n,\boldsymbol\eta)
\in\mathbb N_0^r\times\mathbb N_0^s\), put
\(N=|\boldsymbol n|+|\boldsymbol\eta|\). The index is called
Laguerre admissible when the full multi-index is near the diagonal and
	\(\boldsymbol\eta=\boldsymbol\sigma_s(|\boldsymbol\eta|)\) lies on the
	ordered \(s\)-step-line. In particular, the two
inequalities needed later are
\begin{align}
	n_h-1&\le n_u,
	&&n_h\ge1,
	\label{eq:Laguerre-admissibility-nn}
	\\
	\eta_\ell-1&\le n_u,
	&&\eta_\ell\ge1.
	\label{eq:Laguerre-admissibility-etan}
\end{align}

\Needspace{6\baselineskip}
The following determinant formula and normality statements are recalled
from \cite[Section~4]{Manas2026HypergeometricMixed}, where they are proved.
They specify the conditions under which the Laguerre representatives of the first kind
used below have the required nonzero normalization.

For \(\boldsymbol\lambda=(\boldsymbol n,\boldsymbol\eta)\), put
\(N\coloneq|\boldsymbol\lambda|\) and define
\begin{equation}
 D_{\boldsymbol\lambda}(z)\coloneq(z\one_r+\boldsymbol b)_{\boldsymbol n},\;
 G_0(z)\coloneq\frac{\Gamma(z\one_q+\boldsymbol a)}{\Gamma(z\one_r+\boldsymbol b)},\;
 c_{\boldsymbol\lambda}\coloneq
 \prod_{h<j\le r}(b_j-b_h)^{\min(n_h,n_j)}
 \prod_{\substack{1\le j\le r\\1\le\rho\le q\\0\le u<n_j}}(b_j+1-a_\rho)_u.
 \label{eq:Laguerre-normality-coefficient-determinant}
\end{equation}
For \(|\boldsymbol m|=N\), order the column nodes
\(z_v=\beta_i+k+1\), \(0\le k<m_i\), by increasing degree and then
component, and use the analogous row order.
In this order the moment matrix for Laguerre of the first kind is
\[
G_{\boldsymbol\lambda,\boldsymbol m}\coloneq
\left[\int_0^\infty x^{k+\ell+\beta_i}U_j^{\mathrm L}(x)\,\mathrm dx
\right]_{(j,k),(i,\ell)},
\qquad 0\le k<\lambda_j,\quad 0\le\ell<m_i,
\]
where \(\lambda_j=n_j\) for \(1\le j\le r\) and
\(\lambda_{r+\ell}=\eta_\ell\) for \(1\le\ell\le s\).

\begin{proposition}[Exact Laguerre determinants and normality]
\label{prop:AT-normality-Laguerre-admissible}
Assume the standing integrability hypotheses and that
\(\boldsymbol\lambda\) is Laguerre admissible.
For \(|\boldsymbol m|=N\), the square mixed moment matrix satisfies
\begin{equation}
 \det G_{\boldsymbol\lambda,\boldsymbol m}
 =c_{\boldsymbol\lambda}\prod_{u<v}(z_v-z_u)
   \prod_{v=1}^N\frac{G_0(z_v)}{D_{\boldsymbol\lambda}(z_v)}.
 \label{eq:Laguerre-normality-moment-determinant}
\end{equation}
Suppose \(c_{\boldsymbol\lambda}\ne0\) and
\(\beta_i-\beta_j\notin\mathbb Z\) for \(i\ne j\). Then the
\(A\)-balance \(|\boldsymbol m|=N+1\) has a one-dimensional solution
space, and every component with positive prescribed length has its
maximal degree. For \(N\ge1\), the \(B\)-balance
\(|\boldsymbol m|=N-1\) also has a one-dimensional solution space and
its complete form is nonzero. All its components attain their maximal
degrees if and only if the minors obtained by deleting each highest-degree
weighted row from the rectangular moment matrix are nonzero.
\end{proposition}

The condition \(c_{\boldsymbol\lambda}\ne0\) concerns the whole row
family, not only the paired differences \(b_h-a_h\). For example,
\(\boldsymbol a=(0,7/2)\), \(b_1=5/2\), and
\(\boldsymbol\lambda=(2,1)\) give
\(w_1=\e^{-x}\), \(v_1=(x+5/2)\e^{-x}\), and
\(c_{\boldsymbol\lambda}=0\). Thus these rows are dependent even
though \(b_1-a_1> |\boldsymbol\lambda|-1\).

\begin{remark}[Reduction to classical multiple Laguerre of the first kind]
	\label{rem:multiple-Laguerre-I-edge}
	For \(q=1\), necessarily \(r=0\) and \(s=1\), and
	\[
		x^{\beta_i}U_1^{\mathrm L}(x)
		=x^{a_1+\beta_i}\e^{-x}.
\]
	Thus the one-row specialization is the classical system of multiple
	Laguerre polynomials of the first kind.
\end{remark}

For \(\boldsymbol m\in\mathbb N_0^p\), set
\[
	D_{\boldsymbol\beta,\boldsymbol m}(t)
	\coloneq
	\prod_{i=1}^p(\beta_i-t)_{m_i}.
\]

The formulas in the next theorem are those of
\cite[Section~4]{Manas2026HypergeometricMixed}. Their normalization
under the present hypotheses is ensured by
Proposition~\ref{prop:AT-normality-Laguerre-admissible}.

\begin{theorem}[Normalized Laguerre representatives of the first kind]
	\label{thm:mixed-Laguerre-forms}
	Let \(\boldsymbol\lambda=(\boldsymbol n,\boldsymbol\eta)\) be Laguerre
	admissible. Assume the standing integrability conditions,
	\(c_{\boldsymbol\lambda}\ne0\) in
	\eqref{eq:Laguerre-normality-coefficient-determinant}, and
	\(\beta_i-\beta_j\notin\mathbb Z\) for \(i\ne j\). If
	\(|\boldsymbol m|=N+1\) and the nodes
	\[
		\{\beta_i+k:1\le i\le p,\ 0\le k<m_i\}
\]
	are pairwise distinct, then
	\begin{equation}
		\label{eq:Laguerre-A-contour}
		\mathcal A_{\boldsymbol\lambda,\boldsymbol m}^{\mathrm L}(x)
		=
		\frac{1}{2\pi\mathrm i}\int_\Sigma
		\frac{\Gamma(t\one_r+\boldsymbol b+\boldsymbol n+\one_r)}
		{\Gamma(t\one_q+\boldsymbol a+\one_q)}
		\frac{x^t}{D_{\boldsymbol\beta,\boldsymbol m}(t)}\,\mathrm dt.
	\end{equation}
	Here \(\Sigma\) is a positively oriented finite union of simple closed
	contours enclosing precisely those nodes and no other pole of the
	integrand. If \(N\ge1\) and \(|\boldsymbol m|=N-1\), then
	\begin{equation}
		\label{eq:Laguerre-B-Mellin}
		\M[\mathcal B_{\boldsymbol\lambda,\boldsymbol m}^{\mathrm L}](z)
		=
		\frac{\Gamma(z\one_q+\boldsymbol a)}
		{\Gamma(z\one_r+\boldsymbol b+\boldsymbol n)}
		(\boldsymbol\beta+(1-z)\one_p)_{\boldsymbol m}.
	\end{equation}
	When the \(a_\rho\)'s are pairwise nonresonant, the same complete form is
	\begin{multline}
		\label{eq:Laguerre-B-complete-hypergeometric}
		\mathcal B_{\boldsymbol\lambda,\boldsymbol m}^{\mathrm L}(x)
		=\sum_{\rho=1}^q
		\frac{\Gamma(\boldsymbol a^{*\rho}-a_\rho\one_{q-1})
		(\boldsymbol\beta+(a_\rho+1)\one_p)_{\boldsymbol m}}
		{\Gamma(\boldsymbol b+\boldsymbol n-a_\rho\one_r)}x^{a_\rho}
		\\*[-2pt]
		\times
		\pFq{r+p}{q+p-1}
		{(a_\rho+1)\one_r-\boldsymbol b-\boldsymbol n,\,
		\boldsymbol\beta+\boldsymbol m+(a_\rho+1)\one_p}
		{(a_\rho+1)\one_{q-1}-\boldsymbol a^{*\rho},\,
		\boldsymbol\beta+(a_\rho+1)\one_p}
		{(-1)^s x}.
	\end{multline}
	Equivalently,
	\begin{equation}
		\label{eq:Laguerre-B-complete-Meijer-G}
		\mathcal B_{\boldsymbol\lambda,\boldsymbol m}^{\mathrm L}(x)
		=
		G_{p+r,p+q}^{q,p}
		\left(x\,\middle|\,
		\begin{matrix}-\boldsymbol\beta-\boldsymbol m,\,\boldsymbol b+\boldsymbol n\\
		\boldsymbol a,\,-\boldsymbol\beta\end{matrix}\right).
	\end{equation}
\end{theorem}

In particular, the individual \(A\)-components in
\cite[Section~4]{Manas2026HypergeometricMixed} are terminating
\({}_{p+r}F_{p+q-1}\) polynomials. For \(m_i\ge1\), define
\[
 \kappa_i^{\mathrm L}\coloneq-
 \frac{\Gamma((\beta_i+1)\one_r+\boldsymbol b+\boldsymbol n)}
 {(m_i-1)!\,\Gamma((\beta_i+1)\one_q+\boldsymbol a)
  (\boldsymbol\beta^{*i}-\beta_i\one_{p-1})_{\boldsymbol m^{*i}}}.
\]
With the normalization of \eqref{eq:Laguerre-A-contour}, the formula is
\begin{equation}
 A_{\boldsymbol\lambda,\boldsymbol m}^{(i),\mathrm L}(x)
 =\kappa_i^{\mathrm L}
 \pFq{p+r}{p+q-1}
 {1-m_i,\;(\beta_i+1)\one_r+\boldsymbol b+\boldsymbol n,\;
  (\beta_i+1)\one_{p-1}-\boldsymbol\beta^{*i}-\boldsymbol m^{*i}}
 {(\beta_i+1)\one_q+\boldsymbol a,\;
  (\beta_i+1)\one_{p-1}-\boldsymbol\beta^{*i}}{x}.
 \label{eq:LI-individual-A-terminating}
\end{equation}
In \eqref{eq:LI-individual-A-terminating}, the upper parameter
\(1-m_i\) terminates the series at degree \(m_i-1\);
the component is zero when \(m_i=0\). This is the residue evaluation of
the cited contour formula, with the same normalization.

If
\[
	B^{(h)}(x)=\sum_{k=0}^{n_h-1}b_{h,k}x^k,
	\qquad
	D^{(\ell)}(x)=\sum_{k=0}^{\eta_\ell-1}d_{\ell,k}x^k,
\]
then the polynomial reconstruction in
\cite[Section~4]{Manas2026HypergeometricMixed} determines the
components of the \(B\)-form through
\begin{multline}
	\label{eq:B-component-polynomial-system}
	\sum_{h=1}^r\sum_{k=0}^{n_h-1}
	b_{h,k}(z\one_q+\boldsymbol a)_k
	(z\one_r+\boldsymbol b+k\one_r+\boldsymbol e_h)_{\boldsymbol n-k\one_r-\boldsymbol e_h}
	\\*
	+
	\sum_{\ell=1}^s\sum_{k=0}^{\eta_\ell-1}
	d_{\ell,k}(z\one_q+\boldsymbol a)_k
	(z\one_r+\boldsymbol b+k\one_r)_{\boldsymbol n-k\one_r}(z+k)^{\ell-1}
	=
	(\boldsymbol\beta+(1-z)\one_p)_{\boldsymbol m}.
\end{multline}

\section{From Jacobi to Laguerre of the first kind by rescaling near zero}
\label{sec:left-endpoint-confluence}

The limit is first established for one Mellin-normalized beta weight and
then applied simultaneously to the last \(s\) factors of the Jacobi weight vector.
Writing \(b=a+1+\kappa\), for \(0<x<\kappa\) this weight satisfies
\begin{equation}
\label{eq:beta-to-gamma-confluence}
\kappa^{a}\Gamma(\kappa+1)
\mathcal B_{a,b}\!\left(\frac{x}{\kappa}\right)
=
x^{a}\left(1-\frac{x}{\kappa}\right)^{\!\kappa}
\;\xrightarrow[\kappa\to+\infty]{}\;
x^{a}\e^{-x}
=
\mathcal G_a(x),
\end{equation}
uniformly on compact subsets of \((0,\infty)\), with relative error
\(\mathrm{O}(\kappa^{-1})\) as \(\kappa\to+\infty\).
The \(q=1\) Jacobi weight used below satisfies
\(w_1(\,\cdot\,;a,b)=\mathcal B_{a,b+1}\). Hence the choice
\(b=a+\kappa\) for \(w_1\) is exactly the choice
\(b+1=a+1+\kappa\) in \eqref{eq:beta-to-gamma-confluence}.

Let \(q=r+s\).  Fix pairwise distinct positive numbers
\(\omega_1,\ldots,\omega_s\) and introduce the Jacobi lower parameters
\begin{equation}
	\label{eq:Jacobi-Laguerre-confluent-parameters}
	b_{r+\ell}(t)\coloneq t\omega_\ell,
	\qquad
	\kappa_\ell(t)\coloneq
	t\omega_\ell-a_{r+\ell}-1,
	\qquad
	\ell\in\{1,\ldots,s\}.
\end{equation}
They are admissible for all sufficiently large \(t\).  Set
\begin{equation}
	\label{eq:Jacobi-Laguerre-scaling-constants}
	c_t\coloneq\prod_{\ell=1}^{s}\kappa_\ell(t),
	\qquad
	d_t\coloneq
	\prod_{\ell=1}^{s}
	\kappa_\ell(t)^{a_{r+\ell}}\Gamma(\kappa_\ell(t)+1),
\end{equation}
and write
\[
\boldsymbol b^{\,\mathrm J}(t)
\coloneq
\bigl(b_1,\ldots,b_r,b_{r+1}(t),\ldots,b_q(t)\bigr).
\]
Unless stated otherwise, \(t\to+\infty\), while \(x>0\),
\(\boldsymbol a\), \(\boldsymbol b\), \(\omega_1,\ldots,\omega_s\), the power
exponents, and the multi-indices remain fixed. When Mellin transforms are
used, \(z\) ranges over a compact subset of a strip on which all transforms
in question are analytic. Only \(b_{r+\ell}(t)\), \(\kappa_\ell(t)\),
\(c_t\), and \(d_t\) depend on \(t\).
For the scaled Jacobi base weight
\[
W_0^{(t)}(x)
\coloneq
d_t\,w_0\!\left(\frac{x}{c_t};
	\boldsymbol a,\boldsymbol b^{\,\mathrm J}(t)\right)
\]
the Mellin transform is
\begin{equation}
	\label{eq:Jacobi-Laguerre-base-Mellin-limit}
	\mathcal M[W_0^{(t)}](z)
	=
	d_t c_t^z
	\frac{\Gamma(z\one_q+\boldsymbol a)}
	{\Gamma(z\one_r+\boldsymbol b)
	 \prod_{\ell=1}^{s}\Gamma(z+t\omega_\ell)}
	\xrightarrow[t\to+\infty]{}
	\frac{\Gamma(z\one_q+\boldsymbol a)}{\Gamma(z\one_r+\boldsymbol b)}.
\end{equation}
Indeed,
\[
\frac{
	\kappa_\ell(t)^{z+a_{r+\ell}}
	\Gamma(\kappa_\ell(t)+1)
}{
	\Gamma(\kappa_\ell(t)+z+a_{r+\ell}+1)
}
\xrightarrow[t\to+\infty]{}1
\]
locally uniformly in \(z\). Thus
\(\mathcal M[W_0^{(t)}](z)\xrightarrow[t\to+\infty]{}\mathcal M[w_0](z)\) locally uniformly on
compact subsets of every strip where both transforms are analytic, where \(w_0\) is the
Laguerre base weight of the first kind \eqref{eq:w0-explicit-convolution}. In particular,
every fixed Mellin moment converges.

For \(h\in\{1,\ldots,r\}\), define
\[
W_h^{(t)}(x)
\coloneq
d_t\,w_h\!\left(\frac{x}{c_t};
	\boldsymbol a,\boldsymbol b^{\,\mathrm J}(t)\right).
\]
Its Mellin transform is
\[
\mathcal M[W_h^{(t)}](z)
=
\frac{\mathcal M[W_0^{(t)}](z)}{z+b_h},
\]
so every fixed Mellin moment of \(W_h^{(t)}\) converges to the corresponding
moment of \(w_h\) in \eqref{eq:wh-def}.

After this common change of scale, the moments of each of the last \(s\)
Jacobi weights are of order \(t^{-1}\). Put
\[
	\varepsilon_\ell(t)\coloneq (t\omega_\ell)^{-1},
	\qquad
	Z_\ell^{(t)}(x)\coloneq
	d_t\,w_{r+\ell}\!\left(
		\frac{x}{c_t};\boldsymbol a,\boldsymbol b^{\,\mathrm J}(t)
	\right).
\]
The Mellin moments of \(t\omega_\ell Z_\ell^{(t)}\) tend to those of \(w_0\).
Successive linear combinations recover the Laguerre weights of the first kind
\(w_0,\thetaop w_0,\ldots,\thetaop^{s-1}w_0\).
The \(k\)-th new weight uses only the first \(k\) old weights;
this is the reason for the triangular coefficients below.

For distinct nodes \(z_1,\ldots,z_k\), the Newton divided difference
of a function \(f\) is
\[
[z_1,\ldots,z_k]f\coloneq
\sum_{\ell=1}^k\frac{f(z_\ell)}{\prod_{h\ne\ell}(z_\ell-z_h)},
\qquad [z_1]f=f(z_1).
\]
For an analytic function, the value at coinciding nodes is
the continuous extension; in particular,
\([z,\ldots,z]f=f^{(k-1)}(z)/(k-1)!\) with \(k\) repeated nodes.
The same coefficients occur in the following change of weights.

For \(1\le \ell\le k\le s\), set
	\begin{equation}
		\label{eq:Jacobi-Laguerre-Newton-matrix}
		R_{k,\ell}^{(t)}
		\coloneq
		\frac{(-1)^{k-1}}{
			\varepsilon_\ell(t)
			\displaystyle\prod_{\substack{1\le h\le k\\h\ne\ell}}
			\bigl(\varepsilon_\ell(t)-\varepsilon_h(t)\bigr)},
		\qquad
		R_{k,\ell}^{(t)}\coloneq0\quad(\ell>k).
	\end{equation}
	Set \(R_t\coloneq[R_{k,\ell}^{(t)}]_{k,\ell=1}^{s}\), and define
	\begin{equation}
		\label{eq:Jacobi-Laguerre-Newton-weights}
		\widehat V_k^{(t)}(x)
		\coloneq
		\sum_{\ell=1}^{k}R_{k,\ell}^{(t)}Z_\ell^{(t)}(x).
	\end{equation}

\begin{lemma}[Linear combinations recovering the derivative weights]
	\label{lem:Jacobi-Laguerre-Newton-confluence}
	For \(1\le k\le s\), the coefficients above satisfy, identically in \(z\),
	\begin{equation}
		\label{eq:Jacobi-Laguerre-Newton-identity}
		\sum_{\ell=1}^{k}
		\frac{R_{k,\ell}^{(t)}}{z+t\omega_\ell}
		=
		\frac{z^{k-1}}{
			\prod_{h=1}^{k}
			\bigl(1+\varepsilon_h(t)z\bigr)}.
	\end{equation}
	Consequently,
	\begin{equation}
		\label{eq:Jacobi-Laguerre-Newton-Mellin-limit}
		\mathcal M[\widehat V_k^{(t)}](z)
		=
		\mathcal M[W_0^{(t)}](z)
		\frac{z^{k-1}}{
			\prod_{h=1}^{k}(1+\varepsilon_h(t)z)}
		\xrightarrow[t\to+\infty]{}
		z^{k-1}\mathcal M[w_0](z)
		=
		\mathcal M[v_k](z)
	\end{equation}
	locally uniformly on every compact subset of a strip where both transforms
	are analytic.
\end{lemma}

\begin{proof}
	Equation~\eqref{eq:Jacobi-Laguerre-Newton-identity} is the partial-fraction
	decomposition of its right-hand side. Indeed, the coefficient of
	\((z+t\omega_\ell)^{-1}\) is precisely \(R_{k,\ell}^{(t)}\). Since
	\(
	\mathcal M[Z_\ell^{(t)}](z)
	=
	\mathcal M[W_0^{(t)}](z)/(z+t\omega_\ell)
	\), this proves \eqref{eq:Jacobi-Laguerre-Newton-Mellin-limit}.

	The matrix \(R_t\) is lower
	triangular, with
	\[
	R_{k,k}^{(t)}
	=
	\frac{(-1)^{k-1}}{
		\varepsilon_k(t)\prod_{h<k}
		(\varepsilon_k(t)-\varepsilon_h(t))}
	\ne0.
\]
	This change of weights preserves every finite polynomial space used on
	the step-line. Indeed, if
	\(m_1\ge\cdots\ge m_s\), \(\widehat{\boldsymbol v}=R_t\boldsymbol v\),
	and \(Q_k\in\mathbb P_{m_k-1}\), then the coefficient of \(v_\ell\) in
	\(\sum_kQ_k\widehat v_k\) is
	\(\sum_{k\ge\ell}R_{k,\ell}^{(t)}Q_k\in\mathbb P_{m_\ell-1}\).
	The inverse matrix has the same triangularity, so the two polynomial
	spaces coincide. On the step-line, \(m_1\ge\cdots\ge m_s\). Hence, at each
	polynomial degree, the components whose degree bounds reach that degree
	are \(1,\ldots,k\) for some \(k\).
	The leading \(k\times k\) block of \(R_t\) is invertible, so the
	orthogonality conditions before and after the change of weights are
	equivalent.
\end{proof}

Thus the transformed vector
\begin{equation}
	\label{eq:Jacobi-Laguerre-Newton-vector}
	\widehat{\mathbf U}^{(t)}
	\coloneq
	\bigl(W_1^{(t)},\ldots,W_r^{(t)},
	\widehat V_1^{(t)},\ldots,\widehat V_s^{(t)}\bigr)
\end{equation}
converges to the Laguerre vector of the first kind
\(\mathbf U^{\mathrm L}=(w_1,\ldots,w_r,v_1,\ldots,v_s)\) at every fixed
Mellin moment. Since the first \(k\) transformed derivative weights use only
the first \(k\) original derivative weights, the transformation preserves all the
polynomial spaces that occur in the admissible step-line range. This uses
the ordering \(\eta_1\ge\cdots\ge\eta_s\), which is part of the requirement
that \(\boldsymbol\eta\) lie on the step-line.

Put \(\boldsymbol\lambda=(\boldsymbol n,\boldsymbol\eta)\), set
\(\boldsymbol\alpha=\boldsymbol\beta\), and identify the Jacobi indices by
\(\boldsymbol n^{\,\mathrm J}=\boldsymbol m\) and
\(\boldsymbol m^{\,\mathrm J}=\boldsymbol\lambda\). Thus the reversal in the
Laguerre notation is explicit: the same pair is written
\((\boldsymbol m,\boldsymbol\lambda)\) on the Jacobi side and
\((\boldsymbol\lambda,\boldsymbol m)\) on the Laguerre side. Assume that
\(\boldsymbol\lambda\) is Laguerre admissible and that the parameter and
node-distinctness assumptions in
Theorems~\ref{thm:mixed-Jacobi-like-forms} and
\ref{thm:mixed-Laguerre-forms} hold. The two
mixed-type forms have different balance conditions, so their limits are stated
separately.

The condition \(c_{\boldsymbol\lambda}\ne0\) in
Proposition~\ref{prop:AT-normality-Laguerre-admissible} ensures a nonzero
limiting representative in each balance. In the second balance, a moment
at an additional node outside the prescribed set is nonzero because its
Mellin numerator has precisely the prescribed zeros. This provides a
nonvanishing normalization for the component limits. Suppose first that
\(|\boldsymbol m|=|\boldsymbol\lambda|+1\). Then
\begin{equation}
	\label{eq:Jacobi-Laguerre-A-form-limit}
	\frac{
		\mathcal A_{\boldsymbol m,(\boldsymbol n,\boldsymbol\eta)}^{\mathrm J}(x/c_t)
	}{
		\prod_{\ell=1}^{s}
		\Gamma(t\omega_\ell+\eta_\ell+1)
	}
	\xrightarrow[t\to+\infty]{}
	\mathcal A_{(\boldsymbol n,\boldsymbol\eta),\boldsymbol m}^{\mathrm L}(x)
\end{equation}
coefficientwise in its polynomial components. More precisely, for the
coefficient of \(x^{\beta_i+k}\), \(0\le k<m_i\), the part that depends on
the large parameters is
\[
	\prod_{\ell=1}^{s}
	\frac{\Gamma(\beta_i+k+t\omega_\ell+\eta_\ell+1)}
	{\Gamma(t\omega_\ell+\eta_\ell+1)\,
	 \kappa_\ell(t)^{\beta_i+k}}
	\xrightarrow[t\to+\infty]{}1.
\]

For the \(\mathcal B\)-form, suppose that
\(|\boldsymbol\lambda|\ge1\) and
\(|\boldsymbol\lambda|=|\boldsymbol m|+1\). Then
\begin{equation}
	\label{eq:Jacobi-Laguerre-B-form-limit}
	\mathcal M\!\left[
	-d_t\prod_{\ell=1}^{s}\kappa_\ell(t)^{\eta_\ell}
	\mathcal B_{\boldsymbol m,(\boldsymbol n,\boldsymbol\eta)}^{\mathrm J}
	(\,\cdot\,/c_t)
	\right](z)
	\xrightarrow[t\to+\infty]{}
	\mathcal M\!\left[
	\mathcal B_{(\boldsymbol n,\boldsymbol\eta),\boldsymbol m}^{\mathrm L}
	\right](z)
\end{equation}
locally uniformly on compact subsets of any common strip of analyticity;
hence every fixed Mellin moment converges. The transform on the left is
\begin{equation}
	\label{eq:Jacobi-Laguerre-B-direct-Mellin-limit}
	\mathcal M[W_0^{(t)}](z)
	\frac{
		(\boldsymbol\beta+(1-z)\one_p)_{\boldsymbol m}
	}{
		\prod_{h=1}^{r}(z+b_h)_{n_h}
	}
	\prod_{\ell=1}^{s}
	\frac{\kappa_\ell(t)^{\eta_\ell}}
		{(z+t\omega_\ell)_{\eta_\ell}},
\end{equation}
which converges locally uniformly to
\eqref{eq:Laguerre-B-Mellin}. If
\(a_\rho-a_\sigma\notin\mathbb Z\) for \(\rho\ne\sigma\), its residue expansion
is \eqref{eq:Laguerre-B-complete-hypergeometric}. In the resonant case, the
same limiting form is represented by \eqref{eq:Laguerre-B-complete-Meijer-G}.

\paragraph{Terminating coefficient formulas for the \(B\)-components of Laguerre of the first kind.}
The polynomial components of Laguerre of the first kind are obtained in
\cite[Section~4]{Manas2026HypergeometricMixed} by a finite triangular
reconstruction. An alternative proof is given below, together with a
direct evaluation of their coefficients by terminating hypergeometric
sums. These formulas satisfy the same polynomial identity
\eqref{eq:B-component-polynomial-system}, with the same normalization.
They concern the individual polynomial components, rather than the
complete weighted form. Retain \(q=r+s\), \(N=|\boldsymbol n|+|\boldsymbol\eta|\),
and \(|\boldsymbol m|=N-1\).
Define
\[
 \mathsf P(t)\coloneq\prod_{\rho=1}^q(t+a_\rho),\qquad
 \mathsf Q(t)\coloneq\prod_{h=1}^r(t+b_h),\qquad
 c_h\coloneq\frac{\prod_\rho(a_\rho-b_h)}
                  {\prod_{u\ne h}(b_u-b_h)}.
\]
Let \(\sum_{j=0}^s e_jt^j\) be the polynomial part of
\(\mathsf P(t)/\mathsf Q(t)\). Its coefficients are explicitly
\[
 e_j\coloneq
 \mathfrak C_{s-j}((-\one_q,\one_r),(-\boldsymbol a,-\boldsymbol b);0),
 \qquad 0\le j\le s.
\]
For \(\rho\in\{1,\ldots,q\}\) and \(k\in\mathbb N_0\), put
\begin{equation}
 \Omega_{\rho,k}\coloneq
 \frac{(\boldsymbol\beta+(a_\rho+1)\one_p)_{\boldsymbol m}
       (\boldsymbol b-a_\rho\one_r)_{(k+1)\one_r}}
      {k!\,(\boldsymbol b-a_\rho\one_r)_{\boldsymbol n}
       (\boldsymbol a^{*\rho}-a_\rho\one_{q-1})_{(k+1)\one_{q-1}}}.
 \label{eq:LI-B-explicit-residue-prefactor}
\end{equation}
The following ordered parameter strings have lengths \(d=p+r+q\) and
\(d-1\), respectively:
\[
\begin{aligned}
 \boldsymbol u_{\rho,k}&\coloneq
 (-k,\ \boldsymbol\beta+(a_\rho+1)\one_p+\boldsymbol m,\
       (a_\rho+1)\one_r-\boldsymbol b-\boldsymbol n,\
       (a_\rho-k)\one_{q-1}-\boldsymbol a^{*\rho}),\\
 \boldsymbol v_{\rho,k}&\coloneq
 (\boldsymbol\beta+(a_\rho+1)\one_p,\
       (a_\rho-k)\one_r-\boldsymbol b,\
       (a_\rho+1)\one_{q-1}-\boldsymbol a^{*\rho}).
\end{aligned}
\]
The notation \(\boldsymbol v_{\rho,k}+\boldsymbol e_{p+h}^{(d-1)}\)
increments only its entry corresponding to \(b_h\).

\begin{proposition}[Terminating hypergeometric coefficients for Laguerre of the first kind]
 \label{prop:LI-B-terminating-coefficients}
 Assume the hypotheses of Theorem~\ref{thm:mixed-Laguerre-forms}.
 In addition, assume that \(b_h\) are pairwise distinct, that
 \(a_\rho-a_\sigma\notin\mathbb Z\) for \(\rho\ne\sigma\),
 that \(a_\rho-b_h\notin\mathbb Z\), and that the displayed
 hypergeometric denominators are nonzero.
 The individual components
 \[
 B^{(h)}(x)=\sum_{k=0}^{n_h-1}b_{h,k}x^k,\qquad
 D^{(\ell)}(x)=\sum_{k=0}^{\eta_\ell-1}d_{\ell,k}x^k
 \]
 have the finite coefficient formulas
 \begin{equation}
 b_{h,k}=-c_h\sum_{\rho=1}^q
 \frac{\Omega_{\rho,k}}{b_h+k-a_\rho}
 \pFq{d}{d-1}{\boldsymbol u_{\rho,k}}
 {\boldsymbol v_{\rho,k}+\boldsymbol e_{p+h}^{(d-1)}}{1},
 \qquad 0\le k<n_h,
 \label{eq:LI-B-beta-terminating-coefficients}
 \end{equation}
 and
 \begin{equation}
 d_{\ell,k}=\sum_{\rho=1}^q\Omega_{\rho,k}
 \sum_{v=0}^{s-\ell}e_{\ell+v}(k-a_\rho)^v
 \pFq{d+v}{d+v-1}
 {\boldsymbol u_{\rho,k},\;(a_\rho-k+1)\one_v}
 {\boldsymbol v_{\rho,k},\;(a_\rho-k)\one_v}{1},
 \qquad 0\le k<\eta_\ell.
 \label{eq:LI-B-Euler-terminating-coefficients}
 \end{equation}
 Every series terminates at \(j=k\). These components satisfy
 \eqref{eq:B-component-polynomial-system}, with its stated normalization.
\end{proposition}

\begin{proof}
 Write
 \[
 f(t)\coloneq
 \frac{(\boldsymbol\beta+(1-t)\one_p)_{\boldsymbol m}}
      {(t\one_r+\boldsymbol b)_{\boldsymbol n}},\qquad
 R_k(t)\coloneq
 \frac{(t\one_q+\boldsymbol a)_{k\one_q}}
      {(t\one_r+\boldsymbol b)_{k\one_r}}.
 \]
 Let \(K\) be the largest entry of \((\boldsymbol n,\boldsymbol\eta)\).
 A telescoping rational identity gives
 \[
 \frac1{t-z}
 =\sum_{k=0}^{K-1}\frac{R_k(z)}{R_{k+1}(t)}
   \frac{\mathsf P(t+k)\mathsf Q(z+k)-\mathsf P(z+k)\mathsf Q(t+k)}
        {(t-z)\mathsf Q(z+k)\mathsf Q(t+k)}
   +\frac{R_K(z)}{R_K(t)(t-z)}.
 \]
 Indeed, the \(k\)-th summand is
 \(\{R_k(z)/R_k(t)-R_{k+1}(z)/R_{k+1}(t)\}/(t-z)\).
 Let \(\Gamma\) be a positively oriented union of small closed contours
 enclosing all nodes \(-a_\rho-j\), \(0\le j<K\), but none of the
 beta poles or the point \(z\). Multiply the identity by \(f(t)\) and
 integrate over \(\Gamma\). The integral on the left is zero.
 The rational function \(f/R_K\) has only the enclosed gamma nodes as
 poles and vanishes at infinity, so the final integral equals \(-f(z)\).

 The rational kernel in each remaining summand has a simple explicit
 decomposition. If
 \[
 E_\ell(v)\coloneq\sum_{j=\ell}^s e_jv^{j-\ell},
 \qquad 1\le\ell\le s,
 \]
 partial fractions of \(\mathsf P/\mathsf Q\) give
 \[
 \frac{\mathsf P(v)\mathsf Q(u)-\mathsf P(u)\mathsf Q(v)}
      {(v-u)\mathsf Q(u)\mathsf Q(v)}
 =\sum_{\ell=1}^s E_\ell(v)u^{\ell-1}
   -\sum_{h=1}^r\frac{c_h}{(v+b_h)(u+b_h)}.
 \]
 Thus the preceding contour identity expresses \(f(z)\) as
 \[
 \sum_{k=0}^{K-1}R_k(z)
 \left(\sum_{h=1}^r\frac{b_{h,k}}{z+k+b_h}
             +\sum_{\ell=1}^s d_{\ell,k}(z+k)^{\ell-1}\right),
 \]
 where
 \[
 b_{h,k}=-\frac{c_h}{2\pi\mathrm i}
       \int_\Gamma\frac{f(t)}{R_{k+1}(t)(t+k+b_h)}\,\mathrm dt,\qquad
 d_{\ell,k}=\frac1{2\pi\mathrm i}
       \int_\Gamma\frac{f(t)E_\ell(t+k)}{R_{k+1}(t)}\,\mathrm dt.
 \]
 Only the nodes with \(0\le j\le k\) contribute. The residue of
 \(f/R_{k+1}\) at \(-a_\rho-j\), divided by its residue at
 \(-a_\rho\), is
 \[
 \frac{(\boldsymbol u_{\rho,k})_j}
      {(\boldsymbol v_{\rho,k})_j\,j!}.
 \]
 The initial residue is \(\Omega_{\rho,k}\).
 Multiplication by \((t+k+b_h)^{-1}\) increments precisely the
 denominator parameter \(a_\rho-k-b_h\), with prefactor
 \((b_h+k-a_\rho)^{-1}\). This proves
 \eqref{eq:LI-B-beta-terminating-coefficients}.
 For the Euler coefficient, use
 \[
 (k-a_\rho-j)^v
 =(k-a_\rho)^v
 \left(\frac{(a_\rho-k+1)_j}{(a_\rho-k)_j}\right)^v.
 \]
 Expansion of \(E_\ell\) now gives
 \eqref{eq:LI-B-Euler-terminating-coefficients}.

 It remains to justify the component degree bounds rather than merely a
 common bound \(K-1\). If \(k\ge n_h\), near diagonality cancels every
 beta denominator in the integrand for \(b_{h,k}\), and its degree at
 infinity is \(|\boldsymbol\eta|-s(k+1)-2\le-2\).
 Its contour integral is therefore zero. If \(k\ge\eta_\ell\),
 the same cancellation holds for \(d_{\ell,k}\).
 The ordered step-line condition gives
 \(|\boldsymbol\eta|\le sk+\ell-1\), so this integrand also has degree
 at most \(-2\) at infinity. Hence \(d_{\ell,k}=0\) in that range.
 Finally, multiplication by the common Mellin factor \(G_0(z)\)
 identifies \(R_k(z)/(z+k+b_h)\) with the transform of \(x^kw_h\),
 and \(R_k(z)(z+k)^{\ell-1}\) with that of \(x^kv_\ell\).
 This proves the claimed polynomial identity and normalization.
\end{proof}

\paragraph{Moments and recurrence coefficients on the step-line.}
For a rank-one system with row vector
\(\mathbf U=(U_1,\ldots,U_q)\) and power vector
\(\mathbf V=(V_1,\ldots,V_p)\) on an interval \(I\), enumerate the two scalar bases by
\[
	\mathscr U_N(x)
	\coloneq x^{\rho_q(N)}U_{\ell_q(N)+1}(x),
	\qquad
	\mathscr V_M(x)
	\coloneq x^{\rho_p(M)}V_{\ell_p(M)+1}(x),
\]
and define the scalar moment matrix
\[
	\mathscr M_{N,M}
	\coloneq\int_I\mathscr U_N(x)\mathscr V_M(x)\,\mathrm dx,
	\qquad N,M\ge0.
\]
The step-line index pairs are
\[
	(\boldsymbol\sigma_p(N),\boldsymbol\sigma_q(N+1))
	\quad\text{on the \(B\)-side},\qquad
	(\boldsymbol\sigma_p(N+1),\boldsymbol\sigma_q(N))
	\quad\text{on the \(A\)-side}.
\]
The representatives fixed earlier by the contour and Mellin formulas will
be denoted here by an additional subscript \({\rm exp}\). Thus
\begin{align*}
	\mathcal B_{N,{\rm exp}}^{\mathrm J}
	\coloneq
	\mathcal B^{\mathrm J}_{\boldsymbol\sigma_p(N),
		\boldsymbol\sigma_q(N+1)},
	&
	\mathcal A_{N,{\rm exp}}^{\mathrm J}
	\coloneq
	\mathcal A^{\mathrm J}_{\boldsymbol\sigma_p(N+1),
		\boldsymbol\sigma_q(N)},\\
	\mathcal B_{N,{\rm exp}}^{\mathrm L}
	\coloneq
	\mathcal B^{\mathrm L}_{\boldsymbol\sigma_q(N+1),
		\boldsymbol\sigma_p(N)},
	&
	\mathcal A_{N,{\rm exp}}^{\mathrm L}
	\coloneq
	\mathcal A^{\mathrm L}_{\boldsymbol\sigma_q(N),
		\boldsymbol\sigma_p(N+1)}.
\end{align*}
The second line displays once more the reversed Laguerre index order.

A leading minor of order \(n\) is the determinant of the submatrix
with indices \(0,\ldots,n-1\) on both sides. When all required leading
minors are nonzero, the Gauss--Borel factorization is
\[
\mathscr M=S^{-1}H\widetilde S^{-\top},
\]
where \(S\) and \(\widetilde S\) are lower triangular with unit
diagonal and \(H\) is nonsingular and diagonal. For the infinite moment
matrix, this notation is understood through its compatible finite
leading submatrices.

These explicit representatives need not have the Gauss--Borel
normalization. To separate the two conventions, put
\[
	j_N\coloneq\ell_q(N)+1
\]
and write
\[
	\mathcal B_{N,{\rm exp}}^{\mathrm S}(x)
	=\sum_{j=1}^q
	B_{N,{\rm exp}}^{\mathrm S,(j)}(x)U_j^{\mathrm S}(x),
	\qquad \mathrm S\in\{\mathrm J,\mathrm L\}.
\]
The integration intervals are \(I_{\mathrm J}\coloneq(0,1)\) and
\(I_{\mathrm L}\coloneq(0,\infty)\). Define
\begin{equation}
	h_N^{\mathrm S}
	\coloneq
	[x^{\rho_q(N)}]B_{N,{\rm exp}}^{\mathrm S,(j_N)}(x),
	\label{eq:step-line-h-normalizer}\qquad
	g_N^{\mathrm S}
	\coloneq
	\int_{I_{\mathrm S}}
	\mathcal A_{N,{\rm exp}}^{\mathrm S}(x)
	\frac{\mathcal B_{N,{\rm exp}}^{\mathrm S}(x)}
	{h_N^{\mathrm S}}\,\mathrm dx.
\end{equation}
Whenever the relevant leading minors are nonzero, both constants are
nonzero. The Gauss--Borel normalized forms are, by definition,
\begin{equation}
	\mathsf B_N^{\mathrm S}
	\coloneq
	\frac{\mathcal B_{N,{\rm exp}}^{\mathrm S}}{h_N^{\mathrm S}},
	\qquad
	\mathsf A_N^{\mathrm S}
	\coloneq
	\frac{\mathcal A_{N,{\rm exp}}^{\mathrm S}}{g_N^{\mathrm S}},
	\qquad \mathrm S\in\{\mathrm J,\mathrm L\}.
	\label{eq:explicit-to-step-line-normalization}
\end{equation}
	Thus the coefficient of \(x^{\rho_q(N)}\) in component \(j_N\) of
	\(\mathsf B_N^{\mathrm S}\) is one. The biorthogonality relation
	\cite[Proposition~6.1]{Manas2026HypergeometricMixed} reads
	\[
		\int_{I_{\mathrm S}}\mathsf A_N^{\mathrm S}(x)\mathsf B_M^{\mathrm S}(x)\,\mathrm dx
		=\delta_{N,M}.
\]
The symbols \(\widehat{\mathsf A}_N^{(t)}\) and
\(\widehat{\mathsf B}_N^{(t)}\) denote the forms selected by these same two
conditions for the transformed vector
\(\widehat{\mathbf U}^{(t)}\), with integration over \((0,c_t)\).
These conditions fix their normalization.

Let
\[
	\mathsf{\mathbf B}=(\mathsf B_0,\mathsf B_1,\ldots)^\top,
	\qquad
	\mathsf{\mathbf A}=(\mathsf A_0,\mathsf A_1,\ldots)^\top.
\]
By \cite[Theorem~6.2]{Manas2026HypergeometricMixed}, the recurrence
matrix in this normalization satisfies
\[
	x\mathsf{\mathbf B}=T\mathsf{\mathbf B},
	\qquad
	x\mathsf{\mathbf A}=T^\top\mathsf{\mathbf A};
\]
It is \((p,q)\)-banded, meaning \(T_{N,M}=0\) unless
\(N-p\le M\le N+q\), and \(T_{N,N+q}=1\).
Accordingly, \(T_{\mathrm J}(t)\), \(\widehat T_t\), and \(T_{\mathrm L}\)
denote the matrices of
\(\mathsf B_N^{\mathrm J}\),
\(\widehat{\mathsf B}_N^{(t)}\), and
\(\mathsf B_N^{\mathrm L}\), respectively.

Finally, let \(\Lambda\) be the unilateral shift,
\(\Lambda_{N,M}=\delta_{N+1,M}\) for \(N,M\ge0\). The finite portions of the
row and column Christoffel chains used below are the Gauss--Borel
factorizations of
\[
	\Lambda^a\mathscr M,\quad 0\le a\le q,
	\qquad\text{and}\qquad
	\mathscr M(\Lambda^b)^\top,\quad 0\le b\le p.
\]
The factorization in \cite[Section~7]{Manas2026HypergeometricMixed}
associates with consecutive steps the normalized upper bidiagonal factors
\(U_1,\ldots,U_q\) and lower bidiagonal factors
\(L_1,\ldots,L_p\), in the order
\[
	T=L_1\cdots L_p\,U_q\cdots U_1.
\]
Here every \(L_i\) has unit main diagonal and every \(U_j\) has unit first
superdiagonal; these conditions remove the otherwise possible diagonal
redistribution between consecutive factors.
On every fixed truncation, the entries of these factors, as well as those
of \(T\), are rational functions of the finitely many leading minors in the
two displayed chains.

\begin{proposition}[Stability at each fixed index]
	\label{prop:Jacobi-Laguerre-finite-GB-stability}
	Let \(\widehat{\mathscr M}^{(t)}\) be the moment matrix of
	\eqref{eq:Jacobi-Laguerre-Newton-vector}, and let
	\(\mathscr M^{\mathrm L}\) be the moment matrix for Laguerre of the first kind. Fix \(K\).
	Assume that every leading minor of the following shifted matrices used up
	to level \(K\) in either Christoffel chain is nonzero:
	\[
	\mathscr M^{\mathrm L},
	\qquad
	\Lambda^a\mathscr M^{\mathrm L}\quad(0\le a\le q),
	\qquad
	\mathscr M^{\mathrm L}(\Lambda^b)^{\top}\quad(0\le b\le p).
\]
	Then the corresponding minors for
	\(\widehat{\mathscr M}^{(t)}\) are nonzero for all sufficiently large \(t\).
	At every index not exceeding
	\(K\), the Gauss--Borel forms
	\(\widehat{\mathsf A}_N^{(t)},\widehat{\mathsf B}_N^{(t)}\), the recurrence
	coefficients, and the entries of the normalized lower and upper bidiagonal
	factors converge to their counterparts for Laguerre of the first kind
	\(\mathsf A_N^{\mathrm L},\mathsf B_N^{\mathrm L}\).
\end{proposition}

\begin{proof}
	Equation \eqref{eq:Jacobi-Laguerre-Newton-Mellin-limit} and the analogous
	limits for \(W_h^{(t)}\) give
	\(\widehat{\mathscr M}^{(t)}\xrightarrow[t\to+\infty]{}\mathscr M^{\mathrm L}\) entry by entry;
	the same holds after every fixed left or right shift. Each determinant in
	the statement is a polynomial in finitely many moments, so its nonzero limit
	remains nonzero for large \(t\). Gaussian elimination on each fixed
	truncation, Cramer's formulas for the normalized forms and recurrence
	coefficients, and the formulas for the bidiagonal factors express all the
	quantities in the statement as ratios of determinants involving finitely
	many of these moments. Their denominators have nonzero limits, so the ratios
	converge.
\end{proof}

The recurrence matrix for the transformed weights can now be compared with
the original Jacobi recurrence matrix. Put
\[
	\begin{gathered}
	\mathsf R_t\coloneq\operatorname{diag}(I_r,R_t),
	\qquad
	\gamma_{t,N}\coloneq
	(\mathsf R_t)_{j_N,j_N}c_t^{\rho_q(N)},
	\qquad
	G_t\coloneq\operatorname{diag}
	(\gamma_{t,0},\gamma_{t,1},\ldots).
	\end{gathered}
\]
Because \(\mathsf R_t\) is lower triangular, applying it block by block
preserves each leading step-line truncation. Hence the
Gauss--Borel normalized step-line forms satisfy the following identities.
To see the normalizing factor, let \(j=j_N\) and
\(\rho=\rho_q(N)\). Replacing \(x\) by \(x/c_t\) multiplies the chosen
leading coefficient by \(c_t^{-\rho}\), while rewriting the old weights in
the new triangular basis multiplies it by
\((\mathsf R_t)_{j,j}^{-1}\). Thus multiplication by
\(\gamma_{t,N}=(\mathsf R_t)_{j,j}c_t^\rho\) restores that leading
coefficient to one. The common factor \(d_t\) in the transformed weights, together
with the change of variables, fixes the reciprocal normalization on the
dual side. Consequently,
\begin{equation}
	\label{eq:Jacobi-Laguerre-gauged-forms}
	\widehat{\mathsf B}_N^{(t)}(x)
	=d_t\gamma_{t,N}\mathsf B_N^{\mathrm J}(x/c_t),
	\qquad
	\widehat{\mathsf A}_N^{(t)}(x)
	=\frac{\mathsf A_N^{\mathrm J}(x/c_t)}
		{c_td_t\gamma_{t,N}}.
\end{equation}
Therefore their recurrence matrix is
\begin{equation}
	\label{eq:Jacobi-Laguerre-recurrence-gauge}
	\widehat T_t
	=G_t\bigl(c_tT_{\mathrm J}(t)\bigr)G_t^{-1}
	\xrightarrow[t\to+\infty]{\textnormal{entrywise at fixed indices}}
	T_{\mathrm L}.
\end{equation}
Since
\(\gamma_{t,N+q}=c_t\gamma_{t,N}\), the \(q\)-th superdiagonal remains equal
to one.
Proposition~\ref{prop:Jacobi-Laguerre-finite-GB-stability} therefore applies,
at each fixed index, to the normalized bidiagonal factors of \(\widehat T_t\).
The matrix \(G_t\) simply records the change of normalization caused by the
rescaling of the variable and by the triangular change of the last \(s\)
weights.

\section{From Jacobi to Laguerre of the second kind by rescaling near one}
\label{sec:right-LII-confluence}

The preceding limit near \(x=0\) gives a Laguerre system of the first kind. The change
of variable \(y=\tau(1-x)\) now expands a neighbourhood of \(x=1\) to a
half-line. In this scaling, the product of the original variables becomes
a sum of their rescaled distances \(\tau(1-x_h)\) from \(1\). Consequently,
the limiting Laguerre rows of the second kind are additive convolutions.

For \(\rho,d>0\), set
\begin{equation}
	\label{eq:LII-gamma-kernel}
	g_{\rho,d}(y)
	\coloneq
	\frac{y^{d-1}\e^{-\rho y}}{\Gamma(d)}
	\boldsymbol 1_{(0,\infty)}(y),
	\qquad
	\Lap[g_{\rho,d}](z)=(z+\rho)^{-d},
\end{equation}
where
\[
	\Lap[f](z)\coloneq\int_0^\infty \e^{-zy}f(y)\,\mathrm dy.
\]
The symbol \(*_+\) has the additive-convolution meaning fixed in
Section~\ref{subsec:hypergeometric-notation}. Whenever the integrals
converge absolutely, \(\Lap[f*_+g]=\Lap[f]\Lap[g]\).
For complex \(z\), each power \((z+\rho)^a\) is taken using the
principal logarithm, with cut \((-\infty,-\rho]\); it is positive
for real \(z> -\rho\) and real \(a\).

Fix \(d_h,\rho_h>0\), \(1\le h\le q\), and real numbers
\(\xi_1,\ldots,\xi_p\) such that
\begin{equation}
	\label{eq:LII-integrability}
	\xi_i+\rho_h>0,
	\qquad 1\le i\le p,\quad 1\le h\le q.
\end{equation}
For \(\tau\to+\infty\), take the Jacobi parameters
\begin{equation}
	\label{eq:Jacobi-LII-parameter-scaling}
	a_h(\tau)=\tau\rho_h-1,
	\qquad
	b_h(\tau)=\tau\rho_h+d_h-1,
	\qquad
	\alpha_i(\tau)=\tau\xi_i.
\end{equation}
Put \(D\coloneq\sum_{h=1}^q d_h\) and define
\begin{align}
	F_0
	&\coloneq
	g_{\rho_1,d_1}*_+\cdots *_+g_{\rho_q,d_q},
	\label{eq:LII-base-row}
	\\
	F_j
	&\coloneq
	g_{\rho_1,d_1}*_+\cdots *_+g_{\rho_j,d_j+1}*_+
	\cdots *_+g_{\rho_q,d_q},
	\qquad 1\le j\le q.
	\label{eq:LII-shifted-rows}
\end{align}
The limiting matrix of measures is
\begin{equation}
	\label{eq:LII-matrix-measure}
	\mathrm d\LIImat(y)
	\coloneq
	[F_j(y)\e^{-\xi_i y}]_{j,i}\,\mathrm dy,
	\qquad y>0.
\end{equation}

\begin{theorem}[Limit of the matrix measure from Jacobi to Laguerre of the second kind]
	\label{thm:Jacobi-LII-weight-limit}
	Let \(w_{0,\tau}^{\mathrm J},w_{1,\tau}^{\mathrm J},\ldots,w_{q,\tau}^{\mathrm J}\)
	be the Jacobi weights with parameters
	\eqref{eq:Jacobi-LII-parameter-scaling}. Extended by zero outside
	\(0<y<\tau\), their reflected densities satisfy
	\begin{align}
		\tau^{D-1}w_{0,\tau}^{\mathrm J}\left(1-\frac y\tau\right)
		&\xrightarrow[\tau\to+\infty]{} F_0(y),
		\label{eq:Jacobi-LII-base-limit}
		\\
		\tau^Dw_{j,\tau}^{\mathrm J}\left(1-\frac y\tau\right)
		&\xrightarrow[\tau\to+\infty]{} F_j(y),
		\qquad 1\le j\le q,
		\label{eq:Jacobi-LII-row-limit}
	\end{align}
	locally uniformly on \((0,\infty)\). The corresponding finite measures
	converge weakly, and every fixed polynomial moment converges.

	For the complete mixed-type matrix,
	\begin{equation}
		\label{eq:Jacobi-LII-matrix-limit}
		\tau^D\left(1-\frac y\tau\right)^{\alpha_i(\tau)}
		w_{j,\tau}^{\mathrm J}\left(1-\frac y\tau\right)
		\xrightarrow[\tau\to+\infty]{}
		\e^{-\xi_i y}F_j(y)
	\end{equation}
	locally uniformly. Moreover, for every \(n\in\mathbb N_0\),
	\begin{equation}
		\label{eq:Jacobi-LII-mixed-moment-limit}
		\int_0^\tau y^n \tau^D
		\left(1-\frac y\tau\right)^{\alpha_i(\tau)}
		w_{j,\tau}^{\mathrm J}\left(1-\frac y\tau\right)\,\mathrm dy
		\xrightarrow[\tau\to+\infty]{}
		\int_0^\infty y^n\e^{-\xi_i y}F_j(y)\,\mathrm dy.
	\end{equation}
	Equivalently, after the change of variables \(x=1-y/\tau\), the original
	matrix measure is multiplied by \(\tau^{D+1}\); one power of \(\tau\) is the
	Jacobian.
\end{theorem}

\begin{proof}
	For one beta factor, the identity
	\begin{equation}
		\label{eq:single-beta-right-limit}
		\tau^{d-1}\mathcal B_{\tau\rho-1,\tau\rho+d-1}
		\left(1-\frac y\tau\right)
		=
		\frac{y^{d-1}}{\Gamma(d)}
		\left(1-\frac y\tau\right)^{\tau\rho-1}
	\end{equation}
	gives local uniform convergence to \(g_{\rho,d}\). Its reflected moments
	are also explicit:
	\begin{equation}
		\label{eq:single-beta-right-moments}
		\tau^{d+n}\int_0^1(1-x)^n
		\mathcal B_{\tau\rho-1,\tau\rho+d-1}(x)\,\mathrm dx
		=
		\tau^{d+n}(d)_n\frac{\Gamma(\tau\rho)}
		{\Gamma(\tau\rho+d+n)}
		\xrightarrow[\tau\to+\infty]{}
		\frac{(d)_n}{\rho^{d+n}}.
	\end{equation}

	The Mellin convolution defining \(w_{0,\tau}^{\mathrm J}\) is the image of
	the product of the \(q\) beta measures under multiplication. Writing
	\(x_h=1-y_h/\tau\), one has
	\[
		\tau\left(1-\prod_{h=1}^q\left(1-\frac{y_h}{\tau}\right)\right)
		\xrightarrow[\tau\to+\infty]{} y_1+\cdots+y_q
\]
	uniformly on compact sets. Hence the multiplicative convolution converges
	to the additive convolution \eqref{eq:LII-base-row}. Replacing
	\(b_j(\tau)\) by \(b_j(\tau)+1\) raises the \(j\)-th gap from \(d_j\) to
	\(d_j+1\).

	To prove local convergence of the densities, some auxiliary operations
	and functions are defined next. Throughout this proof, the subscript
	\(\tau\) indicates dependence on the positive scaling parameter in
	\eqref{eq:Jacobi-LII-parameter-scaling}. Under the substitutions
	\(x_1=1-u/\tau\) and \(x_2=1-v/\tau\), multiplication of \(x_1\) and
	\(x_2\) becomes the operation
	\[
		u\oplus_\tau v
		\coloneq
		u+v-\frac{uv}{\tau}
		=
		\tau(1-x_1x_2),
		\qquad u,v\in(0,\tau).
\]
	It is associative because multiplication is associative. For nonnegative
	integrable functions \(f,g\) on \((0,\tau)\), define the corresponding
	convolution by
	\begin{equation}
		\label{eq:right-endpoint-binary-density}
		(f\star_\tau g)(y)
		\coloneq
		\int_0^y
		f(u)
		g\left(\frac{y-u}{1-u/\tau}\right)
		\frac{\mathrm du}{1-u/\tau},
		\qquad 0<y<\tau,
	\end{equation}
	and set it equal to zero outside \((0,\tau)\). Indeed, solving
	\(y=u\oplus_\tau v\) gives \(v=(y-u)/(1-u/\tau)\) and the Jacobian
	\(\partial v/\partial y=(1-u/\tau)^{-1}\). Thus the displayed integral is
	the density obtained from \(f(u)g(v)\,\mathrm du\,\mathrm dv\) by this
	change of variables. Its integration interval \(0<u<y\) is independent
	of \(\tau\).

	Set \(e_h\coloneq d_h\) for the base row and
	\(e_h\coloneq d_h+\delta_{h,j}\) for the \(j\)-th shifted row. Define
	the reflected density of the \(h\)-th scaled beta factor by
	\[
		f_{h,\tau}(u)\coloneq
		\frac{u^{e_h-1}}{\Gamma(e_h)}
		\left(1-\frac u\tau\right)^{\tau\rho_h-1},
		\qquad 0<u<\tau,
\]
	and extend it by zero outside \((0,\tau)\). The successive convolutions
	are defined recursively by
	\[
		H_{1,\tau}\coloneq f_{1,\tau},
		\qquad H_{r+1,\tau}\coloneq H_{r,\tau}\star_\tau f_{r+1,\tau},
		\qquad 1\le r<q.
\]
	For comparison, define the corresponding additive convolutions and
	the sums of their shape parameters by
	\[
		H_r
		\coloneq
		g_{\rho_1,e_1}*_+\cdots *_+g_{\rho_r,e_r},
		\qquad
		E_r\coloneq e_1+\cdots+e_r,
		\qquad 1\le r\le q.
\]
	For \(0<y<\tau\), put
	\[
		\phi_{r,\tau}(y)\coloneq y^{1-E_r}H_{r,\tau}(y),
		\qquad
		\phi_r(y)\coloneq y^{1-E_r}H_r(y),
\]
	with the second definition valid for every \(y>0\). Induction proves
	that these functions extend continuously to \(y=0\) and that, for
	every \(K>0\) and \(\tau>K\),
	\begin{equation}
		\label{eq:right-endpoint-normalized-density-limit}
		\sup_{0\le y\le K}
		|\phi_{r,\tau}(y)-\phi_r(y)|\xrightarrow[\tau\to+\infty]{}0.
	\end{equation}
	For a fixed \(L>0\), Taylor's formula gives
	\[
		(\tau\rho_h-1)\log(1-v/\tau)
		=-\rho_hv+\mathrm O_L(\tau^{-1}),
		\qquad 0\le v\le L\qquad (\tau\to+\infty),
\]
	where the remainder is bounded by \(C_L/\tau\), uniformly for
	\(0\le v\le L\). The constant may depend on \(L\) and the fixed
	parameters, but not on \(\tau\) or \(v\). Consequently,
	\begin{equation}
		\label{eq:right-endpoint-single-factor-uniform}
		b_{h,\tau}(v)
		\coloneq
		\left(1-\frac v\tau\right)^{\tau\rho_h-1}
		\xrightarrow[\tau\to+\infty]{} \e^{-\rho_hv}
	\end{equation}
	uniformly on \([0,L]\). This proves the claim for \(r=1\), with
	\(\phi_{1,\tau}=b_{1,\tau}/\Gamma(e_1)\).

	Assume the claim at level \(r\), and put \(E=E_r\),
	\(e=e_{r+1}\). Formula~\eqref{eq:right-endpoint-binary-density}, followed
	by \(u=ys\), gives
	\[
		H_{r+1,\tau}(y)=y^{E+e-1}\phi_{r+1,\tau}(y),
\]
	where
	\[
		\phi_{r+1,\tau}(y)
		={}\frac1{\Gamma(e)}\int_0^1
		s^{E-1}(1-s)^{e-1}\phi_{r,\tau}(ys)
		\left(1-\frac{ys}{\tau}\right)^{-e}
		b_{r+1,\tau}\left(
		\frac{y(1-s)}{1-ys/\tau}\right)\,\mathrm ds.
\]
	The additive convolution has the analogous representation
	\[
		\phi_{r+1}(y)
		=
		\frac1{\Gamma(e)}\int_0^1
		s^{E-1}(1-s)^{e-1}\phi_r(ys)
		\e^{-\rho_{r+1}y(1-s)}\,\mathrm ds.
\]
	Fix \(K>0\) and take \(\tau>2K\). Uniformly for
	\((y,s)\in[0,K]\times[0,1]\),
	\[
		0\le\frac{y(1-s)}{1-ys/\tau}\le2K,
		\qquad
		\frac{y(1-s)}{1-ys/\tau}-y(1-s)=\mathrm O_K(\tau^{-1})\qquad (\tau\to+\infty).
\]
	The induction hypothesis and
	\eqref{eq:right-endpoint-single-factor-uniform} therefore give uniform
	convergence of the factors multiplying
	\(s^{E-1}(1-s)^{e-1}\). They are uniformly bounded, and
	\[
		\int_0^1s^{E-1}(1-s)^{e-1}\,\mathrm ds
		=B(E,e)<\infty
\]
	because \(E,e>0\). Hence
	\eqref{eq:right-endpoint-normalized-density-limit} follows at level
	\(r+1\). The induction proves the stronger weighted convergence
	\[
		\sup_{0<y\le K}y^{1-E_q}|H_{q,\tau}(y)-H_q(y)|
		\xrightarrow[\tau\to+\infty]{}0.
\]
	In particular, the convergence is uniform on every compact subset of
	\((0,\infty)\). Taking \(e_h=d_h\) gives
	\eqref{eq:Jacobi-LII-base-limit}, while taking
	\(e_h=d_h+\delta_{h,j}\) gives
	\eqref{eq:Jacobi-LII-row-limit}.

	For moment convergence, use
	\[
		0\le \tau\left(1-\prod_{h=1}^q x_h\right)
		\le\sum_{h=1}^q \tau(1-x_h),
		\qquad 0<x_h<1,
\]
	together with \eqref{eq:single-beta-right-moments}. The single-factor
	measures converge weakly, hence so do their product measures and their
	push-forwards, since \(\oplus_\tau\) tends to addition uniformly on compact
	sets. Applying the same bound to one higher power gives uniform
	integrability of each fixed power. This proves convergence of all fixed
	polynomial moments. Finally,
	\[
		\left(1-\frac y\tau\right)^{\tau\xi_i}\xrightarrow[\tau\to+\infty]{}\e^{-\xi_i y}.
\]
	Multiplication by \(x^{\tau\xi_i}\) distributes over the Mellin convolution
	and replaces the rate \(\rho_h\) in every beta factor by
	\(\rho_h+\xi_i\). Condition~\eqref{eq:LII-integrability} makes all these
	rates positive. The preceding density and moment arguments therefore apply
	with \(\rho_h\) replaced by \(\rho_h+\xi_i\), which proves
	\eqref{eq:Jacobi-LII-matrix-limit} and
	\eqref{eq:Jacobi-LII-mixed-moment-limit}.
\end{proof}

\begin{proposition}[Laplace transforms, differential relations, and moments]
	\label{prop:LII-Laplace-transforms}
	The limiting rows satisfy
	\begin{equation}
		\Lap[F_0](z)
		=\prod_{h=1}^q(z+\rho_h)^{-d_h},
		\label{eq:LII-F0-Laplace}\qquad
		\Lap[F_j](z)
		=\frac{1}{z+\rho_j}
		\prod_{h=1}^q(z+\rho_h)^{-d_h},
		\qquad 1\le j\le q.
	\end{equation}
	Consequently,
	\begin{equation}
		\label{eq:LII-first-order-system}
		F_j'(y)+\rho_jF_j(y)=F_0(y),
		\qquad F_j(0)=0,
	\end{equation}
	and
	\begin{equation}
		\label{eq:LII-explicit-moments}
		\int_0^\infty y^n\e^{-\xi_i y}F_j(y)\,\mathrm dy
		=
		(-1)^n\left.
		\frac{\mathrm d^n}{\mathrm dz^n}
		\left[
			\frac{1}{z+\rho_j}
			\prod_{h=1}^q(z+\rho_h)^{-d_h}
		\right]\right|_{z=\xi_i}.
	\end{equation}
\end{proposition}

\begin{proof}
	The convolution theorem and \eqref{eq:LII-gamma-kernel} give the two
	Laplace transforms. Multiplication of \eqref{eq:LII-F0-Laplace} by
	\(z+\rho_j\) gives \eqref{eq:LII-first-order-system}; the boundary value
	follows from \(F_j(y)=\mathrm O(y^D)\) at the origin. Differentiation under
	the integral gives \eqref{eq:LII-explicit-moments}.
\end{proof}

When the rates \(\rho_1,\ldots,\rho_q\) are pairwise distinct, partial
fractions give
\[
	\operatorname{span}
	\left\{
		\frac{1}{z+\rho_j}\prod_{h=1}^q(z+\rho_h)^{-d_h}
	\right\}_{j=1}^q
	=
	\operatorname{span}
	\left\{
		\frac{z^{j-1}}{\prod_{h=1}^{j}(z+\rho_h)}
		\prod_{h=1}^q(z+\rho_h)^{-d_h}
	\right\}_{j=1}^q.
\]
Thus the limiting row span is an instance of the additive-derivative-type
class characterized in \cite[Proposition~2.4]{Wolfs2025DerivativeType}.
That unilateral identification does not include the independent exponential
column vector, the rectangular mixed-type problem, or the confluence from the
Jacobi matrix treated here.

The limiting rows are also explicitly hypergeometric. For \(R\ge1\), define
the confluent Lauricella function by
\[
	\Phi_2^{(R)}(a_1,\ldots,a_R;c;x_1,\ldots,x_R)
	\coloneq
	\sum_{k_1,\ldots,k_R\ge0}
	\frac{\prod_{h=1}^R(a_h)_{k_h}}{(c)_{|\boldsymbol k|}}
	\prod_{h=1}^R\frac{x_h^{k_h}}{k_h!}.
\]
For \(q\ge2\), let
\[
 \Delta_{q-1}\coloneq
 \{(u_1,\ldots,u_{q-1})\in(0,1)^{q-1}:u_1+\cdots+u_{q-1}<1\},
 \qquad u_q\coloneq1-\sum_{h=1}^{q-1}u_h.
\]
For positive \(c_h\), put \(C\coloneq\sum_hc_h\).
The substitution \(v_h=yu_h\) in the additive convolution integral gives
\begin{equation}
 \bigl(g_{\rho_1,c_1}*_+\cdots*_+g_{\rho_q,c_q}\bigr)(y)
 =\frac{y^{C-1}}{\prod_{h=1}^q\Gamma(c_h)}
 \int_{\Delta_{q-1}}\e^{-y\sum_{h=1}^q\rho_hu_h}
 \prod_{h=1}^qu_h^{c_h-1}\,\mathrm du_1\cdots\mathrm du_{q-1}.
 \label{eq:gamma-convolution-simplex}
\end{equation}
The gamma kernels are those of \eqref{eq:LII-gamma-kernel}. The elementary beta
integral over \(\Delta_{q-1}\) is
\[
 \int_{\Delta_{q-1}}\prod_{h=1}^qu_h^{c_h+k_h-1}\,
 \mathrm du_1\cdots\mathrm du_{q-1}
 =\frac{\prod_{h=1}^q\Gamma(c_h+k_h)}{\Gamma(C+|\boldsymbol k|)}.
\]
Expanding the exponential in \eqref{eq:gamma-convolution-simplex} and using
this identity gives the Lauricella series. For \(q=1\), the convolution
consists of a single gamma kernel and the same series follows directly.
Taking \(c_h=d_h\), or \(c_h=d_h+\delta_{h,j}\), respectively, yields
\begin{align}
	F_0(y)
	&=
	\frac{y^{D-1}}{\Gamma(D)}
	\Phi_2^{(q)}(d_1,\ldots,d_q;D;
	-\rho_1y,\ldots,-\rho_qy),
	\label{eq:LII-F0-Phi2}
	\\
	F_j(y)
	&=
	\frac{y^D}{\Gamma(D+1)}
	\Phi_2^{(q)}(d_1,\ldots,d_j+1,\ldots,d_q;D+1;
	-\rho_1y,\ldots,-\rho_qy).
	\label{eq:LII-Fj-Phi2}
\end{align}

\subsection{The limiting mixed-type forms}

Let \(\boldsymbol n=(n_1,\ldots,n_p)\) be the column multi-index and
\(\boldsymbol m=(m_1,\ldots,m_q)\) the row multi-index. The limiting forms
are
\[
	\mathcal A_{\boldsymbol n,\boldsymbol m}^{\mathrm{LII}}(y)
	=
	\sum_{i=1}^p A_{\boldsymbol n,\boldsymbol m}^{(i),\mathrm{LII}}(y)
	\e^{-\xi_i y}
\]
and
\[
	\mathcal B_{\boldsymbol n,\boldsymbol m}^{\mathrm{LII}}(y)
	=
	\sum_{j=1}^q B_{\boldsymbol n,\boldsymbol m}^{(j),\mathrm{LII}}(y)F_j(y),
\]
	where
	\[
		A_{\boldsymbol n,\boldsymbol m}^{(i),\mathrm{LII}}
		\in\mathbb P_{n_i-1},
		\qquad
		B_{\boldsymbol n,\boldsymbol m}^{(j),\mathrm{LII}}
		\in\mathbb P_{m_j-1}.
\]
	A component is understood to be zero when its corresponding multi-index
	entry is zero.

For the contour representation, let \(\Sigma\) be a positively oriented
finite union of simple closed contours enclosing the \(\xi_i\)'s and
lying to the right of the cuts issuing from \(-\rho_h\).
For the Rodrigues representation, define the shifted convolution
\begin{equation}
	\label{eq:LII-shifted-seed}
	S_{\boldsymbol m}
	\coloneq
	g_{\rho_1,d_1+m_1}*_+\cdots *_+g_{\rho_q,d_q+m_q},
\end{equation}
and put \(C\coloneq D+|\boldsymbol m|\).

\begin{theorem}[Mixed-type forms for Laguerre of the second kind and their Jacobi confluence]
	\label{thm:Jacobi-LII-form-confluence}
	Assume that \(\boldsymbol m\) is near the diagonal and that the \(\xi_i\)'s
	are pairwise distinct. Assume also that the rates
	\(\rho_j\), \(m_j>0\), are pairwise distinct.

	If \(|\boldsymbol n|=|\boldsymbol m|+1\), then
	\begin{equation}
		\label{eq:LII-A-contour}
		\mathcal A_{\boldsymbol n,\boldsymbol m}^{\mathrm{LII}}(y)
		=
		\frac{(-1)^{|\boldsymbol n|}}{2\pi\mathrm i}
		\int_\Sigma
		\frac{\prod_{h=1}^q(z+\rho_h)^{d_h+m_h}}
		{\prod_{i=1}^p(z-\xi_i)^{n_i}}
		\e^{-zy}\,\mathrm dz.
	\end{equation}
	This form satisfies
	\[
		\int_0^\infty y^kF_j(y)
		\mathcal A_{\boldsymbol n,\boldsymbol m}^{\mathrm{LII}}(y)\,\mathrm dy=0,
		\qquad 0\le k<m_j,
\]
	and, for \(n_i\ge1\),
	\begin{equation}
		\label{eq:LII-A-components}
		A_{\boldsymbol n,\boldsymbol m}^{(i),\mathrm{LII}}(y)
		=
		\frac{(-1)^{|\boldsymbol n|}}{(n_i-1)!}
		\left.\frac{\mathrm d^{n_i-1}}{\mathrm dz^{n_i-1}}
		\left[
			\frac{\prod_h(z+\rho_h)^{d_h+m_h}}
			{\prod_{k\ne i}(z-\xi_k)^{n_k}}
			\e^{-(z-\xi_i)y}
		\right]\right|_{z=\xi_i}.
	\end{equation}

	If \(|\boldsymbol m|=|\boldsymbol n|+1\), then the complete \(B\)-form is
	characterized by
	\begin{equation}
		\label{eq:LII-B-Laplace}
		\Lap[\mathcal B_{\boldsymbol n,\boldsymbol m}^{\mathrm{LII}}](z)
		=
		-
		\frac{\prod_{i=1}^p(\xi_i-z)^{n_i}}
		{\prod_{h=1}^q(z+\rho_h)^{d_h+m_h}}.
	\end{equation}
	It satisfies
	\[
		\int_0^\infty y^k\e^{-\xi_i y}
		\mathcal B_{\boldsymbol n,\boldsymbol m}^{\mathrm{LII}}(y)\,\mathrm dy=0,
		\qquad 0\le k<n_i.
\]
	Its Rodrigues representation is
	\begin{equation}
		\label{eq:LII-B-Rodrigues}
		\mathcal B_{\boldsymbol n,\boldsymbol m}^{\mathrm{LII}}(y)
		=
		-
		\prod_{i=1}^p
		\left(\xi_i-\frac{\mathrm d}{\mathrm dy}\right)^{n_i}
		S_{\boldsymbol m}(y).
	\end{equation}
	The same complete form has the
	finite sum of confluent Lauricella functions
	\begin{equation}
		\label{eq:LII-B-complete-Phi2}
		\mathcal B_{\boldsymbol n,\boldsymbol m}^{\mathrm{LII}}(y)
		={}-
		\sum_{\substack{0\le \ell_i\le n_i\\1\le i\le p}}
		(-1)^{|\boldsymbol\ell|}
		\left[
			\prod_{i=1}^p
			\binom{n_i}{\ell_i}\xi_i^{n_i-\ell_i}
		\right]
		\frac{y^{C-|\boldsymbol\ell|-1}}
		{\Gamma(C-|\boldsymbol\ell|)}
		\times
		\Phi_2^{(q)}
		\left(
		\boldsymbol d+\boldsymbol m;
		C-|\boldsymbol\ell|;
		-\rho_1y,\ldots,-\rho_qy
		\right).
	\end{equation}

	With the Jacobi normalization in
	Theorem~\ref{thm:mixed-Jacobi-like-forms},
	\begin{equation}
		\tau^{-D}\mathcal A_{\boldsymbol n,\boldsymbol m}^{\mathrm J}
		\left(1-\frac y\tau\right)
		\xrightarrow[\tau\to+\infty]{}
		\mathcal A_{\boldsymbol n,\boldsymbol m}^{\mathrm{LII}}(y),
		\label{eq:Jacobi-LII-A-form-limit}\qquad
		\tau^D\mathcal B_{\boldsymbol n,\boldsymbol m}^{\mathrm J}
		\left(1-\frac y\tau\right)
		\xrightarrow[\tau\to+\infty]{}
		\mathcal B_{\boldsymbol n,\boldsymbol m}^{\mathrm{LII}}(y).
	\end{equation}
	The first convergence is locally uniform. Its individual polynomial
	components satisfy, coefficientwise,
	\begin{equation}
		\tau^{-D}A_{\boldsymbol n,\boldsymbol m}^{(i),\mathrm J}(1-y/\tau)
		\xrightarrow[\tau\to+\infty]{}
		A_{\boldsymbol n,\boldsymbol m}^{(i),\mathrm{LII}}(y)
		\quad\text{in }\mathbb P_{n_i-1},\qquad n_i>0.
		\label{eq:Jacobi-LII-A-component-limit}
	\end{equation}
	After extension by zero, the convergence of the \(B\)-form holds
	in \(\mathcal D'(\mathbb R)\) and in the moments against
	\(y^k\e^{-\zeta y}\), for every
	\(k\in\mathbb N_0\) and every real \(\zeta\) such that
	\(\zeta+\rho_h>0\) for all \(h\).
	The polynomial row decomposition of the limiting \(B\)-form is established
	and computed explicitly in
	Proposition~\ref{prop:LII-explicit-B-components} below.
\end{theorem}

\begin{proof}
	In \eqref{eq:A-contour}, put \(t=\tau z\) and
	\(x=1-y/\tau\). Uniformly on the fixed scaled contour,
	\begin{align*}
		\frac{P_{\boldsymbol b(\tau),\boldsymbol m}(\tau z)}
		{D_{\boldsymbol\alpha(\tau),\boldsymbol n}(\tau z)}
		&=
		\tau^{|\boldsymbol m|-|\boldsymbol n|}
		\frac{\prod_h(z+\rho_h)^{m_h}}
		{\prod_i(z-\xi_i)^{n_i}}\bigl(1+\mathrm O(\tau^{-1})\bigr)\qquad (\tau\to+\infty),
\\
		\frac{\Gamma(\tau z\one_q+\boldsymbol b(\tau)+\one_q)}
		{\Gamma(\tau z\one_q+\boldsymbol a(\tau)+\one_q)}
		&=
		\tau^D\prod_h(z+\rho_h)^{d_h}
		\bigl(1+\mathrm O(\tau^{-1})\bigr)\qquad (\tau\to+\infty),
\\
		\left(1-\frac y\tau\right)^{\tau z}
		&=\e^{-zy}\bigl(1+\mathrm O_K(\tau^{-1})\bigr)
		\qquad(\tau\to+\infty).
\end{align*}
	The last estimate is uniform for \(y\) in any fixed compact
	\(K\subset(0,\infty)\), by the Taylor expansion of
	\(\log(1-y/\tau)\).
	The balance \(|\boldsymbol n|=|\boldsymbol m|+1\) cancels the factor
	\(\tau^{-1}\) against \(\mathrm dt=\tau\,\mathrm dz\). This proves the contour
	and the \(A\)-form limit.

	\emph{Convergence of the individual \(A\)-components.}
	To prove \eqref{eq:Jacobi-LII-A-component-limit}, choose disjoint
	small circles around the distinct nodes \(\xi_i\), all contained
	in \(\operatorname{Re}z> -\min_h\rho_h\). For large \(\tau\), the circle
	around \(\xi_i\) encloses exactly the poles \(\xi_i+k/\tau\),
	\(0\le k<n_i\). Its integral is
	\(\tau^{-D}A_i^{\mathrm J}(1-y/\tau)(1-y/\tau)^{\tau\xi_i}\).
	The three estimates above hold separately on every circle, with error
	\(\mathrm O_K(\tau^{-1})\) for \(y\) in a fixed compact set
	\(K\subset(0,\infty)\), as \(\tau\to+\infty\).
	The integral on the limiting circle is
	\(A_{\boldsymbol n,\boldsymbol m}^{(i),\mathrm{LII}}(y)\e^{-\xi_i y}\)
	by \eqref{eq:LII-A-components}. Dividing by
	\((1-y/\tau)^{\tau\xi_i}\), which is bounded away from zero on
	\(K\) for large \(\tau\), therefore proves locally uniform convergence
	of the reflected polynomial component, with the same error bound.
	Both polynomials have degree at most \(n_i-1\). Write their
	coefficients as \(a_{i,k}(\tau)\) and \(a_{i,k}\), and choose
	\(n_i\) fixed distinct points \(y_0,\ldots,y_{n_i-1}>0\).
	The Vandermonde matrix \(V_i\coloneq[y_v^k]_{0\le v,k<n_i}\)
	is invertible and independent of \(\tau\). Hence
	\[
	 [a_{i,k}(\tau)-a_{i,k}]_{k=0}^{n_i-1}
	 =V_i^{-1}\bigl[\tau^{-D}A_{\boldsymbol n,\boldsymbol m}^{(i),\mathrm J}(1-y_v/\tau)
	       -A_{\boldsymbol n,\boldsymbol m}^{(i),\mathrm{LII}}(y_v)\bigr]_{v=0}^{n_i-1}
	 =\mathrm O(\tau^{-1})\qquad(\tau\to+\infty).
	\]
	This proves convergence of every coefficient, not merely convergence
	of the complete linear form.

	\emph{Orthogonality of the limiting \(A\)-form.}
	The same choice of circles justifies interchanging the contour integral
	with each moment integral: their real parts stay a positive distance
	from \(-\min_h\rho_h\), so the absolute integrals of
	\(y^k\e^{-zy}F_j(y)\) are uniformly bounded there. Put
	\[
		\gamma(z)=\prod_h(z+\rho_h)^{-d_h},
		\qquad
		Q_{\boldsymbol m}(z)=\prod_h(z+\rho_h)^{m_h}.
\]
	For an index \(j\) with \(m_j>0\), fix \(0\le k<m_j\).
	Multiplication of
	\((-\partial_z)^k[\gamma(z)/(z+\rho_j)]\) by
	\(Q_{\boldsymbol m}(z)/\gamma(z)\) gives
	\[
		\frac{k!Q_{\boldsymbol m}(z)}{z+\rho_j}
		\sum_{|\boldsymbol r|=k}
		\prod_{h=1}^q
		\frac{(d_h+\delta_{h,j})_{r_h}}
		{r_h!(z+\rho_h)^{r_h}}.
\]
	For every \(|\boldsymbol r|=k\),
	\[
		r_j+1\le k+1\le m_j,
		\qquad
		r_h\le k\le m_j-1\le m_h
		\quad(h\ne j),
\]
	where the last inequality is precisely
	Definition~\ref{def:near-diagonal}. Thus every summand is a
	polynomial of degree \(|\boldsymbol m|-k-1\), which is at most
	\(|\boldsymbol m|-1\). Since
	\(|\boldsymbol n|=|\boldsymbol m|+1\),
	the resulting contour integrand has zero residue at infinity. The factors
	\(\gamma\) have cancelled, so this integrand is rational despite the
	arbitrary real shapes \(d_h\). This proves
	the stated orthogonality.

	For the \(B\)-form, evaluate \eqref{eq:B-Mellin} at \(s=\tau z+1\). Then
	\begin{align*}
		\frac{\Gamma((\tau z+1)\one_q+\boldsymbol a(\tau))}
		{\Gamma((\tau z+1)\one_q+\boldsymbol b(\tau)+\boldsymbol m)}
		&=
		\tau^{-D-|\boldsymbol m|}
		\prod_h(z+\rho_h)^{-d_h-m_h}
		\bigl(1+\mathrm O(\tau^{-1})\bigr)\qquad (\tau\to+\infty),
\\
		(\boldsymbol\alpha(\tau)-\tau z\one_p)_{\boldsymbol n}
		&=
		\tau^{|\boldsymbol n|}\prod_i(\xi_i-z)^{n_i}
		\bigl(1+\mathrm O(\tau^{-1})\bigr)\qquad (\tau\to+\infty).
\end{align*}
	Thus \(\tau^{D+1}\M[\mathcal B^{\mathrm J}](\tau z+1)\) converges locally
	uniformly on compact subsets of
	\(\operatorname{Re}z>-\min_h\rho_h\) to
	\eqref{eq:LII-B-Laplace}. To deduce convergence of the forms after the
	change \(x=1-y/\tau\), it is necessary to pass from the Mellin transform
	evaluated at \(s=\tau z+1\) to the functions themselves.

	Introduce the auxiliary Jacobi weight with shifted parameters
	\[
		S_{\boldsymbol m,\tau}^{\mathrm J}(x)
		\coloneq
		w_0^{\mathrm J}
		(x;\boldsymbol a(\tau),\boldsymbol b(\tau)+\boldsymbol m).
\]
	Its Mellin transform is the gamma quotient in the first line above.
	Since \(\M[\thetaop f](s)=s\M[f](s)\), where
	\(\thetaop=-x\mathrm d/\mathrm dx\), formula~\eqref{eq:B-Mellin} is
	equivalently the finite Rodrigues identity
	\begin{equation}
		\label{eq:Jacobi-LII-finite-B-Rodrigues}
		\mathcal B_{\boldsymbol n,\boldsymbol m}^{\mathrm J}(x)
		=
		-
		\prod_{i=1}^p\prod_{\ell=0}^{n_i-1}
		\bigl(\alpha_i(\tau)+1+\ell-\thetaop\bigr)
		S_{\boldsymbol m,\tau}^{\mathrm J}(x).
	\end{equation}
	Put \(C=D+|\boldsymbol m|\), extend the reflected functions by zero, and
	set
	\[
		\widetilde S_\tau(y)
		\coloneq
		\tau^{C-1}S_{\boldsymbol m,\tau}^{\mathrm J}(1-y/\tau),
		\qquad
		\widetilde{\mathcal B}_\tau(y)
		\coloneq
		\tau^D\mathcal B_{\boldsymbol n,\boldsymbol m}^{\mathrm J}(1-y/\tau).
\]
	Under \(x=1-y/\tau\), one has
	\(\thetaop=(\tau-y)\partial_y\). Since
	\(|\boldsymbol n|=|\boldsymbol m|-1\),
	\eqref{eq:Jacobi-LII-finite-B-Rodrigues} becomes
	\begin{equation}
		\label{eq:Jacobi-LII-reflected-B-Rodrigues}
		\widetilde{\mathcal B}_\tau
		=-\mathscr P_\tau\widetilde S_\tau,
		\qquad
		\mathscr P_\tau
		\coloneq
		\prod_{i=1}^p\prod_{\ell=0}^{n_i-1}
		\left(
		\xi_i-\partial_y
		+\frac{1+\ell+y\partial_y}{\tau}
		\right).
	\end{equation}
	The identity holds also after extension by zero, in the sense of
	distributions, for all sufficiently large \(\tau\): the function
	\(\widetilde S_\tau\) and all its derivatives of order less than
	\(|\boldsymbol n|\) vanish at \(y=0\) and \(y=\tau\). More explicitly, for
	\(0\le j<|\boldsymbol n|\),
	\[
		\partial_y^j\widetilde S_\tau(y)=\mathrm O(y^{C-1-j})
		\quad(y\downarrow0),
\]
	while, with \(\underline a_\tau=\min_h a_h(\tau)\),
	\[
		\partial_y^j\widetilde S_\tau(y)
		=
		\mathrm O\!\left(
		(\tau-y)^{\underline a_\tau-j}
		\bigl(1+|\log(\tau-y)|^{M_{j,\tau}}\bigr)
		\right)
		\quad(y\uparrow \tau)
\]
	for some finite \(M_{j,\tau}\). At \(y=0\), the full differential
	order \(|\boldsymbol n|=|\boldsymbol m|-1\) leaves the positive exponent
	\(D\); at \(y=\tau\), \(\underline a_\tau=\tau\min_h\rho_h-1\) tends to
	infinity. Hence all the required derivatives vanish at both ends of
	the interval.

	Theorem~\ref{thm:Jacobi-LII-weight-limit}, applied with shapes
	\(d_h+m_h\), and the stronger estimate
	\eqref{eq:right-endpoint-normalized-density-limit} give, after extension
	by zero,
	\[
		\widetilde S_\tau\xrightarrow[\tau\to+\infty]{} S_{\boldsymbol m}
\]
	locally uniformly on \(\mathbb R\). Moreover, the coefficients of
	\(\mathscr P_\tau\), together with all their derivatives on compact sets,
	converge to those of
	\[
		\mathscr P
		\coloneq
		\prod_{i=1}^p(\xi_i-\partial_y)^{n_i}.
\]
	Here
	\(\Lap[S_{\boldsymbol m}](z)=
	\prod_h(z+\rho_h)^{-d_h-m_h}\). Since
	\(S_{\boldsymbol m}(y)=\mathrm O(y^{D+|\boldsymbol m|-1})\) at the
	origin and \(|\boldsymbol n|=|\boldsymbol m|-1\), all boundary terms in
	the Laplace transform of \(\mathscr P S_{\boldsymbol m}\) vanish.
	Consequently, \(-\mathscr P S_{\boldsymbol m}\) has transform
	\eqref{eq:LII-B-Laplace} and is precisely
	\(\mathcal B_{\boldsymbol n,\boldsymbol m}^{\mathrm{LII}}\). This also
	proves \eqref{eq:LII-B-Rodrigues}.
	For \(\varphi\in C_c^\infty(\mathbb R)\), formal adjoints in
	\eqref{eq:Jacobi-LII-reflected-B-Rodrigues} therefore give
	\[
		\langle\widetilde{\mathcal B}_\tau,\varphi\rangle
		=-\langle\widetilde S_\tau,\mathscr P_\tau^*\varphi\rangle
		\xrightarrow[\tau\to+\infty]{}
		-\langle S_{\boldsymbol m},\mathscr P^*\varphi\rangle
		=\langle\mathcal B_{\boldsymbol n,\boldsymbol m}^{\mathrm{LII}},
		\varphi\rangle.
\]
	This proves \eqref{eq:Jacobi-LII-A-form-limit} in
	\(\mathcal D'(\mathbb R)\).

	The same argument gives the claimed moments without using a compactly
	supported test function. If
	\(\psi(y)=y^k\e^{-\zeta y}\), with
	\(\zeta+\rho_h>0\) for every \(h\), then
	\(\mathscr P_\tau^*\psi\) is \(\e^{-\zeta y}\) times a polynomial of fixed
	degree whose coefficients converge to those of \(\mathscr P^*\psi\).
	The required uniform integrability follows directly from the factor
	densities. Indeed, for all sufficiently large \(\tau\),
	\[
		\left(1-\frac u\tau\right)^{\tau\rho_h-1}
		\le \exp\bigl(-(\rho_h-\tau^{-1})u\bigr),
		\qquad 0<u<\tau.
\]
	Together with
	\(Y_{q,\tau}\le U_{1,\tau}+\cdots+U_{q,\tau}\), this gives uniform exponential
	moment bounds at every growth rate strictly below
	\(\min_h\rho_h\). Thus the condition
	\(\zeta+\rho_h>0\) for all \(h\), with a slightly stronger intermediate
	exponent, gives uniform integrability of every polynomial times
	\(\e^{-\zeta y}\). Weak convergence of the measures
	\(\widetilde S_\tau(y)\,\mathrm dy\), already proved
	in Theorem~\ref{thm:Jacobi-LII-weight-limit}, therefore yields convergence
	of all the pairings with \(\mathscr P_\tau^*\psi\). Hence
	\[
		\int_0^\tau y^k\e^{-\zeta y}\widetilde{\mathcal B}_\tau(y)\,\mathrm dy
		\xrightarrow[\tau\to+\infty]{}
		\int_0^\infty y^k\e^{-\zeta y}
		\mathcal B_{\boldsymbol n,\boldsymbol m}^{\mathrm{LII}}(y)\,\mathrm dy.
\]

	Finally, the convolution integral \eqref{eq:gamma-convolution-simplex} gives
	\[
		S_{\boldsymbol m}(y)
		=
		\frac{y^{C-1}}{\Gamma(C)}
		\Phi_2^{(q)}
		\left(
		\boldsymbol d+\boldsymbol m;C;
		-\rho_1y,\ldots,-\rho_qy
		\right).
\]
	Termwise differentiation of the entire Lauricella series gives,
	for \(0\le r\le |\boldsymbol n|\),
	\[
		\frac{\mathrm d^r}{\mathrm dy^r}S_{\boldsymbol m}(y)
		=
		\frac{y^{C-r-1}}{\Gamma(C-r)}
		\Phi_2^{(q)}
		\left(
		\boldsymbol d+\boldsymbol m;C-r;
		-\rho_1y,\ldots,-\rho_qy
		\right).
\]
	Expanding the differential operator in
	\eqref{eq:LII-B-Rodrigues} and grouping the terms with
	\(r=|\boldsymbol\ell|\) proves
	\eqref{eq:LII-B-complete-Phi2}.
\end{proof}

The residue derivatives in \eqref{eq:LII-A-components} can be written
as terminating multiple Kamp\'e de F\'eriet series. Retain the parameter
hypotheses of Theorem~\ref{thm:Jacobi-LII-form-confluence} and suppose
\(|\boldsymbol n|=|\boldsymbol m|+1\). For \(n_i\ge1\), put
	\[
		M_i\coloneq n_i-1,
		\qquad
		\lambda_h\coloneq d_h+m_h,
		\qquad
		c_{i,h}\coloneq\xi_i+\rho_h,
		\qquad
		\Delta_{i,k}\coloneq\xi_i-\xi_k,
\]
	and define
	\[
		\mathcal H_i
		\coloneq
		\frac{\prod_{h=1}^q c_{i,h}^{\lambda_h}}
		{\prod_{k\ne i}\Delta_{i,k}^{n_k}}.
\]
	Choose \(h_0\in\{1,\ldots,q\}\) and set
	\(\chi_{i,h_0}\coloneq\lambda_{h_0}-M_i+1\).
	With \(T\coloneq\ell+\sum_{h\ne h_0}r_h+\sum_{k\ne i}s_k\), define
	\begin{equation}
		\label{eq:LII-A-multiple-KdF-series}
		\mathfrak F_{i,h_0}(y)
		\coloneq{}
		\sum_{\ell,\boldsymbol r,\boldsymbol s\ge0}
		\frac{(-M_i)_T}{(\chi_{i,h_0})_T}
		\frac{(c_{i,h_0}y)^\ell}{\ell!}
		\prod_{\substack{1\le h\le q\\h\ne h_0}}
		\frac{(-\lambda_h)_{r_h}}{r_h!}
		\left(\frac{c_{i,h_0}}{c_{i,h}}\right)^{r_h}
		\prod_{\substack{1\le k\le p\\k\ne i}}
		\frac{(n_k)_{s_k}}{s_k!}
		\left(\frac{c_{i,h_0}}{\Delta_{i,k}}\right)^{s_k}.
	\end{equation}
	The factor \((-M_i)_T\) makes this a terminating multiple Kamp\'e de
	F\'eriet polynomial of the form
	\(F_{1:0;0;\ldots;0}^{1:0;1;\ldots;1}\) in the convention
	\eqref{eq:multiple-KdF-definition}.

\begin{corollary}[Multiple Kamp\'e de F\'eriet form of the
\(A\)-components]
	\label{cor:LII-A-multiple-KdF}
	Under the hypotheses of Theorem~\ref{thm:Jacobi-LII-form-confluence},
	with \(|\boldsymbol n|=|\boldsymbol m|+1\), the components with \(n_i\ge1\) satisfy
	\begin{equation}
		\label{eq:LII-A-multiple-KdF}
		A_{\boldsymbol n,\boldsymbol m}^{(i),\mathrm{LII}}(y)
		=
		(-1)^{|\boldsymbol n|}\mathcal H_i
		\frac{(\chi_{i,h_0})_{M_i}}
		{M_i!\,c_{i,h_0}^{M_i}}
		\mathfrak F_{i,h_0}(y).
	\end{equation}
	The value of
	\eqref{eq:LII-A-multiple-KdF} is independent of \(h_0\); if a displayed
	denominator parameter is exceptional, the formula is understood by
	polynomial continuation.
\end{corollary}

\begin{proof}
	In \eqref{eq:LII-A-components}, write \(z=\xi_i+u\).  Expanding
	\((c_{i,h}+u)^{\lambda_h}\),
	\((\Delta_{i,k}+u)^{-n_k}\), and \(\e^{-uy}\), and then extracting the
	coefficient of \(u^{M_i}\), gives the nonsingular finite formula
	\begin{equation}
		\label{eq:LII-A-homogeneous-finite-sum}
		\begin{aligned}
		A_{\boldsymbol n,\boldsymbol m}^{(i),\mathrm{LII}}(y)
		={}&(-1)^{|\boldsymbol n|+M_i}\mathcal H_i
		\sum_{\ell+|\boldsymbol r|+|\boldsymbol s|=M_i}
		\frac{y^\ell}{\ell!}
		\prod_{h=1}^q
		\frac{(-\lambda_h)_{r_h}}
		{r_h!c_{i,h}^{r_h}}
		\prod_{k\ne i}
		\frac{(n_k)_{s_k}}
		{s_k!\Delta_{i,k}^{s_k}}.
		\end{aligned}
	\end{equation}
	Choose \(h_0\) and put
	\(r_{h_0}=M_i-T\).  The identity
	\[
		\frac{(-\lambda_{h_0})_{M_i-T}}
		{(M_i-T)!c_{i,h_0}^{M_i-T}}
		=
		\frac{(-\lambda_{h_0})_{M_i}}
		{M_i!c_{i,h_0}^{M_i}}
		\frac{(-M_i)_T}{(\chi_{i,h_0})_T}
		c_{i,h_0}^{T}
\]
	rewrites \eqref{eq:LII-A-homogeneous-finite-sum} as
	\eqref{eq:LII-A-multiple-KdF-series}.  Finally,
	\((-\lambda_{h_0})_{M_i}=(-1)^{M_i}(\chi_{i,h_0})_{M_i}\), which gives
	the prefactor in \eqref{eq:LII-A-multiple-KdF}.
\end{proof}

\subsection{Rational reconstruction and normality}

Put
\[
	\gamma(z)\coloneq\prod_{h=1}^q(z+\rho_h)^{-d_h},
	\qquad
	Q_{\boldsymbol m}(z)\coloneq\prod_{h=1}^q(z+\rho_h)^{m_h},
\]
and, for \(0\le k<m_j\), define
\begin{equation}
	\label{eq:LII-rational-basis}
	R_{j,k}(z)
	\coloneq
	\gamma(z)^{-1}(-\partial_z)^k
	\left[\frac{\gamma(z)}{z+\rho_j}\right],
	\qquad
	\Phi_{j,k}(z)\coloneq Q_{\boldsymbol m}(z)R_{j,k}(z).
\end{equation}

For a near-diagonal \(\boldsymbol m\), put \(N\coloneq|\boldsymbol m|\)
and order the pairs \((j,k)\)
lexicographically, and let \(C_{\boldsymbol m}\) be the matrix whose
\((j,k)\)-th column consists of the coefficients of \(\Phi_{j,k}\)
at \(1,z,\ldots,z^{N-1}\).

\begin{proposition}[Near-diagonal rational basis]
	\label{prop:LII-rational-basis}
	If \(\boldsymbol m\) is near the diagonal, every
	\(\Phi_{j,k}\) belongs to \(\mathbb P_{N-1}\), and
	\begin{equation}
		\label{eq:LII-rational-basis-determinant}
		\det C_{\boldsymbol m}
		=
		(-1)^{\sum_j\binom{m_j}{2}}
		\left[\prod_{j=1}^q\prod_{k=0}^{m_j-1}(d_j+1)_k\right]
		\prod_{1\le j<h\le q}
		(\rho_h-\rho_j)^{m_jm_h}.
	\end{equation}
	Consequently, these \(N\) polynomials form a basis if and only if
	\[
		\rho_j\ne\rho_h
		\qquad\textnormal{whenever }m_jm_h>0.
\]
\end{proposition}

\begin{proof}
	The multinomial Leibniz rule gives
	\[
		R_{j,k}(z)
		=
		\frac{k!}{z+\rho_j}
		\sum_{|\boldsymbol r|=k}
		\prod_{h=1}^q
		\frac{(d_h+\delta_{h,j})_{r_h}}
		{r_h!(z+\rho_h)^{r_h}}.
\]
	For every indexing pair \((j,k)\), one has \(m_j\ge1\) and
	\(0\le k<m_j\). Near-diagonality gives \(m_j-1\le m_h\) for every
	\(h\). Hence the pole order is at most \(k+1\le m_j\) at
	\(-\rho_j\), and at most \(k\le m_j-1\le m_h\) at
	\(-\rho_h\), \(h\ne j\). Therefore
	\(Q_{\boldsymbol m}R_{j,k}\) is a polynomial of degree at most \(N-1\).

	To identify the determinant of the change of basis, temporarily order the
	pairs first by increasing \(k\) and then, within each level, by increasing
	\(j\); use the same order for the partial fractions
	\((z+\rho_j)^{-k-1}\). At level \(k\), reduction modulo partial fractions
	of lower pole order leaves
	\((d_j+1)_k(z+\rho_j)^{-k-1}\). All the remaining terms belong to earlier
	levels. Hence the change from the functions \(R_{j,k}\) to the
	partial-fraction basis is triangular in this level-first order, with
	diagonal entries \((d_j+1)_k\). Returning both ordered lists to the
	lexicographic order fixed above conjugates the change-of-basis
	matrix by the same permutation matrix, and therefore does not change its
	determinant. The numerators
	\(Q_{\boldsymbol m}(z)/(z+\rho_j)^\ell\), \(1\le\ell\le m_j\),
	have coefficient determinant
	\[
		(-1)^{\sum_j\binom{m_j}{2}}
		\prod_{j<h}(\rho_h-\rho_j)^{m_jm_h}.
\]
	For completeness, evaluate these numerators and their normalized
	derivatives \(\partial_z^v/v!\), \(0\le v<m_u\), at each
	\(-\rho_u\). Blocks with \(u\ne j\) vanish. In the \(j\)-th
	block, reversing the columns gives a triangular matrix with diagonal
	\(\prod_{h\ne j}(\rho_h-\rho_j)^{m_h}\). Its reversal contributes
	\((-1)^{\binom{m_j}{2}}\). Dividing the resulting block determinant
	by the normalized confluent Vandermonde determinant
	\(\prod_{j<h}(\rho_j-\rho_h)^{m_jm_h}\) gives the displayed
	coefficient determinant, with its stated sign.
	Multiplication of the two change-of-basis determinants gives
	\eqref{eq:LII-rational-basis-determinant}.
	The calculation was made for distinct rates. The coefficient matrix
	depends continuously on all positive rates, so the formula extends to
	coincident rates as well. Since every \((d_j+1)_k\) is positive,
	its product is nonzero exactly under the stated separation condition.
\end{proof}

The rational basis determines the component problem itself, not only the
square moment determinants. For a fixed near-diagonal row index of length
\(N\ge1\), denote the solution spaces of the homogeneous polynomial
\(A\)- and \(B\)-problems by \(\mathscr S_{\mathrm A}^{\mathrm{LII}}\)
and \(\mathscr S_{\mathrm B}^{\mathrm{LII}}\), with respective column
lengths \(N+1\) and \(N-1\).

\begin{corollary}[Polynomial components and normality for Laguerre of the second kind]
 \label{cor:LII-component-normality}
 Suppose that the rates with positive row indices are pairwise distinct,
 that the column nodes with positive indices are pairwise distinct, and
 that \eqref{eq:LII-integrability} holds. For a near-diagonal row index,
 \begin{equation}
 \dim\mathscr S_{\mathrm A}^{\mathrm{LII}}=1,
 \qquad \dim\mathscr S_{\mathrm B}^{\mathrm{LII}}=1.
 \label{eq:LII-solution-dimensions}
 \end{equation}
 The \(A\)-problem is strongly normal. The \(B\)-problem is weakly
 normal, and its nonzero vectors represent nonzero forms. In particular,
 the Rodrigues function has a unique decomposition
 \(\sum_jB_j(y)F_j(y)\), with \(\deg B_j<m_j\).
 If all the rates are distinct, the functions \(F_1,\ldots,F_q\) are
 linearly independent over the polynomial ring.
\end{corollary}

\begin{proof}
 Put \(g\coloneq\gamma/Q_{\boldsymbol m}\). For a polynomial row
 combination, its Laplace transform is
 \[
 \Lap\left[\sum_{j=1}^q\sum_{k=0}^{m_j-1}b_{j,k}y^kF_j\right](z)
 =g(z)\sum_{j,k}b_{j,k}\Phi_{j,k}(z).
 \]
 Since \(g\) is nonzero at each column node, Leibniz's rule transforms
 the moment conditions into vanishing of the corresponding derivatives
 of the polynomial on the right. In the \(B\)-balance their total
 multiplicity is \(N-1\). A polynomial of degree at most \(N-1\)
 with these zeros is a scalar multiple of
 \(P_{\boldsymbol n}(z)=\prod_i(\xi_i-z)^{n_i}\).
 Proposition~\ref{prop:LII-rational-basis} gives one and only one
 coefficient vector for each such polynomial. The normalization
 \(\Lap[\mathcal B^{\mathrm{LII}}]=-gP_{\boldsymbol n}\) therefore
 gives precisely the stated Rodrigues decomposition. Its transform is
 nonzero, and its first unused moment in any chosen column is nonzero,
 because that derivative removes the full multiplicity at that node.

 In the \(A\)-balance there are \(N+1\) column functionals acting
 on \(\mathbb P_{N-1}\). Their evaluation map is injective: a
 polynomial in its kernel would have at least \(N+1\) zeros counted
 with multiplicities. Proposition~\ref{prop:LII-rational-basis}
 identifies the \(N\) row combinations with all of \(\mathbb P_{N-1}\):
 its coefficient matrix \(C_{\boldsymbol m}\) is invertible
 under the near-diagonal and separation hypotheses used here.
 Thus the rectangular moment matrix has rank \(N\), and its nullspace
 of component vectors has dimension one. Theorem~\ref{thm:Jacobi-LII-form-confluence}
 shows that the contour components satisfy these moment conditions;
 its proof also uses near-diagonality to cancel all the row poles.
 The highest coefficient of the \(i\)-th such component is
 \[
 \frac{(-1)^{N+1+n_i-1}}{(n_i-1)!}
 \frac{\prod_h(\xi_i+\rho_h)^{d_h+m_h}}
 {\prod_{v\ne i}(\xi_i-\xi_v)^{n_v}},\qquad n_i>0.
 \]
 All its factors are nonzero, proving strong normality in this
 near-diagonal range. The same formal residue can be written for other
 row indices, but need not then satisfy the orthogonality conditions.

 Finally, any finite polynomial relation among distinct-rate rows fits
 within some diagonal index \((K,\ldots,K)\). Taking its Laplace
 transform gives a zero linear combination of that index's basis
 \(\Phi_{j,k}\), so all its coefficients vanish. This proves polynomial
 independence at arbitrary finite degrees. In particular,
 \eqref{eq:LII-first-order-system}, which also involves derivatives
 and \(F_0\), does not produce a Pearson-type polynomial relation among
 these rows. Polynomial independence and normality for every multi-index
 are different properties, as the example below shows.
\end{proof}

The determinant underlying the following evaluation is the classical
Cauchy--Vandermonde determinant with multiple nodes and multiple prescribed
poles; see \cite[Theorem~3, equation~(14)]{CarstensenMuehlbach1992NevilleAitken}
and \cite[Proposition~5, equations~(55)--(56)]{Muehlbach2000Interpolation}.
The arbitrary positive real parameters \(d_j\) enter through the analytic
factor \(\gamma\) and the triangular change of basis in
Proposition~\ref{prop:LII-rational-basis}. The resulting product evaluates
the mixed-type moment determinant for Laguerre of the second kind and supplies the determinants
needed to normalize the forms.

For \(|\boldsymbol n|=|\boldsymbol m|=N\), order the row indices
\((j,k)\), \(1\le j\le q\), \(0\le k<m_j\), and column indices
\((i,\ell)\), \(1\le i\le p\), \(0\le\ell<n_i\), lexicographically.
Define the square mixed moment matrix by
	\[
		\mathscr M_{\boldsymbol m,\boldsymbol n}^{\mathrm{LII}}
		\coloneq
		\left[
		\int_0^\infty y^{k+\ell}F_j(y)\e^{-\xi_i y}\,\mathrm dy
		\right]_{(j,k),(i,\ell)}.
\]
\begin{theorem}[Cauchy--Vandermonde factorization of the mixed-type moment
determinant]
	\label{thm:LII-moment-determinants}
	If \(|\boldsymbol n|=|\boldsymbol m|=N\) and
	\(\boldsymbol m\) is near the diagonal, then
	\begin{multline}
		\det\mathscr M_{\boldsymbol m,\boldsymbol n}^{\mathrm{LII}}
		={}
		(-1)^{\sum_j\binom{m_j}{2}+\sum_i\binom{n_i}{2}}
		\left[\prod_j\prod_{k=0}^{m_j-1}(d_j+1)_k\right]
		\left[\prod_{j<h}(\rho_h-\rho_j)^{m_jm_h}\right]
		\\*
		\times
		\left[\prod_i\prod_{\ell=0}^{n_i-1}\ell!\right]
		\left[\prod_{i<r}(\xi_r-\xi_i)^{n_in_r}\right]
		\prod_{i=1}^p\prod_{j=1}^q
		(\xi_i+\rho_j)^{-(d_j+m_j)n_i}.
		\label{eq:LII-square-moment-determinant}
	\end{multline}
	In particular, it is nonzero when the rates
	\(\rho_j\), \(m_j>0\), and the column nodes
	\(\xi_i\), \(n_i>0\), are pairwise distinct and
	\eqref{eq:LII-integrability} holds.
\end{theorem}

\begin{proof}
	Put \(g\coloneq\gamma/Q_{\boldsymbol m}\). Differentiating the
	Laplace transform shows that the entry in row \((j,k)\) and column
	\((i,\ell)\) is \((-1)^\ell\partial_z^\ell[g\Phi_{j,k}](\xi_i)\).
	Define the derivative evaluation matrix and the block triangular
	Leibniz matrix by
	\[
	 V_{(i,\ell),a}\coloneq\left.\partial_z^\ell z^a\right|_{z=\xi_i},
	 \qquad J_{(i,\ell),(i,v)}\coloneq\binom\ell v g^{(\ell-v)}(\xi_i)
	 \quad(0\le v\le\ell<n_i),
	\]
	where \(0\le a<N\), and all other entries of \(J\) are zero.
	Let \(S\) be diagonal with \(S_{(i,\ell),(i,\ell)}=(-1)^\ell\).
	The entrywise Leibniz formula is precisely
	\[
	 (\mathscr M_{\boldsymbol m,\boldsymbol n}^{\mathrm{LII}})^\top
	 =S J V C_{\boldsymbol m}.
	\]
	The diagonal entries of each \(i\)-th block of \(J\) are all
	\(g(\xi_i)=\prod_j(\xi_i+\rho_j)^{-d_j-m_j}\), so
	\(\det J=\prod_i g(\xi_i)^{n_i}\). The derivative rather than
	normalized-derivative convention gives
	\[
	 \det V=\prod_i\prod_{\ell=0}^{n_i-1}\ell!
	 \prod_{i<v}(\xi_v-\xi_i)^{n_in_v},\qquad
	 \det S=(-1)^{\sum_i\binom{n_i}{2}}.
	\]
	Multiplying these factors by
	\eqref{eq:LII-rational-basis-determinant} proves the formula, including
	its sign. Integrability makes every \(g(\xi_i)\) positive, and the
	remaining factors are nonzero exactly under the stated separation
	conditions.
\end{proof}

The row restriction is essential. The following example has positive
shapes, distinct rates, and admissible column nodes, but fails strong
normality outside the near-diagonal range. In the one-column case,
perfectness means that for every row index \(\boldsymbol m\in\mathbb N_0^q\)
there is a unique monic polynomial of degree \(|\boldsymbol m|\)
satisfying its \(|\boldsymbol m|\) moment conditions.

\begin{example}[Failure of perfectness outside the near-diagonal range]
 \label{ex:LII-not-perfect}
 Take \(q=2\), \(p=1\), \(\boldsymbol\rho=(1,3)\),
 \(\xi_1=1/2\), and \(\boldsymbol m=(3,1)\). Put
 \(\mu_{j,v}\coloneq\int_0^\infty y^vF_j(y)\e^{-y/2}\,\mathrm dy\)
 and \(g_*\coloneq\gamma(1/2)\). Direct differentiation gives
 \[
 \frac{\mu_{j,v}}{g_*}
 =\sum_{a=0}^v\binom va
 \frac{(d_1+\delta_{j,1})_a(d_2+\delta_{j,2})_{v-a}}
 {(3/2)^{a+\delta_{j,1}}(7/2)^{v-a+\delta_{j,2}}}.
 \]
 Let \(M\) be the \(4\times5\) matrix with entries
 \(\mu_{j,k+\ell}/g_*\), rows
 \((j,k)=(1,0),(1,1),(1,2),(2,0)\), and columns
 \(\ell=0,\ldots,4\). Its first four columns have determinant
 \[
 \det M_{[0,1,2,3]}
 =-\frac{96(d_1+1)^2(d_1+2)}{(3/2)^{12}(7/2)^4}
 \frac{2401d_1^2+7203d_1-81d_2^2-81d_2+4802}
 {2401(d_1+1)(d_1+2)}.
 \]
 This identity follows by substituting the finite moment sum above and
 collecting powers of \(d_1,d_2\). Choose
 \[
 d_1=\frac12,\qquad d_2=\frac{8\sqrt{141}}9-\frac12>0.
 \]
 Then \(324d_2(d_2+1)=36015\), so this determinant vanishes.
 With the same parameters the remaining minor
 \[
 \det M_{[0,1,2,4]}=-\frac{41943040}{992436543}\ne0
 \]
 proves that \(M\) still has rank four. In the \(A\)-balance
 \(n_1=5\), its kernel is therefore one-dimensional. The vanishing
 first minor forces the coefficient of \(y^4\) to be zero, whereas
 independence of columns \(0,1,2\) forces the coefficient of \(y^3\)
 to be nonzero. The unique polynomial has degree three, not four.
 Thus this Laguerre system of the second kind is not perfect and is not strongly normal
 at all multi-indices.

 The step that fails outside the hypotheses of
 Corollary~\ref{cor:LII-component-normality} can be seen before computing
 the moments. For \(\boldsymbol m=(3,1)\), one has
 \[
 \Phi_{1,2}(z)
 =(d_1+1)(d_1+2)(z+3)+2(d_1+1)d_2(z+1)
   +\frac{d_2(d_2+1)(z+1)^2}{z+3}.
 \]
 Its residue at \(z=-3\) is \(4d_2(d_2+1)\ne0\), so it is not a
 polynomial and Proposition~\ref{prop:LII-rational-basis} does not apply.
 More explicitly, inserting \(\boldsymbol m=(3,1)\) and \(n_1=5\)
 into the right-hand side of \eqref{eq:LII-A-components} defines a
 formal polynomial \(A_{\mathrm{cont}}\). The residue at \(-3\) gives
 \[
 \int_0^\infty y^2F_1(y)A_{\mathrm{cont}}(y)\e^{-y/2}\,\mathrm dy
 =\frac{4d_2(d_2+1)}{(-7/2)^5}
 =-\frac{128d_2(d_2+1)}{16807}\ne0.
 \]
 Thus its nonzero coefficient of \(y^4\) belongs to a polynomial
 that fails orthogonality, not to the unique orthogonal polynomial
 of degree three found above. The rectangular matrix retains rank four; only the
 square minor controlling the coefficient of \(y^4\) vanishes.
\end{example}

\begin{remark}[Relation with the separable mixed-type Laguerre system of the second kind]
	The separable mixed-type Laguerre matrix of the second kind explicitly exhibited in Medina
	Peralta's thesis has entries of the form
	\[
		y^{\alpha+\beta}\e^{-(c_j+d_i)y},
		\qquad \alpha,\beta>-1,
\]
	where the power is common to the complete matrix
	\cite[Section~2.3.2]{MedinaPeralta2014Thesis}. It is covered by the mixed-type
	AT normality results of Fidalgo Prieto, Medina Peralta, and
	M\'inguez~\cite{FidalgoMedinaMinguez2013}. The matrix
	\eqref{eq:LII-matrix-measure} is generally different: each row is an
	additive convolution of gamma kernels.
	Theorem~\ref{thm:LII-moment-determinants} gives normality in the
	near-diagonal range, while Example~\ref{ex:LII-not-perfect} distinguishes
	this conclusion from the perfectness of the separable family.
\end{remark}

The exact scaling also gives a quantitative version of the confluence.
Write
\[
 \widetilde w_{j,\tau}(y)\coloneq \tau^Dw_{j,\tau}^{\mathrm J}(1-y/\tau),\quad
 \widetilde A_{i,\tau}(y)\coloneq \tau^{-D}A_i^{\mathrm J}(1-y/\tau),\quad
 \widetilde B_{j,\tau}(y)\coloneq B_j^{\mathrm J}(1-y/\tau).
\]
The last two definitions use the normalized representatives of
Theorem~\ref{thm:mixed-Jacobi-like-forms} in their respective balances.

\begin{proposition}[First-order confluence to Laguerre of the second kind]
 \label{prop:LII-convergence-rate}
 Under \eqref{eq:Jacobi-LII-parameter-scaling}, with fixed near-diagonal
 row index and the separation hypotheses of
 Corollary~\ref{cor:LII-component-normality}, the finite normalized
 components exist for all sufficiently large \(\tau\), and
 \begin{equation}
 \widetilde A_{i,\tau}=A_i^{\mathrm{LII}}+\mathrm O(\tau^{-1}),\qquad
 \widetilde B_{j,\tau}=B_j^{\mathrm{LII}}+\mathrm O(\tau^{-1})
 \qquad(\tau\to+\infty)
 \label{eq:LII-component-rate}
 \end{equation}
 coefficientwise. The density convergence in
 \eqref{eq:Jacobi-LII-row-limit} has error \(\mathrm O_K(\tau^{-1})\)
 on each compact \(K\subset(0,\infty)\), and every fixed mixed
 moment in \eqref{eq:Jacobi-LII-mixed-moment-limit} has error
 \(\mathrm O(\tau^{-1})\), as \(\tau\to+\infty\).
 For each \(\varphi\in C_c^\infty(\mathbb R)\), the complete
 \(B\)-form limit also satisfies
 \[
 \left\langle\widetilde{\mathcal B}_\tau-
 \mathcal B^{\mathrm{LII}},\varphi\right\rangle
 =\mathrm O_\varphi(\tau^{-1})\qquad(\tau\to+\infty).
 \]
\end{proposition}

\begin{proof}
 In the induction proving
 \eqref{eq:right-endpoint-normalized-density-limit}, Taylor's estimate
 gives \(\|\phi_{1,\tau}-\phi_1\|_{[0,K]}\le C_K/\tau\).
 At each subsequent step, all the other factors are uniformly bounded,
 their differences are bounded by \(C_K/\tau\), and their common
 multiplier has integral \(B(E,e)<\infty\). Hence
 \[
 \|\phi_{r+1,\tau}-\phi_{r+1}\|_{[0,K]}
 \le C_K\|\phi_{r,\tau}-\phi_r\|_{[0,K]}+C_K/\tau.
 \]
 Induction proves the same first-order bound for every fixed \(r\).
 This proves the density assertion, including the weighted estimate at zero.

 For the moments, set \(e_h=d_h+\delta_{h,j}\) and
 \(Y_\tau(\boldsymbol u)\coloneq \tau[1-\prod_h(1-u_h/\tau)]\).
 The mixed moment is exactly
 \[
 \int_{(0,\tau)^q}Y_\tau(\boldsymbol u)^v
 \prod_{h=1}^q\frac{u_h^{e_h-1}}{\Gamma(e_h)}
 (1-u_h/\tau)^{\tau(\rho_h+\xi_i)-1}\,\mathrm d\boldsymbol u.
 \]
 For fixed integer \(v\), \(Y_\tau^v\) is a finite polynomial in the
 \(u_h\)'s and \(\tau^{-1}\), whose constant term in \(\tau^{-1}\) is
 \((\sum_hu_h)^v\). Every monomial integral factors into terms
 \[
 \int_0^\tau\frac{u^{e+a-1}}{\Gamma(e)}(1-u/\tau)^{\tau\sigma-1}\,\mathrm du
 =\tau^{e+a}(e)_a\frac{\Gamma(\tau\sigma)}{\Gamma(\tau\sigma+e+a)}
 =\frac{(e)_a}{\sigma^{e+a}}+\mathrm O(\tau^{-1})
 \quad(\tau\to+\infty),
 \]
 where \(\sigma=\rho_h+\xi_i>0\). There are only finitely many
 such terms, proving the mixed-moment estimate.

 The contour-circle and fixed Vandermonde argument in the proof of
 Theorem~\ref{thm:Jacobi-LII-form-confluence} already gives the asserted
 rate for each \(A\)-component. For \(B\), augment the column index
 by one unit in any column. The limiting square matrix is invertible
 by Theorem~\ref{thm:LII-moment-determinants}. The moment estimate
 just proved gives its finite counterpart as
 \(M_\tau=M+\mathrm O(\tau^{-1})\); thus \(M_\tau\) is invertible for large
 \(\tau\). Theorem~\ref{thm:mixed-Jacobi-like-forms} identifies its
 normalized solution with the finite Jacobi representative.

 To check the normalization as well as the homogeneous conditions, put
 \(H_\tau(z)\coloneq \tau^{D+1}\M[\mathcal B^{\mathrm J}](\tau z+1)\).
 The gamma-quotient calculation above gives
 \(H_\tau=\Lap[\mathcal B^{\mathrm{LII}}]+\mathrm O(\tau^{-1})\)
 locally uniformly in \(\operatorname{Re}z> -\min_h\rho_h\).
 Cauchy's formula gives the same bound for each fixed derivative.
 If \(\Delta_h f(z)\coloneq f(z+h)-f(z)\), the finite normalizing
 moments are
 \[
 \int_0^\tau y^v(1-y/\tau)^{\tau\xi_i}\widetilde{\mathcal B}_\tau(y)\,\mathrm dy
 =(-1)^v \tau^v\Delta_{1/\tau}^{v}H_\tau(\xi_i).
 \]
 The integral formula for finite differences expresses the right side
 as an average of \((-1)^vH_\tau^{(v)}\) at points within \(v/\tau\)
 of \(\xi_i\). Its error is therefore \(\mathrm O(\tau^{-1})\).
 Thus the full right-hand side vector of the normalized system is
 \(y_\tau=y+\mathrm O(\tau^{-1})\). If \(b_\tau,b\) are the component
 coefficient vectors, then
 \[
 b_\tau-b=M_\tau^{-1}\bigl[(y_\tau-y)-(M_\tau-M)b\bigr]
 =\mathrm O(\tau^{-1})\qquad(\tau\to+\infty),
 \]
 since the inverses \(M_\tau^{-1}\) remain bounded. This proves the
 \(B\)-component rate with the prescribed normalization.

 Finally, use \eqref{eq:Jacobi-LII-reflected-B-Rodrigues} and formal
 adjoints. On the compact support of \(\varphi\), the coefficients
 of \(\mathscr P_\tau^*-\mathscr P^*\) are \(\mathrm O(\tau^{-1})\).
 The density estimate with shapes \(d_h+m_h\) bounds
 \(\widetilde S_\tau-S_{\boldsymbol m}\) by \(C \tau^{-1}y^{D+|\boldsymbol m|-1}\)
 near zero, an integrable function there. The two terms obtained by
 subtracting \(\langle\widetilde S_\tau,\mathscr P_\tau^*\varphi\rangle\)
 and \(\langle S_{\boldsymbol m},\mathscr P^*\varphi\rangle\) are
 each \(\mathrm O_\varphi(\tau^{-1})\), proving the last assertion.
\end{proof}

\subsection{Terminating hypergeometric formulas for the \texorpdfstring{\(B\)}{B}-components}
\label{subsec:LII-B-components}
In the Jacobi formulas, the simple poles in each family approach the
same point \(-\rho_J\). Combining their contributions before taking the
limit first gives a sum over the orders of that pole. The following
representation performs this sum as well: each polynomial component is
a combination of at most \(q+1\) terminating Srivastava--Daoust functions.

\Needspace{8\baselineskip}
Throughout this construction, \(\boldsymbol m\) is near the diagonal,
\(|\boldsymbol m|=|\boldsymbol n|+1\), and the integrability condition
\eqref{eq:LII-integrability} holds. Initially assume \(q\ge2\), all
\(m_h\ge1\), and pairwise distinct rates. Abbreviate
\(B_j\coloneq B_{\boldsymbol n,\boldsymbol m}^{(j),\mathrm{LII}}\), and put
\begin{equation}
 c_{i,J}\coloneq\xi_i+\rho_J,\qquad
 x_h^{(J)}\coloneq-\frac1{\rho_h-\rho_J}\quad(h\ne J),\qquad
 \kappa_J\coloneq-\frac{\prod_{i=1}^pc_{i,J}^{n_i}}
                         {\prod_{h\ne J}(\rho_h-\rho_J)^{m_h}}.
 \label{eq:LII-B-SD-constants}
\end{equation}
Fix distinct row indices \(J,j\), let
\(I_{J,j}\coloneq\{h\in\{1,\ldots,q\}:h\ne J,j\}\), and set
\(r_0\coloneq p+q-1\), \(t_0\coloneq q-2\), and
\(R\coloneq p+2q-1=r_0+t_0+2\).
Order the \(R\) summation indices as
\((\boldsymbol\eta,\boldsymbol\beta,r,s)\): the first \(p\) entries of
\(\boldsymbol\eta\) correspond to the columns, its remaining \(q-1\)
entries to the rows \(h\ne J\), and \(\boldsymbol\beta\) to
\(h\in I_{J,j}\). Row indices within each group are in increasing order.
Define the coefficient vectors
\begin{equation}
 \boldsymbol e\coloneq(\one_{r_0},\one_{t_0},1,1),\qquad
 \boldsymbol u\coloneq(\boldsymbol0_{r_0},\one_{t_0},1,1),\qquad
 \boldsymbol v\coloneq(\boldsymbol0_{r_0},\one_{t_0},0,1),
 \label{eq:LII-B-SD-vectors}
\end{equation}
where \(\boldsymbol0_r\) is the zero vector of length \(r\), and a group
of length zero is omitted. For \(h\ne J\), set
\(\widetilde m_j\coloneq m_j\) and
\(\widetilde m_h\coloneq m_h+d_h\) when \(h\in I_{J,j}\).
The individual upper parameter vector is
\[
 \boldsymbol b^{J,j}\coloneq
 \bigl((-n_i)_{i=1}^p,(\widetilde m_h)_{h\ne J},
                         (-d_h)_{h\in I_{J,j}},1\bigr).
\]
It has \(R-1\) entries; the last summation index \(s\) has no individual
upper parameter. All individual lower parameter strings are empty.
In the same ordering, the arguments are
\begin{equation}
 \boldsymbol z^{J,j}(y)\coloneq
 \left(\left(-\frac1{c_{i,J}y}\right)_{i=1}^p,
       \left(-\frac{x_h^{(J)}}y\right)_{h\ne J},
       \left(-\frac{x_h^{(J)}}y\right)_{h\in I_{J,j}},
       -\frac{x_j^{(J)}}y,1\right).
 \label{eq:LII-B-SD-parameters}
\end{equation}
For \(a>-1\) and \(N\in\mathbb N_0\), define
\begin{equation}
 \mathcal V_N^{J,j}(a;y)\coloneq\frac{y^N}{N!}
 F_{2:0;\ldots;0}^{3:1;\ldots;1;0}
 \left[\begin{array}{c}
 (-N:\boldsymbol e),(a+d_j+1:\boldsymbol u),(a+1:\boldsymbol v):
 b_1^{J,j};\ldots;b_{R-1}^{J,j};-\\
 (a+2:\boldsymbol u),(a+d_j+1:\boldsymbol v):
 -;\ldots;-
 \end{array}\,\middle|\,\boldsymbol z^{J,j}(y)\right].
 \label{eq:LII-B-SD-polynomial}
\end{equation}
The notation is that of \eqref{eq:Srivastava-Daoust-definition}.
The coupled parameter \(-N\) restricts the sum to
\(|\boldsymbol\eta|+|\boldsymbol\beta|+r+s\le N\).
After multiplication by \(y^N\), every term contains the nonnegative
power \(y^{N-|\boldsymbol\eta|-|\boldsymbol\beta|-r}\).
Consequently \eqref{eq:LII-B-SD-polynomial} defines a polynomial also at
\(y=0\), without evaluating its reciprocal arguments there.
Its lower parameters are positive. Finally, set
\(\mathcal V_{-1}^{J,j}(a;y)\coloneq0\).

\begin{proposition}[Srivastava--Daoust formulas for the \(B\)-components]
 \label{prop:LII-explicit-B-components}
 Under the preceding assumptions, for every \(j\in\{1,\ldots,q\}\)
 and any \(h\ne j\),
 \begin{equation}
 B_j(y)=\kappa_j\mathcal V_{m_j-1}^{j,h}(d_j-1;y)
 -\frac{\kappa_jd_jx_h^{(j)}}{d_j+1}\mathcal V_{m_j-2}^{j,h}(d_j;y)
 -d_j\sum_{J\ne j}\frac{\kappa_Jx_j^{(J)}}{d_J+1}
                       \mathcal V_{m_J-2}^{J,j}(d_J;y).
 \label{eq:LII-B-components-hypergeometric}
 \end{equation}
 This expression is independent of \(h\), belongs to
 \(\mathbb P_{m_j-1}\), and represents each component by at most \(q+1\)
 terminating functions of \(p+2q-1\) arguments.
\end{proposition}

\begin{proof}
 The proof first reconstructs polynomial components from the poles of
 the rational transform, and then evaluates the resulting finite sums.
 Put \(P_{\boldsymbol n}(z)\coloneq\prod_i(\xi_i-z)^{n_i}\) and define
 \(\Pi_{J,K}\) by
 \begin{equation}
 -\frac{P_{\boldsymbol n}(z)}{Q_{\boldsymbol m}(z)}
 =\sum_{J=1}^q\sum_{K=0}^{m_J-1}\frac{\Pi_{J,K}}{(z+\rho_J)^{K+1}},
 \qquad
 \Pi_{J,K}=\kappa_J[w^{m_J-K-1}]
       \prod_i(1-w/c_{i,J})^{n_i}\prod_{h\ne J}(1-x_h^{(J)}w)^{-m_h}.
 \label{eq:LII-B-partial-fraction-families}
 \end{equation}
 The coefficient formula follows by expanding at \(z=-\rho_J\).

 \emph{Polynomial reconstruction.}
 Fix \(J\) and \(K\ge1\). Define
 \(h_J(w;y)\coloneq\e^{yw}\prod_{h\ne J}(1-x_h^{(J)}w)^{-d_h}\)
 near \(w=0\), and, for \(0\le r\le K\), set
 \[
 T_{J,K;r}(y)\coloneq[w^r]\left(h_J(w;y)^{-1}
       \sum_{v=0}^K\frac{[u^v]h_J(u;y)}{v-d_J-K}w^v\right).
 \]
 All denominators are nonzero. Differentiating this coefficient identity
 gives \(T_{J,K;0}=-1/(d_J+K)\) and, for \(1\le r\le K\),
 \[
 (d_J+K-r)T_{J,K;r}
 =yT_{J,K;r-1}
  +\sum_{h\ne J}d_h\sum_{v=0}^{r-1}(x_h^{(J)})^{r-v}T_{J,K;v}.
 \]
 Define \(U_{j;J,0}\coloneq\delta_{j,J}\), and, for \(K\ge1\),
 \[
 U_{j;J,K}\coloneq d_j\sum_{r=0}^{K-1}(x_j^{(J)})^{K-r}T_{J,K;r}
 \quad(j\ne J),\qquad
 U_{J;J,K}\coloneq-d_JT_{J,K;K}
              =-yT_{J,K;K-1}-\sum_{j\ne J}U_{j;J,K}.
 \]
 To verify the reconstruction, put
 \(S_{J,K}(z;y)\coloneq\sum_{r=0}^{K-1}T_{J,K;r}(y)(z+\rho_J)^{r-K}\).
 The coefficient recurrence cancels all poles of order at least two in
 \[
 \frac1{(z+\rho_J)^{K+1}}-\partial_zS_{J,K}
       -\left(y-\sum_h\frac{d_h}{z+\rho_h}\right)S_{J,K}.
 \]
 Its residues at \(-\rho_j\), \(j\ne J\), are \(U_{j;J,K}\).
 The coefficient of \(z^{-1}\) at infinity gives \(U_{J;J,K}\);
 there is no polynomial part. Multiplying by
 \(\gamma(z)=\prod_h(z+\rho_h)^{-d_h}\) therefore yields
 \begin{equation}
 \frac{\gamma(z)}{(z+\rho_J)^{K+1}}
 =\sum_j U_{j;J,K}(y)\frac{\gamma(z)}{z+\rho_j}
   +\partial_z\bigl(\gamma(z)S_{J,K}(z;y)\bigr)
   +y\gamma(z)S_{J,K}(z;y).
 \label{eq:LII-B-scalar-reduction-identity}
 \end{equation}
 Invert the Laplace transform with respect to a variable \(v>0\),
 treating \(y\) as a parameter, and then set \(v=y\). The last two
 terms cancel, since differentiation in \(z\) corresponds to
 multiplication by \(-v\). Using
 \eqref{eq:LII-B-Laplace} and
 \eqref{eq:LII-B-partial-fraction-families} proves
 \begin{equation}
 B_j(y)=\sum_{J=1}^q\sum_{K=0}^{m_J-1}\Pi_{J,K}U_{j;J,K}(y).
 \label{eq:LII-B-pole-reconstruction}
 \end{equation}
 The remaining steps evaluate the Taylor coefficients and all sums
 in \eqref{eq:LII-B-pole-reconstruction}.

 \emph{The two-row calculation.}
 For two rows, write \(a\coloneq d_J\), \(b\coloneq d_j\), and
 \(x\coloneq x_j^{(J)}\), and
 denote the component with \(j\ne J\) by \(u_K(a,b,x;y)\).
 The preceding coefficient identities give
 \[
 (a+K+1)u_{K+1}=\{y+x(a+b+K)\}u_K-xyu_{K-1}\quad(K\ge1),
 \qquad u_0=0,\quad u_1=-\frac{bx}{a+1}.
 \]
 On setting \(z\coloneq y/x\), the solution takes the form
 \[
 u_K=-\frac{bx^K}{(a+1)_K}P_{K-1}(z),\qquad
 P_N(z)=(z+a+b+N)P_{N-1}(z)-z(a+N)P_{N-2}(z),
 \]
 with the recurrence valid for \(N\ge2\), \(P_0=1\), and \(P_1=z+a+b+1\).
 Thus \(P_N(z)=z^N\mathscr C_N(-b;z,a+1)\), where
 \(\mathscr C_N(X;\tau,\gamma)\) are the associated Charlier polynomials
 normalized by \(\mathscr C_{-1}=0\), \(\mathscr C_0=1\), and
 \(\tau\mathscr C_{N+1}=(N+\gamma+\tau-X)\mathscr C_N-(N+\gamma)\mathscr C_{N-1}\).
 Expanding their terminating \({}_3F_2(1)\) representation
 \cite[Section~3.1]{Ahbli2021Associated} gives
 \begin{equation}
 u_K(a,b,x;y)=-\frac{bx}{(a+1)(K-1)!}\mathcal K_{K-1}(a,b;x,y),
 \qquad K\ge1,
 \label{eq:LII-B-two-row-evaluation}
 \end{equation}
 where the finite polynomial is
 \[
 \mathcal K_N(a,b;x,y)\coloneq
 \sum_{r+s\le N}
 \frac{(-N)_{r+s}(a+b+1)_{r+s}(a+1)_s}
      {(a+2)_{r+s}(a+b+1)_s}
 \frac{(-x)^ry^{N-r}}{s!}.
 \]
 Equivalently, for \(y\ne0\),
 \begin{equation}
 \mathcal K_N(a,b;x,y)=y^N F_{1:0;1}^{2:1;1}
 \left[\begin{array}{c}-N,a+b+1:1;a+1\\a+2:-;a+b+1\end{array}
 \,\middle|-\frac{x}{y},1\right].
 \label{eq:LII-B-two-row-KdF}
 \end{equation}
 To see the identification directly, expand the terminating
 \({}_3F_2\) and use
 \((a+b+1)_r(a+b+1+r)_s=(a+b+1)_{r+s}\) and
 \((-N)_r(-N+r)_s=(-N)_{r+s}\).
 The individual parameter \(1\) cancels \(r!\).
 All identities are polynomial identities in \(y\);
 the displayed denominators also permit \(a>-1\), which will be
 needed for the diagonal component.

 \emph{Adding the other rows.}
 For fixed \(J,j\), expand the factors of \(h_J\) with
 \(h\in I_{J,j}\) and their reciprocals. If
 \(\boldsymbol\alpha,\boldsymbol\beta\in\mathbb N_0^{q-2}\), direct
 multiplication of the two finite Taylor expansions gives
 \begin{equation}
 U_{j;J,K}(y)=
 \sum_{|\boldsymbol\alpha|+|\boldsymbol\beta|\le K-1}
 \left[\prod_{h\in I_{J,j}}
 \frac{(d_h)_{\alpha_h}(-d_h)_{\beta_h}
       (x_h^{(J)})^{\alpha_h+\beta_h}}{\alpha_h!\beta_h!}\right]
 u_{K-|\boldsymbol\alpha|-|\boldsymbol\beta|}
       (d_J+|\boldsymbol\beta|,d_j,x_j^{(J)};y).
 \label{eq:LII-B-additional-rows}
 \end{equation}
 Indeed, a coefficient of order \(|\boldsymbol\alpha|\) in \(h_J\)
 changes the denominator \(v-d_J-K\) to
 \(v_0-(d_J+K-|\boldsymbol\alpha|)\).
 After extracting a coefficient of order \(|\boldsymbol\beta|\)
 from \(h_J^{-1}\), this is precisely the denominator for the
 two-row polynomial on the right.

 \emph{Combining the finite sums.}
 First consider \(j\ne J\), so the term with \(K=0\) vanishes. Insert
 \eqref{eq:LII-B-two-row-evaluation} into
 \eqref{eq:LII-B-additional-rows}, and expand the coefficient
 \(\Pi_{J,K}\) in \eqref{eq:LII-B-partial-fraction-families}.
 If \(L\coloneq m_J-K-1\), the identity
 \[
 \frac{(-N+L)_v}{(N-L)!}
 =\frac{(-1)^L(-N)_{L+v}}{N!},
 \qquad 0\le L\le N,\quad 0\le v\le N-L,
 \]
 combines the sum over \(K\) with all remaining indices, with
 \(N=m_J-2\). For each \(h\in I_{J,j}\), the index from
 \(\Pi_{J,K}\) and \(\alpha_h\) have the same argument and occur
 elsewhere only through their sum. Vandermonde's coefficient identity
 \[
 \sum_{u+v=t}\frac{(m_h)_u(d_h)_v}{u!\,v!}
 =\frac{(m_h+d_h)_t}{t!}
 \]
 combines them into one index. This explains the parameter
 \(\widetilde m_h=m_h+d_h\) and reduces the number of indices to
 \(R=p+2q-1\).

 In the ordering of \eqref{eq:LII-B-SD-vectors}, the coupled
 Pochhammer factors that remain are
 \[
 \frac{(-N)_{|\boldsymbol\eta|+|\boldsymbol\beta|+r+s}
       (a+d_j+1)_{|\boldsymbol\beta|+r+s}(a+1)_{|\boldsymbol\beta|+s}}
      {(a+2)_{|\boldsymbol\beta|+r+s}(a+d_j+1)_{|\boldsymbol\beta|+s}},
 \qquad a=d_J.
 \]
 The individual factors and arguments are exactly
 \eqref{eq:LII-B-SD-parameters}. Hence
 \begin{equation}
 \sum_{K=0}^{m_J-1}\Pi_{J,K}U_{j;J,K}(y)
 =-\frac{\kappa_Jd_jx_j^{(J)}}{d_J+1}
                  \mathcal V_{m_J-2}^{J,j}(d_J;y),\qquad j\ne J.
 \label{eq:LII-B-SD-off-diagonal}
 \end{equation}
 When \(m_J=1\), both sides vanish.

 For the diagonal contribution, fix \(h\ne J\). In the coefficient
 construction regard \(d_J=a\) as a variable, leaving the other shape
 parameters fixed. The quantities \(U_{h;J,K}(a;y)\) and
 \(U_{h;J,K+1}(a-1;y)\) use the same denominator \(v-a-K\).
 Subtracting their finite sums gives
 \[
 U_{h;J,K+1}(a-1;y)-x_h^{(J)}U_{h;J,K}(a;y)
       =d_hx_h^{(J)}T_{J,K;K}(y).
 \]
 Since \(U_{J;J,K}=-aT_{J,K;K}\), it follows that
 \[
 U_{J;J,K}(a;y)
 =-\frac{a}{d_hx_h^{(J)}}U_{h;J,K+1}(a-1;y)
       +\frac{a}{d_h}U_{h;J,K}(a;y).
 \]
 The same identity holds at \(K=0\), using \(u_0=0\).
 Applying the preceding summation calculation to these two terms gives
 the first two terms of \eqref{eq:LII-B-components-hypergeometric}.
 Together with \eqref{eq:LII-B-SD-off-diagonal}, this proves the formula
 and its independence of \(h\).

 Finally, \(\deg\mathcal V_N^{J,j}\le N\) follows from its finite
 expansion. The first term has degree at most \(m_j-1\), the second
 at most \(m_j-2\), and every term with \(J\ne j\) has degree at most
 \(m_J-2\le m_j-1\), by near diagonality. The reconstruction has
 the prescribed Laplace transform, so these are the normalized
 polynomial components of the limiting \(B\)-form.
\end{proof}

If a near-diagonal \(\boldsymbol m\) has a zero entry, all its entries
are \(0\) or \(1\). Only simple poles occur, and the same reconstruction
immediately gives
\begin{equation}
 B_j(y)=
 \begin{cases}
 -\displaystyle\frac{\prod_i(\xi_i+\rho_j)^{n_i}}
                   {\prod_{\substack{h\ne j\\m_h=1}}(\rho_h-\rho_j)},
       &m_j=1,\\[6pt]
 0,&m_j=0.
 \end{cases}
 \label{eq:LII-B-zero-row-indices}
\end{equation}
Only the rates with \(m_h=1\) need be distinct in this case.

\begin{remark}[Confluence and the number of hypergeometric terms]
 The sum over the coalescing Jacobi poles in each family becomes
 \(\sum_K\Pi_{J,K}U_{j;J,K}\) in
 \eqref{eq:LII-B-pole-reconstruction}.
 Formula \eqref{eq:LII-B-components-hypergeometric} includes this
 sum in the parameters of a single Srivastava--Daoust function for
 \(J\ne j\), and two such functions for \(J=j\).
 Thus the number of hypergeometric terms is bounded by \(q+1\),
 independently of the degrees. For two rows, the intermediate
 polynomials are the bivariate Kamp\'e de F\'eriet polynomials
 \eqref{eq:LII-B-two-row-KdF}; after summation over the pole orders,
 the full components use Srivastava--Daoust functions of \(p+3\)
 arguments. In every case the components themselves are univariate
 polynomials in \(y\).
\end{remark}

For the one-row reduction, take \(q=1\) and write
\(d\coloneq d_1\), \(\rho\coloneq\rho_1\),
\(c_i\coloneq\rho+\xi_i\), and \(N\coloneq|\boldsymbol n|\).
Let \(L_{\boldsymbol n}^{(i)}(y;\boldsymbol c,d)\) denote the classical
type-I components of multiple Laguerre of the second kind with the normalization of
\cite[Equation~(20)]{BranquinhoDiazFoulquieManas2025Classical}, and let
\(L_{\boldsymbol n}(y;\boldsymbol c,d)\) be the monic type-II polynomial
in equation~(19) of that reference.

\begin{corollary}[Reduction to classical multiple Laguerre of the second kind]
	\label{cor:q1-LII-reduction}
	Under the distinctness and integrability hypotheses of
	Theorem~\ref{thm:Jacobi-LII-form-confluence}, with \(q=1\),
	\begin{equation}
		\label{eq:q1-classical-LII-weights}
		F_1(y)\e^{-\xi_i y}
		=
		\frac{y^d\e^{-c_i y}}{\Gamma(d+1)}.
	\end{equation}
	Up to the common factor \(\Gamma(d+1)^{-1}\), these are precisely the
	classical multiple Laguerre weights of the second kind
	\cite[Section~3.3]{VanAsscheCoussement2001}
	\cite[Equations~(19)--(21)]{BranquinhoDiazFoulquieManas2025Classical}.
	In the \(A\)-balance \(m_1=N-1\), the
	components satisfy, for \(n_i\ge1\),
	\begin{equation}
		\label{eq:q1-LII-A-classical-reduction}
		A_{\boldsymbol n,(N-1)}^{(i),\mathrm{LII}}(y)
		=
		-
		\frac{\Gamma(d+N)}{\prod_{k=1}^pc_k^{n_k}}
		L_{\boldsymbol n}^{(i)}(y;\boldsymbol c,d).
	\end{equation}
	In the \(B\)-balance \(m_1=N+1\),
	\begin{equation}
		\label{eq:q1-LII-B-Phi2}
		B_{\boldsymbol n,(N+1)}^{(1),\mathrm{LII}}(y)
		=
		(-1)^{N+1}
		\Phi_2^{(p)}
		\left(
		-n_1,\ldots,-n_p;
		d+1;
		c_1y,\ldots,c_py
		\right)
		=
		-
		\frac{\prod_{i=1}^pc_i^{n_i}}{(d+1)_N}
		L_{\boldsymbol n}(y;\boldsymbol c,d).
	\end{equation}
\end{corollary}

\begin{proof}
	Formula \eqref{eq:q1-LII-A-classical-reduction} follows by specializing
	\eqref{eq:LII-A-multiple-KdF} to \(q=1\) and comparing the normalizing type-I
	moment.  For the \(B\)-form, \eqref{eq:LII-B-Laplace} and
	\eqref{eq:q1-classical-LII-weights} give
	\[
		\Lap\left[
		F_1(y)
		\Phi_2^{(p)}
		\left(-\boldsymbol n;d+1;c_1y,\ldots,c_py\right)
		\right](z)
		=
		\frac{\prod_i(z-\xi_i)^{n_i}}
		{(z+\rho)^{d+N+1}}.
\]
	Multiplication by \((-1)^{N+1}\) proves the first equality in
	\eqref{eq:q1-LII-B-Phi2}; the second is the normalization in
	\cite[Equation~(19)]{BranquinhoDiazFoulquieManas2025Classical}.
\end{proof}

\begin{remark}
	When \(q=1\), the shifted row \(F_1\) remains meaningful for \(d>-1\),
	which is the usual scalar Laguerre parameter range. When \(q>1\), the
	assumptions \(d_h>0\) ensure that the base convolution \(F_0\) is a
	function. If two rates \(\rho_j\) coincide, the corresponding rows
	coincide; if two \(\xi_i\)'s coincide, the corresponding columns coincide.
\end{remark}

\section{Hermite systems and their three confluence routes}
\label{subsec:Laguerre-to-Hermite-confluence}

The Hermite system studied in this section has \(q\) row weights: one
Gaussian--gamma convolution \(F\) and \(q-1\) functions
\(K_{\rho_h,1}F\). Put \(r\coloneq q-1\). In the notation of
Section~\ref{sec:source-families}, its first-kind Laguerre source is the
case \(s=1\), so its last row is the base weight \(w_0^{\mathrm L}\).
This choice will remain fixed throughout the section; no further
derivative row is included in the system.

The construction first gives the limit from Laguerre of the first kind,
including the polynomial components, their terminating series, and
normality at every admissible near-diagonal index under the stated
separation conditions. The one-row case recovers the multiple Hermite
systems of \cite[Section~5]{BranquinhoDiazFoulquieManas2025Classical}.
The same matrix and normalized polynomial vectors are then obtained
directly from Laguerre of the second kind and from Jacobi.

The more general first-kind limit has a different outcome. If its
\(s\) Euler rows are retained, the last rows tend to
\(F,F',\ldots,F^{(s-1)}\). For \(s>1\), only finitely many admissible
indices admit a nonzero normalized \(B\)-form. Appendix~\ref{app:Hermite-derivative-limits}
proves this restriction, the exact loss of rank, and divergence of
finite-parameter components. These results describe the boundary of the
construction, separately from the family treated here.

\subsection{Rescaled variables and convergence of the weights}

For \(\rho>0\) and \(d>0\), set
\begin{equation}
	\label{eq:Hermite-gamma-convolution-operator}
	\kappa_{\rho,d}(u)
	\coloneq
	\frac{\rho^d}{\Gamma(d)}u^{d-1}\e^{-\rho u}
	\boldsymbol 1_{(0,\infty)}(u),
	\qquad
	(K_{\rho,d}f)(t)
	\coloneq
	\int_0^\infty\kappa_{\rho,d}(u)f(t+u)\,\mathrm du,
	\qquad
	K_{\rho,0}\coloneq I.
\end{equation}
The kernel \(\kappa_{\rho,d}\) has integral one. For \(d,e>0\),
\[
	\int_0^u v^{d-1}(u-v)^{e-1}\,\mathrm dv
	=
	u^{d+e-1}\frac{\Gamma(d)\Gamma(e)}{\Gamma(d+e)},
\]
and therefore
\[
	K_{\rho,d}K_{\rho,e}=K_{\rho,d+e},
	\qquad
	K_{\rho,1}=\rho J_\rho,
	\qquad
	(J_\rho f)(t)=\int_0^\infty\e^{-\rho u}f(t+u)\,\mathrm du.
\]
Thus \(K_{\rho,d}\) is a normalized one-sided additive convolution; the
semigroup identity remains valid when \(d=0\) or \(e=0\), because
\(K_{\rho,0}=I\).
More precisely, if \(\check\kappa_{\rho,d}(u)\coloneq
\kappa_{\rho,d}(-u)\), then
\(K_{\rho,d}f=\check\kappa_{\rho,d}*_+f\), with convolution on
\(\mathbb R\). The reflection accounts for the argument \(t+u\)
in the defining integral.

Throughout this subsection, \(\tau\to+\infty\), and the parameter families
are assumed to satisfy, for all sufficiently large \(\tau\),
\[
	a_\nu(\tau)>-1,
	\qquad 1\le\nu\le r+1,
	\qquad
	b_h(\tau)>a_h(\tau),
	\qquad 1\le h\le r.
\]
For the remaining gamma factor, assume that
\[
 a_q(\tau)+1=\tau+\sqrt{2\tau}\,\widehat a_*+\mathrm O(1),
 \qquad \tau\to+\infty.
\]
Set
\[
	\varepsilon_\tau\coloneq\sqrt{\frac{2}{\tau}},
	\qquad
	c_\tau\coloneq \tau,
	\qquad
	x_\tau(t)\coloneq c_\tau(1+\varepsilon_\tau t).
\]
Thus \(t=(x-c_\tau)/(c_\tau\varepsilon_\tau)\): the change of variable subtracts
\(c_\tau\) and divides by \(c_\tau\varepsilon_\tau\). The new variable ranges over
\[
	I_\tau
	\coloneq
	x_\tau^{-1}\bigl((0,\infty)\bigr)
	=
	(-\varepsilon_\tau^{-1},\infty).
\]
For \(1\le h\le r\), assume that
\[
	\varepsilon_\tau(a_h(\tau)+1)\xrightarrow[\tau\to+\infty]{}\rho_h>0,
	\qquad
	b_h(\tau)-a_h(\tau)\xrightarrow[\tau\to+\infty]{} d_h\ge0.
\]
Set
\[
	P_{h,\tau}\coloneq a_h(\tau)+1,
	\qquad
	D_{h,\tau}\coloneq b_h(\tau)-a_h(\tau),
	\qquad
	u_{h,\tau}\coloneq\varepsilon_\tau P_{h,\tau}.
\]
Denote the corresponding weights by \(w_{0,\tau}\), \(w_{h,\tau}\), and
\(v_{1,\tau}=w_{0,\tau}\), and define
\[
	Z_\tau
	\coloneq
	\mathcal M[w_{0,\tau}](1)
	=
	\frac{\prod_{\nu=1}^{r+1}\Gamma(a_\nu(\tau)+1)}
	{\prod_{h=1}^r\Gamma(b_h(\tau)+1)},
	\qquad
	Z_{h,\tau}\coloneq\frac{Z_\tau}{b_h(\tau)+1}.
\]
For \(a\in\mathbb R\), define the Gaussian weight
\[
	\phi_a(t)
	\coloneq
	\pi^{-1/2}\e^{-(t-a)^2}.
\]
The limiting base weight is
\[
	F
	\coloneq
	K_{\rho_1,d_1}\cdots K_{\rho_r,d_r}
	\phi_{\widehat a_*}.
\]

The bilateral Laplace transform integrates over the whole real line:
\[
 L_F(z)\coloneq\int_{-\infty}^{+\infty}\e^{zt}F(t)\,\mathrm dt,
\]
whenever the integral converges absolutely. Throughout this section the
kernel is \(\e^{zt}\), in contrast to the kernel \(\e^{-zy}\) used for
the Laplace transform on \((0,\infty)\) in the construction of Laguerre of the second kind.
With this convention, differentiation of \(F\) gives \(-zL_F(z)\),
whereas multiplication of \(F(t)\) by \(t\) gives \(L_F'(z)\).

\begin{lemma}[Laplace transform of the limiting density]
	The bilateral Laplace transform of \(F\) is
	\begin{equation}
		\label{eq:canonical-Hermite-LF}
		L_F(z)
		=
		\exp\!\left(\widehat a_*z+\frac{z^2}{4}\right)
		\prod_{h=1}^r
		\left(\frac{\rho_h}{\rho_h+z}\right)^{d_h}.
	\end{equation}
	For \(r>0\), the transform is holomorphic in
	\(\operatorname{Re}z>-\min_h\rho_h\), with the powers defined by the
	branch that is positive on the real axis to the right of every
	\(-\rho_h\). For \(r=0\), it is entire.
\end{lemma}

\begin{proof}
	Completing the square gives
	\[
		\int_{\mathbb R}\e^{zt}\phi_{\widehat a_*}(t)\,\mathrm dt
		=
		\exp\!\left(\widehat a_*z+\frac{z^2}{4}\right).
\]
	If
	\(\e^{\operatorname{Re}z\,t}f(t)\in L^1(\mathbb R)\) and
	\(\operatorname{Re}(\rho+z)>0\), both iterated integrals are absolutely
	convergent. Fubini's theorem and the definition of \(K_{\rho,d}\) then give
	\[
		\int_{\mathbb R}\e^{zt}(K_{\rho,d}f)(t)\,\mathrm dt
		=
		\left(\frac{\rho}{\rho+z}\right)^d
		\int_{\mathbb R}\e^{zt}f(t)\,\mathrm dt.
\]
	Applying this identity in the common half-plane of absolute convergence to every
	factor in the definition of \(F\) proves
	\eqref{eq:canonical-Hermite-LF} and its stated domain of holomorphy.
\end{proof}

The normalized Mellin transform used in the next estimate is
	\[
		Q_{0,\tau}(z)
		\coloneq
		\frac{
			c_\tau^{-z/\varepsilon_\tau}
			\mathcal M[w_{0,\tau}](1+z/\varepsilon_\tau)
		}{Z_\tau}.
\]
For all sufficiently large \(\tau\), powers involving the moving parameters
are defined on \(\operatorname{Re}z>-u_{h,\tau}\) by
	\[
		\left(\frac{u_{h,\tau}}{u_{h,\tau}+z}\right)^{D_{h,\tau}}
		\coloneq
		\exp\!\left(
			D_{h,\tau}\bigl[\log u_{h,\tau}-\log(u_{h,\tau}+z)\bigr]
		\right),
\]
	where \(\log(u_{h,\tau}+z)\) is real for real
	\(z>-u_{h,\tau}\).  Set
	\begin{equation}
		\label{eq:canonical-Hermite-Omega-rate}
		\Omega_\tau
		\coloneq
		\varepsilon_\tau+
		\sum_{h=1}^r
		\left(
			|u_{h,\tau}-\rho_h|+|D_{h,\tau}-d_h|
		\right).
	\end{equation}

\begin{lemma}[Uniform asymptotics of the normalized Mellin transform]
	\label{lem:canonical-Hermite-uniform-Gamma-quotient}
	If \(K\) is a compact subset of
	\(\operatorname{Re}z>-\min_h\rho_h\) when \(r>0\), or a compact subset
	of \(\mathbb C\) when \(r=0\), then, uniformly for \(z\in K\),
	\begin{equation}
		\label{eq:canonical-Hermite-moving-parameter-rate}
		Q_{0,\tau}(z)
		=
		\exp\!\left(\widehat a_*z+\frac{z^2}{4}\right)
		\prod_{h=1}^r
		\left(\frac{u_{h,\tau}}{u_{h,\tau}+z}\right)^{D_{h,\tau}}
		\bigl(1+\mathrm O_K(\varepsilon_\tau)\bigr)\qquad (\tau\to+\infty,\ \varepsilon_\tau\downarrow0).
\end{equation}
	For every such \(K\), there is \(C_K>0\) such that
	\begin{equation}
		\label{eq:canonical-Hermite-Q0-uniform-rate}
		\sup_{z\in K}|Q_{0,\tau}(z)-L_F(z)|
		\le C_K\Omega_\tau
	\end{equation}
	for all sufficiently large \(\tau\).  If \(m\in\mathbb N_0\) and \(K'\)
	is a compact subset of the same domain, then
	\[
		\sup_{z\in K'}
		\left|
			\partial_z^mQ_{0,\tau}(z)-\partial_z^mL_F(z)
		\right|
		\le C_{K',m}\Omega_\tau.
\]
\end{lemma}

\begin{proof}
	Put
	\[
		\alpha_\tau\coloneq a_q(\tau)+1
		=
		\tau+\sqrt{2\tau}\,\widehat a_*+\delta_\tau,
		\qquad
		\delta_\tau=\mathrm O(1),\qquad \tau\to+\infty.
\]
	The Mellin formula \eqref{eq:w0-Laguerre-Mellin} gives the exact
	factorization
	\begin{equation}
		Q_{0,\tau}(z)
		=
		\frac{
			\Gamma(\alpha_\tau+z/\varepsilon_\tau)
		}{
			\Gamma(\alpha_\tau)\tau^{z/\varepsilon_\tau}
		}
		\prod_{h=1}^r
		\frac{
			\Gamma(P_{h,\tau}+z/\varepsilon_\tau)
			\Gamma(P_{h,\tau}+D_{h,\tau})
		}{
			\Gamma(P_{h,\tau})
			\Gamma(P_{h,\tau}+D_{h,\tau}+z/\varepsilon_\tau)
		}.
		\label{eq:canonical-Hermite-Q0-exact-factorization}
	\end{equation}

	The gamma factor and the beta product in
	\eqref{eq:canonical-Hermite-Q0-exact-factorization} are estimated
	separately. Let
	\(\psi(x)\coloneq(\log\Gamma(x))'\). As \(|x|\to\infty\), uniformly
	on each fixed sector \(|\arg x|\le\pi-\delta\),
	\begin{equation}
		\label{eq:canonical-Hermite-psi-expansions}
		\psi(x)=\log x-\tfrac1{2x}+\mathrm O(x^{-2}),\;
		\psi'(x)=\tfrac1x+\tfrac1{2x^2}+\mathrm O(x^{-3}),\;
		\psi''(x)=-x^{-2}+\mathrm O(x^{-3}),\;
		\psi'''(x)=\mathrm O(x^{-3}).
	\end{equation}
	Suppose that \(c\tau\le X\le C\tau\) for fixed constants \(c,C>0\) and
	all sufficiently large \(\tau\), and that
	\(|y|\le C_K\sqrt \tau\), with \(C_K\) independent of \(\tau\).
	Taylor's formula with integral remainder gives
	\begin{equation}
		\label{eq:canonical-Hermite-log-Gamma-Taylor}
			\log\Gamma(X+y)-\log\Gamma(X)
			={}
			y\psi(X)+\frac{y^2}{2}\psi'(X)
			+\frac{y^3}{6}\psi''(X)
			+
			\frac16\int_0^y(y-v)^3\psi'''(X+v)\,\mathrm dv.
	\end{equation}
	For \(y=z/\varepsilon_\tau\), the integration segment remains within the
	sector above, and its length is at most \(C_K\sqrt \tau\). Since
	\(|\psi'''(X+v)|\le C'_K \tau^{-3}\) there for large \(\tau\), the integral
	remainder is bounded by \(C''_K|y|^4 \tau^{-3}\le C'''_K \tau^{-1}\).  Substitution of
	\eqref{eq:canonical-Hermite-psi-expansions} gives
	\[
		\log\Gamma(X+y)-\log\Gamma(X)
		=
		y\log X+\frac{y^2-y}{2X}
		+\frac{3y^2-2y^3}{12X^2}
		+\mathrm O_K(\tau^{-1})\qquad (\tau\to+\infty).
\]
 Taking \(X=\alpha_\tau\), expanding
 \(\log(\alpha_\tau/\tau)\), \(1/\alpha_\tau\), and
 \(1/\alpha_\tau^2\), and using \(\varepsilon_\tau=\sqrt{2/\tau}\), gives
 \begin{equation}
 \log\frac{\Gamma(\alpha_\tau+z/\varepsilon_\tau)}
               {\Gamma(\alpha_\tau)\tau^{z/\varepsilon_\tau}}
 =\widehat a_*z+\frac{z^2}{4}
   +\varepsilon_\tau R_{G,\tau}(z)
   +\mathrm O_K(\varepsilon_\tau^2),\qquad \tau\to+\infty,
 \label{eq:canonical-Hermite-Gaussian-Gamma-expansion}
 \end{equation}
 where
 \[
 R_{G,\tau}(z)\coloneq
 \frac z2\left(\delta_\tau-\widehat a_*^{\,2}-\frac12\right)
 -\frac{\widehat a_*z^2}{4}-\frac{z^3}{24}.
 \]
 Since \(\delta_\tau\) is bounded, these polynomials are uniformly
 bounded on each fixed compact set.

	For a factor in the beta product of
	\eqref{eq:canonical-Hermite-Q0-exact-factorization}, write
	\(P=P_{h,\tau}\), \(D=D_{h,\tau}\), and \(y=z/\varepsilon_\tau\).  The
	fundamental theorem of calculus gives the exact identity
	\[
		\log
		\frac{\Gamma(P+y)\Gamma(P+D)}
		{\Gamma(P)\Gamma(P+D+y)}
		=
		\int_0^D
		\bigl[\psi(P+v)-\psi(P+y+v)\bigr]\,\mathrm dv.
\]
	The variable \(D\) remains in a fixed compact subset of
	\([0,\infty)\).  Applying the first line of
	\eqref{eq:canonical-Hermite-psi-expansions} and expanding only in the
	bounded variable \(v\) gives
	\[
		D\log\frac{P}{P+y}
		+\frac{D(D-1)}2
		\left(\frac1P-\frac1{P+y}\right)
		+\mathrm O_K(P^{-2}),\qquad P\to+\infty.
\]
	Since \(P=u_{h,\tau}/\varepsilon_\tau\), exponentiation yields
	\begin{multline}
		\frac{
			\Gamma(P_{h,\tau}+z/\varepsilon_\tau)
			\Gamma(P_{h,\tau}+D_{h,\tau})
		}{
			\Gamma(P_{h,\tau})
			\Gamma(P_{h,\tau}+D_{h,\tau}+z/\varepsilon_\tau)
		}
		\\*
		=
		\left(\frac{u_{h,\tau}}{u_{h,\tau}+z}\right)^{D_{h,\tau}}
		\left(
			1+
			\frac{\varepsilon_\tau D_{h,\tau}(D_{h,\tau}-1)}2
			\left(
				\frac1{u_{h,\tau}}-\frac1{u_{h,\tau}+z}
			\right)
			+\mathrm O_K(\varepsilon_\tau^2)
		\right)\qquad (\tau\to+\infty).
		\label{eq:canonical-Hermite-beta-Gamma-expansion}
\end{multline}

	Exponentiating
	\eqref{eq:canonical-Hermite-Gaussian-Gamma-expansion} and multiplying the
	result by \eqref{eq:canonical-Hermite-beta-Gamma-expansion} proves
	\eqref{eq:canonical-Hermite-moving-parameter-rate}.  On \(K\), the map
	\[
		(u,D,z)\longmapsto
		\exp\!\left(D[\log u-\log(u+z)]\right)
\]
	has bounded first derivatives in \(u\) and \(D\) near
	\((\rho_h,d_h)\).  The mean-value theorem bounds the difference between
	the moving and limiting factors by
	\[
		C_K\bigl(|u_{h,\tau}-\rho_h|+|D_{h,\tau}-d_h|\bigr),
\]
	which proves \eqref{eq:canonical-Hermite-Q0-uniform-rate}.  Cauchy's
	integral formula on a slightly larger compact set proves the derivative
	estimates.
\end{proof}

\begin{theorem}[Hermite limit of the Laguerre row weights of the first kind]
	\label{thm:simultaneous-Laguerre-Hermite-confluence}
	The normalized row weights satisfy
	\begin{align}
		\frac{c_\tau\varepsilon_\tau}{Z_{h,\tau}}
		w_{h,\tau}(x_\tau(t))
		&\xrightarrow[\tau\to+\infty]{}
		K_{\rho_h,1}F(t),
		&&1\le h\le r,
		\label{eq:canonical-Hermite-shifted-beta-limit}
		\\
		\frac{c_\tau\varepsilon_\tau}{Z_\tau}
		w_{0,\tau}(x_\tau(t))
		&\xrightarrow[\tau\to+\infty]{}
		F(t).
		&&
		\label{eq:canonical-Hermite-base-limit}
	\end{align}
 Both limits hold weakly as measures and in every fixed polynomial moment.
\end{theorem}

\begin{proof}
	Set
	\[
		\mathscr H
		\coloneq
		\begin{cases}
			\mathbb C,&r=0,\\
			\{z\in\mathbb C:\operatorname{Re}z>-\min_h\rho_h\},&r>0,
		\end{cases}
\]
	and extend every centered row by zero outside \(I_\tau\). For the base row,
	put
	\[
		f_{0,\tau}(t)
		\coloneq
		\frac{c_\tau\varepsilon_\tau}{Z_\tau}w_{0,\tau}(x_\tau(t))
\]
	and define its normalized transform by
	\[
		Q_{0,\tau}(z)
		\coloneq
		\int_{I_\tau}
		(1+\varepsilon_\tau t)^{z/\varepsilon_\tau}f_{0,\tau}(t)\,\mathrm dt
		=
		\frac{
			c_\tau^{-z/\varepsilon_\tau}
			\mathcal M[w_{0,\tau}](1+z/\varepsilon_\tau)
		}{Z_\tau}.
\]
	Lemma~\ref{lem:canonical-Hermite-uniform-Gamma-quotient} proves, locally
	uniformly in \(\mathscr H\), that
	\begin{equation}
		\label{eq:canonical-Hermite-Q0-qualitative-limit}
		Q_{0,\tau}(z)
		\xrightarrow[\tau\to+\infty]{}
		L_F(z).
	\end{equation}

	The transforms of the displayed rows are computed from the Mellin
	identities. Define
	\[
		f_{h,\tau}(t)
		\coloneq
		\frac{c_\tau\varepsilon_\tau}{Z_{h,\tau}}w_{h,\tau}(x_\tau(t)),
		\qquad 1\le h\le r,
\]
 and put \(f_{q,\tau}\coloneq f_{0,\tau}\). The Mellin identity
 \(\mathcal M[w_{h,\tau}](z)=\mathcal M[w_{0,\tau}](z)/(z+b_h(\tau))\)
 gives
 \begin{equation}
 Q_{h,\tau}(z)\coloneq
 \int_{I_\tau}(1+\varepsilon_\tau t)^{z/\varepsilon_\tau}
 f_{h,\tau}(t)\,\mathrm dt
 =\frac{\varepsilon_\tau(b_h(\tau)+1)}
        {\varepsilon_\tau(b_h(\tau)+1)+z}Q_{0,\tau}(z),
 \qquad Q_{q,\tau}=Q_{0,\tau}.
 \label{eq:canonical-Hermite-shifted-row-transform}
 \end{equation}
	Because
	\(\varepsilon_\tau(b_h(\tau)+1)\xrightarrow[\tau\to+\infty]{}\rho_h\), these transforms converge
	locally uniformly in \(\mathscr H\) to
	\[
		\frac{\rho_h}{\rho_h+z}L_F(z)
		\quad\textnormal{and}\quad
		L_F(z),
\]
	which are the transforms of \(K_{\rho_h,1}F\) and
	\(F\), respectively.

	The transform limits imply convergence of ordinary moments through the
	following general identity. Let \(f_\tau\) be any locally integrable centered
	form, possibly signed and not normalized, and put
	\[
		Q_\tau(z)=\int_{I_\tau}
		(1+\varepsilon_\tau t)^{z/\varepsilon_\tau}f_\tau(t)\,\mathrm dt.
\]
	Whenever the integrals defining \(Q_\tau(z+\nu\varepsilon_\tau)\),
	\(0\le\nu\le m\), and the integral on the left below converge absolutely,
	the binomial theorem gives
	\begin{equation}
		\label{eq:canonical-Hermite-finite-difference-moments}
			\int_{I_\tau}
			t^m(1+\varepsilon_\tau t)^{z/\varepsilon_\tau}
			f_\tau(t)\,\mathrm dt
			=
			\varepsilon_\tau^{-m}
			\sum_{\nu=0}^m
			(-1)^{m-\nu}\binom m\nu
			Q_\tau(z+\nu\varepsilon_\tau).
	\end{equation}
	To see this directly, insert the integral defining \(Q_\tau\) into the
	sum; the factor that remains inside the integral is
	\[
		\varepsilon_\tau^{-m}
		\bigl((1+\varepsilon_\tau t)-1\bigr)^m=t^m.
\]
	For \(\varepsilon>0\), define the forward difference by
	\(\Delta_\varepsilon Q(z)\coloneq Q(z+\varepsilon)-Q(z)\).
	Its \(m\)-fold iterate is
	\[
	\Delta_\varepsilon^m Q(z)=\sum_{\nu=0}^m
	(-1)^{m-\nu}\binom m\nu Q(z+\nu\varepsilon),
	\qquad \Delta_\varepsilon^0Q\coloneq Q.
	\]
	The finite difference on the right also has the exact integral form
	\begin{equation}
		\label{eq:canonical-Hermite-finite-difference-integral}
		\varepsilon^{-m}\Delta_\varepsilon^mQ(z)
		=
		\int_{[0,1]^m}
		Q^{(m)}
		\left(z+\varepsilon(u_1+\cdots+u_m)\right)
		\,\mathrm d\boldsymbol u.
	\end{equation}
	It follows by applying the fundamental theorem of calculus \(m\) times.
	Here \(\mathrm d\boldsymbol u\coloneq\mathrm du_1\cdots\mathrm du_m\);
	for \(m=0\), the right-hand side is understood as \(Q(z)\).
	Local uniform holomorphic convergence, followed by Cauchy's formula on a
	slightly larger compact subset of \(\mathscr H\), gives local uniform
	convergence of every derivative.  Taking \(z=0\) in
	\eqref{eq:canonical-Hermite-finite-difference-moments} and
	\eqref{eq:canonical-Hermite-finite-difference-integral} proves convergence
	of all fixed polynomial moments.

	The base row \(f_{0,\tau}\) and the shifted rows \(f_{h,\tau}\),
	\(1\le h\le r\), are nonnegative and satisfy
	\(\int_{\mathbb R}f_{j,\tau}(t)\,\mathrm dt=1\) for \(0\le j\le r\).
	Convergence of their
	second moments gives a constant \(C\), independent of \(\tau\), such that
	\[
		\int_{|t|>R}f_{j,\tau}(t)\,\mathrm dt\le\frac{C}{R^2},
		\qquad j\in\{0,1,\ldots,r\}.
\]
	Thus every sequence contains a weakly convergent subsequence. Fix a moment
	of order \(m\), and choose an even integer \(2M>m\). The elementary bound
	\[
		\int_{|t|>R}|t|^mf_{j,\tau}(t)\,\mathrm dt
		\le
		R^{m-2M}\int_{\mathbb R}t^{2M}f_{j,\tau}(t)\,\mathrm dt
\]
	shows that the \(m\)-th moment passes to every subsequential limit, because
	the last integrals converge and are therefore uniformly bounded. These
	moments are the derivatives at zero of the limiting transform already
	computed. That transform is finite on an open real interval containing
	zero, and hence its Taylor series there determines the Laplace transform
	and the positive measure uniquely. Every subsequence therefore has the
	same limit. Consequently
	\[
		f_{0,\tau}(t)\,\mathrm dt
		\xrightarrow[\tau\to+\infty]{\mathrm w}
		F(t)\,\mathrm dt,
		\qquad
		f_{h,\tau}(t)\,\mathrm dt
		\xrightarrow[\tau\to+\infty]{\mathrm w}
		K_{\rho_h,1}F(t)\,\mathrm dt,
\]
	where the superscript \(\mathrm w\) denotes weak convergence of the
	absolutely continuous measures. This proves both
	\eqref{eq:canonical-Hermite-shifted-beta-limit} and
	\eqref{eq:canonical-Hermite-base-limit}.
\end{proof}

For the matrix limit, let \(\widehat U_{j,\tau}\) denote the normalized
finite-\(\tau\) row on the left of
\eqref{eq:canonical-Hermite-shifted-beta-limit} or
\eqref{eq:canonical-Hermite-base-limit}. Denote the corresponding
limiting row by \(U_j^{\mathrm H}\): it is \(K_{\rho_j,1}F\) for
\(1\le j\le r\) and \(F\) for \(j=q\).

\begin{corollary}[Hermite limit of the mixed-type Laguerre matrix of the first kind]
	\label{cor:simultaneous-Laguerre-Hermite-matrix-confluence}
	Suppose additionally that
	\[
		\varepsilon_\tau\beta_i(\tau)\xrightarrow[\tau\to+\infty]{}\xi_i,
		\qquad
		\xi_i+\rho_h>0
		\quad
		(1\le i\le p,\ 1\le h\le r).
\]
	Then, locally uniformly for \(t\) in compact subsets of \(\mathbb R\),
	\[
		c_\tau^{-\beta_i(\tau)}x_\tau(t)^{\beta_i(\tau)}
		=
		(1+\varepsilon_\tau t)^{\beta_i(\tau)}
		\xrightarrow[\tau\to+\infty]{}\e^{\xi_i t}.
\]
	After the row normalizations in
	\eqref{eq:canonical-Hermite-shifted-beta-limit}--
	\eqref{eq:canonical-Hermite-base-limit} and multiplication of the
	\(i\)-th column by \(c_\tau^{-\beta_i(\tau)}\), the matrix of measures converges
	in every fixed mixed-type moment to the matrix generated by
	\begin{equation}
		\label{eq:canonical-Hermite-limiting-vectors}
		\left(
		K_{\rho_1,1}F,\ldots,K_{\rho_r,1}F,
		F
		\right)
		\quad\textnormal{and}\quad
		\left(\e^{\xi_1t},\ldots,\e^{\xi_pt}\right).
	\end{equation}
	For every fixed
	\(m\in\mathbb N_0\), \(1\le i\le p\), and \(1\le j\le r+1\),
	\begin{equation}
		\label{eq:canonical-Hermite-mixed-moment-convergence}
		\int_{-\varepsilon_\tau^{-1}}^\infty
		t^m(1+\varepsilon_\tau t)^{\beta_i(\tau)}
		\widehat U_{j,\tau}(t)\,\mathrm dt
		\xrightarrow[\tau\to+\infty]{}
		\int_{\mathbb R}
		t^m\e^{\xi_it}U_j^{\mathrm H}(t)\,\mathrm dt.
	\end{equation}
\end{corollary}

\begin{proof}
	For \(t\) in a fixed compact set,
	\[
		\beta_i(\tau)\log(1+\varepsilon_\tau t)
		=
		\varepsilon_\tau\beta_i(\tau)t+
		\mathrm O\!\left(\varepsilon_\tau^2|\beta_i(\tau)|\right)
		\xrightarrow[\tau\to+\infty]{}\xi_it,
\]
	which proves the column limit. For the mixed-type moments, take
	\(z=z_{i,\tau}\coloneq\varepsilon_\tau\beta_i(\tau)\) in
	\eqref{eq:canonical-Hermite-finite-difference-moments}. The conditions
	\(z_{i,\tau}\xrightarrow[\tau\to+\infty]{}\xi_i\) and \(\xi_i+\rho_h>0\) place all the required points
	inside one compact subset of the holomorphy domain for sufficiently large
	\(\tau\). The integral formula for finite differences in the proof of
	Theorem~\ref{thm:simultaneous-Laguerre-Hermite-confluence} then gives,
	with \(Q_j\) denoting the bilateral Laplace transform of
	\(U_j^{\mathrm H}\),
	\[
		\int_{-\varepsilon_\tau^{-1}}^\infty t^m(1+\varepsilon_\tau t)^{\beta_i(\tau)}
		\widehat U_{j,\tau}(t)\,\mathrm dt
		\xrightarrow[\tau\to+\infty]{}
		\left.\partial_z^mQ_j(z)\right|_{z=\xi_i}
		=
		\int_{\mathbb R}t^m\e^{\xi_it}U_j^{\mathrm H}(t)\,\mathrm dt,
\]
	for every normalized row \(j\). This proves
	\eqref{eq:canonical-Hermite-mixed-moment-convergence}.
\end{proof}

\subsection{Confluence of the normalized forms and their components}

For the fixed-index results below, retain the column scaling of
Corollary~\ref{cor:simultaneous-Laguerre-Hermite-matrix-confluence}, assume
that \(\xi_1,\ldots,\xi_p\) are pairwise distinct, and fix a
Laguerre-admissible index
\[
	\boldsymbol\lambda
	=(\boldsymbol n,\eta)
	\in\mathbb N_0^r\times\mathbb N_0,
	\qquad
	N=|\boldsymbol n|+\eta.
\]
Define the fixed-index error by
\begin{equation}
	\label{eq:canonical-Hermite-Theta-rate}
	\chi_\tau
	\coloneq
	\max_{1\le i\le p}
	|\varepsilon_\tau\beta_i(\tau)-\xi_i|,
	\qquad
	\Theta_\tau\coloneq\Omega_\tau+\chi_\tau.
\end{equation}

The contour and Mellin representatives used below are defined directly by
\eqref{eq:Laguerre-A-contour} and \eqref{eq:Laguerre-B-Mellin}; their
confluence does not require an AT hypothesis. If, at a fixed finite \(\tau\),
the determinant and node-separation conditions in
Proposition~\ref{prop:AT-normality-Laguerre-admissible} hold, that proposition
identifies these representatives with the unique normalized mixed-type forms.
Under the limiting rate-separation conditions, the \(A\)-form is unique
with this normalization. The \(B\)-form has a unique polynomial
decomposition of the prescribed degrees. These assertions are proved in
Corollary~\ref{cor:Hermite-main-normality}.
With the multi-indices fixed, the shorter symbols
\(\mathcal A_\tau^{\mathrm L}\) and \(\mathcal B_\tau^{\mathrm L}\) denote
these representatives at parameters \(\boldsymbol a(\tau)\),
\(\boldsymbol b(\tau)\), and \(\boldsymbol\beta(\tau)\).

\begin{lemma}[Orthogonality of the direct contour representative]
 \label{lem:canonical-Hermite-direct-A-orthogonality}
 Let \(|\boldsymbol m|=N+1\). For all sufficiently large \(\tau\),
 the representative \eqref{eq:Laguerre-A-contour} satisfies
 \[
 \int_0^\infty x^kw_{h,\tau}(x)\mathcal A_\tau^{\mathrm L}(x)\,\mathrm dx=0
 \quad(0\le k<n_h),\qquad
 \int_0^\infty x^kw_{0,\tau}(x)\mathcal A_\tau^{\mathrm L}(x)\,\mathrm dx=0
 \quad(0\le k<\eta).
 \]
 These identities hold without assuming the finite-parameter
 normality conditions in Proposition~\ref{prop:AT-normality-Laguerre-admissible}.
\end{lemma}

\begin{proof}
 The column circles lie inside the common Mellin strip for all large
 \(\tau\), by \(\xi_i+\rho_h>0\). Absolute convergence there permits
 Fubini's theorem. Write
 \(D_\tau(\upsilon)=D_{\boldsymbol\beta(\tau),\boldsymbol m}(\upsilon)\).
 For \(0\le k<n_h\), inserting the Mellin transform of \(w_{h,\tau}\)
 gives the contour integral of \(P_{h,k,\tau}/D_\tau\), where
 \[
 P_{h,k,\tau}(\upsilon)\coloneq
 \prod_{\nu=1}^{q}(\upsilon+a_\nu(\tau)+1)_k
 (\upsilon+b_h(\tau)+k+2)_{n_h-k-1}
 \prod_{u\ne h}(\upsilon+b_u(\tau)+k+1)_{n_u-k}.
 \]
 Near diagonality gives \(n_u\ge n_h-1\ge k\) for \(u\ne h\)
 and \(\eta\ge k\). Every displayed Pochhammer index is therefore
 nonnegative, and the numerator has degree
 \(|\boldsymbol n|+k-1\le N-1\).

 For the base row and \(0\le k<\eta\), the corresponding numerator is
 \[
 P_{q,k,\tau}(\upsilon)\coloneq
 \prod_{\nu=1}^{q}(\upsilon+a_\nu(\tau)+1)_k
 \prod_{u=1}^r(\upsilon+b_u(\tau)+k+1)_{n_u-k}.
 \]
 Here \(n_u\ge\eta-1\ge k\), and the degree is
 \(|\boldsymbol n|+k\le N-1\). In both cases the denominator has
 degree \(N+1\), and the contour encloses all its zeros. The rational
 function is \(\mathrm O(\upsilon^{-2})\) as
 \(|\upsilon|\to\infty\), so its contour integral vanishes by the
 residue theorem.
\end{proof}

The limiting polynomial components will be expressed in the monic Hermite
basis. These polynomials are defined by the generating function
\begin{equation}
\label{eq:monic-Hermite-generating-function}
	\e^{uz-z^2/4}
	=
	\sum_{\ell=0}^\infty H_\ell(u)\frac{z^\ell}{\ell!}.
\end{equation}

Separating even and odd powers in this generating function gives the
terminating hypergeometric expressions
\begin{equation}
 H_{2m}(u)=(-1)^m(1/2)_m\pFq{1}{1}{-m}{1/2}{u^2},
 \qquad
 H_{2m+1}(u)=(-1)^m(3/2)_m\,u\pFq{1}{1}{-m}{3/2}{u^2}.
 \label{eq:Hermite-terminating-1F1}
\end{equation}
By \eqref{eq:Hermite-terminating-1F1}, every finite Hermite sum below is
also a finite sum of terminating \({}_1F_1\) polynomials.

For the limiting contour formula, choose a positively oriented contour
\(\Sigma_\xi\) enclosing \(\xi_1,\ldots,\xi_p\) and, when \(r>0\), lying
in \(\operatorname{Re}z>-\min_{1\le h\le r}\rho_h\). The powers are
taken on the branch that is positive on the real interval containing
the \(\xi_i\)'s.

\begin{theorem}[Confluence of the canonical \(A\)-representative]
	\label{thm:canonical-Hermite-explicit-forms}
	Let \(\boldsymbol m\in\mathbb N_0^p\) satisfy
	\(|\boldsymbol m|=N+1\). Then the contour-normalized form
	\(\mathcal A_\tau^{\mathrm L}\) in \eqref{eq:Laguerre-A-contour} satisfies
	\begin{equation}
		\label{eq:canonical-Hermite-A-scaled-limit}
		Z_\tau\varepsilon_\tau^{-\eta}
		\mathcal A_\tau^{\mathrm L}(x_\tau(t))
		\xrightarrow[\tau\to+\infty]{}
		\mathcal A^{\mathrm H}_{\boldsymbol\lambda,\boldsymbol m}(t)
	\end{equation}
	locally uniformly in \(t\), where
	\begin{equation}
		\label{eq:canonical-Hermite-A-limit-contour}
		\mathcal A^{\mathrm H}_{\boldsymbol\lambda,\boldsymbol m}(t)
		=
		\frac{1}{2\pi\mathrm i}
		\oint_{\Sigma_\xi}
		\frac{
			\e^{zt}\prod_{h=1}^r(\rho_h+z)^{n_h}
		}{
			L_F(z)\prod_{i=1}^p(\xi_i-z)^{m_i}
		}\,\mathrm dz.
	\end{equation}
\end{theorem}

\begin{proof}
	Put \(z_{i,\tau}=\varepsilon_\tau\beta_i(\tau)\), and define the fixed-degree
	polynomials
	\[
		P_{h,\tau}^{[n_h]}(z)
		\coloneq
		\prod_{\nu=0}^{n_h-1}
		\bigl[z+\varepsilon_\tau(b_h(\tau)+1+\nu)\bigr],
		\qquad
		D_{i,\tau}^{[m_i]}(z)
		\coloneq
		\prod_{\nu=0}^{m_i-1}
		\bigl[z_{i,\tau}-z+\varepsilon_\tau\nu\bigr].
\]
	Empty products are one.  Choose disjoint positively oriented circles
	around the distinct points \(\xi_i\), contained in the holomorphy domain
	of \(L_F\), and let their union be \(\Sigma_\xi\).  For all sufficiently
	large \(\tau\), the circle around \(\xi_i\) contains precisely the scaled
	poles \(z_{i,\tau}+\varepsilon_\tau\nu\), \(0\le\nu<m_i\).

	The Mellin formula for \(w_{0,\tau}\) rewrites the quotient of gamma
	functions in \eqref{eq:Laguerre-A-contour} as
	\[
		\frac{
			\prod_{h=1}^r
			(\upsilon+b_h(\tau)+1)_{n_h}
		}{
			\mathcal M[w_{0,\tau}](1+\upsilon)
		}.
\]
	After the exact substitution \(\upsilon=z/\varepsilon_\tau\), cancellation of
	the powers of \(c_\tau\), and use of
	\(|\boldsymbol m|=|\boldsymbol n|+\eta+1\), one obtains
	\begin{equation}
		\label{eq:canonical-Hermite-A-exact-scaled-contour}
			Z_\tau\varepsilon_\tau^{-\eta}
			\mathcal A_\tau^{\mathrm L}(x_\tau(t))
			=
			\frac1{2\pi\mathrm i}
			\oint_{\Sigma_\xi}
			\frac{
				(1+\varepsilon_\tau t)^{z/\varepsilon_\tau}
				\prod_{h=1}^rP_{h,\tau}^{[n_h]}(z)
			}{
				Q_{0,\tau}(z)
				\prod_{i=1}^pD_{i,\tau}^{[m_i]}(z)
			}\,\mathrm dz.
	\end{equation}
	This identity displays every normalization; no asymptotic factor remains
	outside the integral.

	Uniformly on \(\Sigma_\xi\),
	\[
		P_{h,\tau}^{[n_h]}(z)\xrightarrow[\tau\to+\infty]{}(z+\rho_h)^{n_h},
		\qquad
		D_{i,\tau}^{[m_i]}(z)\xrightarrow[\tau\to+\infty]{}(\xi_i-z)^{m_i},
\]
	and
	\[
		(1+\varepsilon_\tau t)^{z/\varepsilon_\tau}
		\xrightarrow[\tau\to+\infty]{}\e^{zt}
\]
	locally uniformly for \(t\in\mathbb C\).
	Lemma~\ref{lem:canonical-Hermite-uniform-Gamma-quotient} gives
	\(Q_{0,\tau}\xrightarrow[\tau\to+\infty]{}L_F\) uniformly on the circles, and \(L_F\) does not vanish
	there.  Thus the integrand in
	\eqref{eq:canonical-Hermite-A-exact-scaled-contour} converges uniformly,
	which proves
	\eqref{eq:canonical-Hermite-A-scaled-limit} and
	\eqref{eq:canonical-Hermite-A-limit-contour}.
\end{proof}

For the residue expansion of \eqref{eq:canonical-Hermite-A-limit-contour},
introduce, for every \(i\) with \(m_i\ge1\),
	\[
		M_i\coloneq m_i-1,
		\qquad
		G_i(z)
		\coloneq
		\frac{
			\prod_{h=1}^r
			(\rho_h+z)^{n_h+d_h}
		}{
			\prod_{\substack{1\le k\le p\\k\ne i}}
			(\xi_k-z)^{m_k}
		}.
\]
\begin{corollary}[Explicit limiting \(A\)-components]
	\label{cor:canonical-Hermite-explicit-A-components}
	Under the hypotheses of
	Theorem~\ref{thm:canonical-Hermite-explicit-forms}, the residue expansion
	\[
		\mathcal A^{\mathrm H}_{\boldsymbol\lambda,\boldsymbol m}(t)
		=
		\sum_{i=1}^p
		A_{\boldsymbol\lambda,\boldsymbol m}^{(i),\mathrm H}(t)
		\e^{\xi_it}
	\]
	has polynomial components of degree exactly \(m_i-1\) for \(m_i\ge1\),
	given by
	\begin{equation}
		\label{eq:canonical-Hermite-A-components}
			A_{\boldsymbol\lambda,\boldsymbol m}^{(i),\mathrm H}(t)
			=
			\frac{
				(-1)^{m_i}
				\e^{-\widehat a_*\xi_i-\xi_i^2/4}
			}{
				M_i!\prod_{h=1}^r\rho_h^{d_h}
			}
			\sum_{\ell=0}^{M_i}
			\binom{M_i}{\ell}
			H_\ell\!\left(t-\widehat a_*-\frac{\xi_i}{2}\right)
			G_i^{(M_i-\ell)}(\xi_i).
	\end{equation}
	If \(m_i=0\), then
	\(A_{\boldsymbol\lambda,\boldsymbol m}^{(i),\mathrm H}\equiv0\).
	For the \(i\)-th finite-\(\tau\) component,
	\begin{equation}
		\label{eq:canonical-Hermite-A-component-convergence}
		Z_\tau\varepsilon_\tau^{-\eta}
		c_\tau^{\beta_i(\tau)}A_\tau^{(i),\mathrm L}(x_\tau(t))
		\xrightarrow[\tau\to+\infty]{}
		A_{\boldsymbol\lambda,\boldsymbol m}^{(i),\mathrm H}(t)
	\end{equation}
	coefficientwise. The limiting form satisfies
	\begin{align*}
		\int_{\mathbb R}
		t^kK_{\rho_h,1}F(t)
		\mathcal A^{\mathrm H}_{\boldsymbol\lambda,\boldsymbol m}(t)
		\,\mathrm dt
		&=0,
		&&0\le k<n_h,
		\quad 1\le h\le r,
		\\
		\int_{\mathbb R}
		t^kF(t)
		\mathcal A^{\mathrm H}_{\boldsymbol\lambda,\boldsymbol m}(t)
		\,\mathrm dt
		&=0,
		&&0\le k<\eta.
	\end{align*}
	These components are unique with the stated normalization under
	the rate-separation conditions of
	Corollary~\ref{cor:Hermite-main-normality}.
\end{corollary}

\begin{proof}
	Using \eqref{eq:canonical-Hermite-LF}, the limiting integrand becomes
	\[
		\frac{1}{\prod_h\rho_h^{d_h}}
		\frac{
			\e^{z(t-\widehat a_*)-z^2/4}
			\prod_h(\rho_h+z)^{n_h+d_h}
		}{
			\prod_i(\xi_i-z)^{m_i}
		}.
\]
	If \(m_i=0\), both the finite and limiting \(i\)-th components vanish.
	Fix \(i\) with \(m_i\ge1\). At \(z=\xi_i\),
	\((\xi_i-z)^{-m_i}=(-1)^{m_i}(z-\xi_i)^{-m_i}\), and
	\[
		\partial_z^\ell
		\e^{z(t-\widehat a_*)-z^2/4}
		=
		H_\ell\!\left(t-\widehat a_*-\frac z2\right)
		\e^{z(t-\widehat a_*)-z^2/4}.
\]
	The formula for a pole of order \(m_i\), followed by the ordinary
	Leibniz rule, is exactly
	\eqref{eq:canonical-Hermite-A-components}. Since \(H_{M_i}\) is monic,
	\begin{equation}
	 [t^{M_i}]A_{\boldsymbol\lambda,\boldsymbol m}^{(i),\mathrm H}
	 =\frac{(-1)^{m_i}\e^{-\widehat a_*\xi_i-\xi_i^2/4}G_i(\xi_i)}
	 {M_i!\prod_h\rho_h^{d_h}}\ne0.
	 \label{eq:Hermite-A-leading-coefficient}
	\end{equation}
	The inequality follows from the distinct column nodes and
	\(\xi_i+\rho_h>0\), and proves the exact degree assertion. If
	\(A_\tau^{(i),\mathrm L}\) denotes the \(i\)-th finite-\(\tau\) component,
	let
	\[
		\widehat A_{i,\tau}(t)
		\coloneq
		Z_\tau\varepsilon_\tau^{-\eta}
		c_\tau^{\beta_i(\tau)}A_\tau^{(i),\mathrm L}(x_\tau(t)).
\]
	The contribution of the \(i\)-th circle in
	\eqref{eq:canonical-Hermite-A-exact-scaled-contour} is exactly
	\[
		\widehat A_{i,\tau}(t)
		(1+\varepsilon_\tau t)^{\beta_i(\tau)}.
\]
	Uniform convergence of the integrand on that circle makes this product
	converge locally uniformly to
	\(A_{\boldsymbol\lambda,\boldsymbol m}^{(i),\mathrm H}(t)
	\e^{\xi_it}\). Since the second factor converges locally uniformly to
	the nowhere-zero function \(\e^{\xi_it}\), it follows that
	\[
		\widehat A_{i,\tau}(t)
		\xrightarrow[\tau\to+\infty]{}
		A_{\boldsymbol\lambda,\boldsymbol m}^{(i),\mathrm H}(t)
\]
	locally uniformly. These polynomials have the fixed degree bound
	\(m_i-1\); evaluation at \(m_i\) fixed distinct points gives an
	invertible Vandermonde system, and therefore the convergence is
	coefficientwise. This proves
	\eqref{eq:canonical-Hermite-A-component-convergence}.

	For the limiting orthogonality, write
	\[
		\widehat A_{i,\tau}(t)
		=
		\sum_{a=0}^{m_i-1}c_{i,a,\tau}t^a.
\]
	The affine change \(x=x_\tau(t)\) maps \(\mathbb P_k(x)\) onto
	\(\mathbb P_k(t)\). Thus, by
	Lemma~\ref{lem:canonical-Hermite-direct-A-orthogonality}, for each limiting
	row index \(j\) and every admissible \(k\), the finite-\(\tau\)
	orthogonality is exactly
	\[
		0
		=
		\sum_{i=1}^p\sum_{a=0}^{m_i-1}
		c_{i,a,\tau}
		\int_{I_\tau}
		t^{k+a}(1+\varepsilon_\tau t)^{\beta_i(\tau)}
		\widehat U_{j,\tau}(t)\,\mathrm dt.
\]
	The coefficients \(c_{i,a,\tau}\) converge by the preceding paragraph,
	and every integral converges by
	Corollary~\ref{cor:simultaneous-Laguerre-Hermite-matrix-confluence}.
	Passing to the limit gives precisely the two orthogonality relations in
	the statement.
\end{proof}

The next theorem defines the limiting function by its Rodrigues formula
and proves its moment conditions. The notation
\(\mathcal B^{\mathrm H}_{\boldsymbol\lambda,\boldsymbol m}\) is used for
this function throughout. It is a mixed-type \(B\)-form when it admits a
decomposition with the prescribed polynomial degrees; that question is
settled in the following subsection.

\begin{theorem}[Confluence of the Rodrigues function]
	\label{thm:canonical-Hermite-B-form-confluence}
	Let \(\boldsymbol m\in\mathbb N_0^p\) satisfy
	\(|\boldsymbol m|=N-1\). The complete finite-\(\tau\) representative
	\(\mathcal B_\tau^{\mathrm L}\), normalized by
	\eqref{eq:Laguerre-B-Mellin}, satisfies
	\begin{equation}
		\label{eq:canonical-Hermite-B-scaled-limit}
		\frac{c_\tau\varepsilon_\tau^{\eta}}{Z_\tau}
		\mathcal B_\tau^{\mathrm L}(x_\tau(t))
		\xrightarrow[\tau\to+\infty]{}
		\mathcal B^{\mathrm H}_{\boldsymbol\lambda,\boldsymbol m}(t)
	\end{equation}
	in the sense of distributions. For every \(k\in\mathbb N_0\) and
	\(1\le i\le p\),
	\begin{equation}
		\label{eq:canonical-Hermite-B-moment-convergence}
			\int_{-\varepsilon_\tau^{-1}}^\infty
			t^k(1+\varepsilon_\tau t)^{\beta_i(\tau)}
			\frac{c_\tau\varepsilon_\tau^{\eta}}{Z_\tau}
			\mathcal B_\tau^{\mathrm L}(x_\tau(t))\,\mathrm dt
			\xrightarrow[\tau\to+\infty]{}
			\int_{\mathbb R}t^k\e^{\xi_it}
			\mathcal B^{\mathrm H}_{\boldsymbol\lambda,\boldsymbol m}(t)
			\,\mathrm dt.
	\end{equation}
	For
	\(\operatorname{Re}z>-\min_{1\le h\le r}\rho_h\) when \(r>0\), and for all
	\(z\in\mathbb C\) when \(r=0\), its bilateral Laplace transform is
	\begin{equation}
		\label{eq:canonical-Hermite-B-transform}
		\int_{\mathbb R}\e^{zt}
		\mathcal B^{\mathrm H}_{\boldsymbol\lambda,\boldsymbol m}(t)\,\mathrm dt
		=
		L_F(z)
		\frac{
			\prod_{i=1}^p(\xi_i-z)^{m_i}
		}{
			\prod_{h=1}^r(\rho_h+z)^{n_h}
		}.
	\end{equation}
	The limiting function satisfies the Rodrigues representation
	\begin{equation}
		\label{eq:canonical-Hermite-B-Rodrigues}
		\mathcal B^{\mathrm H}_{\boldsymbol\lambda,\boldsymbol m}
		=
		\left(
			\prod_{i=1}^p(\xi_i+\partial_t)^{m_i}
		\right)
		\left(
			\prod_{h=1}^r\rho_h^{-n_h}K_{\rho_h,n_h}
		\right)F.
	\end{equation}
	It also satisfies the limiting mixed-type orthogonality conditions
	\begin{equation}
		\label{eq:canonical-Hermite-B-orthogonality}
		\int_{\mathbb R}
		t^k\e^{\xi_it}
		\mathcal B^{\mathrm H}_{\boldsymbol\lambda,\boldsymbol m}(t)
		\,\mathrm dt
		=0,
		\qquad
		0\le k<m_i,
		\quad 1\le i\le p.
	\end{equation}
\end{theorem}

\begin{proof}
	Set
	\[
		P_{h,\tau}^{[n_h]}(z)
		\coloneq
		\prod_{\nu=0}^{n_h-1}
		\bigl[z+\varepsilon_\tau(b_h(\tau)+1+\nu)\bigr].
\]
	Put \(n_\Sigma=|\boldsymbol n|\) and define
	\[
	S_\tau(t)
	\coloneq
		\frac{c_\tau\varepsilon_\tau^{1-n_\Sigma}}{Z_\tau}
		w_0\!\left(
			x_\tau(t);
			\boldsymbol a(\tau),
			\boldsymbol b(\tau)+\boldsymbol n
		\right).
\]
	Set
	\[
		Z_\tau^{(\boldsymbol n)}
		\coloneq
		\frac{Z_\tau}{\prod_{h=1}^r(b_h(\tau)+1)_{n_h}}.
\]
	Then the following identity is exact:
	\[
		S_\tau(t)
		=
		\frac{\varepsilon_\tau^{-n_\Sigma}}
		{\prod_{h=1}^r(b_h(\tau)+1)_{n_h}}
		\frac{c_\tau\varepsilon_\tau}{Z_\tau^{(\boldsymbol n)}}
		w_0\!\left(
			x_\tau(t);\boldsymbol a(\tau),\boldsymbol b(\tau)+\boldsymbol n
		\right).
\]
	The factor containing \(w_0\) is the normalized base row in
	Theorem~\ref{thm:simultaneous-Laguerre-Hermite-confluence}, with
	\(d_h\) replaced by \(d_h+n_h\). Hence it converges weakly, with every
	fixed polynomial moment, to
	\[
		\left[\prod_{h=1}^rK_{\rho_h,d_h+n_h}\right]
		\phi_{\widehat a_*}.
\]
	Moreover,
	\[
		\frac{\varepsilon_\tau^{-n_\Sigma}}
		{\prod_{h=1}^r(b_h(\tau)+1)_{n_h}}
		\xrightarrow[\tau\to+\infty]{}
		\prod_{h=1}^r\rho_h^{-n_h}.
\]
The semigroup identity for the operators \(K_{\rho,d}\) therefore gives,
for every \(\varphi\in C_b(\mathbb R)\),
\[
\lim_{\tau\to\infty}
\int_{I_\tau}\varphi(t)S_\tau(t)\,\mathrm dt
=
\int_{\mathbb R}\varphi(t)S(t)\,\mathrm dt,
\]
where
\[
S(t)
\coloneq
\left[
\prod_{h=1}^r
\rho_h^{-n_h}K_{\rho_h,n_h}
\right]F(t).
\]
Moreover, for every \(k\in\mathbb N_0\),
\[
\lim_{\tau\to\infty}
\int_{I_\tau}t^kS_\tau(t)\,\mathrm dt
=
\int_{\mathbb R}t^kS(t)\,\mathrm dt.
\]

	Its centered transform is exactly
	\begin{equation}
		\label{eq:canonical-Hermite-SA-transform}
		Q_{S,\tau}(z)
		\coloneq
		\int_{I_\tau}
		(1+\varepsilon_\tau t)^{z/\varepsilon_\tau}S_\tau(t)\,\mathrm dt
		=
		\frac{Q_{0,\tau}(z)}
		{\prod_{h=1}^r
		 P_{h,\tau}^{[n_h]}(z)}.
	\end{equation}
	Indeed, replacing \(b_h(\tau)\) by \(b_h(\tau)+n_h\) divides the Mellin
	transform by
	\((b_h(\tau)+1+z/\varepsilon_\tau)_{n_h}\), while the prefactor in \(S_\tau\)
	supplies \(\varepsilon_\tau^{-n_\Sigma}\).  Hence
	\[
		Q_{S,\tau}(z)
		\xrightarrow[\tau\to+\infty]{}
		\frac{L_F(z)}
		{\prod_{h=1}^r(\rho_h+z)^{n_h}}
\]
	locally uniformly in the common holomorphy domain. Fubini's theorem and
	the transform rule for \(K_{\rho,d}\) identify the limit with the
	bilateral Laplace transform of the displayed function \(S\).

	The Rodrigues formula for Laguerre of the first kind
	\cite[Theorem~5.6]{Manas2026HypergeometricMixed}, after the change of
	variables \(x=x_\tau(t)\), becomes
	\begin{equation}
		\label{eq:canonical-Hermite-B-finite-A-Rodrigues}
			\widehat{\mathcal B}_\tau(t)
			\coloneq
			\frac{c_\tau\varepsilon_\tau^{\eta}}{Z_\tau}
			\mathcal B_\tau^{\mathrm L}(x_\tau(t))=
			\prod_{i=1}^p\prod_{j=0}^{m_i-1}
			\Bigl[
				(1+\varepsilon_\tau t)\partial_t
				+\varepsilon_\tau(\beta_i(\tau)+1+j)
			\Bigr]S_\tau(t).
	\end{equation}
	Replacing \(\boldsymbol b(\tau)\) by
	\(\boldsymbol b(\tau)+\boldsymbol n\) does not change the lower exponents
	\(a_\nu(\tau)\). Hence the estimate as \(x\downarrow0\) established in the proof of
	Theorem~\ref{thm:simultaneous-Laguerre-Hermite-confluence} applies to
	\(S_\tau\) and to every Euler derivative occurring here. It shows that the
	product of each such derivative with \(1+\varepsilon_\tau t\) tends to zero
	as \(t\downarrow-\varepsilon_\tau^{-1}\). Repeated integration by parts
	therefore gives the displayed Rodrigues identity after extension by zero,
	without an additional distribution supported at that point.
	For \(\varphi\in C_c^\infty(\mathbb R)\), the formal adjoint of each
	factor in brackets converges uniformly on the support of \(\varphi\),
	together with all derivatives, to \(-\partial_t+\xi_i\), the adjoint of
	\(\partial_t+\xi_i\).  Applying the finitely many adjoints in
	\eqref{eq:canonical-Hermite-B-finite-A-Rodrigues} and using
	\(S_\tau\xrightarrow[\tau\to+\infty]{}S\) proves
	\[
		\widehat{\mathcal B}_\tau
		\xrightarrow[\tau\to+\infty]{}
		\prod_{i=1}^p(\xi_i+\partial_t)^{m_i}S
\]
	in distributions.  This is
	\eqref{eq:canonical-Hermite-B-scaled-limit} and
	\eqref{eq:canonical-Hermite-B-Rodrigues}.

	For the moment statement, evaluate the Mellin transform in
	\eqref{eq:Laguerre-B-Mellin} at \(1+z/\varepsilon_\tau\) and use
	the definitions of \(S_\tau\) and
	\(\widehat{\mathcal B}_\tau\).  This gives
	\begin{equation}
		\label{eq:canonical-Hermite-finite-B-transform}
		Q_{B,\tau}(z)
		\coloneq
		\int_{I_\tau}
		(1+\varepsilon_\tau t)^{z/\varepsilon_\tau}
		\widehat{\mathcal B}_\tau(t)\,\mathrm dt
		=
		Q_{S,\tau}(z)
		\prod_{i=1}^p\prod_{j=0}^{m_i-1}
		\bigl[\varepsilon_\tau(\beta_i(\tau)+j)-z\bigr].
	\end{equation}
	The last display converges locally uniformly,
	together with every fixed derivative, to
	\[
		L_F(z)
		\frac{\prod_{i=1}^p(\xi_i-z)^{m_i}}
		{\prod_{h=1}^r(\rho_h+z)^{n_h}}.
\]
	This proves \eqref{eq:canonical-Hermite-B-transform}. The exact Mellin
	formula gives absolute convergence at the finitely many shifted arguments
	used below. Thus the general identity
	\eqref{eq:canonical-Hermite-finite-difference-moments} applies to the
	possibly signed, nonnormalized form \(Q_{B,\tau}\). Applying it
	at \(z=\varepsilon_\tau\beta_i(\tau)\) proves convergence of every
	polynomial moment with the indicated exponential factor. The limiting transform has a
	zero of order \(m_i\) at \(z=\xi_i\); its derivatives of orders
	\(0,\ldots,m_i-1\) therefore vanish. These derivatives are exactly the integrals in
	\eqref{eq:canonical-Hermite-B-orthogonality}, which proves the last
	assertion.
\end{proof}

\subsection{Polynomial components and normality at admissible indices}
\label{subsec:Hermite-components-normality}
\label{subsec:Hermite-normal-components}

Each \(A\)-component has both a finite Hermite expansion and a single
terminating Srivastava--Daoust representation. For \(B\), multiplication
of the row transforms by a common
factor reduces the component problem to a polynomial identity. Its
coefficient determinant is nonzero for every admissible index under
rate separation. Its product evaluation gives normality, followed by the
explicit series and quantitative convergence of the components.
Throughout, existence of the Rodrigues decomposition, uniqueness of its
components, and attainment of their maximal degrees are distinguished.

For \(m_i\ge1\), set \(u_i(t)\coloneq t-\widehat a_*-\xi_i/2\).
In the following vectors, list first the \(r\) rows and then the columns
\(k\ne i\), each in increasing order:
\[
 \boldsymbol b_i\coloneq((-n_h-d_h)_{h=1}^r,(m_k)_{k\ne i}),\qquad
 \boldsymbol x_i\coloneq
 \left(\left(-\frac1{\rho_h+\xi_i}\right)_{h=1}^r,
       \left(\frac1{\xi_k-\xi_i}\right)_{k\ne i}\right).
\]
\begin{proposition}[One terminating series for each Hermite \(A\)-component]
 \label{prop:Hermite-A-SD}
 Under the hypotheses of
 Corollary~\ref{cor:canonical-Hermite-explicit-A-components}, for \(m_i\ge1\),
 \begin{equation}
 A_{\boldsymbol\lambda,\boldsymbol m}^{(i),\mathrm H}(t)
 =\frac{(-1)^{m_i}\e^{-\widehat a_*\xi_i-\xi_i^2/4}G_i(\xi_i)}
             {\prod_{h=1}^r\rho_h^{d_h}}
   \mathfrak E_{m_i-1}(\boldsymbol b_i,\boldsymbol x_i;u_i(t),-1/4).
 \label{eq:Hermite-A-single-SD}
 \end{equation}
 Thus each component is one terminating Srivastava--Daoust function
 of \(r+p\) arguments, multiplied by the displayed constant.
\end{proposition}

\begin{proof}
 Put \(z=\xi_i+w\) in the residue giving the \(i\)-th component.
 The nonexponential factor becomes
 \[
 \frac{G_i(\xi_i+w)}{G_i(\xi_i)}
 =\prod_{h=1}^r\left(1+\frac{w}{\rho_h+\xi_i}\right)^{n_h+d_h}
  \prod_{k\ne i}\left(1-\frac{w}{\xi_k-\xi_i}\right)^{-m_k}.
 \]
 After taking out the constant and \(\e^{\xi_i t}\), the exponential
 is \(\e^{u_i(t)w-w^2/4}\). The residue is therefore the coefficient
 of \(w^{m_i-1}\) in this product, which gives
 \eqref{eq:Hermite-A-single-SD} by \eqref{eq:finite-quadratic-KdF}.
 Its identification in
 \eqref{eq:finite-quadratic-SD-identification} proves the assertion.
 Expanding the exponential first instead recovers the finite Hermite
 sum \eqref{eq:canonical-Hermite-A-components}.
\end{proof}

For the row weights \(U_j^{\mathrm H}\) in
\eqref{eq:canonical-Hermite-limiting-vectors}, denote their bilateral
Laplace transforms by
\[
	\Lambda_j(z)\coloneq\int_{\mathbb R}\e^{zt}U_j^{\mathrm H}(t)\,\mathrm dt,
	\qquad 1\le j\le r+1.
\]
Their mixed moments, for \(k,\ell\in\mathbb N_0\), are denoted by
\[
	\mathfrak M_{(j,k),(i,\ell)}\coloneq
	\int_{\mathbb R}t^{k+\ell}\e^{\xi_i t}U_j^{\mathrm H}(t)\,\mathrm dt.
\]

\begin{lemma}[Transforms and moments of the limiting rows]
	The row transforms satisfy
	\begin{equation}
		\label{eq:canonical-Hermite-row-transforms}
		\Lambda_h(z)
		=
		\frac{\rho_h}{\rho_h+z}L_F(z),
		\quad 1\le h\le r,
		\qquad
		\Lambda_q(z)=L_F(z).
	\end{equation}
	The mixed-type moments are
	\begin{equation}
		\label{eq:canonical-Hermite-limiting-moments}
		\mathfrak M_{(j,k),(i,\ell)}
		=
		\left.
		\partial_z^{k+\ell}\Lambda_j(z)
		\right|_{z=\xi_i}.
	\end{equation}
\end{lemma}

\begin{proof}
	Fubini's theorem and the definition of \(K_{\rho,1}\) give
	\[
		\int_{\mathbb R}\e^{zt}K_{\rho,1}F(t)\,\mathrm dt
		=
		\frac{\rho}{\rho+z}L_F(z).
\]
 The final identity is the definition of \(L_F\), since \(U_q^{\mathrm H}=F\).
	Finally, differentiation \(k+\ell\) times with respect to \(z\)
	multiplies the integrand by \(t^{k+\ell}\), which proves
	\eqref{eq:canonical-Hermite-limiting-moments}.
\end{proof}

To study the \(B\)-components, retain the hypotheses of
Theorem~\ref{thm:canonical-Hermite-B-form-confluence}. The coefficient
matrix and its relation to the moments are defined as follows. Put
	\[
		\lambda_j
		\coloneq
		\begin{cases}
			n_j,&1\le j\le r,\\
			\eta,&j=q,
		\end{cases}
		\qquad
		\mathcal I_{\boldsymbol\lambda}
		\coloneq
		\left\{(j,k):1\le j\le r+1,\ 0\le k<\lambda_j\right\}.
\]
	Then \(\lvert\mathcal I_{\boldsymbol\lambda}\rvert=N\). Define
	\begin{equation}
		\label{eq:canonical-Hermite-DP-polynomials}
		D_{\boldsymbol n}(z)
		\coloneq
		\prod_{h=1}^r(\rho_h+z)^{n_h},
		\qquad
		P_{\boldsymbol m}(z)
		\coloneq
		\prod_{i=1}^p(\xi_i-z)^{m_i}.
	\end{equation}
	For \((j,k)\in\mathcal I_{\boldsymbol\lambda}\), define
	\begin{equation}
		\label{eq:canonical-Hermite-connection-polynomials}
		\Phi_{j,k}(z)
		\coloneq
		\frac{D_{\boldsymbol n}(z)}{L_F(z)}
		H_k(\partial_z)\Lambda_j(z).
	\end{equation}
	Order \(\mathcal I_{\boldsymbol\lambda}\) lexicographically. Define the
	\(N\times N\) coefficient matrix \(C\) of the polynomials \(\Phi_{j,k}\)
	and the coefficient vector \(\mathbf p\) of \(P_{\boldsymbol m}\) by
	\begin{equation}
		\label{eq:canonical-Hermite-B-connection-matrix}
		C_{\nu,(j,k)}
		\coloneq
		[z^\nu]\Phi_{j,k}(z),
		\qquad
		p_\nu
		\coloneq
		[z^\nu]P_{\boldsymbol m}(z),
		\qquad
		0\le\nu\le N-1.
	\end{equation}
	Set \(\Delta_{\mathrm H}\coloneq\det C\). Expanding \(P_{\boldsymbol m}\) gives
	\begin{equation}
		\label{eq:canonical-Hermite-B-target-column}
		p_\nu
		=
		(-1)^\nu
		\sum_{\substack{
			0\le\alpha_i\le m_i\\
			\alpha_1+\cdots+\alpha_p=\nu
		}}
		\prod_{i=1}^p
		\binom{m_i}{\alpha_i}\xi_i^{m_i-\alpha_i}.
	\end{equation}
	For a fixed \(i_0\in\{1,\ldots,p\}\), put
	\[
		q_i\coloneq m_i+\delta_{i,i_0},
		\qquad
		\mathcal J_{\boldsymbol m,i_0}
		\coloneq
		\{(i,\ell):1\le i\le p,\ 0\le\ell<q_i\}.
\]
	Order \(\mathcal J_{\boldsymbol m,i_0}\) lexicographically, with both
	indices increasing, and use the same column order as in \(C\). Define the adjacent
	moment matrix by
	\[
		\mathcal N_{i_0}
		\coloneq
		\left[
			\mathfrak M_{(j,k),(i,\ell)}
		\right]_{
			(i,\ell)\in\mathcal J_{\boldsymbol m,i_0},\,
			(j,k)\in\mathcal I_{\boldsymbol\lambda}}.
\]
\begin{proposition}[Connection determinant for the limiting \(B\)-problem]
	\label{prop:canonical-Hermite-B-components}
	The functions \(\Phi_{j,k}\) are polynomials of degree at most \(N-1\).
	They form a basis of \(\mathbb P_{N-1}\) if and only if
	\(\Delta_{\mathrm H}\ne0\). For every \(i_0\in\{1,\ldots,p\}\),
	\begin{equation}
		\label{eq:canonical-Hermite-B-connection-determinant}
			\det\mathcal N_{i_0}
			=
			\left[
				\prod_{i=1}^p
				\left(
					\frac{L_F(\xi_i)}{D_{\boldsymbol n}(\xi_i)}
				\right)^{q_i}
			\right]
			\left[
				\prod_{i=1}^p\prod_{\ell=0}^{q_i-1}\ell!
			\right]
			\left[
				\prod_{1\le i<j\le p}
				(\xi_j-\xi_i)^{q_iq_j}
			\right]
			\Delta_{\mathrm H}.
	\end{equation}
	Consequently,
	\(\det\mathcal N_{i_0}\ne0\) if and only if
	\(\Delta_{\mathrm H}\ne0\).
	When these determinants are nonzero, the limiting \(B\)-problem is
	weakly normal and any one of its adjacent moments fixes a unique
	component vector.
\end{proposition}

\begin{proof}
	Expanding each factor \((\xi_i-z)^{m_i}\) and collecting the coefficient
	of \(z^\nu\) proves
	\eqref{eq:canonical-Hermite-B-target-column}.
	Put
	\[
		g(z)\coloneq\frac{L_F(z)}{D_{\boldsymbol n}(z)}.
\]
	The inequalities \(\xi_i+\rho_h>0\) imply \(g(\xi_i)\ne0\).
	Replacing the monomial basis in each polynomial component by the monic
	Hermite basis does not change the determinant of
	\(\mathcal N_{i_0}\), because the change-of-basis matrix is triangular
	with diagonal entries one. In the Hermite basis the entries are
	\[
		\widetilde{\mathcal N}_{(i,\ell),(j,k)}
		=
		\left.
		\partial_z^\ell\bigl(g(z)\Phi_{j,k}(z)\bigr)
		\right|_{z=\xi_i},
		\qquad
		0\le\ell<q_i.
\]
	Let \(J=\operatorname{diag}(J_1,\ldots,J_p)\), where
	\[
		(J_i)_{\ell,a}
		\coloneq
		\binom{\ell}{a}g^{(\ell-a)}(\xi_i),
		\qquad
		0\le a\le\ell<q_i,
\]
	with zero entries for \(a>\ell\), and define
	\[
		V_{(i,a),\nu}
		\coloneq
		\left.\partial_z^a z^\nu\right|_{z=\xi_i}
		=\begin{cases}
		\dfrac{\nu!}{(\nu-a)!}\xi_i^{\nu-a},&\nu\ge a,\\
		0,&\nu<a,
		\end{cases}
		\qquad
		0\le\nu\le N-1.
\]
	Leibniz's rule gives the exact factorization
	\[
		\widetilde{\mathcal N}_{i_0}=JVC.
\]
	Each \(J_i\) is lower triangular, so
	\[
		\det J
		=
		\prod_{i=1}^pg(\xi_i)^{q_i}.
\]
	With the lexicographic row order fixed above, the confluent
	Vandermonde determinant is
	\[
		\det V
		=
		\left(
			\prod_{i=1}^p\prod_{\ell=0}^{q_i-1}\ell!
		\right)
		\prod_{1\le i<j\le p}
		(\xi_j-\xi_i)^{q_iq_j}.
\]
	Multiplication of these determinants proves
	\eqref{eq:canonical-Hermite-B-connection-determinant}.
	The polynomial character and degree bound asserted in the statement follow
	from
	Lemma~\ref{lem:canonical-Hermite-explicit-connection-polynomials} below.
	Finally, deleting the row of \(\mathcal N_{i_0}\) corresponding to the
	additional moment leaves the \(N-1\) homogeneous conditions. Since
	\(\mathcal N_{i_0}\) is invertible, those rows are independent and their
	kernel is one-dimensional. Its nonzero vectors have a nonzero additional
	moment, which proves the weak normality and normalization assertions.
\end{proof}

The connection polynomials can be evaluated using the auxiliary
polynomials \(\mathcal H_v^{(+)}\) and functions \(\Psi_{j,k}\), defined,
for \((j,k)\in\mathcal I_{\boldsymbol\lambda}\), by
	\[
		\e^{xu+u^2/4}
		=
		\sum_{v=0}^\infty
		\mathcal H_v^{(+)}(x)\frac{u^v}{v!},
		\qquad
		\Psi_{j,k}(z)
		\coloneq
		\frac{D_{\boldsymbol n}(z)}{L_F(z)}
		\partial_z^k\Lambda_j(z).
\]
\begin{lemma}[Explicit connection polynomials]
	\label{lem:canonical-Hermite-explicit-connection-polynomials}
	The polynomials \(\mathcal H_v^{(+)}\) are related to the monic scalar Hermite polynomials in
	\eqref{eq:monic-Hermite-generating-function} by
	\[
		\mathcal H_v^{(+)}(x)=(-\mathrm i)^vH_v(\mathrm i x).
\]
	The scalar Hermite polynomials and the connection polynomials satisfy,
	respectively,
	\[
		H_k(x)
		=
		k!\sum_{v=0}^{\lfloor k/2\rfloor}
		\frac{(-1)^v x^{k-2v}}{4^v v!(k-2v)!},
\]
	and
	\begin{equation}
		\label{eq:canonical-Hermite-Phi-from-Psi}
		\Phi_{j,k}(z)
		=
		k!\sum_{v=0}^{\lfloor k/2\rfloor}
		\frac{(-1)^v}{4^v v!(k-2v)!}
		\Psi_{j,k-2v}(z).
	\end{equation}
	For \(1\le h\le r\) and \(0\le k<n_h\),
	\begin{equation}
		\label{eq:canonical-Hermite-Psi-beta-explicit}
			\Psi_{h,k}(z)
			=
			\rho_h
			\sum_{\substack{
				\alpha_0,\ldots,\alpha_r\ge0\\
				\alpha_0+\cdots+\alpha_r=k
			}}
			\binom{k}{\alpha_0,\ldots,\alpha_r}
			\mathcal H_{\alpha_0}^{(+)}\!\left(\widehat a_*+\frac z2\right)
			\prod_{u=1}^r
			(-1)^{\alpha_u}
			(d_u+\delta_{u,h})_{\alpha_u}
			(\rho_u+z)^{n_u-\alpha_u-\delta_{u,h}}.
	\end{equation}
 For \(0\le k<\eta\), the last row gives
 \begin{equation}
 \Psi_{q,k}(z)=
 \sum_{\substack{\alpha_0,\ldots,\alpha_r\ge0\\\alpha_0+\cdots+\alpha_r=k}}
 \binom{k}{\alpha_0,\ldots,\alpha_r}
 \mathcal H_{\alpha_0}^{(+)}\!\left(\widehat a_*+\frac z2\right)
 \prod_{u=1}^r(-1)^{\alpha_u}(d_u)_{\alpha_u}
                 (\rho_u+z)^{n_u-\alpha_u}.
 \label{eq:canonical-Hermite-Psi-base-explicit}
 \end{equation}
	These finite identities give every entry of \(C\) directly.
\end{lemma}

\begin{proof}
	Expanding \eqref{eq:monic-Hermite-generating-function} proves the scalar
	identity in the statement. Substitution of \(x=\partial_z\) proves
	\eqref{eq:canonical-Hermite-Phi-from-Psi}.

	For \(1\le h\le r\), write
	\[
		\Lambda_h(z)
		=
		\rho_h\left(\prod_{u=1}^r\rho_u^{d_u}\right)
		\e^{\widehat a_*z+z^2/4}
		\prod_{u=1}^r
		(\rho_u+z)^{-d_u-\delta_{u,h}}.
\]
	The identities
	\[
		\partial_z^a\e^{\widehat a_*z+z^2/4}
		=
		\mathcal H_a^{(+)}\!\left(\widehat a_*+\frac z2\right)
		\e^{\widehat a_*z+z^2/4}
\]
	and
	\[
		\partial_z^a(\rho_u+z)^{-d}
		=
		(-1)^a(d)_a(\rho_u+z)^{-d-a}
\]
	turn the multinomial Leibniz rule into
	\eqref{eq:canonical-Hermite-Psi-beta-explicit} after multiplication by
	\(D_{\boldsymbol n}/L_F\).

 For \(\Lambda_q=L_F\), the same multinomial Leibniz rule,
 without the extra simple-pole factor, proves
 \eqref{eq:canonical-Hermite-Psi-base-explicit}.
	The inequalities
	\eqref{eq:Laguerre-admissibility-nn} and
	\eqref{eq:Laguerre-admissibility-etan} make every displayed exponent
	nonnegative.  A term in
	\eqref{eq:canonical-Hermite-Psi-beta-explicit} has degree at most
	\(|\boldsymbol n|-1+k\), and a term in
	\eqref{eq:canonical-Hermite-Psi-base-explicit} has degree at most
	\(|\boldsymbol n|+k\).  The remaining admissibility inequalities
	therefore give \(\deg\Phi_{j,k}\le N-1\).
\end{proof}

\begin{proposition}[Exact Hermite connection determinant]
	\label{prop:canonical-Hermite-vector-edge-determinant}
 For the row index \(\boldsymbol\lambda=(n_1,\ldots,n_r,\eta)\)
 in Proposition~\ref{prop:canonical-Hermite-B-components}, let \(p\ge1\).
	Order \(\mathcal I_{\boldsymbol\lambda}\) lexicographically, first by
	the row index and then by the polynomial degree. If the rates
	\(\{\rho_h:n_h>0\}\) are pairwise distinct, then the connection
	determinant has the closed product evaluation
	\begin{equation}
		\label{eq:canonical-Hermite-vector-edge-determinant}
		\Delta_{\mathrm H}
		=
		2^{-\binom{\eta}{2}}
		\prod_{h=1}^r
		\left(
			\rho_h^{n_h}
			\prod_{k=0}^{n_h-1}(d_h+1)_k
		\right)
		\prod_{1\le j<h\le r}
		(\rho_h-\rho_j)^{n_jn_h}.
	\end{equation}
	In particular, \(\Delta_{\mathrm H}\ne0\) for every
	Laguerre-admissible vector index under the stated hypotheses. Consequently,
	for every \(B\)-balance
	\(|\boldsymbol\lambda|=|\boldsymbol m|+1\), every adjacent moment matrix
	\(\mathcal N_{i_0}\) in
	\eqref{eq:canonical-Hermite-B-connection-determinant} is nonsingular.
	Thus the limiting \(B\)-problem is weakly normal, and its Rodrigues
	function has a unique decomposition with the prescribed component degrees.
\end{proposition}

\begin{proof}
	Put
	\[
		L\coloneq|\boldsymbol n|,
		\qquad R\coloneq\eta,
		\qquad D(z)\coloneq\prod_{h=1}^r(\rho_h+z)^{n_h}.
\]
	The determinant calculation first uses the following polynomials:
	\[
		E_{h,k}(z)
		\coloneq
		\frac{D(z)}{(\rho_h+z)^{k+1}},
		\qquad
		1\le h\le r,\quad 0\le k<n_h,
\]
	and
	\[
		E_{r+1,k}(z)
		\coloneq z^kD(z),
		\quad 0\le k<R.
\]
	The columns are ordered by increasing \(h\) and then \(k\), followed by
	\(E_{r+1,0},\ldots,E_{r+1,R-1}\). Their coefficient determinant is
	\begin{equation}
		\label{eq:canonical-Hermite-vector-edge-cleared-determinant}
		\det\bigl([z^\nu]E_{j,k}(z)\bigr)
		=
		(-1)^{\sum_{h=1}^r\binom{n_h}{2}}
		\prod_{1\le j<h\le r}
		(\rho_h-\rho_j)^{n_jn_h}.
	\end{equation}
	Here and below, factors with zero exponent are omitted. To prove the
	identity, observe that the first \(L\) columns have degree below \(L\),
	whereas \(E_{r+1,k}\) is monic of degree \(L+k\). The coefficient
	matrix is therefore block upper triangular, with a final triangular block
	of determinant one. It remains to compute the coefficient determinant
	\(E_0\) of the first \(L\) columns in \(1,z,\ldots,z^{L-1}\).

	For each \(h\) with \(n_h>0\), apply the normalized evaluations
	\(P\mapsto P^{(a)}(-\rho_h)/a!\), \(0\le a<n_h\). Their matrix
	\(V_0\) on the monomials has determinant
	\[
		\det V_0=\prod_{j<h}(\rho_j-\rho_h)^{n_jn_h}.
\]
	The matrix \(V_0E_0\) is block diagonal: a polynomial \(E_{u,k}\),
	\(u\ne h\), is divisible by \((z+\rho_h)^{n_h}\). Within the \(h\)-th
	block write
	\[
		E_{h,k}(z)=(z+\rho_h)^{n_h-k-1}B_h(z),
		\qquad B_h(z)\coloneq\prod_{u\ne h}(z+\rho_u)^{n_u}.
\]
	Its entry in row \(a\) vanishes when \(a<n_h-k-1\), and the
	antidiagonal entries all equal
	\(B_h(-\rho_h)=\prod_{u\ne h}(\rho_u-\rho_h)^{n_u}\). Thus
	\[
		\det(V_0E_0)=(-1)^{\sum_h\binom{n_h}{2}}
		\prod_h B_h(-\rho_h)^{n_h}.
\]
	Dividing by \(\det V_0\) and grouping the two factors from each pair
	\(j<h\) proves \eqref{eq:canonical-Hermite-vector-edge-cleared-determinant}.
	The same argument works for any nonnegative number \(R\) of final
	columns; if \(L=0\), it is simply the determinant of a monic triangular
	polynomial family.

	It remains to express the connection polynomials in this basis. Consider
	\(\Phi_{h,k}/D=L_F^{-1}H_k(\partial_z)\Lambda_h\). The Leibniz
	expansion is rational, with a pole of order at most \(k+1\) at
	\(-\rho_h\) and order at most \(k\) at the other rates. The coefficient
	of \((z+\rho_h)^{-k-1}\) is
	\((-1)^k\rho_h(d_h+1)_k\): it comes from applying all \(k\) derivatives
	to that factor. Its polynomial part has degree at most \(k-1\).
	Partial fractions consequently give
	\[
		\Phi_{h,k}
		-
		(-1)^k\rho_h(d_h+1)_kE_{h,k}
		\in\operatorname{span}\{E_{u,a}:a<k\},
		\qquad 1\le h\le r.
\]
	Only existing columns are included in this span, with \(u=r+1\)
	allowed. They suffice: if \(k<n_h\), admissibility gives
	\(n_u\ge n_h-1\ge k\) for \(u\ne h\), and \(\eta\ge k\).
	Hence every pole of lower order and every polynomial term just described
	has its corresponding column.

	For the final family, the generating function
	\eqref{eq:monic-Hermite-generating-function} gives the exact identity
	\[
		H_k(\partial_z)
		\exp\!\left(\widehat a_*z+\frac{z^2}{4}\right)
		=
		2^{-k}(z+2\widehat a_*)^k
		\exp\!\left(\widehat a_*z+\frac{z^2}{4}\right).
\]
	In \(L_F^{-1}H_k(\partial_z)L_F\), the polynomial part consequently
	has degree \(k\) and leading coefficient \(2^{-k}\). The remaining
	poles have order at most \(k\). Since \(k<\eta\) implies
	\(n_h\ge\eta-1\ge k\), partial fractions give
	\[
		\Phi_{r+1,k}-2^{-k}E_{r+1,k}
		\in\operatorname{span}\{E_{u,a}:a<k\}.
\]
	After a simultaneous reordering of the two polynomial families by
	increasing \(k\), these relations define a triangular change of basis.
	Its diagonal determinant is
	\[
		(-1)^{\sum_{h=1}^r\binom{n_h}{2}}
		2^{-\binom{\eta}{2}}
		\prod_{h=1}^r
			\left[\rho_h^{n_h}\prod_{k=0}^{n_h-1}(d_h+1)_k\right].
\]
	Multiplication by
	\eqref{eq:canonical-Hermite-vector-edge-cleared-determinant} cancels the
	two identical signs and proves
	\eqref{eq:canonical-Hermite-vector-edge-determinant}. Every factor in
	the resulting product is nonzero under the hypotheses. Weak normality
	of the \(B\)-problem for arbitrary \(p\) now follows from
	Proposition~\ref{prop:canonical-Hermite-B-components} and
	\eqref{eq:canonical-Hermite-B-connection-determinant}.
\end{proof}

The determinant evaluation also settles the \(A\)-problem. Normality
holds at every admissible near-diagonal row index, and these indices
form an infinite set.

\begin{corollary}[Normality of the Hermite system]
 \label{cor:Hermite-main-normality}
 Suppose that the column nodes are pairwise distinct,
 \(\xi_i+\rho_h>0\), and the rates with \(n_h>0\) are pairwise distinct.
 At every Laguerre-admissible row index
 \(\boldsymbol\lambda=(\boldsymbol n,\eta)\), the \(A\)-problem in
 the balance \(|\boldsymbol m|=N+1\) is strongly normal, and the
 \(B\)-problem in the balance \(|\boldsymbol m|=N-1\), \(N\ge1\),
 is weakly normal. Their normalized component vectors are the ones
 constructed above and below. Strong normality of \(B\) is characterized
 by \eqref{eq:Hermite-B-strong-normality-criterion}.
\end{corollary}

\begin{proof}
 If \(N=0\), the \(A\)-balance has exactly one positive column
 index, equal to one, and its assertion is immediate. Assume \(N\ge1\).
 Proposition~\ref{prop:canonical-Hermite-vector-edge-determinant} gives
 \(\operatorname{rank}C=N\). In the \(A\)-balance the moment matrix,
 transposed if necessary, factors as \(JVC\), by the Leibniz calculation
 in Proposition~\ref{prop:canonical-Hermite-B-components}. The matrix
 \(J\) is invertible. The \(N+1\) evaluations represented by \(V\)
 are injective on \(\mathbb P_{N-1}\): a polynomial in their kernel
 would be divisible by \(\prod_i(z-\xi_i)^{m_i}\), whose degree is
 \(N+1\). The moment matrix therefore has rank \(N\), so the space
 of \(A\)-vectors has dimension one. Distinct real exponentials with
 polynomial coefficients are linearly independent, and the coefficients
 in \eqref{eq:Hermite-A-leading-coefficient} are nonzero. This proves
 both nonvanishing of the form and attainment of every prescribed degree.

 For \(B\), the adjacent moment determinant
 \eqref{eq:canonical-Hermite-B-connection-determinant} is nonzero.
 Deleting its additional moment leaves independent homogeneous
 conditions with a one-dimensional kernel; the additional moment
 is nonzero on each nonzero vector in that kernel. The Rodrigues
 transform fixes this normalization. Finally, expansion in the monic
 Hermite basis and Cramer's rule give the criterion cited in the statement.
\end{proof}

The polynomial components can now be given by terminating series, with
no sum over pole orders. Two Srivastava--Daoust specializations are used:
one for the finite poles, and one for the polynomial part at infinity.
The number of functions in each component depends only on the number of
rows, not on the degrees of the polynomials.

Put \(L\coloneq|\boldsymbol n|\), \(R\coloneq\eta\),
\(M\coloneq|\boldsymbol m|=L+R-1\), and
\(\mathcal I\coloneq\{h:n_h>0\}\). The partial-fraction coefficients
are specified by
\begin{equation}
 \frac{P_{\boldsymbol m}(z)}{D_{\boldsymbol n}(z)}
 =\sum_{l=0}^{R-1}\pi_lz^l+
   \sum_{J\in\mathcal I}\sum_{K=0}^{n_J-1}
       \frac{\zeta_{J,K}}{(z+\rho_J)^{K+1}}.
 \label{eq:Hermite-B-explicit-partial-fractions}
\end{equation}
Only \(\zeta_{J,0}\) occurs separately in the final component formulas.
For use in the proof and in Appendix~\ref{app:Hermite-derivative-limits}, define
\begin{equation}
 \pi_l\coloneq(-1)^M
 \mathfrak C_{R-1-l}
 \bigl((-\boldsymbol m,\boldsymbol n),(\boldsymbol\xi,-\boldsymbol\rho);0\bigr),
 \qquad 0\le l<R,
 \label{eq:Hermite-B-polynomial-part-KdF}
\end{equation}
and set \(\pi_l=0\) outside this range. For \(J\in\mathcal I\), put
\[
 \kappa_J\coloneq
 \frac{\prod_i(\xi_i+\rho_J)^{m_i}}
      {\prod_{h\in\mathcal I\setminus\{J\}}(\rho_h-\rho_J)^{n_h}},\qquad
 c_J(t)\coloneq\widehat a_*-\rho_J/2-t,\qquad
 x_h^{(J)}\coloneq-\frac1{\rho_h-\rho_J}\quad(\rho_h\ne\rho_J).
\]
Then, with row indices in increasing order,
\begin{equation}
 \zeta_{J,0}=\kappa_J\mathfrak C_{n_J-1}
 \left(((-m_i)_i,(n_h)_{h\in\mathcal I\setminus\{J\}}),
       ((\xi_i+\rho_J)^{-1}_i,(x_h^{(J)})_{h\in\mathcal I\setminus\{J\}});0\right).
 \label{eq:Hermite-B-simple-pole-coefficient}
\end{equation}
Set \(\zeta_{J,0}=0\) when \(n_J=0\).

The first specialization is needed only when \(n_J\ge2\). Admissibility
then gives \(n_h\ge1\) for every \(h\), so the distinct-rate condition
applies to all the differences just introduced. Order its
\(d_Z\coloneq p+2r+2\) indices as
\((\boldsymbol\alpha,\boldsymbol\beta,l,j,\mu,\nu)\), where
\(\boldsymbol\alpha\) has \(p+r-1\) entries (columns, then rows
\(h\ne J\)), and \(\boldsymbol\beta\) has \(r-1\) entries (rows
\(h\ne J\)). Define
\[
 \boldsymbol e_Z\coloneq(\one_{p+2r},2,2),\quad
 \boldsymbol v_Z\coloneq
 (\boldsymbol0_{p+r-1},\one_{r-1},1,1,0,2),\quad
 \boldsymbol a_Z\coloneq((-m_i)_i,(n_h+d_h)_{h\ne J},(-d_h)_{h\ne J},1).
\]
The vector \(\boldsymbol a_Z\) contains the \(d_Z-3\) individual upper
parameters; the final three indices have none. For \(c_J(t)\ne0\), put
\[
 \boldsymbol z_J(\upsilon;t)\coloneq
 \left(\left(-\frac1{(\xi_i+\rho_J)c_J(t)}\right)_i,
       \left(-\frac{x_h^{(J)}}{c_J(t)}\right)_{h\ne J},
       \left(-\frac{x_h^{(J)}}{c_J(t)}\right)_{h\ne J},
       -\frac\upsilon{c_J(t)},1,\frac1{4c_J(t)^2},-\frac1{4c_J(t)^2}\right).
\]
For \(b>0\) and \(N\in\mathbb N_0\), define
\begin{equation}
 \mathcal Z_{J,N}(b,\upsilon;t)\coloneq\frac{c_J(t)^N}{N!}
 F_{1:0;\ldots;0}^{2:1;\ldots;1;0;0;0}
 \left[\begin{array}{c}
 (-N:\boldsymbol e_Z),(b:\boldsymbol v_Z):a_{Z,1};\ldots;a_{Z,d_Z-3};-;-;-\\
 (b+1:\boldsymbol v_Z):-;\ldots;-
 \end{array}\,\middle|\,\boldsymbol z_J(\upsilon;t)\right].
 \label{eq:Hermite-B-pole-SD}
\end{equation}
Writing \(E_Z\coloneq|\boldsymbol\alpha|+|\boldsymbol\beta|+l+j+2\mu+2\nu\),
the sum is restricted to \(E_Z\le N\). After multiplication by the
prefactor, the power of \(c_J(t)\) in each term is
\(N-E_Z+j\ge0\). At \(c_J(t)=0\), precisely the terms with
\(E_Z=N\) and \(j=0\) remain. Thus the complete finite expression
defines a polynomial there without evaluating any infinite argument.
For \(N<0\), set \(\mathcal Z_{J,N}=0\), without evaluating its
parameters or arguments. Finally, define
\begin{equation}
 \mathcal T_J(t)\coloneq\frac{\kappa_J}{d_J+1}
       \mathcal Z_{J,n_J-2}(d_J+1,0;t),\qquad
 \mathcal W_J(t)\coloneq\frac{\kappa_J}{d_J+2}
       \mathcal Z_{J,n_J-3}(d_J+2,0;t),
 \label{eq:Hermite-B-pole-TW}
\end{equation}
and, for \(h\ne J\),
\begin{equation}
 \mathcal U_{h,J}(t)\coloneq
 -\frac{\kappa_Jd_hx_h^{(J)}}{d_J+1}
        \mathcal Z_{J,n_J-2}(d_J+1,x_h^{(J)};t).
 \label{eq:Hermite-B-pole-U}
\end{equation}
All three quantities are zero when \(n_J<2\). Empty row groups in these
definitions are omitted, so the same formulas apply for \(r=1\).

For the polynomial part, put
\(d_*\coloneq\sum_h d_h\) and \(c(t)\coloneq\widehat a_*-t\).
Order \(d_Y\coloneq p+2r+3\) indices as
\((\boldsymbol\alpha,\boldsymbol\beta,l,k,v)\), with
\(\boldsymbol\alpha\) of length \(p+r\) (columns, then rows) and
\(\boldsymbol\beta\) of length \(r\). Set
\[
 \boldsymbol e_Y\coloneq(\one_{p+2r+1},2,2),\qquad
 \boldsymbol f_Y\coloneq(\one_{p+2r+1},1,0),\qquad
 \boldsymbol a_Y\coloneq((-m_i)_i,(n_h+d_h)_h,(-d_h)_h,1),
\]
\[
 \boldsymbol u_Y\coloneq(\one_{p+r},\boldsymbol0_r,0,1,0),\qquad
 \boldsymbol v_Y\coloneq(\one_{p+r},\boldsymbol0_r,0,0,0).
\]
Here \(\boldsymbol a_Y\) lists the \(d_Y-2\) individual upper
parameters. The arguments, for \(c(t)\ne0\), are
\[
 \boldsymbol y(\upsilon;t)\coloneq
 \left(\left(-\frac{\xi_i}{2c(t)}\right)_i,
       \left(\frac{\rho_h}{2c(t)}\right)_h,
       \left(\frac{\rho_h}{2c(t)}\right)_h,
       -\frac\upsilon{2c(t)},-\frac1{2c(t)^2},\frac1{4c(t)^2}\right).
\]
For \(N\in\mathbb N_0\), define
\begin{equation}
 \mathcal Y_N(b,\upsilon;t)\coloneq(-2c(t))^N
 F_{2:0;\ldots;0}^{2:1;\ldots;1;0;0}
 \left[\begin{array}{c}
 (-N:\boldsymbol e_Y),(b:\boldsymbol u_Y):a_{Y,1};\ldots;a_{Y,d_Y-2};-;-\\
 (-N:\boldsymbol f_Y),(b:\boldsymbol v_Y):-;\ldots;-
 \end{array}\,\middle|\,\boldsymbol y(\upsilon;t)\right].
 \label{eq:Hermite-B-infinity-SD}
\end{equation}
This denotes the finite sum with
\(E\coloneq|\boldsymbol\alpha|+|\boldsymbol\beta|+l+2k+2v\le N\).
Its lower factor \((-N)_{|\boldsymbol\alpha|+|\boldsymbol\beta|+l+k}\)
is nonzero throughout that range, since its index is \(E-k-2v\le N\).
The common parameter \(b\) is
cancelled term by term according to
\[
 \frac{(b)_{|\boldsymbol\alpha|+k}}{(b)_{|\boldsymbol\alpha|}}
 =(b+|\boldsymbol\alpha|)_k.
\]
Consequently the specialization is polynomial in \(b\) for every
\(b\in\mathbb C\). Its powers of \(c(t)\) are \(N-E\), so its value
at \(c(t)=0\) is the sum of the terms with \(E=N\), after the
prefactor and arguments have been combined. All other terms vanish.
Set
\(\mathcal Y_N=0\) for \(N<0\).

The contributions from infinity are
\begin{equation}
 \mathcal P_h(t)\coloneq\frac{2(-1)^Md_h}{\rho_h}
        \mathcal Y_{R-2}(d_*-R+1,-\rho_h;t),\qquad
 \mathcal Q_1(t)\coloneq(-1)^M\mathcal Y_{R-1}(d_*-R+1,0;t).
 \label{eq:Hermite-B-polynomial-contributions-SD}
\end{equation}
For the strong-normality criterion,
\(C[(j,k)\leftarrow\mathbf p]\) denotes \(C\) with column \((j,k)\)
replaced by \(\mathbf p\).

\begin{corollary}[Terminating series for the Hermite \(B\)-components]
 \label{cor:canonical-Hermite-explicit-B-components}
 Under the hypotheses of
 Proposition~\ref{prop:canonical-Hermite-B-components}, assume
 \(\Delta_{\mathrm H}\ne0\).
 The individual components are
 \begin{equation}
 B^{(h),\mathrm H}_{\boldsymbol\lambda,\boldsymbol m}(t)
 =\mathcal P_h(t)+\frac1{\rho_h}
  \left[\zeta_{h,0}+c_h(t)\mathcal T_h(t)+\frac{\mathcal W_h(t)}2
         +\sum_{J\ne h}\{\mathcal U_{h,J}(t)-\mathcal U_{J,h}(t)\}\right],
 \qquad 1\le h\le r,
 \label{eq:canonical-Hermite-B-beta-components}
 \end{equation}
 \begin{equation}
 D^{(1),\mathrm H}_{\boldsymbol\lambda,\boldsymbol m}(t)
 =\mathcal Q_1(t)+\frac12\sum_{J\in\mathcal I}\mathcal T_J(t).
 \label{eq:canonical-Hermite-B-base-component}
 \end{equation}
 They have the prescribed degree bounds, with zero components for zero
 indices, and give the Rodrigues decomposition
 \begin{equation}
 \mathcal B^{\mathrm H}_{\boldsymbol\lambda,\boldsymbol m}(t)
 =\sum_{h=1}^rB^{(h),\mathrm H}_{\boldsymbol\lambda,\boldsymbol m}(t)
                      K_{\rho_h,1}F(t)
  +D^{(1),\mathrm H}_{\boldsymbol\lambda,\boldsymbol m}(t)F(t).
 \label{eq:canonical-Hermite-B-component-decomposition}
 \end{equation}
 For \(r\ge1\), each \(B^{(h),\mathrm H}\) contains at
 most \(2r+2\) terminating functions, and \(D^{(1),\mathrm H}\) at
 most \(r+1\), with at most \(p+2r+3\) arguments. The problem is weakly
 normal. Strong normality holds if and only if
 \begin{equation}
 \det C[(j,\lambda_j-1)\leftarrow\mathbf p]\ne0
 \qquad\text{for every }j\text{ with }\lambda_j>0.
 \label{eq:Hermite-B-strong-normality-criterion}
 \end{equation}
\end{corollary}

\begin{proof}
 Put \(\gamma(z)\coloneq L_F(z)\). The logarithmic derivative is
 \[
 \frac{\gamma'(z)}{\gamma(z)}-t
 =\widehat a_*-t+\frac z2-\sum_{h=1}^r\frac{d_h}{z+\rho_h}.
 \]
 Inverse bilateral Laplace transformation sends
 \(\partial_z(\gamma S)-t\gamma S\) to zero after its argument is set
 equal to \(t\). Scalar rational reductions modulo this expression
 therefore give polynomial coefficients of the row functions.

 \emph{Finite poles.}
 For a pole with \(n_J\ge2\), fix \(J\) and put
 \[
 R_J(w)\coloneq
 \prod_i\left(1-\frac w{\xi_i+\rho_J}\right)^{m_i}
 \prod_{h\ne J}(1-x_h^{(J)}w)^{-n_h},\qquad
 h_J(w;t)\coloneq\e^{c_J(t)w+w^2/4}
                    \prod_{h\ne J}(1-x_h^{(J)}w)^{-d_h}.
 \]
 Expansion at \(z=-\rho_J\) gives
 \(\zeta_{J,K}=\kappa_J[w^{n_J-K-1}]R_J(w)\), including
 \eqref{eq:Hermite-B-simple-pole-coefficient}. Write
 \(h_J(w;t)=\sum_{v\ge0}H_vw^v\) and
 \(h_J(w;t)^{-1}=\sum_{v\ge0}\widetilde H_vw^v\).
 For \(K\ge1\), define, locally within the proof,
 \[
 \Theta_{K,a}\coloneq\sum_{v=0}^a
       \frac{H_v\widetilde H_{a-v}}{v-d_J-K}\quad(0\le a<K),\qquad
 S_{J,K}(z;t)\coloneq\sum_{a=0}^{K-1}\Theta_{K,a}(z+\rho_J)^{a-K},
 \]
 and take \(\Theta_{K,-1}=0\). Every denominator is strictly negative,
 including when \(d_J=0\). Direct differentiation cancels the poles of
 order at least two in
 \((z+\rho_J)^{-K-1}-S_{J,K}'-(\gamma'/\gamma-t)S_{J,K}\).
 Its residue at \(-\rho_h\), for \(h\ne J\), is
 \(d_h\sum_{a=0}^{K-1}(x_h^{(J)})^{K-a}\Theta_{K,a}\).
 Its constant term is \(-\Theta_{K,K-1}/2\), and its residue at
 \(-\rho_J\) is
 \[
 -c_J(t)\Theta_{K,K-1}-\tfrac12\Theta_{K,K-2}
 -\sum_{h\ne J}d_h\sum_{a=0}^{K-1}(x_h^{(J)})^{K-a}\Theta_{K,a}.
 \]

 To sum these expressions over \(K\), expand the following integral
 near \(w=0\):
 \begin{equation}
 \mathcal Z_{J,N}(b,\upsilon;t)
 =b[w^N]\left(R_J(w)h_J(w;t)
     \int_0^1\frac{v^{b-1}}{h_J(vw;t)(1-\upsilon vw)}\,\mathrm dv\right).
 \label{eq:Hermite-B-pole-SD-coefficient}
 \end{equation}
 This identity follows directly from the binomial and exponential
 series. The integral contributes
 \(b/(b+|\boldsymbol\beta|+l+j+2\nu)\), which is the coupled
 Pochhammer ratio in \eqref{eq:Hermite-B-pole-SD}. The remaining
 coefficient of \(\e^{c_J(t)w}\) supplies the parameter \(-N\),
 the prefactor \(c_J(t)^N/N!\), and the stated arguments.
 Thus \eqref{eq:Hermite-B-pole-SD-coefficient} is exactly the
 Srivastava--Daoust specialization, not an additional sum.

 In \(\Theta_{K,K-1}\), replacing \(v\) by \(K-1-v\) gives the
 denominator \(-(d_J+1+v)\). Multiplication by \(\zeta_{J,K}\) and
 coefficient multiplication in \eqref{eq:Hermite-B-pole-SD-coefficient}
 show that its sum is \(-\mathcal T_J\). The same argument with
 \(K-2\) gives \(-\mathcal W_J\). In an off-diagonal residue, the
 additional power of \(x_h^{(J)}\) gives the geometric factor
 \((1-x_h^{(J)}vw)^{-1}\); its summed value is
 \(\mathcal U_{h,J}\). Consequently, reducing all the terms at the
 pole \(-\rho_J\) leaves the residue
 \[
 \zeta_{J,0}+c_J(t)\mathcal T_J(t)+\frac12\mathcal W_J(t)
       -\sum_{h\ne J}\mathcal U_{h,J}(t)
 \]
 at that pole, the residue \(\mathcal U_{h,J}\) at each other pole
 \(-\rho_h\), and the constant \(\mathcal T_J/2\). Summing over
 \(J\) gives the difference
 \(\mathcal U_{h,J}-\mathcal U_{J,h}\) in row \(h\).
 It also separates the two factors \(1/2\): \(\mathcal W_h/2\)
 belongs to the residue at \(-\rho_h\), whereas
 \(\tfrac12\sum_J\mathcal T_J\) is the constant term and therefore
 multiplies the row \(F\). This proves every finite-pole term in the
 two component formulas, including the factors \(1/\rho_h\)
 required by the bilateral transform identity
 \(\int_{\mathbb R}\e^{zv}(K_{\rho_h,1}F)(v)\,\mathrm dv
 =\rho_h\gamma(z)/(z+\rho_h)\).

 \emph{The polynomial part.}
 The scalar coefficient polynomials needed at infinity can be written
 as one double sum:
 \begin{equation}
 a_n(\nu,c)\coloneq
 (-1)^n\sum_{2k+2v\le n}
 \frac{2^{n-k-2v}(-\nu)_k(n-k)!}
      {k!v!(n-2k-2v)!}\,c^{n-2k-2v}.
 \label{eq:Hermite-B-infinity-coefficients}
 \end{equation}
 Substitution, with the indices shifted in the terms of degree
 \(n-2\), gives
 \(a_0=1\), \(a_1=-2c\), and
 \[
 a_n(\nu,c)=-2c\,a_{n-1}(\nu,c)
                -2(\nu+1-n)a_{n-2}(\nu,c)\qquad(n\ge2).
 \]
 Consequently the formal series
 \(Q_\nu(z;c)\coloneq2z^{\nu-1}\sum_{n\ge0}a_n(\nu,c)z^{-n}\)
 satisfies \(Q_\nu'+(z/2+c)Q_\nu=z^\nu\). Only finitely many
 coefficients of this formal series are used.

 Set \(\mathcal E(z)\coloneq\prod_h(1+\rho_h/z)^{-d_h}\), and define
 \(\chi_u\) by
 \[
 \sum_{u\ge0}\chi_uw^u
 =\prod_i(1-\xi_iw)^{m_i}\prod_h(1+\rho_hw)^{-n_h-d_h}.
 \]
 A formal solution of
 \(S'+(\gamma'/\gamma-t)S=P_{\boldsymbol m}/D_{\boldsymbol n}\)
 at infinity is
 \[
 S(z;t)=(-1)^M z^{d_*}\mathcal E(z)^{-1}
       \sum_{u\ge0}\chi_uQ_{R-1-d_*-u}(z;c(t)).
 \]
 The prefactors cancel all noninteger powers of \(z\).
 The finite-pole part of the right-hand side contributes only powers
 \(z^{-2},z^{-3},\ldots\) to \(S\). Hence its polynomial part
 \(S_+(z;t)\) and its coefficient of \(z^{-1}\) depend only on
 \(\sum_l\pi_lz^l\). Expand \(\mathcal E^{-1}\) by the binomial theorem and
 substitute \eqref{eq:Hermite-B-infinity-coefficients}. The coefficient
 calculation gives
 \[
 S_+(\upsilon;t)=2(-1)^M\mathcal Y_{R-2}(d_*-R+1,\upsilon;t),\qquad
 \tfrac12[z^{-1}]S(z;t)=(-1)^M\mathcal Y_{R-1}(d_*-R+1,0;t).
 \]
 For clarity, the factorial ratio in this calculation is
 \[
 \frac{(N-|\boldsymbol\alpha|-|\boldsymbol\beta|-l-k)!}{(N-E)!}
 =(-1)^k\frac{(-N)_E}
              {(-N)_{|\boldsymbol\alpha|+|\boldsymbol\beta|+l+k}}.
 \]
 The other coupled factor is
 \((d_*-R+1+|\boldsymbol\alpha|)_k\). These are precisely the
 two coupled ratios of \eqref{eq:Hermite-B-infinity-SD}; the binomial
 factors give its individual parameters. Removing the negative powers
 of \(S\) leaves a constant and simple-pole terms. The latter have
 coefficients \(d_hS_+(-\rho_h;t)\). Inverse transformation therefore
 gives \(\mathcal P_h\) and \(\mathcal Q_1\) in
 \eqref{eq:Hermite-B-polynomial-contributions-SD}.

 \emph{Degrees and normality.}
 The degrees of \(\mathcal Z_{J,N}\) and \(\mathcal Y_N\) are at most
 \(N\). Near diagonality gives
 \(n_J-2\le n_h-1\), \(R-2\le n_h-1\), and \(n_J-2\le R-1\),
 so the formulas have all prescribed degree bounds. Zero indices give
 zero components.

 Expanding the components in the monic Hermite basis gives coefficients
 \(b_{j,k}\) satisfying the polynomial identity
 \begin{equation}
 P_{\boldsymbol m}(z)
 =\sum_{(j,k)\in\mathcal I_{\boldsymbol\lambda}}
                b_{j,k}\Phi_{j,k}(z).
 \label{eq:canonical-Hermite-B-polynomial-identity}
 \end{equation}
 Its coefficient matrix is \(C\), so \(\Delta_{\mathrm H}\ne0\)
 proves uniqueness. The adjacent moment is
 \begin{equation}
 \int_{\mathbb R}t^{m_{i_0}}\e^{\xi_{i_0}t}
      \mathcal B^{\mathrm H}_{\boldsymbol\lambda,\boldsymbol m}(t)\,\mathrm dt
 =\frac{L_F(\xi_{i_0})}{D_{\boldsymbol n}(\xi_{i_0})}
   (-1)^{m_{i_0}}m_{i_0}!\prod_{i\ne i_0}(\xi_i-\xi_{i_0})^{m_i}\ne0,
 \label{eq:canonical-Hermite-B-adjacent-moment}
 \end{equation}
 obtained by differentiating \(L_FP_{\boldsymbol m}/D_{\boldsymbol n}\)
 at \(\xi_{i_0}\). This proves weak normality. The top Hermite
 coefficient in component \(j\) is
 \(\det C[(j,\lambda_j-1)\leftarrow\mathbf p]/\Delta_{\mathrm H}\),
 proving the strong-normality criterion.
\end{proof}

For \(r=0\), there are no finite poles, and the single polynomial
\(\mathcal Q_1\) is the classical multiple Hermite polynomial in the
normalization of this section.

For fixed multi-indices and fixed compact sets, the estimates below have
constants independent of \(\tau\), with the uniform and coefficientwise
conventions of Section~\ref{subsec:hypergeometric-notation}.
For the estimates concerning the
\(B\)-form, with \(|\boldsymbol m|=N-1\), let \(Q_{B,\tau}\) be the
transform in \eqref{eq:canonical-Hermite-finite-B-transform}, and put
\[
	Q_B(z)\coloneq
	L_F(z)\frac{\prod_{i=1}^p(\xi_i-z)^{m_i}}
	{\prod_{h=1}^r(\rho_h+z)^{n_h}}.
\]

\begin{proposition}[Quantitative fixed-index confluence]
	\label{prop:canonical-Hermite-quantitative-confluence}
	As \({\tau\to+\infty}\), the fixed-index error satisfies
	\(\Theta_\tau=\mathrm o(1)\) and the following estimates hold.
	\begin{enumerate}[label=\textnormal{(\arabic*)},beginpenalty=10000]
	\item For
	\(1\le i\le p\), \(1\le j\le r+1\), and \(v\in\mathbb N_0\),
	\[
		\int_{-\varepsilon_\tau^{-1}}^\infty
		t^v(1+\varepsilon_\tau t)^{\beta_i(\tau)}
		\widehat U_{j,\tau}(t)\,\mathrm dt
		=
		\int_{\mathbb R}
		t^v\e^{\xi_it}U_j^{\mathrm H}(t)\,\mathrm dt
		+\mathrm O(\Theta_\tau)\qquad (\tau\to+\infty).
\]

	\item For the \(A\)-balance \(|\boldsymbol m|=N+1\) and every compact set
	\(T\subset\mathbb C\),
	\[
		\sup_{t\in T}
		\left|
			Z_\tau\varepsilon_\tau^{-\eta}
			\mathcal A_\tau^{\mathrm L}(x_\tau(t))
			-
			\mathcal A^{\mathrm H}_{\boldsymbol\lambda,\boldsymbol m}(t)
		\right|
		=
		\mathrm O_T(\Theta_\tau)\qquad (\tau\to+\infty),
\]
	and, coefficientwise,
	\[
		Z_\tau\varepsilon_\tau^{-\eta}
		c_\tau^{\beta_i(\tau)}A_\tau^{(i),\mathrm L}(x_\tau(t))
		-
		A_{\boldsymbol\lambda,\boldsymbol m}^{(i),\mathrm H}(t)
		=\mathrm O(\Theta_\tau)\qquad (\tau\to+\infty).
\]

	\item For the \(B\)-balance \(|\boldsymbol m|=N-1\), every compact
	subset \(K\) of the common holomorphy domain, and every
	fixed \(v\in\mathbb N_0\),
	\[
		\sup_{z\in K}
		\left|
			\partial_z^vQ_{B,\tau}(z)-\partial_z^vQ_B(z)
		\right|
		=
		\mathrm O_{K,v}(\Theta_\tau)\qquad (\tau\to+\infty).
\]
	The difference between the two sides of
	\eqref{eq:canonical-Hermite-B-moment-convergence} is consequently
	\(\mathrm O(\Theta_\tau)\ {(\tau\to+\infty)}\).

	\item If \(\Delta_{\mathrm H}\ne0\), then the finite-\(\tau\) \(B\)-problem is
	weakly normal for every sufficiently large \(\tau\). Its polynomial components
	satisfy, coefficientwise in the monomial basis in \(t\),
	\begin{align}
		\frac{\varepsilon_\tau^{\eta-1}}{b_h(\tau)+1}
		B_\tau^{(h),\mathrm L}(x_\tau(t))
		&=
		B^{(h),\mathrm H}_{\boldsymbol\lambda,\boldsymbol m}(t)
		+\mathrm O(\Theta_\tau)\qquad (\tau\to+\infty),
		\label{eq:canonical-Hermite-B-beta-component-limit}
\\
		\varepsilon_\tau^{\eta-1}
		D_\tau^{(1),\mathrm L}(x_\tau(t))
		&=
		D^{(1),\mathrm H}_{\boldsymbol\lambda,\boldsymbol m}(t)
		+\mathrm O(\Theta_\tau)\qquad (\tau\to+\infty).
		\label{eq:canonical-Hermite-B-base-component-limit}
\end{align}
	\end{enumerate}
	If, in addition,
	\[
		u_{h,\tau}=\rho_h+\mathrm O(\varepsilon_\tau),
		\qquad
		D_{h,\tau}=d_h+\mathrm O(\varepsilon_\tau),
		\qquad
		\varepsilon_\tau\beta_i(\tau)=\xi_i+\mathrm O(\varepsilon_\tau)\qquad (\tau\to+\infty),
\]
	for all \(h\) and \(i\), then every error above is
	\(\mathrm O(\varepsilon_\tau)=\mathrm O(\tau^{-1/2})\ {(\tau\to+\infty)}\).
\end{proposition}

\begin{proof}
	\textnormal{(1)} Lemma~\ref{lem:canonical-Hermite-uniform-Gamma-quotient} gives
	\(Q_{0,\tau}-L_F=\mathrm O(\Omega_\tau)\ {(\tau\to+\infty)}\), locally uniformly with every fixed
	derivative.  Put
	\[
		r_{h,\tau}(z)
		\coloneq
		\frac{\varepsilon_\tau(b_h(\tau)+1)}
		{\varepsilon_\tau(b_h(\tau)+1)+z},
		\qquad
		r_h(z)\coloneq\frac{\rho_h}{\rho_h+z}.
\]
	On every compact subset of the common holomorphy domain,
	\(r_{h,\tau}-r_h=\mathrm O(\Omega_\tau)\ {(\tau\to+\infty)}\), and therefore
	\[
		Q_{h,\tau}-\Lambda_h
		=
		r_{h,\tau}(Q_{0,\tau}-L_F)+(r_{h,\tau}-r_h)L_F
		=
		\mathrm O(\Omega_\tau)\qquad (\tau\to+\infty).
\]
 For the last row, \(Q_{q,\tau}-\Lambda_q=Q_{0,\tau}-L_F\), so it
 obeys the same bound.
	Cauchy's formula gives these estimates for every fixed derivative.
	Evaluation at
	\(z_{i,\tau}=\varepsilon_\tau\beta_i(\tau)=\xi_i+\mathrm O(\chi_\tau)\ {(\tau\to+\infty)}\), followed by
	the exact moment identity
	\eqref{eq:canonical-Hermite-finite-difference-moments} and the integral
	formula
	\eqref{eq:canonical-Hermite-finite-difference-integral}, proves the mixed-type
	moment estimate.

	\textnormal{(2)} In the exact contour identity
	\eqref{eq:canonical-Hermite-A-exact-scaled-contour}, all factors converge
	with error \(\mathrm O(\Theta_\tau)\ {(\tau\to+\infty)}\), uniformly on the fixed circles and
	for \(t\in T\). The limiting denominator has a positive minimum in
	modulus on these circles. Its finite-\(\tau\) counterpart is therefore
	bounded away from zero for \(\tau\ge \tau_T\). Subtraction of the two
	integrands gives a bound \(C_T\Theta_\tau\); integration over the fixed
	contour multiplies this bound by its length divided by \(2\pi\).
	This proves the estimate for the complete \(A\)-form.
	The contribution of the circle surrounding the poles
	\(\varepsilon_\tau\beta_i(\tau)+\nu\varepsilon_\tau\), \(0\le\nu<m_i\), is the
	product of the scaled polynomial component and its moving column factor.
	This product differs from its limit by \(\mathrm O_T(\Theta_\tau)\ {(\tau\to+\infty)}\). The moving
	column factor is uniformly nonzero on every fixed compact set and differs
	from \(\e^{\xi_it}\) by \(\mathrm O_T(\Theta_\tau)\ {(\tau\to+\infty)}\); division therefore gives
	an \(\mathrm O_T(\Theta_\tau)\ {(\tau\to+\infty)}\) estimate for the polynomial component.
	Evaluation at \(m_i\) fixed
	distinct points, followed by inversion of the resulting fixed Vandermonde
	matrix, proves the coefficient estimates.

	\textnormal{(3)} For fixed \(n_h\) and \(m_i\), the factors in
	\eqref{eq:canonical-Hermite-finite-B-transform} satisfy, uniformly on
	\(K\),
	\[
		P_{h,\tau}^{[n_h]}(z)
		=
		(\rho_h+z)^{n_h}+\mathrm O_K(\Omega_\tau),
		\qquad
		\prod_{j=0}^{m_i-1}
		[\varepsilon_\tau(\beta_i(\tau)+j)-z]
		=
		(\xi_i-z)^{m_i}+\mathrm O_K(\Theta_\tau)\qquad (\tau\to+\infty).
\]
	Together with the estimate for \(Q_{0,\tau}\), the exact formula
	\eqref{eq:canonical-Hermite-finite-B-transform} now gives
	\(Q_{B,\tau}-Q_B=\mathrm O_K(\Theta_\tau)\ {(\tau\to+\infty)}\). Cauchy's formula on a slightly
	larger compact set gives the derivative estimates, and the finite
	difference identity gives the estimate for the moments with exponential factors.

	\textnormal{(4)} Assume \(\Delta_{\mathrm H}\ne0\), and fix
	\(i_0\in\{1,\ldots,p\}\). Put
	\[
		\widehat{\mathcal B}_\tau(t)
		\coloneq
		\frac{c_\tau\varepsilon_\tau^{\eta}}{Z_\tau}
		\mathcal B_\tau^{\mathrm L}(x_\tau(t)).
\]
	For the lexicographically ordered set
	\(\mathcal J_{\boldsymbol m,i_0}\) in
	Proposition~\ref{prop:canonical-Hermite-B-components}, define
	\[
		(\widetilde{\mathcal N}_{i_0,\tau})_{(i,\ell),(j,k)}
		\coloneq
		\int_{I_\tau}t^\ell(1+\varepsilon_\tau t)^{\beta_i(\tau)}
		H_k(t)\widehat U_{j,\tau}(t)\,\mathrm dt
\]
	and
	\[
		(y_\tau)_{(i,\ell)}
		\coloneq
		\int_{I_\tau}t^\ell(1+\varepsilon_\tau t)^{\beta_i(\tau)}
		\widehat{\mathcal B}_\tau(t)\,\mathrm dt.
\]
	Let
	\[
		y_{(i,\ell)}
		\coloneq
		\int_{\mathbb R}t^\ell\e^{\xi_it}
		\mathcal B^{\mathrm H}_{\boldsymbol\lambda,\boldsymbol m}(t)
		\,\mathrm dt.
\]
	The mixed-type moment estimate already proved gives
	\[
		\widetilde{\mathcal N}_{i_0,\tau}
		=
		\widetilde{\mathcal N}_{i_0}
		+\mathrm O(\Theta_\tau)\qquad (\tau\to+\infty),
\]
	entrywise, while the moment estimate for \(Q_{B,\tau}\) gives
	\(y_\tau=y+\mathrm O(\Theta_\tau)\ {(\tau\to+\infty)}\). The matrix
	\(\widetilde{\mathcal N}_{i_0}\) is the Hermite-basis version of
	\(\mathcal N_{i_0}\); hence it is invertible by
	\eqref{eq:canonical-Hermite-B-connection-determinant}. Since
	\[
		\det\widetilde{\mathcal N}_{i_0,\tau}
		\xrightarrow[\tau\to+\infty]{}
		\det\widetilde{\mathcal N}_{i_0}\ne0,
\]
	the finite matrices are invertible for all sufficiently large \(\tau\).
	This proves weak normality of the finite row-normalized \(B\)-problem in
	the Hermite polynomial basis. To make the identification with the complete
	representative explicit, apply
	\eqref{eq:canonical-Hermite-finite-difference-moments} to the rows of
	\(\widetilde{\mathcal N}_{i_0,\tau}\). It expresses those \(N\) functionals
	as invertible triangular combinations of centered transforms at the points
	\[
		\varepsilon_\tau\beta_i(\tau)+\nu\varepsilon_\tau,
		\qquad 0\le\nu<q_i.
\]
	These points are distinct for all sufficiently large \(\tau\), because the
	limits \(\xi_i\) are pairwise distinct. The Mellin calculation in
	\eqref{eq:B-component-polynomial-system} shows that the centered transform
	of every admissible row combination is a nonzero common factor times a
	polynomial of degree at most \(N-1\). Equation
	\eqref{eq:canonical-Hermite-finite-B-transform} gives this polynomial
	for \(\widehat{\mathcal B}_\tau\). Interpolation at the displayed \(N\) points
	therefore proves that the unique solution of the finite moment problem is
	exactly \(\widehat{\mathcal B}_\tau\).

	Denote its scaled polynomial components by
	\[
		\widetilde B_{h,\tau}(t)
		\coloneq
		\frac{\varepsilon_\tau^{\eta-1}}{b_h(\tau)+1}
		B_\tau^{(h),\mathrm L}(x_\tau(t)),
		\qquad
		\widetilde D_{1,\tau}(t)
		\coloneq
		\varepsilon_\tau^{\eta-1}
		D_\tau^{(1),\mathrm L}(x_\tau(t)),
\]
	and write
	\[
		\widetilde B_{h,\tau}(t)
		=
		\sum_{k=0}^{n_h-1}(b_\tau)_{h,k}H_k(t),
		\qquad
		\widetilde D_{1,\tau}(t)
		=
		\sum_{k=0}^{\eta-1}(b_\tau)_{q,k}H_k(t).
\]
	Then
	\[
		\widehat{\mathcal B}_\tau(t)
		=
		\sum_{(j,k)\in\mathcal I_{\boldsymbol\lambda}}
		(b_\tau)_{j,k}H_k(t)\widehat U_{j,\tau}(t)
\]
	and hence
	\(\widetilde{\mathcal N}_{i_0,\tau}b_\tau=y_\tau\) exactly. Let \(b\) be the
	Hermite coefficient vector in
	\eqref{eq:canonical-Hermite-B-beta-components}--
	\eqref{eq:canonical-Hermite-B-base-component}; then
	\(\widetilde{\mathcal N}_{i_0}b=y\). Write
	\(M_\tau=\widetilde{\mathcal N}_{i_0,\tau}\) and
	\(M=\widetilde{\mathcal N}_{i_0}\). In the estimates below, vector
	norms are Euclidean and matrix norms are the induced operator norms.
	Since \(M\) is invertible and
	\(\|M_\tau-M\|\le C\Theta_\tau\), for sufficiently large \(\tau\) one has
	\(\|M^{-1}(M_\tau-M)\|\le1/2\), hence
	\(\|M_\tau^{-1}\|\le2\|M^{-1}\|\).
	The elementary identity
	\[
		M_\tau^{-1}-M^{-1}
		=
		-M_\tau^{-1}(M_\tau-M)M^{-1}
\]
	gives \(\|M_\tau^{-1}-M^{-1}\|\le2C\|M^{-1}\|^2\Theta_\tau\).
	Likewise, \(b_\tau-b=M_\tau^{-1}[(y_\tau-y)-(M_\tau-M)b]\), so that
	\(\|b_\tau-b\|\le C'\Theta_\tau\) for all sufficiently large \(\tau\).
	This proves
	\eqref{eq:canonical-Hermite-B-beta-component-limit} and
	\eqref{eq:canonical-Hermite-B-base-component-limit}.

	The scaling limits and
	\eqref{eq:canonical-Hermite-Omega-rate} give \(\Theta_\tau=\mathrm o(1)\ {(\tau\to+\infty)}\). Under the
	additional first-order estimates in the statement,
	\(\Theta_\tau=\mathrm O(\varepsilon_\tau)=\mathrm O(\tau^{-1/2})\ {(\tau\to+\infty)}\).
\end{proof}

\subsection{Classical multiple Hermite reductions}
\label{subsec:Hermite-classical-reductions}

When only one row remains, the system reduces to the classical multiple
Hermite weights, for which all multi-indices are strongly normal. The
comparison with coinciding Gaussian centres is given separately in
Appendix~\ref{app:Hermite-derivative-limits}.

For the scalar-row reduction below, specialize
Corollary~\ref{cor:simultaneous-Laguerre-Hermite-matrix-confluence} to
\(r=0\), hence \(q=1\), and suppose that its limiting exponents
\(\xi_1,\ldots,\xi_p\) are pairwise distinct. Set
\[
	a\coloneq\widehat a_*,
	\qquad
	c_i\coloneq\xi_i+2a.
\]
Then \(F=\phi_a\), and the entries of the limiting matrix are
\begin{equation}
	\label{eq:q1-Hermite-limiting-weights}
	\phi_a(t)\e^{\xi_it}
	=
	\pi^{-1/2}\e^{-a^2}\e^{-t^2+c_it},
	\qquad 1\le i\le p.
\end{equation}
Thus, up to a common positive constant, the limiting weights are the
classical multiple Hermite weights
\[
	w_i^{\mathrm H}(t)\coloneq\e^{-t^2+c_it},
	\qquad 1\le i\le p.
\]
For functions \(f_1,\ldots,f_n\) with \(n-1\) continuous derivatives,
their Wronskian is
\[
W(f_1,\ldots,f_n)(t)\coloneq
\det[f_j^{(k-1)}(t)]_{k,j=1}^n.
\]

\begin{lemma}[Strong normality of the multiple Hermite system]
	\label{lem:q1-multiple-Hermite-normality}
	For every
	\(\boldsymbol\nu=(\nu_1,\ldots,\nu_p)\in\mathbb N_0^p\) with
	\(\lvert\boldsymbol\nu\rvert\ge1\), the family
	\[
		\left\{
			t^kw_i^{\mathrm H}(t):
			1\le i\le p,\ 0\le k<\nu_i
		\right\}
\]
	is an extended complete Chebyshev system on \(\mathbb R\). Consequently,
	the corresponding type-I and type-II multiple Hermite problems are
	strongly normal. The same holds for every multi-index obtained by
	increasing or decreasing a component, provided all components remain
	nonnegative and their sum is positive.
\end{lemma}

\begin{proof}
	After ordering the \(c_i\)'s increasingly, the Wronskian of the
	exponential polynomials
	\[
		t^k\e^{c_it},
		\qquad
		1\le i\le p,\quad 0\le k<\nu_i,
\]
	is
	\[
		\e^{(\sum_i\nu_ic_i)t}
		\left(\prod_{i=1}^p\prod_{k=0}^{\nu_i-1}k!\right)
		\prod_{1\le i<j\le p}(c_j-c_i)^{\nu_i\nu_j},
\]
	and is therefore strictly positive in this block order. It follows by
	confluence from the Vandermonde determinant in the exponents; equivalently,
	the columns are successive derivatives with respect to \(c_i\) of the
	exponential column. The same formula with the last multiplicity truncated
	applies to every initial subfamily.

	For completeness, the zero-counting property can be proved directly.
	A nonzero exponential polynomial \(f=\sum_i p_i\e^{c_it}\), with
	\(\deg p_i<\nu_i\), has at most \(|\boldsymbol\nu|-1\) real zeros,
	counted with multiplicity. Induct on \(|\boldsymbol\nu|\), omitting zero
	multiplicities. Multiplication by \(\e^{-c_1t}\) preserves its zeros.
	Differentiation then decreases the first polynomial degree by one and
	replaces each other polynomial by \(p_i'+(c_i-c_1)p_i\), without
	increasing its degree. The resulting space has dimension
	\(|\boldsymbol\nu|-1\). If the derivative is nonzero, induction and
	Rolle's theorem give the asserted bound. If it is zero, the original
	function is a nonzero multiple of \(\e^{c_1t}\), and has no zeros.
	The same argument applies to each initial subfamily, proving the
	extended complete Chebyshev property. Multiplication by
	\(\e^{-t^2}>0\) preserves the zeros and their multiplicities.

	To pass from zeros to moment determinants, put
	\(n\coloneq|\boldsymbol\nu|\) and list the exponential polynomials in
	the chosen order as \(f_1,\ldots,f_n\). Define
	\[
		G_{\boldsymbol\nu}\coloneq
		\left[\int_{\mathbb R}t^a f_b(t)\e^{-t^2}\,\mathrm dt\right]_
		{0\le a<n,\,1\le b\le n}.
\]
	Expansion of the determinants and Fubini's theorem give
	\[
		\det G_{\boldsymbol\nu}
		=\frac1{n!}\int_{\mathbb R^n}
		\det[t_j^a]_{a,j}\det[f_b(t_j)]_{j,b}
		\prod_{j=1}^n\e^{-t_j^2}\,\mathrm dt_1\cdots\mathrm dt_n.
\]
	All terms are absolutely integrable because a Gaussian dominates each
	polynomial times an exponential. On \(t_1<\cdots<t_n\), the monomial
	determinant is positive. The other determinant never vanishes, since a
	vanishing determinant would yield a nonzero exponential polynomial with
	\(n\) distinct zeros. Its sign is constant on this connected region;
	letting the points approach one another and using the positive Wronskian
	shows that the sign is positive. The product is invariant under
	permutations of the points. Hence \(\det G_{\boldsymbol\nu}>0\).

	This proves existence and uniqueness of the monic type-II polynomial of
	degree \(n\), and of the type-I vector normalized by its moment of order
	\(n-1\). To check its component degrees, the coefficient of
	\(t^{\nu_i-1}\) in the \(i\)-th type-I component is, by Cramer's rule,
	up to sign,
	\[
		\frac{\det G_{\boldsymbol\nu-\boldsymbol e_i}}
		{\det G_{\boldsymbol\nu}},\qquad \nu_i>0,
\]
	where the determinant for the zero multi-index is one. Both determinants
	are nonzero by the preceding argument. Every prescribed type-I degree is
	therefore attained, and the type-II degree is fixed by its monic
	normalization. Applying the same reasoning to any adjacent multi-index
	proves all the stated strong-normality assertions.
\end{proof}

Fix \(\boldsymbol\nu\in\mathbb N_0^p\), put
\(n\coloneq|\boldsymbol\nu|\), and assume \(n\ge1\). Let
\(H_{\boldsymbol\nu}(t;\boldsymbol c)\) denote the unique monic polynomial of degree
\(n\) such that
\[
	\int_{\mathbb R}
	t^kH_{\boldsymbol\nu}(t;\boldsymbol c)
	w_i^{\mathrm H}(t)\,\mathrm dt=0,
	\qquad
	0\le k<\nu_i,
\]
and by \(H_{\boldsymbol\nu}^{(i)}(t;\boldsymbol c)\) the type-I components,
normalized by
\[
	\int_{\mathbb R}t^k
	\sum_{i=1}^p
	H_{\boldsymbol\nu}^{(i)}(t;\boldsymbol c)
	w_i^{\mathrm H}(t)\,\mathrm dt
	=
	\delta_{k,n-1},
	\qquad
	0\le k\le n-1.
\]
	Here
	\(\deg H_{\boldsymbol\nu}^{(i)}\le\nu_i-1\), with
	\(H_{\boldsymbol\nu}^{(i)}\equiv0\) when \(\nu_i=0\). These are the
	normalizations used in
	\cite[Equation~(24) and Theorem~5.1, equation~(28)]
	{BranquinhoDiazFoulquieManas2025Classical}. For \(p=1\), the same
	normalization gives the scalar Hermite specialization.

In the monic scalar normalization fixed by
\eqref{eq:monic-Hermite-generating-function}, which agrees with
\cite[Equation~(25)]{BranquinhoDiazFoulquieManas2025Classical},
distribute the exterior factor \(\prod_i c_i^{\nu_i}\) in equation~(24) of
that paper inside its finite sum. The resulting polynomial identity is
\begin{equation}
	\label{eq:q1-Hermite-type-II-explicit}
	H_{\boldsymbol\nu}(t;\boldsymbol c)
	=
	\left(-\frac12\right)^n
	\sum_{\ell_1=0}^{\nu_1}\cdots
	\sum_{\ell_p=0}^{\nu_p}
	2^{|\boldsymbol\ell|}
	\prod_{i=1}^p
	\frac{(-\nu_i)_{\ell_i}c_i^{\nu_i-\ell_i}}{\ell_i!}
	H_{|\boldsymbol\ell|}(t).
\end{equation}

Set
\[
	\boldsymbol\lambda_{\mathrm A}\coloneq(n-1),
	\qquad
	\boldsymbol\lambda_{\mathrm B}\coloneq(n+1).
\]

\begin{corollary}[Exact reduction to multiple Hermite type I and II]
	\label{cor:q1-canonical-Hermite-reduction}
	For
	\((\boldsymbol m,\boldsymbol\lambda)
	=(\boldsymbol\nu,\boldsymbol\lambda_{\mathrm A})\), the components satisfy
	\begin{equation}
		\label{eq:q1-A-reduction-to-SAPM-type-I}
		A_{\boldsymbol\lambda_{\mathrm A},\boldsymbol\nu}^{(i),\mathrm H}(t)
		=
		\frac{
			(-1)^n\sqrt\pi\,\e^{a^2}
		}{
			2^{n-1}
		}
		H_{\boldsymbol\nu}^{(i)}(t;\boldsymbol c).
	\end{equation}
	If \(\nu_i=0\), both sides are zero. For
	\((\boldsymbol m,\boldsymbol\lambda)
	=(\boldsymbol\nu,\boldsymbol\lambda_{\mathrm B})\), the complete form satisfies
	\begin{equation}
		\label{eq:q1-B-reduction-to-SAPM-type-II}
		\mathcal B_{\boldsymbol\lambda_{\mathrm B},\boldsymbol\nu}^{\mathrm H}(t)
		=
		(-2)^n
		H_{\boldsymbol\nu}(t;\boldsymbol c)\phi_a(t).
	\end{equation}
\end{corollary}

\begin{proof}
	For the type-I form, set
	\[
		L_a(z)=\e^{az+z^2/4},
		\qquad
		P_{\boldsymbol\nu}(z)
		=\prod_{i=1}^p(\xi_i-z)^{\nu_i}.
\]
	Equation~\eqref{eq:canonical-Hermite-A-limit-contour} becomes
	\[
		\mathcal A^{\mathrm H}(t)
		=
		\frac1{2\pi\mathrm i}
		\oint_{\Sigma_\xi}
		\frac{\e^{zt}}{L_a(z)P_{\boldsymbol\nu}(z)}\,\mathrm dz.
\]
	Fubini's theorem and differentiation of the Gaussian transform give,
	for \(0\le k\le n-1\),
	\[
		\int_{\mathbb R}t^k\phi_a(t)\mathcal A^{\mathrm H}(t)\,\mathrm dt
		=
		\frac1{2\pi\mathrm i}
		\oint_{\Sigma_\xi}
		\frac{L_a^{(k)}(z)}
		{L_a(z)P_{\boldsymbol\nu}(z)}\,\mathrm dz.
\]
	The quotient \(L_a^{(k)}/L_a\) is a polynomial of degree \(k\) with
	leading coefficient \(2^{-k}\). The residue at infinity shows that the
	last integral is zero for \(k<n-1\) and equals
	\[
		\frac{(-1)^n}{2^{n-1}}
\]
	for \(k=n-1\). Moreover,
	\[
		\phi_a(t)\e^{\xi_it}
		=
		\pi^{-1/2}\e^{-a^2}\e^{-t^2+c_it}.
\]
	The type-I normalization gives moment one at order \(n-1\). Uniqueness
	from Lemma~\ref{lem:q1-multiple-Hermite-normality} and comparison of those
	moments give
	\eqref{eq:q1-A-reduction-to-SAPM-type-I}.

	For the type-II form, Rodrigues' formula gives
	\[
		\mathcal B^{\mathrm H}(t)
		=
		\prod_{i=1}^p
		(\xi_i+\partial_t)^{\nu_i}\phi_a(t).
\]
	Write \(\mathcal B^{\mathrm H}=D\phi_a\). Conjugation by \(\phi_a\)
	shows that \(D\) has degree \(n\) and leading coefficient \((-2)^n\).
	For \(0\le k<\nu_i\), repeated integration by parts gives
	\[
		\int_{\mathbb R}
		t^k\e^{\xi_it}\mathcal B^{\mathrm H}(t)\,\mathrm dt=0,
\]
	because the formal adjoint of \(\xi_i+\partial_t\) sends
	\(t^k\e^{\xi_it}\) to
	\(-kt^{k-1}\e^{\xi_it}\). Lemma~\ref{lem:q1-multiple-Hermite-normality}
	therefore identifies \((-2)^{-n}D\) with the monic type-II multiple
	Hermite polynomial, which proves
	\eqref{eq:q1-B-reduction-to-SAPM-type-II}.
\end{proof}

\begin{remark}
	On the one-row parameter path used in
	\cite{BranquinhoDiazFoulquieManas2025Classical},
	\eqref{eq:canonical-Hermite-A-component-convergence}, with the
	normalization in \eqref{eq:q1-A-reduction-to-SAPM-type-I}, recovers
	equation~(30). Likewise,
	\eqref{eq:canonical-Hermite-B-scaled-limit}, together with the scalar row
	limit and \eqref{eq:q1-B-reduction-to-SAPM-type-II}, recovers
	equation~(27).
\end{remark}

\subsection{Direct confluence from Laguerre of the second kind}
\label{subsec:LII-Hermite-direct}

The Hermite family defined above can also be obtained directly from the
additive-convolution rows of Laguerre of the second kind. One gamma shape
tends to infinity, while the remaining shapes stay fixed and their rates
decrease on the square-root scale. Centering, reflection, and rescaling
then give the same Gaussian--gamma convolution and exponential
convolutions as in \eqref{eq:canonical-Hermite-limiting-vectors}.
Only constant row and column normalizations are needed; the polynomial
degree bounds are therefore preserved.

Fix \(r\ge0\), put \(q=r+1\), and take \(d_h,\rho_h>0\) for
\(1\le h\le r\), together with \(\widehat a_*\in\mathbb R\).
Let the real column nodes \(\xi_i\) satisfy
\(\xi_i+\rho_h>0\) for every \(i,h\). For \(\tau>0\), define
\begin{equation}
 h_\tau\coloneq\sqrt{2\tau},\qquad
 \boldsymbol d_\tau\coloneq(d_1,\ldots,d_r,\tau),\qquad
 \boldsymbol\varrho_\tau\coloneq
 (\rho_1/h_\tau,\ldots,\rho_r/h_\tau,1).
 \label{eq:LII-Hermite-direct-parameters}
\end{equation}
Let \(F_{j;\tau}^{\mathrm{II}}\) denote the row
\eqref{eq:LII-shifted-rows} with these shapes and rates. With the kernels
of \eqref{eq:Hermite-gamma-convolution-operator}, define the functions
of integral one
\[
 f_{0,\tau}\coloneq
 \kappa_{\varrho_{1,\tau},d_{1,\tau}}*_+\cdots*_+
 \kappa_{\varrho_{q,\tau},d_{q,\tau}},
\]
\[
 f_{j,\tau}\coloneq
 \kappa_{\varrho_{1,\tau},d_{1,\tau}+\delta_{1,j}}*_+\cdots*_+
 \kappa_{\varrho_{q,\tau},d_{q,\tau}+\delta_{q,j}},
 \qquad 1\le j\le q.
\]
Thus \(f_{j,\tau}=\sigma_{j,\tau}F_{j;\tau}^{\mathrm{II}}\), where
\[
 \sigma_{j,\tau}\coloneq
 \varrho_{j,\tau}\prod_{h=1}^q\varrho_{h,\tau}^{d_{h,\tau}}.
\]
The change of variable and the transformed rows are
\begin{equation}
 y_\tau(t)\coloneq\tau+h_\tau(\widehat a_*-t),\qquad
 V_{j,\tau}(t)\coloneq h_\tau f_{j,\tau}(y_\tau(t)),
 \qquad 0\le j\le q.
 \label{eq:LII-Hermite-direct-rows}
\end{equation}
Each \(V_{j,\tau}\) is zero outside
\(I_\tau^{\mathrm{II}}\coloneq(-\infty,\widehat a_*+\tau/h_\tau)\).
Choose the original columns to be \(\e^{-\xi_i y/h_\tau}\) and put
\(\chi_{i,\tau}\coloneq\e^{\xi_i(\tau/h_\tau+\widehat a_*)}\).
Since
\[
 \chi_{i,\tau}\e^{-\xi_i y_\tau(t)/h_\tau}=\e^{\xi_i t},
\]
the row and column normalizations \(\sigma_{j,\tau}\) and
\(\chi_{i,\tau}\), followed by the change of variable, give the matrix
\[
 [V_{j,\tau}(t)\e^{\xi_i t}]_{j,i}\,\mathrm dt.
\]
Its entries are integrable for all sufficiently large \(\tau\).
For the limiting rows, retain the notation
\[
 F=K_{\rho_1,d_1}\cdots K_{\rho_r,d_r}\phi_{\widehat a_*},\qquad
 U_h^{\mathrm H}=K_{\rho_h,1}F\ (1\le h\le r),\qquad
 U_q^{\mathrm H}=F.
\]

\Needspace{8\baselineskip}
\begin{theorem}[Direct Hermite limit of the second-kind matrix]
 \label{thm:LII-Hermite-direct-matrix}
 For the parameters and rows defined in
 \eqref{eq:LII-Hermite-direct-parameters}--\eqref{eq:LII-Hermite-direct-rows},
 the following statements hold.
 \begin{enumerate}[label=\textnormal{(\arabic*)}]
 \item For each \(j\), the rows converge locally uniformly on
 \(\mathbb R\), and for each \(i,j\),
 \begin{equation}
 V_{j,\tau}(t)\xrightarrow[\tau\to+\infty]{}U_j^{\mathrm H}(t),\qquad
 \int_{\mathbb R}\e^{\xi_i t}|V_{j,\tau}(t)-U_j^{\mathrm H}(t)|\,\mathrm dt
 \xrightarrow[\tau\to+\infty]{}0.
 \label{eq:LII-Hermite-direct-weight-limit}
 \end{equation}
 \item For every fixed \(v\in\mathbb N_0\), all mixed moments satisfy
 \begin{equation}
 \int_{\mathbb R}t^v\e^{\xi_i t}
       (V_{j,\tau}(t)-U_j^{\mathrm H}(t))\,\mathrm dt
 =\mathrm O(\tau^{-1/2}),\qquad \tau\to+\infty.
 \label{eq:LII-Hermite-direct-moment-limit}
 \end{equation}
 \end{enumerate}
\end{theorem}

\begin{proof}
 First compute the bilateral transform of the base row. Substituting
 \(y=y_\tau(t)\), reversing the integration limits, and using the
 transforms of the gamma factors gives
 \begin{equation}
 G_\tau(z)\coloneq\int_{\mathbb R}\e^{zt}V_{0,\tau}(t)\,\mathrm dt
 =\e^{\widehat a_*z+\tau z/h_\tau}(1+z/h_\tau)^{-\tau}
   \prod_{h=1}^r\left(\frac{\rho_h}{\rho_h+z}\right)^{d_h}.
 \label{eq:LII-Hermite-direct-base-transform}
 \end{equation}
 This identity holds when \(\operatorname{Re}z>-h_\tau\) and
 \(\operatorname{Re}(z+\rho_h)>0\) for every \(h\le r\).
 Increasing one shape by one multiplies its gamma transform by one
 additional linear factor. Consequently,
 \begin{equation}
 \int_{\mathbb R}\e^{zt}V_{h,\tau}(t)\,\mathrm dt
 =\frac{\rho_h}{\rho_h+z}G_\tau(z)\ (h\le r),\qquad
 \int_{\mathbb R}\e^{zt}V_{q,\tau}(t)\,\mathrm dt
 =\frac{G_\tau(z)}{1+z/h_\tau}.
 \label{eq:LII-Hermite-direct-row-transforms}
 \end{equation}
 On every fixed compact subset of
 \(\{z:\operatorname{Re}(z+\rho_h)>0\text{ for all }h\le r\}\),
 interpreted as \(\mathbb C\) when \(r=0\), the logarithmic expansion is
 \[
 \frac{\tau z}{h_\tau}-\tau\log(1+z/h_\tau)
 =\frac{z^2}{4}-\frac{z^3}{6h_\tau}
   +\mathrm O(h_\tau^{-2}),\qquad \tau\to+\infty.
 \]
 It follows from \eqref{eq:LII-Hermite-direct-base-transform} that
 \(G_\tau=L_F+\mathrm O(\tau^{-1/2})\) as \(\tau\to+\infty\),
 locally uniformly on this domain. Cauchy's formula gives the same estimate for every
 fixed derivative. Applying \(\partial_z^v\) to
 \eqref{eq:LII-Hermite-direct-row-transforms} at \(z=\xi_i\) and
 comparing with \eqref{eq:canonical-Hermite-row-transforms} proves
 \eqref{eq:LII-Hermite-direct-moment-limit}.

 To prove convergence of the row functions themselves, define
 \(g_\tau(t)\coloneq h_\tau\kappa_{1,\tau}(y_\tau(t))\) and
 \(g_\tau^+(t)\coloneq h_\tau\kappa_{1,\tau+1}(y_\tau(t))\).
 Stirling's formula gives convergence of both functions to
 \(\phi_{\widehat a_*}\), uniformly on compact subsets of \(\mathbb R\).
 Their suprema are bounded independently of sufficiently large \(\tau\),
 as follows by evaluating the gamma factors at their maxima.
 Both functions are nonnegative and have integral one. Dominated
 convergence applied to their minima with \(\phi_{\widehat a_*}\)
 consequently gives convergence in \(L^1(\mathbb R)\).

 In the remaining factors, the substitution \(u=h_\tau v\) converts
 \(\kappa_{\rho_h/h_\tau,d_h}(u)\,\mathrm du\) exactly into
 \(\kappa_{\rho_h,d_h}(v)\,\mathrm dv\). Reflection of the variable
 gives
 \[
 V_{0,\tau}=K_{\rho_1,d_1}\cdots K_{\rho_r,d_r}g_\tau,\qquad
 V_{h,\tau}=K_{\rho_h,1}V_{0,\tau},\qquad
 V_{q,\tau}=K_{\rho_1,d_1}\cdots K_{\rho_r,d_r}g_\tau^+.
 \]
 The convolution kernels have integral one. Their action preserves
 \(L^1\)-convergence; the uniform bounds just obtained also give local
 uniform convergence by dominated convergence. Finally, the case
 \(v=0\) of \eqref{eq:LII-Hermite-direct-moment-limit} gives convergence
 of the integrals of \(\e^{\xi_i t}V_{j,\tau}\).
 These functions and their pointwise limits are nonnegative.
 Applying the same minimum argument with
 \(\e^{\xi_i t}U_j^{\mathrm H}(t)\) proves the weighted
 \(L^1\)-convergence in \eqref{eq:LII-Hermite-direct-weight-limit}.
\end{proof}

The diagonal normalizations also specify the change in each polynomial
component. If \(\boldsymbol A_\tau^{\mathrm{II}}\) and
\(\boldsymbol B_\tau^{\mathrm{II}}\) are component vectors for the
original second-kind matrix with columns \(\e^{-\xi_i y/h_\tau}\),
their transformed components are defined by
\begin{equation}
 \widetilde A_\tau^{(i)}(t)
 \coloneq\chi_{i,\tau}^{-1}A_\tau^{(i),\mathrm{II}}(y_\tau(t)),\qquad
 \widetilde B_\tau^{(j)}(t)
 \coloneq\sigma_{j,\tau}^{-1}B_\tau^{(j),\mathrm{II}}(y_\tau(t)).
 \label{eq:LII-Hermite-direct-component-change}
\end{equation}
The affine substitution preserves every component degree bound and the
span of the testing polynomials. Thus these are exactly the components
for \([V_{j,\tau}(t)\e^{\xi_i t}]_{j,i}\,\mathrm dt\).

At a weakly normal index, the homogeneous moment conditions determine a
nonzero vector up to one common scalar. For each limiting \(A\)- or \(B\)-vector, choose one
nonzero coefficient and normalize it to one; use the same coefficient
to normalize the corresponding transformed finite vectors. This expresses
convergence of the projective classes, in which vectors differing by a
nonzero scalar are identified. Another choice of coefficient selects
proportional limiting representatives.

\begin{corollary}[Convergence of the mixed polynomial components]
 \label{cor:LII-Hermite-direct-components}
 Assume that the column nodes are pairwise distinct. Fix a
 Laguerre-admissible row index
 \(\boldsymbol\lambda=(n_1,\ldots,n_r,\eta)\), and suppose that
 the rates \(\rho_h\) with \(n_h>0\) are pairwise distinct.
 For each of the balances
 \(|\boldsymbol m|=|\boldsymbol\lambda|+1\) and
 \(|\boldsymbol\lambda|=|\boldsymbol m|+1\), respectively, the
 transformed \(A\)- and \(B\)-problems are weakly normal for all
 sufficiently large \(\tau\).

 With these coefficient normalizations, in the respective balances,
 \begin{equation}
 \widetilde{\boldsymbol A}_\tau
 \xrightarrow[\tau\to+\infty]{}\boldsymbol A^{\mathrm H},\qquad
 \widetilde{\boldsymbol B}_\tau
 \xrightarrow[\tau\to+\infty]{}\boldsymbol B^{\mathrm H},
 \label{eq:LII-Hermite-direct-component-limit}
 \end{equation}
 coefficientwise, with error \(\mathrm O(\tau^{-1/2})\) in every
 coefficient as \(\tau\to+\infty\). The \(A\)-problems are strongly normal for sufficiently
 large \(\tau\). The same conclusion holds for \(B\) whenever the
 limiting vector satisfies \eqref{eq:Hermite-B-strong-normality-criterion}.
\end{corollary}

\begin{proof}
 Proposition~\ref{prop:canonical-Hermite-vector-edge-determinant} and
 Corollary~\ref{cor:Hermite-main-normality} give weak normality
 of both limiting problems. In a monomial basis, the finite moment
 matrices converge entrywise to those matrices, with error
 \(\mathrm O(\tau^{-1/2})\), by
 \eqref{eq:LII-Hermite-direct-moment-limit}.

 Adjoin the selected coefficient normalization to each homogeneous
 system. The resulting limiting square matrix is nonsingular: its
 homogeneous kernel would otherwise contain a nonzero multiple of the
 unique Hermite vector whose selected coefficient is zero. That
 contradicts the choice of the coefficient. Hence the finite square
 matrices are nonsingular for all sufficiently large \(\tau\), and
 their inverses remain bounded. If \(M_\tau\mathbf c_\tau=\mathbf e\)
 and \(M\mathbf c=\mathbf e\) denote either normalized system, then
 \[
 \mathbf c_\tau-\mathbf c
 =M_\tau^{-1}(M-M_\tau)\mathbf c
 =\mathrm O(\tau^{-1/2}),\qquad \tau\to+\infty.
 \]
 This proves \eqref{eq:LII-Hermite-direct-component-limit} for every
 component. A nonzero adjacent moment of the limiting \(B\)-form
 remains nonzero for large \(\tau\), so the finite vector represents
 a nonzero form as well. By
 Corollary~\ref{cor:canonical-Hermite-explicit-A-components} and
 \eqref{eq:Hermite-A-leading-coefficient}, every \(A\)-component with
 positive index has a nonzero prescribed leading coefficient: the
 factors \(G_i(\xi_i)\) do not vanish under the node-separation and
 positivity assumptions. These coefficients remain nonzero for
 sufficiently large \(\tau\), proving strong normality of the finite
 \(A\)-problems. For \(B\), the same argument proves strong normality
 for sufficiently large \(\tau\) whenever
 \eqref{eq:Hermite-B-strong-normality-criterion} holds for the limit.
\end{proof}

For \(q\ge2\), the finite components in
\eqref{eq:LII-Hermite-direct-component-change} have the terminating
representations \eqref{eq:LII-A-multiple-KdF} and
\eqref{eq:LII-B-components-hypergeometric}, in the parameter ranges
specified there, with the substitution
\eqref{eq:LII-Hermite-direct-parameters} and column nodes \(\xi_i/h_\tau\).
Their normalized coefficientwise limits are the Hermite components
\eqref{eq:Hermite-A-single-SD} and
\eqref{eq:canonical-Hermite-B-beta-components}--
\eqref{eq:canonical-Hermite-B-base-component}.
Zero row indices are covered by \eqref{eq:LII-B-zero-row-indices}.
The confluence thus
applies to the individual polynomial vectors, not only to their complete
forms. In this scaling, the \(r\) shifted gamma factors with fixed shapes
give the rows \(K_{\rho_h,1}F\), while the shifted factor whose shape
tends to infinity gives \(F\). The resulting vector is therefore
the \(q\)-row Hermite vector used throughout this section.

In the one-row reduction, the reflection can be reversed to use the
usual increasing change of variable. Let \(\xi_1,\ldots,\xi_p\) be
distinct real numbers and define
\[
 c_{i,\tau}\coloneq1-\frac{\xi_i}{h_\tau},\qquad
 w_{i,\tau}(y)\coloneq y^\tau\e^{-c_{i,\tau}y},\qquad y>0.
\]
All \(c_{i,\tau}\) are positive for sufficiently large \(\tau\).
Use the classical notation \(L_{\boldsymbol\nu}\) and
\(L_{\boldsymbol\nu}^{(i)}\) fixed before
Corollary~\ref{cor:q1-LII-reduction}, and the Hermite notation
\(H_{\boldsymbol\nu}\) and \(H_{\boldsymbol\nu}^{(i)}\) of
Section~\ref{subsec:Hermite-classical-reductions}.

\begin{corollary}[Direct limit for the classical multiple polynomials]
 \label{cor:ordinary-LII-Hermite-direct}
 For every fixed \(\boldsymbol\nu\in\mathbb N_0^p\), with
 \(n\coloneq|\boldsymbol\nu|\ge1\), the monic type-II polynomials satisfy
 \begin{equation}
 h_\tau^{-n}L_{\boldsymbol\nu}
 (\tau+h_\tau t;\boldsymbol c_\tau,\tau)
 \xrightarrow[\tau\to+\infty]{}H_{\boldsymbol\nu}(t;\boldsymbol\xi).
 \label{eq:ordinary-LII-Hermite-type-II}
 \end{equation}
 The normalized type-I components satisfy, for each \(i\),
 \begin{equation}
 h_\tau^n w_{i,\tau}(\tau)L_{\boldsymbol\nu}^{(i)}
 (\tau+h_\tau t;\boldsymbol c_\tau,\tau)
 \xrightarrow[\tau\to+\infty]{}H_{\boldsymbol\nu}^{(i)}(t;\boldsymbol\xi).
 \label{eq:ordinary-LII-Hermite-type-I}
 \end{equation}
 Both limits are coefficientwise, with error
 \(\mathrm O(\tau^{-1/2})\) in each coefficient as \(\tau\to+\infty\).
 No near-diagonal restriction on \(\boldsymbol\nu\) is required in this
 one-row case, where the classical weight systems are AT.
\end{corollary}

\begin{proof}
 Extend the transformed weights
 \(\widetilde w_{i,\tau}(t)
 \coloneq w_{i,\tau}(\tau+h_\tau t)/w_{i,\tau}(\tau)\)
 by zero for \(t\le-\tau/h_\tau\). Direct gamma integration gives
 \[
 \int_{\mathbb R}\e^{zt}\widetilde w_{i,\tau}(t)\,\mathrm dt
 =\frac{\Gamma(\tau+1)\e^\tau}{h_\tau\tau^\tau}
   \e^{-\tau(z+\xi_i)/h_\tau}
   \left(1-\frac{z+\xi_i}{h_\tau}\right)^{-\tau-1}.
 \]
 The identity holds for \(\operatorname{Re}(z+\xi_i)<h_\tau\);
 each fixed compact subset of \(\mathbb C\) lies in this half-plane
 for all sufficiently large \(\tau\) and every \(i\).
 Stirling's formula and expansion of the logarithm give convergence to
 \(\sqrt\pi\e^{(z+\xi_i)^2/4}\), locally uniformly in \(z\), with
 error \(\mathrm O(\tau^{-1/2})\) as \(\tau\to+\infty\).
 Consequently, every fixed moment converges to the corresponding
 moment of \(\e^{-t^2+\xi_i t}\) with the same rate.
 Lemma~\ref{lem:q1-multiple-Hermite-normality} makes the limiting
 normalized moment systems nonsingular for every \(\boldsymbol\nu\).

 The affine substitution preserves the testing spaces, and multiplication
 by \(h_\tau^{-n}\) makes the transformed type-II polynomial monic.
 For type I, the transformed complete form is
 \(h_\tau^n\sum_iL_{\boldsymbol\nu}^{(i)}
 (\tau+h_\tau t;\boldsymbol c_\tau,\tau)w_{i,\tau}(\tau+h_\tau t)\).
 Its moments of orders below \(n-1\) vanish by the binomial expansion;
 its moment of order \(n-1\) equals one, since the change of variable
 contributes \(h_\tau^{-1}\) and the leading term of
 \((y-\tau)^{n-1}\) contributes \(h_\tau^{-(n-1)}\).
 This verifies the factor \(h_\tau^n\) in
 \eqref{eq:ordinary-LII-Hermite-type-I}. Continuity of the inverses of
 the two fixed-size moment matrices now proves
 \eqref{eq:ordinary-LII-Hermite-type-II} and
 \eqref{eq:ordinary-LII-Hermite-type-I}, with their coefficientwise rates.
\end{proof}

For \(q=1\), the weight scaling underlying this corollary is the
Laguerre-to-Gaussian transition in
\cite[preprint version, equation~(217)]{AdlerVanMoerbekeWang2013},
up to the normalization of the Gaussian variance and a fixed shift of
the large shape parameter. That work relates it to the Wishart and
Gaussian models associated with the two ordinary multiple families.
Theorem~\ref{thm:LII-Hermite-direct-matrix} and
Corollary~\ref{cor:LII-Hermite-direct-components} establish the
mixed-type extension; the two formulas above state its one-row polynomial
normalizations explicitly.

\subsection{Direct confluence from Jacobi}
\label{subsec:Jacobi-Hermite-direct}

The Jacobi weights also admit a direct limit to the same Hermite system
defined above. One beta factor is concentrated near \(1/2\), while the
remaining factors are concentrated near \(1\). All parameters vary
with one parameter \(\tau\); no Laguerre limit is taken first.
For \(q=1\), this is the direct Jacobi--Pi\~neiro-to-Hermite limit
of \cite[Section~5]{BranquinhoDiazFoulquieManas2025Classical}, up to
fixed shifts of the beta parameters. The construction below retains
all \(q\) row weights.

Fix \(q=r+1\), \(d_h,\rho_h>0\) for \(1\le h\le r\), and
\(\widehat a_*\in\mathbb R\). Assume \(\rho_h+\xi_i>0\) for all
\(i,h\). In \eqref{eq:Jacobi-source-base}, choose
\begin{equation}
 a_h(\tau)+1\coloneq\rho_h\sqrt\tau,\quad
 b_h(\tau)-a_h(\tau)\coloneq d_h\quad(1\le h\le r),\quad
 a_q(\tau)+1\coloneq\tau+\widehat a_*\sqrt\tau,\quad
 b_q(\tau)+1\coloneq2\tau.
 \label{eq:Jacobi-Hermite-direct-parameters}
\end{equation}
All beta parameters are positive for sufficiently large \(\tau\).
For this limit, set
\begin{equation}
 \varepsilon_\tau^{\mathrm J}\coloneq\tau^{-1/2},\qquad
 x_\tau^{\mathrm J}(t)\coloneq\tfrac12(1+\varepsilon_\tau^{\mathrm J}t),\qquad
 \alpha_i(\tau)\coloneq\xi_i\sqrt\tau,\qquad
 I_\tau^{\mathrm J}\coloneq(-\sqrt\tau,\sqrt\tau).
 \label{eq:Jacobi-Hermite-direct-variable}
\end{equation}
Write \(w_{j,\tau}^{\mathrm J}\) for the resulting Jacobi rows and define
\[
 Z_{j,\tau}^{\mathrm J}\coloneq
 \int_0^1w_{j,\tau}^{\mathrm J}(x)\,\mathrm dx
 =\prod_{h=1}^q
 \frac{\Gamma(a_h(\tau)+1)}{\Gamma(b_h(\tau)+1+\delta_{h,j})},
 \qquad 0\le j\le q,
\]
where \(\delta_{h,0}=0\). The transformed rows and columns are
\begin{equation}
 V_{j,\tau}^{\mathrm J}(t)\coloneq
 \frac{\varepsilon_\tau^{\mathrm J}}{2Z_{j,\tau}^{\mathrm J}}
 w_{j,\tau}^{\mathrm J}(x_\tau^{\mathrm J}(t)),\qquad
 C_{i,\tau}^{\mathrm J}(t)\coloneq
 (1+\varepsilon_\tau^{\mathrm J}t)^{\alpha_i(\tau)},
 \qquad t\in I_\tau^{\mathrm J}.
 \label{eq:Jacobi-Hermite-direct-entries}
\end{equation}
The products \(V_{j,\tau}^{\mathrm J}C_{i,\tau}^{\mathrm J}\) are
extended by zero outside \(I_\tau^{\mathrm J}\).
For polynomial vectors of the original Jacobi problems, their
transformed components are
\begin{equation}
 \widetilde A_\tau^{(i)}(t)\coloneq
 2^{-\alpha_i(\tau)}A_\tau^{(i),\mathrm J}(x_\tau^{\mathrm J}(t)),\qquad
 \widetilde B_\tau^{(j)}(t)\coloneq
 Z_{j,\tau}^{\mathrm J}B_\tau^{(j),\mathrm J}(x_\tau^{\mathrm J}(t)).
 \label{eq:Jacobi-Hermite-direct-components}
\end{equation}
These are the components for the transformed matrix, since the change
of variable is affine and the row and column normalizations are constant.

\begin{theorem}[Direct Jacobi--Hermite confluence]
 \label{thm:Jacobi-Hermite-direct}
 Under \eqref{eq:Jacobi-Hermite-direct-parameters}--
 \eqref{eq:Jacobi-Hermite-direct-entries}, the following statements hold.
 \begin{enumerate}[label=\textnormal{(\arabic*)}]
 \item For \(1\le j\le q\), \(1\le i\le p\), and every fixed
 \(v\in\mathbb N_0\),
 \begin{equation}
 \int_{I_\tau^{\mathrm J}}t^vV_{j,\tau}^{\mathrm J}(t)
 C_{i,\tau}^{\mathrm J}(t)\,\mathrm dt
 =\int_{\mathbb R}t^vU_j^{\mathrm H}(t)\e^{\xi_i t}\,\mathrm dt
 +\mathrm O(\tau^{-1/2}),\qquad \tau\to+\infty,
 \label{eq:Jacobi-Hermite-direct-moments}
 \end{equation}
 where \(U_h^{\mathrm H}=K_{\rho_h,1}F\), \(1\le h\le r\), and
 \(U_q^{\mathrm H}=F\). The corresponding positive matrix measures
 converge weakly, entry by entry, to
 \([U_j^{\mathrm H}(t)\e^{\xi_i t}]_{j,i}\,\mathrm dt\).
 \item Fix a Laguerre-admissible row index
 \(\boldsymbol\lambda=(n_1,\ldots,n_r,\eta)\).
 Suppose that the column nodes are pairwise distinct and that the
 rates with \(n_h>0\) are pairwise distinct. In each of the balances
 \(|\boldsymbol m|=|\boldsymbol\lambda|+1\) and
 \(|\boldsymbol\lambda|=|\boldsymbol m|+1\), respectively, the
 finite transformed \(A\)- and \(B\)-problems are weakly normal for
 sufficiently large \(\tau\). With the coefficient normalizations
 of Corollary~\ref{cor:LII-Hermite-direct-components},
 \begin{equation}
 \widetilde{\boldsymbol A}_\tau
 \xrightarrow[\tau\to+\infty]{}\boldsymbol A^{\mathrm H},\qquad
 \widetilde{\boldsymbol B}_\tau
 \xrightarrow[\tau\to+\infty]{}\boldsymbol B^{\mathrm H},
 \label{eq:Jacobi-Hermite-direct-polynomial-limit}
 \end{equation}
 coefficientwise, with error \(\mathrm O(\tau^{-1/2})\) in every
 coefficient as \(\tau\to+\infty\).
 \end{enumerate}
\end{theorem}

\begin{proof}
 Define the normalized Mellin transforms by
 \[
 Q_{j,\tau}^{\mathrm J}(z)\coloneq
 \frac1{Z_{j,\tau}^{\mathrm J}}
 \int_0^1(2x)^{z\sqrt\tau}w_{j,\tau}^{\mathrm J}(x)\,\mathrm dx.
 \]
 In the next formula the Pochhammer notation is extended to complex
 orders by \((u)_v\coloneq\Gamma(u+v)/\Gamma(u)\).
 The Mellin-convolution identity gives
 \begin{equation}
 Q_{0,\tau}^{\mathrm J}(z)
 =2^{z\sqrt\tau}
 \frac{(\tau+\widehat a_*\sqrt\tau)_{z\sqrt\tau}}
      {(2\tau)_{z\sqrt\tau}}
 \prod_{h=1}^r
 \frac{(\rho_h\sqrt\tau)_{z\sqrt\tau}}
      {(\rho_h\sqrt\tau+d_h)_{z\sqrt\tau}},\qquad
 Q_{j,\tau}^{\mathrm J}(z)
 =\frac{b_j(\tau)+1}{b_j(\tau)+1+z\sqrt\tau}Q_{0,\tau}^{\mathrm J}(z).
 \label{eq:Jacobi-Hermite-direct-transform}
 \end{equation}
 These identities hold on every fixed compact subset of
 \(\mathscr H\coloneq\{z:\operatorname{Re}(z+\rho_h)>0, 1\le h\le r\}\)
 for all sufficiently large \(\tau\); when \(r=0\), take
 \(\mathscr H=\mathbb C\).

 Apply the logarithmic Gamma expansion
 \eqref{eq:canonical-Hermite-log-Gamma-Taylor} at
 \(X=\tau+\widehat a_*\sqrt\tau\) and \(X=2\tau\).
 For \(z\) in any fixed compact subset of \(\mathscr H\), it gives
 \[
 \log\!\left(
 2^{z\sqrt\tau}
 \frac{(\tau+\widehat a_*\sqrt\tau)_{z\sqrt\tau}}
      {(2\tau)_{z\sqrt\tau}}\right)
 =\widehat a_*z+\frac{z^2}{4}+\mathrm O(\tau^{-1/2}),
 \qquad \tau\to+\infty.
 \]
 The fixed-shape Gamma ratios in the remaining factors satisfy
 \[
 \frac{(\rho_h\sqrt\tau)_{z\sqrt\tau}}
      {(\rho_h\sqrt\tau+d_h)_{z\sqrt\tau}}
 =\left(\frac{\rho_h}{\rho_h+z}\right)^{d_h}
  \bigl(1+\mathrm O(\tau^{-1/2})\bigr),\qquad \tau\to+\infty.
 \]
 Consequently, \(Q_{0,\tau}^{\mathrm J}=L_F+\mathrm O(\tau^{-1/2})\)
 locally uniformly on \(\mathscr H\) as \(\tau\to+\infty\).
 The row factor in \eqref{eq:Jacobi-Hermite-direct-transform} tends
 to \(\rho_j/(\rho_j+z)\) for \(j\le r\), and to one for \(j=q\),
 with the same error. Thus \(Q_{j,\tau}^{\mathrm J}\) converges to
 the bilateral transform of \(U_j^{\mathrm H}\). Cauchy's formula
 gives the same estimate for every fixed derivative.

 Put \(\varepsilon=\varepsilon_\tau^{\mathrm J}\).
 The step is chosen so that \(\varepsilon\sqrt\tau=1\), hence
 \((2x)^{\varepsilon\sqrt\tau}=2x\) exactly. Together with
 \(t=(2x-1)/\varepsilon\), binomial expansion therefore gives
 \[
 \int_{I_\tau^{\mathrm J}}t^vV_{j,\tau}^{\mathrm J}(t)
 C_{i,\tau}^{\mathrm J}(t)\,\mathrm dt
 =\varepsilon^{-v}\Delta_\varepsilon^v
 Q_{j,\tau}^{\mathrm J}(\xi_i).
 \]
 The integral representation
 \eqref{eq:canonical-Hermite-finite-difference-integral} expresses
 the right-hand side as an average of the \(v\)-th derivative on
 arguments at distance at most \(v\varepsilon\) from \(\xi_i\).
 The derivative estimates just proved therefore give
 \eqref{eq:Jacobi-Hermite-direct-moments}.

 Convergence of the zeroth and second moments makes each family of
 positive measures tight, with bounded total integrals. The bounds
 on all even moments ensure that every fixed moment passes to any
 weak subsequential limit, by the tail estimate used in the proof of
 Theorem~\ref{thm:simultaneous-Laguerre-Hermite-confluence}.
 The limiting measure \(U_j^{\mathrm H}(t)\e^{\xi_i t}\,\mathrm dt\)
 has finite exponential moments in a neighbourhood of zero, because
 \(\xi_i\in\mathscr H\). Its moments determine it uniquely.
 Hence all weak subsequential limits coincide with that measure,
 proving the first assertion.

 For the second assertion, \eqref{eq:Jacobi-Hermite-direct-components}
 preserves the prescribed polynomial degree bounds and the testing
 spaces. The normalized finite moment matrices converge, by
 \eqref{eq:Jacobi-Hermite-direct-moments}, to the same nonsingular
 Hermite systems as in Corollary~\ref{cor:LII-Hermite-direct-components}.
 The matrix-inverse argument in its proof gives weak normality for
 sufficiently large \(\tau\) and the coefficientwise estimate
 \eqref{eq:Jacobi-Hermite-direct-polynomial-limit}.
\end{proof}

The limiting polynomial vectors are therefore the explicit terminating
representations \eqref{eq:Hermite-A-single-SD} and
\eqref{eq:canonical-Hermite-B-beta-components}--
\eqref{eq:canonical-Hermite-B-base-component}.
The vertical arrow in Figure~\ref{fig:Askey-confluence-diagram}
records this one-parameter limit.

The three routes can now be compared at the same limiting parameters
\(\widehat a_*\), \(\boldsymbol d\), \(\boldsymbol\rho\), and
\(\boldsymbol\xi\). For either route through Laguerre, denote the
parameter of the first, Jacobi-to-Laguerre, limit by \(\kappa\).
All intermediate Laguerre parameters are held fixed while
\(\kappa\to+\infty\); the Hermite scaling parameter \(\tau\) tends to
infinity afterwards.

\begin{corollary}[Agreement of the three Hermite limits]
 \label{cor:Hermite-three-route-agreement}
 Under the index and separation hypotheses of
 Corollary~\ref{cor:LII-Hermite-direct-components}, the direct
 Jacobi--Hermite limit and the two iterated limits through Laguerre
 of the first and second kinds give the same matrix
 \([U_j^{\mathrm H}(t)\e^{\xi_i t}]_{j,i}\,\mathrm dt\) in every fixed
 mixed moment. In the respective \(A\)- and \(B\)-balances, their
 transformed polynomial vectors converge coefficientwise to the same
 \(\boldsymbol A^{\mathrm H}\) and \(\boldsymbol B^{\mathrm H}\) when
 the same nonzero coefficients are normalized to one.
\end{corollary}

\begin{proof}
 The first limits follow from
 \eqref{eq:Jacobi-Laguerre-base-Mellin-limit},
 \eqref{eq:Jacobi-Laguerre-Newton-Mellin-limit}, and
 Theorem~\ref{thm:Jacobi-LII-weight-limit}. Applying
 Corollary~\ref{cor:simultaneous-Laguerre-Hermite-matrix-confluence}
 and Theorem~\ref{thm:LII-Hermite-direct-matrix} to the two intermediate
 families gives the same Hermite rows
 \(K_{\rho_1,1}F,\ldots,K_{\rho_r,1}F,F\) and columns
 \(\e^{\xi_i t}\). Theorem~\ref{thm:Jacobi-Hermite-direct} gives this
 matrix by the direct route.

 For the components, fix either balance and the selected coefficient
 normalization. The limiting Hermite moment system is nonsingular by
 Corollary~\ref{cor:Hermite-main-normality}. Moment convergence
 therefore makes both normalized Laguerre systems nonsingular for
 every sufficiently large fixed \(\tau\). At each such \(\tau\),
 the corresponding normalized Jacobi systems converge to the Laguerre
 systems as \(\kappa\to+\infty\). Continuity of matrix inversion gives
 this first coefficientwise limit and then the second limit as
 \(\tau\to+\infty\), exactly as in the proof of
 Corollary~\ref{cor:LII-Hermite-direct-components}. The direct
 coefficientwise limit is
 \eqref{eq:Jacobi-Hermite-direct-polynomial-limit}. Since all three
 limiting systems and their normalizations agree, their solutions agree.
\end{proof}

Thus Figure~\ref{fig:Askey-confluence-diagram} commutes in the sense of
these limiting moments and projective polynomial vectors.
The direct arrow uses a single parameter, whereas the paths through
Laguerre use the stated order of two limits. In particular, the
second-kind parameters in \eqref{eq:LII-Hermite-direct-parameters}
depend on \(\tau\) but are fixed during the first limit
\(\kappa\to+\infty\). The finite-parameter Jacobi families are different;
the corollary identifies their common Hermite limit rather than
interchanging two limits.

\section{Conclusions and outlook}
\label{sec:conclusions}

The limits studied here act on the entire \(q\times p\) matrix of
measures. Near \(x=1\), the change of variable \(y=\tau(1-x)\) gives the
Laguerre matrix of the second kind \([F_j(y)\e^{-\xi_i y}]_{j,i}\,\mathrm dy\).
Its row weights are additive convolutions of gamma kernels. For
near-diagonal row multi-indices, the square moment determinants have a
closed product formula for arbitrary positive real shape parameters.
They are nonzero when the rates and column nodes corresponding to
positive indices are distinct within their respective families. This
also supplies the adjacent normalizations whose row multi-indices remain
near the diagonal. The contour and Rodrigues formulas determine the
complete forms. Their polynomial components have terminating
Kamp\'e de F\'eriet representations on the \(A\)-side and, for \(q\ge2\),
linear combinations of at most \(q+1\) Srivastava--Daoust functions on
the \(B\)-side. All parameters and polynomial prefactors are explicit.
For \(q=1\), these formulas recover the classical multiple Laguerre
system of the second kind.

Near \(x=0\), the last Jacobi rows first require a triangular Newton
transformation to obtain distinct limiting rows. The resulting matrix
converges to the Laguerre system of the first kind. At fixed indices, the normalized forms,
recurrence coefficients and bidiagonal-factorization coefficients
converge whenever the limiting shifted minors used in their normalization
are nonzero.

Centering and rescaling the first-kind Laguerre family with one gamma
factor gives the Hermite rows \(K_{\rho_1,1}F,\ldots,K_{\rho_{q-1},1}F,F\).
Under rate and node separation,
\eqref{eq:canonical-Hermite-vector-edge-determinant} gives normality
at every Laguerre-admissible index. The \(A\)-components attain their
maximal degrees and each is one terminating Srivastava--Daoust function.
The nonzero normalized \(B\)-form has unique components, each a
combination of a number of terminating functions bounded independently
of the degrees. Their finite-pole and polynomial contributions are both
explicit; strong normality is characterized by
\eqref{eq:Hermite-B-strong-normality-criterion}. The one-row case
recovers the classical multiple Hermite formulas.

Appendix~\ref{app:Hermite-derivative-limits} explains the restriction
to one final base row. With several derivative rows, Pearson relations
give an exact loss of rank beyond a finite admissible range. The
canonical forms may still converge while their polynomial components
lose uniqueness or diverge. These conclusions concern the more general
limit and do not alter the infinite admissible normal family studied in
the main text.

The direct limit in Section~\ref{subsec:LII-Hermite-direct} shows that
the same Hermite family is also reached from Laguerre of the
second kind. A single gamma shape tends to infinity while the remaining
shapes stay fixed and their rates decrease as \(\tau^{-1/2}\).
After centering, reflection, and diagonal row and column normalizations,
the mixed moments and all normalized polynomial components converge,
with coefficientwise error \(\mathrm O(\tau^{-1/2})\) as
\(\tau\to+\infty\). In the one-row reduction the weight scaling has
the antecedent \cite[preprint version, equation~(217)]
{AdlerVanMoerbekeWang2013}; the extension established here is of mixed
type. The single large-shape gamma factor produces \(F\), while the
other shifted factors produce its exponential convolutions.

Theorem~\ref{thm:Jacobi-Hermite-direct} also obtains the same
Hermite matrix directly from Jacobi by rescaling near \(x=1/2\).
The parameter choice \eqref{eq:Jacobi-Hermite-direct-parameters} gives
convergence of every fixed mixed moment and every normalized polynomial
component at the stated normal indices.
Corollary~\ref{cor:Hermite-three-route-agreement} identifies this limit
with the two iterated limits through Laguerre, both for the matrix moments
and for the polynomial vectors with a common normalization. The direct
Jacobi limit uses one parameter; each Laguerre route takes its two limits
successively. This agreement holds on the family with infinitely many
admissible normal indices. The general first-kind limits with additional
derivative rows are treated separately in the appendix.

The discrete part of the scheme is developed in
\cite{Manas2026HahnBivariate}. Applying Bernstein kernels in two
discrete variables to the Jacobi matrix measure gives a Hahn system
whose normalized factorial bimoments reproduce the continuous mixed moments
exactly. Its confluences give Kravchuk, Meixner-I, Meixner-II, and
additive and multiplicative Charlier systems. The matrix measures are
bivariate, while each component of the two polynomial vectors depends
on only one variable. In the continuous Jacobi and Laguerre limits,
the rescaled measures concentrate on the diagonal and recover the
corresponding univariate matrix measures studied here. These constructions
connect the discrete and continuous parts of an Askey scheme of mixed
type through both the moment identities and the explicit polynomial
components.

A further direction is to enlarge the Jacobi family by allowing a beta
parameter to vary independently from row to row. Rescaling near zero,
with all beta factors tending to gamma weights, then gives Laguerre
weights with distinct scales. Allowing these scales to vary with the
large parameter separates the centres in the Gaussian limit and leads,
at the level of weights and mixed moments, to
\[
 \bigl[\phi_{c_j}(t)\e^{\xi_i t}\bigr]_{j,i}\,\mathrm dt,
 \qquad
 \phi_c(t)\coloneq\pi^{-1/2}\e^{-(t-c)^2},
 \qquad t\in\mathbb R,
\]
where the real centres \(c_1,\ldots,c_q\) and the real column parameters
\(\xi_1,\ldots,\xi_p\) are pairwise distinct within their respective
families. After a common rescaling of the variable and nonzero row and
column normalizations, this is the Gaussian system of mixed type of
Daems and Kuijlaars~\cite[Section~6]{DaemsKuijlaars2007}.
Taking divided differences and then letting the centres coincide recovers
the Gaussian derivative rows. This route requires additional row
parameters beyond those of the Jacobi vector studied here. The principal
open problem is to obtain terminating hypergeometric formulas for the
individual \(A\)- and \(B\)-components throughout this enlarged Jacobi
and Laguerre family.

\appendix
\section{Limits with several derivative rows}
\label{app:Hermite-derivative-limits}

The general first-kind limit produces successive derivatives of one
density. This appendix records the changes to the confluence proof,
then identifies the resulting polynomial relations and their effect on
normality. The estimates and coefficient-reconstruction arguments of
Section~\ref{subsec:Laguerre-to-Hermite-confluence} are used where they
apply without change.

\subsection{The general limit}

Take \(q=r+s\), with \(r\ge0\) and \(s\ge2\), in the Laguerre family of
Section~\ref{sec:source-families}. Suppose
\[
 a_{r+j}(\tau)+1=\tau+\sqrt{2\tau}\,\widehat a_j+\mathrm O(1)
 \quad(1\le j\le s),\qquad \tau\to+\infty,
\]
and put
\[
 \varepsilon_\tau\coloneq\sqrt{2s/\tau},\qquad
 c_\tau\coloneq\tau^s,\qquad
 \widehat a_*\coloneq s^{-1/2}\sum_{j=1}^s\widehat a_j,\qquad
 x_\tau(t)\coloneq c_\tau(1+\varepsilon_\tau t).
\]
For the remaining parameters retain \(a_h(\tau)>-1\),
\(b_h(\tau)>a_h(\tau)\), and assume
\[
 \varepsilon_\tau(a_h(\tau)+1)\xrightarrow[\tau\to+\infty]{}\rho_h>0,
 \qquad b_h(\tau)-a_h(\tau)\xrightarrow[\tau\to+\infty]{}d_h\ge0.
\]
The column nodes \(\xi_i\) are pairwise distinct, satisfy
\(\xi_i+\rho_h>0\), and obey
\(\varepsilon_\tau\beta_i(\tau)\xrightarrow[\tau\to+\infty]{}\xi_i\).
Let \(I_\tau=(-\varepsilon_\tau^{-1},\infty)\),
\(Z_\tau=\mathcal M[w_{0,\tau}](1)\), and
\(Z_{h,\tau}=Z_\tau/(b_h(\tau)+1)\). With
\(v_{\ell,\tau}=\thetaop^{\ell-1}w_{0,\tau}\), define
\[
 \widehat U_{h,\tau}(t)\coloneq
 \frac{c_\tau\varepsilon_\tau}{Z_{h,\tau}}w_{h,\tau}(x_\tau(t)),
 \qquad
 \widehat U_{r+\ell,\tau}(t)\coloneq
 \frac{c_\tau\varepsilon_\tau}{Z_\tau}
 (-\varepsilon_\tau)^{\ell-1}v_{\ell,\tau}(x_\tau(t)).
\]
All these rows are extended by zero outside \(I_\tau\).
For \(F=K_{\rho_1,d_1}\cdots K_{\rho_r,d_r}\phi_{\widehat a_*}\),
the limiting vector is
\begin{equation}
 \boldsymbol U^{\mathrm H}\coloneq
 (K_{\rho_1,1}F,\ldots,K_{\rho_r,1}F,F,F',\ldots,F^{(s-1)}).
 \label{eq:Hermite-derivative-row-vector}
\end{equation}
Fix a Laguerre-admissible index
\(\boldsymbol\lambda=(\boldsymbol n,\boldsymbol\eta)\), and write
\(L\coloneq|\boldsymbol n|\), \(R\coloneq|\boldsymbol\eta|\),
and \(N\coloneq L+R\). The contour and Rodrigues expressions
\eqref{eq:canonical-Hermite-A-limit-contour} and
\eqref{eq:canonical-Hermite-B-Rodrigues} are used with these indices
and column totals \(N+1\) and \(N-1\), respectively. Denote their finite
counterparts, with the required normalizations, by
\[
 \widehat{\mathcal A}_\tau(t)\coloneq
 Z_\tau\varepsilon_\tau^{-R}\mathcal A_\tau^{\mathrm L}(x_\tau(t)),
 \qquad
 \widehat{\mathcal B}_\tau(t)\coloneq
 \frac{c_\tau\varepsilon_\tau^R}{Z_\tau}
 \mathcal B_\tau^{\mathrm L}(x_\tau(t)).
\]

\begin{theorem}[Confluence with several derivative rows]
 \label{thm:Hermite-derivative-confluence}
 The normalized rows satisfy
 \begin{equation}
 \widehat U_{h,\tau}\xrightarrow[\tau\to+\infty]{\mathrm w}K_{\rho_h,1}F,
 \qquad
 \widehat U_{r+\ell,\tau}
 \xrightarrow[\tau\to+\infty]{\mathcal D'(\mathbb R)}F^{(\ell-1)}.
 \label{eq:canonical-Hermite-Euler-limit}
 \end{equation}
 Every fixed mixed moment converges, with error
 \(\mathrm O(\Theta_\tau)\) as \(\tau\to+\infty\), where
 \(\Theta_\tau\) is defined by \eqref{eq:canonical-Hermite-Theta-rate}
 using the present scalings. In the respective balances,
 \[
 \widehat{\mathcal A}_\tau\xrightarrow[\tau\to+\infty]{}
 \mathcal A^{\mathrm H}_{\boldsymbol\lambda,\boldsymbol m},
 \qquad
 \widehat{\mathcal B}_\tau
 \xrightarrow[\tau\to+\infty]{\mathcal D'(\mathbb R)}
 \mathcal B^{\mathrm H}_{\boldsymbol\lambda,\boldsymbol m}.
 \]
 The \(A\)-convergence is locally uniform; its individual polynomial
 components converge coefficientwise. The limiting functions satisfy
 their moment conditions. The estimates for the forms and
 \(A\)-components in Proposition~\ref{prop:canonical-Hermite-quantitative-confluence}
 hold with the scalar \(\eta\) replaced by \(R\).
\end{theorem}

\begin{proof}
 Set \(Q_{0,\tau}(z)\coloneq
 c_\tau^{-z/\varepsilon_\tau}\mathcal M[w_{0,\tau}](1+z/\varepsilon_\tau)/Z_\tau\).
 Relative to \eqref{eq:canonical-Hermite-Q0-exact-factorization},
 the single gamma quotient is replaced by
 \[
 G_\tau(z)\coloneq\prod_{j=1}^s
 \frac{\Gamma(\alpha_{j,\tau}+z/\varepsilon_\tau)}
      {\Gamma(\alpha_{j,\tau})\tau^{z/\varepsilon_\tau}},
 \qquad \alpha_{j,\tau}\coloneq a_{r+j}(\tau)+1.
 \]
 Taylor's formula \eqref{eq:canonical-Hermite-log-Gamma-Taylor}
 gives, for each factor, the logarithm
 \(\widehat a_jz/\sqrt s+z^2/(4s)+\mathrm O_K(\varepsilon_\tau)\).
 Summing, and applying the unchanged beta estimate
 \eqref{eq:canonical-Hermite-beta-Gamma-expansion}, gives
 \[
 Q_{0,\tau}=L_F+\mathrm O_K(\Omega_\tau),
 \qquad \tau\to+\infty,
 \]
 locally uniformly in the common half-plane and with every fixed
 derivative. Here \(\Omega_\tau\) has the definition
 \eqref{eq:canonical-Hermite-Omega-rate} with the present parameters.

 The normalized transforms of the other rows are exactly
 \[
 Q_{h,\tau}(z)=
 \frac{\varepsilon_\tau(b_h(\tau)+1)}
      {\varepsilon_\tau(b_h(\tau)+1)+z}Q_{0,\tau}(z),\qquad
 Q_{r+\ell,\tau}(z)=(-z-\varepsilon_\tau)^{\ell-1}Q_{0,\tau}(z).
 \]
 The finite-difference identities
 \eqref{eq:canonical-Hermite-finite-difference-moments} and
 \eqref{eq:canonical-Hermite-finite-difference-integral} therefore
 give all the mixed-moment limits and their errors. For the base row
 \(f_{0,\tau}=\widehat U_{r+1,\tau}\) and the convolution rows, the
 bounded-even-moment argument in
 Theorem~\ref{thm:simultaneous-Laguerre-Hermite-confluence} gives weak
 convergence.

 For the derivative rows, the change of variable gives
 \(\widehat U_{r+\ell,\tau}
 =((1+\varepsilon_\tau t)\partial_t)^{\ell-1}f_{0,\tau}\).
 This also holds after extension by zero: near the original origin,
 \(\thetaop^kw_{0,\tau}(x)
 =\mathrm O(x^{\underline a_\tau}(1+|\log x|^{M_{k,\tau}}))\)
 as \(x\downarrow0\), for some finite \(M_{k,\tau}\) and
 \(\underline a_\tau=\min_\nu a_\nu(\tau)>-1\).
 Hence \(x\thetaop^kw_{0,\tau}(x)\xrightarrow[x\downarrow0]{}0\) and integration by parts
 introduces no boundary term. The formal adjoint of
 \((1+\varepsilon_\tau t)\partial_t\) is
 \(-\partial_t((1+\varepsilon_\tau t)\,\cdot\,)\).
 Its fixed powers converge on each compactly supported smooth test
 function to the corresponding powers of \(-\partial_t\). Testing
 against such functions proves \eqref{eq:canonical-Hermite-Euler-limit}.

 The exact contour identity
 \eqref{eq:canonical-Hermite-A-exact-scaled-contour} now has exponent
 \(|\boldsymbol m|-|\boldsymbol n|-1=R\).
 Uniform convergence of \(Q_{0,\tau}\) on disjoint circles around
 the \(\xi_i\)'s proves the \(A\)-limit. Division by the column factor
 and evaluation at \(m_i\) fixed points give coefficient convergence,
 with the same Vandermonde argument as in
 Corollary~\ref{cor:canonical-Hermite-explicit-A-components}.
 The exact \(B\)-transform is
 \eqref{eq:canonical-Hermite-finite-B-transform}, with the present
 \(Q_{0,\tau}\), and tends to \(L_FP_{\boldsymbol m}/D_{\boldsymbol n}\).
 The convolution and differential representation in
 Theorem~\ref{thm:canonical-Hermite-B-form-confluence} uses only
 these finite factors and the weak base-row limit, so it also gives
 distributional convergence here. The moment estimates and the contour
 estimates give the asserted rates. The limiting moment conditions are
 checked in the connection calculation below.
\end{proof}

\subsection{Polynomial relations and normality}
\label{subsec:Hermite-several-derivative-rows}

Use the index set \(\mathcal I_{\boldsymbol\lambda}\) and coefficient
matrix \(C\) of \eqref{eq:canonical-Hermite-B-connection-matrix}, now
with \(\lambda_h=n_h\), \(\lambda_{r+\ell}=\eta_\ell\),
\(\mathcal I_{\boldsymbol\lambda}=\{(j,k):1\le j\le r+s,\ 0\le k<\lambda_j\}\), and
\[
 \Lambda_h(z)=\frac{\rho_h}{\rho_h+z}L_F(z),\qquad
 \Lambda_{r+\ell}(z)=(-z)^{\ell-1}L_F(z).
\]
Thus \(\Phi_{j,k}=(D_{\boldsymbol n}/L_F)H_k(\partial_z)\Lambda_j\)
and \(\Psi_{j,k}\coloneq(D_{\boldsymbol n}/L_F)\partial_z^k\Lambda_j\).
The explicit base-row expression
\eqref{eq:canonical-Hermite-Psi-base-explicit} gives
\(\Psi_{F,v}\coloneq(D_{\boldsymbol n}/L_F)\partial_z^vL_F\);
Leibniz's rule supplies all the derivative rows:
\begin{equation}
 \Psi_{r+\ell,k}(z)=
 \sum_{j=0}^{\min(k,\ell-1)}
 \binom kj(-1)^j(\ell-1)_{\underline j}
 (-z)^{\ell-1-j}\Psi_{F,k-j}(z).
 \label{eq:canonical-Hermite-Psi-Euler-explicit}
\end{equation}
Admissibility cancels every denominator: the rational quotient
\(L_F^{-1}\partial_z^k\Lambda_j\) has pole order at most \(k+1\) at its own convolution-row rate and at most
\(k\) at the other rates. The inequalities
\eqref{eq:Laguerre-admissibility-nn}--\eqref{eq:Laguerre-admissibility-etan}
supply precisely these powers in \(D_{\boldsymbol n}\).
Expansion as \(z\to\infty\), where \(L_F'/L_F=z/2+\mathrm O(1)\), yields
\begin{equation}
 \deg\Phi_{h,k}=L-1+k,\qquad
 \deg\Phi_{r+\ell,k}=L+\ell-1+k,\qquad
 [z^{L+\ell-1+k}]\Phi_{r+\ell,k}=(-1)^{\ell-1}2^{-k}.
 \label{eq:Hermite-derivative-connection-degrees}
\end{equation}
The same degrees hold for \(\Psi\), since the lower derivatives in
the monic polynomial \(H_k\) contribute lower powers.
The ordered step-line bounds put all these degrees at most \(N-1\).

This also verifies the \(A\)-moment conditions asserted above. Fubini's
theorem is applicable on the fixed contour, since it lies inside the
common half-plane of integrability, and gives
\[
 \int_{\mathbb R}t^kU_j^{\mathrm H}(t)
       \mathcal A^{\mathrm H}(t)\,\mathrm dt
 =\frac1{2\pi\mathrm i}\oint_{\Sigma_\xi}
       \frac{\Psi_{j,k}(z)}{P_{\boldsymbol m}(z)}\,\mathrm dz=0.
\]
Indeed, \(\deg P_{\boldsymbol m}=N+1\), and the rational integrand
is \(\mathrm O(z^{-2})\) as \(z\to\infty\). For \(B\), the zeros of
\(P_{\boldsymbol m}\) in its transform give the moment conditions.
The factorization \(JVC\) from
Proposition~\ref{prop:canonical-Hermite-B-components} consequently
applies to these same component spaces.

The dependence among the rows follows from a Pearson identity. Its
derivatives also describe every relation needed below.

\begin{proposition}[Pearson relation]
 \label{prop:canonical-Hermite-Pearson-relation}
 For every \(t\in\mathbb R\),
 \begin{equation}
 F'(t)+2(t-\widehat a_*)F(t)
 +2\sum_{h=1}^r\frac{d_h}{\rho_h}U_h^{\mathrm H}(t)=0.
 \label{eq:canonical-Hermite-Pearson-relation}
 \end{equation}
 More generally, for \(k\in\mathbb N_0\),
 \begin{equation}
 F^{(k+1)}+2(t-\widehat a_*)F^{(k)}+2kF^{(k-1)}
 +2\sum_{h=1}^r\frac{d_h}{\rho_h}K_{\rho_h,1}F^{(k)}=0,
 \label{eq:canonical-Hermite-Pearson-derivative-relations}
 \end{equation}
 where the term \(kF^{(k-1)}\) is zero when \(k=0\).
\end{proposition}

\begin{proof}
 The logarithmic derivative
 \(L_F'/L_F=\widehat a_*+z/2-\sum_h d_h/(\rho_h+z)\)
 makes the bilateral transform of the first left-hand side zero.
 Each term is continuous and integrable: Gaussian derivatives have
 these properties, the gamma kernels have finite first moments, and
 hence \(tF\) is also integrable. Fourier uniqueness, followed by
 continuity, proves the identity for every \(t\).
 Differentiation under each gamma convolution is justified by the
 bounded, integrable Gaussian derivatives. Differentiating \(k\)
 times and using the product rule proves the second formula.
\end{proof}

Integration by parts, with vanishing boundary term at infinity, gives
\begin{equation}
 (U_h^{\mathrm H})'=\rho_h(U_h^{\mathrm H}-F),\qquad 1\le h\le r.
 \label{eq:Hermite-convolution-row-derivative}
\end{equation}
When \(\eta_1\ge2\), the coefficient vector in
\eqref{eq:canonical-Hermite-Pearson-relation} respects the component
degree bounds. On the ordered step-line this is precisely \(R>s\).
The next result first separates that range from the normal indices.

\begin{proposition}[Finite normal range]
 \label{prop:Hermite-finite-normal-range}
 For \(R>s\), the Rodrigues function has no polynomial decomposition
 with the prescribed component degrees. For \(R\le s\),
 \begin{equation}
 \Delta_{\mathrm H}=(-1)^{\binom R2}
 \prod_{h=1}^r\left[\rho_h^{n_h}\prod_{k=0}^{n_h-1}(d_h+1)_k\right]
 \prod_{1\le j<h\le r}(\rho_h-\rho_j)^{n_jn_h}.
 \label{eq:Hermite-connection-determinant-several-derivatives}
 \end{equation}
 Thus the decomposition exists uniquely exactly when the rates with
 positive indices are distinct. In that case the \(B\)-problem is
 weakly normal, and its strong-normality criterion is
 \eqref{eq:Hermite-B-strong-normality-criterion}.
 Its derivative-row components are constant; the convolution-row
 components are constant for \(R<s\), and at most linear for \(R=s\).
 In particular,
 \[
 N\le r+R\quad(R<s),\qquad N\le2r+s\quad(R=s).
 \]
\end{proposition}

\begin{proof}
 Write \(R=sQ+R_0\), \(0\le R_0<s\). The step-line has
 \(\eta_\ell=Q+1\) for \(\ell\le R_0\) and \(\eta_\ell=Q\)
 otherwise. Admissibility gives \(n_h\le Q+1\), while
 \(\ell+\eta_\ell\le Q+s\). The degree formulas
 \eqref{eq:Hermite-derivative-connection-degrees} therefore imply
 \begin{equation}
 \operatorname{span}\{\Phi_{j,k}:(j,k)\in\mathcal I_{\boldsymbol\lambda}\}
 \subseteq\mathbb P_{L+Q+s-2}.
 \label{eq:Hermite-singular-degree-bound}
 \end{equation}
 If \(R>s\), then \(L+Q+s-2<N-1\). No combination of these
 polynomials can equal \(P_{\boldsymbol m}\), which has degree \(N-1\)
 in the \(B\)-balance. Taking transforms of a putative component
 decomposition would give precisely that impossible identity.

 If \(R\le s\), the derivative indices are \(R\) ones followed by
 zeros. Use the basis \(E_{h,k}=D_{\boldsymbol n}/(z+\rho_h)^{k+1}\)
 and \(E_{r+1,k}=z^kD_{\boldsymbol n}\) from
 Proposition~\ref{prop:canonical-Hermite-vector-edge-determinant};
 its coefficient determinant is
 \eqref{eq:canonical-Hermite-vector-edge-cleared-determinant}.
 For \(R<s\), every \(n_h\le1\), and the change of basis is diagonal:
 \(\Phi_{h,0}=\rho_hE_{h,0}\) and
 \(\Phi_{r+\ell,0}=(-1)^{\ell-1}E_{r+1,\ell-1}\).
 For \(R=s\), every \(n_h\le2\). If some \(n_h=2\), all
 convolution-row indices are positive, and
 \[
 \frac{\Phi_{h,1}}{D_{\boldsymbol n}}=
 \rho_h\left[\frac{\widehat a_*+z/2}{z+\rho_h}
 -\frac{d_h+1}{(z+\rho_h)^2}
 -\sum_{u\ne h}\frac{d_u}{(z+\rho_h)(z+\rho_u)}\right].
 \]
 For distinct rates, partial fractions leave the diagonal coefficient
 \(-\rho_h(d_h+1)\) of \(E_{h,1}\); every other term uses a simple-pole
 column or \(D_{\boldsymbol n}\). Multiplying these coefficients with
 the preceding determinant proves
 \eqref{eq:Hermite-connection-determinant-several-derivatives}, including its
 sign. If two positive-index rates coincide, the corresponding
 \(\Phi_{h,0}\)'s coincide, so both determinant and product vanish.
 Nonsingularity gives the decomposition and normalizing moment through
 the same \(JVC\) factorization as in
 Proposition~\ref{prop:canonical-Hermite-B-components}.
 The degree assertions and bounds on \(N\) follow from
 \(\eta_\ell\in\{0,1\}\) and the bounds on \(n_h\).
\end{proof}

The components in this finite normal range require only the constant
and double-pole coefficients in
\eqref{eq:Hermite-B-explicit-partial-fractions}, with the present \(R\).
Use its \(\pi_l,\zeta_{J,K},c_J(t),x_h^{(J)}\), and put
\[
 \mathcal T_J\coloneq\frac{\zeta_{J,1}}{d_J+1},\qquad
 \mathcal U_{h,J}\coloneq-d_hx_h^{(J)}\mathcal T_J\quad(h\ne J).
\]
Set both expressions to zero when \(n_J<2\), without evaluating a
rate difference, and set \(\pi_l=0\) outside \(0\le l<R\).
Then
\begin{equation}
 \mathcal B^{\mathrm H}
 =\sum_{h=1}^r B^{(h),\mathrm H}U_h^{\mathrm H}
  +\sum_{\ell=1}^sD^{(\ell),\mathrm H}F^{(\ell-1)},
 \label{eq:Hermite-derivative-component-decomposition}
\end{equation}
with zero components at zero indices and
\begin{align*}
 B^{(h),\mathrm H}(t)
 &=\frac1{\rho_h}\left[\zeta_{h,0}+c_h(t)\mathcal T_h
            +\sum_{J\ne h}(\mathcal U_{h,J}-\mathcal U_{J,h})\right],\\
 D^{(\ell),\mathrm H}(t)
 &=(-1)^{\ell-1}\pi_{\ell-1}
       +\frac{\delta_{\ell,1}}2\sum_{J:n_J>0}\mathcal T_J.
\end{align*}
Indeed, the finite-pole calculation in
Corollary~\ref{cor:canonical-Hermite-explicit-B-components} now has
pole order at most two, so \(\mathcal W_J=0\) and its other terms
are exactly \(\mathcal T_J,\mathcal U_{h,J}\). The polynomial part is represented
directly by the derivative rows, since \(z^lL_F\) is the transform
of \((-1)^lF^{(l)}\). The constant \(\mathcal T_J/2\) left by each double
pole multiplies \(F\), giving the second formula.

For \(R>s\) and distinct rates at positive indices, define
\[
 R=sQ+R_0,\qquad 0\le R_0<s,\qquad
 \kappa\coloneq(s-1)(Q-1)+R_0>0.
\]
For \(0\le k\le s-2\), let \(\boldsymbol r_k(t)\) be the coefficient
vector, in the row order \eqref{eq:Hermite-derivative-row-vector}, of
\begin{equation}
 F^{(k+1)}+2(t-\widehat a_*)F^{(k)}+2kF^{(k-1)}
 +2\sum_{h=1}^r d_h\rho_h^{k-1}U_h^{\mathrm H}
 -2\sum_{h=1}^r d_h\sum_{a=0}^{k-1}\rho_h^{k-a-1}F^{(a)}=0.
 \label{eq:Hermite-reduced-Pearson-relations}
\end{equation}
All functions are evaluated at \(t\); the term \(kF^{(k-1)}\) and
the inner sum are zero when \(k=0\). These identities follow from
\eqref{eq:canonical-Hermite-Pearson-derivative-relations} by iterating
\eqref{eq:Hermite-convolution-row-derivative}:
\[
 (U_h^{\mathrm H})^{(k)}
 =\rho_h^kU_h^{\mathrm H}
  -\sum_{a=0}^{k-1}\rho_h^{k-a}F^{(a)}.
\]

\begin{theorem}[Exact rank loss and polynomial relations]
 \label{thm:Hermite-exact-rank-defect}
 Under the preceding hypotheses, with \(s>1\) and \(R>s\),
 \begin{equation}
 \operatorname{span}\{\Phi_{j,k}:(j,k)\in\mathcal I_{\boldsymbol\lambda}\}
 =\mathbb P_{N-\kappa-1},\qquad
 \operatorname{rank}C=N-\kappa,\qquad \dim\ker C=\kappa.
 \label{eq:Hermite-exact-rank-defect}
 \end{equation}
 The space of polynomial component vectors satisfying
 \(\deg V_j<\lambda_j\) and
 \(\sum_j V_j(t)U_j^{\mathrm H}(t)\equiv0\) has the basis
 \begin{equation}
 \left\{t^a\boldsymbol r_k(t):
       0\le k\le s-2,\quad 0\le a\le\eta_{k+1}-2\right\}.
 \label{eq:Hermite-Pearson-kernel-basis}
 \end{equation}
 Expanding the component polynomials in the Hermite basis identifies this
 space with \(\ker C\).
\end{theorem}

\begin{proof}
 The bound
 \eqref{eq:Hermite-singular-degree-bound} gives the inclusion in
 \(\mathbb P_{L+Q+s-2}\), and
 \(L+Q+s-2=N-\kappa-1\). It remains to prove that every polynomial
 in this space belongs to the span.

 Retain the \(r\) convolution rows and the single row \(F\), with
 lengths \((\boldsymbol n,\eta_1)\). This index is still near the
 diagonal. Its connection polynomials are the corresponding columns of
 \(C\), since both constructions use the same \(D_{\boldsymbol n}\),
 \(L_F\), and Hermite basis. Proposition~\ref{prop:canonical-Hermite-vector-edge-determinant}
 shows that these columns form a basis of
 \(\mathbb P_{L+\eta_1-1}\).

 If \(R_0=0\), then \(\eta_1=Q\). Append the columns
 \(\Phi_{r+\ell,Q-1}\), \(\ell=2,\ldots,s\). Their degrees are
 successively \(L+Q,L+Q+1,\ldots,L+Q+s-2\). If \(R_0>0\), then
 \(\eta_1=Q+1\); append instead the same columns with
 \(\ell=3,\ldots,s\), whose degrees begin at \(L+Q+1\).
 All these columns exist because \(\eta_\ell\ge Q\). Their leading
 coefficients are \((-1)^{\ell-1}2^{-(Q-1)}\), by
 \eqref{eq:Hermite-derivative-connection-degrees}. Thus each appended column
 adds the next degree, proving the span and rank assertions.

 The identities \eqref{eq:Hermite-reduced-Pearson-relations} show that
 the vectors in \eqref{eq:Hermite-Pearson-kernel-basis} represent zero.
 Their degrees are admissible: the coefficient of \(F^{(k)}\) has
 degree \(a+1\le\eta_{k+1}-1\), the coefficient of \(F^{(k+1)}\)
 has degree \(a\le\eta_{k+2}-1\), and all other nonzero coefficients
 have degree at most \(a\). For the coefficient of \(F^{(k+1)}\),
 the required inequality follows from \(\eta_{k+1}\le\eta_{k+2}+1\):
 \(a\le\eta_{k+1}-2\le\eta_{k+2}-1\).
 For an earlier derivative \(F^{(v)}\), \(v<k\), the ordering gives
 \(\eta_{v+1}\ge\eta_{k+1}\), and therefore
 \(a\le\eta_{v+1}-2<\eta_{v+1}\).
 Finally, near diagonality gives \(n_h\ge\eta_{k+1}-1\), whence
 \(a\le n_h-1\) in every convolution row. These inequalities check
 all the nonzero components of \(t^a\boldsymbol r_k\).

 To prove independence, suppose
 \(\sum_{k=0}^{s-2}P_k(t)\boldsymbol r_k(t)=0\), with not all
 \(P_k\) zero, and let \(a\) be their largest degree. At degree
 \(a+1\), the only contributions are \(2tP_k(t)\), each in the
 distinct row corresponding to \(F^{(k)}\). All other coefficients
 of \(\boldsymbol r_k\) are constant, so none can cancel those terms.
 This is a contradiction. Finally, the number of vectors is
 \[
 \sum_{k=0}^{s-2}(\eta_{k+1}-1)
 =R-\eta_s-(s-1)=R-Q-s+1=\kappa.
 \]
 They therefore form a basis of the kernel. The use of monomials here
 instead of Hermite polynomials changes coordinates by an invertible
 triangular matrix and leaves the kernel dimension unchanged.
\end{proof}

Let \(\mathscr S_{\mathrm A}\) and \(\mathscr S_{\mathrm B}\) denote
the homogeneous component-vector solution spaces in the two separate
balances \(|\boldsymbol m|=N+1\) and \(|\boldsymbol m|=N-1\).
For the latter assume \(N\ge1\). The following consequences distinguish
uniqueness of components from a nonzero normalization of their form.

\begin{corollary}[Dimensions of the two moment problems]
 \label{cor:Hermite-solution-dimensions}
 \begin{enumerate}[label=\textnormal{(\arabic*)}]
 \item If \(\Delta_{\mathrm H}\ne0\), both solution spaces are
 one-dimensional and their nonzero vectors represent nonzero forms.
 The \(A\)-problem is strongly normal, and the \(B\)-problem is weakly
 normal, with strong normality characterized by
 \eqref{eq:Hermite-B-strong-normality-criterion}.
 \item If \(R>s\) and the rates with positive indices are distinct, then
 \begin{equation}
 \dim\mathscr S_{\mathrm A}=\kappa+1,\qquad
 \dim\mathscr S_{\mathrm B}=\kappa.
 \label{eq:Hermite-solution-dimensions}
 \end{equation}
 Every vector in \(\mathscr S_{\mathrm B}\) represents zero.
 The canonical \(A\)-representative remains a nonzero solution,
 but the \(A\)-problem is not weakly normal.
 \item If \(\rho_h=\rho_u\) with \(n_h,n_u>0\), put
 \(v\coloneq\min(n_h,n_u)\). Then
 \begin{equation}
 \dim\ker C\ge v,\qquad
 \dim\mathscr S_{\mathrm A}\ge v+1,\qquad
 \dim\mathscr S_{\mathrm B}\ge v.
 \label{eq:Hermite-coincident-rate-dimensions}
 \end{equation}
 The vectors \(t^k(\boldsymbol e_h-\boldsymbol e_u)\), \(0\le k<v\),
 are independent admissible relations representing zero.
 \end{enumerate}
\end{corollary}

\begin{proof}
 The \(A\)-assertion for \(N=0\) is immediate. Assume \(N\ge1\).
 The rectangular moment matrices, after the Hermite change of basis,
 factor as \(JVC\), where \(J\) is invertible and \(V\) evaluates
 polynomials and derivatives at the column nodes.
 In the \(A\)-balance its \(N+1\) evaluations are injective on
 \(\mathbb P_{N-1}\): their kernel would be divisible by a polynomial
 of degree \(N+1\). Thus the moment matrix has rank \(\operatorname{rank}C\).
 Subtracting from \(N+1\) gives the asserted \(A\)-dimensions.
 Distinct real exponentials with polynomial coefficients are independent,
 and the leading coefficients \eqref{eq:Hermite-A-leading-coefficient}
 are nonzero; hence the one-dimensional case is strongly normal.

 For \(B\), a nonzero \(\Delta_{\mathrm H}\) gives the adjacent
 determinant \eqref{eq:canonical-Hermite-B-connection-determinant}
 and a nonzero normalizing moment. If \(R>s\),
 \eqref{eq:Hermite-exact-rank-defect} identifies the image of \(C\)
 with \(\mathbb P_{N-\kappa-1}\subseteq\mathbb P_{N-2}\).
 The \(N-1\) evaluations are injective on that image, so
 \(\ker(JVC)=\ker C\). The exact-rank theorem identifies this kernel
 with the polynomial relations representing zero.
 Finally, coincident rates give identical rows \(U_h^{\mathrm H}=U_u^{\mathrm H}\).
 The \(v\) displayed independent relations imply
 \(\operatorname{rank}C\le N-v\), the \(A\)-bound through \(JVC\),
 and the \(B\)-bound directly.
\end{proof}

When \(\kappa=1\), the singular \(B\)-problem is weakly normal in the
component-vector sense of Section~\ref{sec:source-families}, but its sole
direction represents zero and cannot be normalized by a nonzero moment.
When \(\kappa>1\), uniqueness of the component vector also fails.

\subsection{Divergence of the components}

The finite-parameter Rodrigues function can converge although its
polynomial components grow without bound. To state this precisely, use
the scaled components
\begin{equation}
 \widetilde B_{h,\tau}(t)\coloneq
 \frac{\varepsilon_\tau^{R-1}}{b_h(\tau)+1}
 B_\tau^{(h),\mathrm L}(x_\tau(t)),\qquad
 \widetilde D_{\ell,\tau}(t)\coloneq
 (-1)^{\ell-1}\varepsilon_\tau^{R-\ell}
 D_\tau^{(\ell),\mathrm L}(x_\tau(t)).
 \label{eq:canonical-Hermite-B-Euler-component-limit}
\end{equation}
Their linear combination with the normalized rows is
\(\widehat{\mathcal B}_\tau\). Let \(\mathbf b_\tau\) collect their
coefficients in the fixed Hermite basis, and fix any norm on this
finite-dimensional coefficient space.
For \(R\le s\) and separated positive-index rates, these components
converge to those in \eqref{eq:Hermite-derivative-component-decomposition},
with coefficient error \(\mathrm O(\Theta_\tau)\) as \(\tau\to+\infty\).
Indeed, the moments converge and the limiting adjacent determinant is
nonzero, so the fixed-matrix inversion argument of
Proposition~\ref{prop:canonical-Hermite-quantitative-confluence} applies.

\begin{proposition}[Divergence in the singular range]
 \label{prop:Hermite-singular-component-divergence}
 Suppose \(R>s\) and the product \(c_{\boldsymbol\lambda}(\tau)\) in
 \eqref{eq:Laguerre-normality-coefficient-determinant} is nonzero for
 all sufficiently large \(\tau\). Then the finite Rodrigues function
 has unique polynomial components of the prescribed degrees, and
 \begin{equation}
 \|\mathbf b_\tau\|\xrightarrow[\tau\to+\infty]{}+\infty.
 \label{eq:Hermite-singular-component-divergence}
 \end{equation}
\end{proposition}

\begin{proof}
 Augment the column index to \(\boldsymbol m+\boldsymbol e_{i_0}\),
 of length \(N\). Its finite nodes are
 \(\beta_i(\tau)+k+1\), \(0\le k<m_i+\delta_{i,i_0}\).
 Within each group they are distinct; between groups their differences,
 multiplied by \(\varepsilon_\tau\), tend to \(\xi_i-\xi_j\ne0\).
 The gamma factors in
 \eqref{eq:Laguerre-normality-moment-determinant} are finite and
 nonzero for large \(\tau\). That determinant formula and
 \(c_{\boldsymbol\lambda}(\tau)\ne0\) give an invertible reconstruction
 system \eqref{eq:B-component-polynomial-system}, hence unique components.

 If their coefficient norms did not tend to infinity, a bounded
 subsequence would converge to a vector \(\mathbf b\).
 The normalized rows converge in distributions by
 Theorem~\ref{thm:Hermite-derivative-confluence}; multiplication
 by each fixed \(H_k\) preserves this convergence. Passing to the limit
 in their finite linear combination would give
 \[
 \mathcal B^{\mathrm H}_{\boldsymbol\lambda,\boldsymbol m}
 =\sum_{(j,k)\in\mathcal I_{\boldsymbol\lambda}}b_{j,k}H_kU_j^{\mathrm H}.
 \]
 Its left-hand side has the nonzero transform \(L_FP_{\boldsymbol m}/D_{\boldsymbol n}\).
 Its right-hand side would be a decomposition excluded by
 Proposition~\ref{prop:Hermite-finite-normal-range}.
 This contradiction rules out every bounded subsequence.
\end{proof}

At a singular finite parameter, existence is instead the compatibility
condition \(\operatorname{rank}C_\tau^{\mathrm L}
=\operatorname{rank}[C_\tau^{\mathrm L}\mid p_\tau]\), where
\(C_\tau^{\mathrm L}\) and \(p_\tau\) are the coefficient matrix and
right-hand side of \eqref{eq:B-component-polynomial-system}.
Along any sequence on which components do exist, the same argument
forces their divergence. The canonical \(A\)-components, by contrast,
still converge coefficientwise.

\begin{example}[A convergent form with diverging components]
 \label{ex:Hermite-diverging-components}
 Take \(r=0\), \(s=2\), \(\boldsymbol\eta=(2,2)\), \(p=1\),
 \(m_1=3\), \(\beta_1(\tau)=0\), and \(a_1(\tau)=a_2(\tau)=\tau-1\).
 Then \(c_\tau=\tau^2\), \(\varepsilon_\tau=2/\sqrt \tau\), and
 \(\widehat a_*=0\). Since \(r=0\), the product
 \(c_{\boldsymbol\lambda}(\tau)\) is empty and equals one. For the
 augmented column length four the finite nodes are \(1,2,3,4\), so
 the existence and uniqueness hypotheses of
 Proposition~\ref{prop:Hermite-singular-component-divergence} hold.
 Write \(\varepsilon=\varepsilon_\tau\) and
 \(f_\tau=\widehat U_{1,\tau}\); the second normalized row is
 \((1+\varepsilon t)f_\tau'(t)\). The exact scaled components
 \eqref{eq:canonical-Hermite-B-Euler-component-limit} are
 \begin{align*}
 \widetilde D_{1,\tau}(t)
 &=\left(\frac{128}{\varepsilon^2}+80\right)t
          +\frac{80}{\varepsilon}+24\varepsilon+2\varepsilon^3,\\
 \widetilde D_{2,\tau}(t)
 &=\frac{64}{\varepsilon^2}+\frac{16t}{\varepsilon}+24+2\varepsilon^2.
 \end{align*}
 To verify these expressions, the gamma factors in the centered transform
 give
 \[
 \frac{Q_{0,\tau}(z+\varepsilon)}{Q_{0,\tau}(z)}
 =\left(1+\frac{\varepsilon z}{4}\right)^2.
 \]
 Multiplication by \(t\) corresponds to the forward difference divided
 by \(\varepsilon\), and the transform of the second row is
 \((-z-\varepsilon)Q_{0,\tau}(z)\). Substitution of the displayed
 components therefore yields exactly
 \(Q_{0,\tau}(z)(-z)(\varepsilon-z)(2\varepsilon-z)\), the transform in
 \eqref{eq:canonical-Hermite-finite-B-transform} for these indices.
 Consequently,
 \[
 \frac1\tau\bigl(\widetilde D_{1,\tau}(t),\widetilde D_{2,\tau}(t)\bigr)
 \xrightarrow[\tau\to+\infty]{}16(2t,1)
 \]
 coefficientwise, whereas
 \[
 \widetilde D_{1,\tau}f_\tau+\widetilde D_{2,\tau}(1+\varepsilon_\tau t)f_\tau'
 \xrightarrow[\tau\to+\infty]{\mathcal D'(\mathbb R)}\phi_0'''.
 \]
 Thus the coefficient norm grows linearly in \(\tau\), and its leading
 direction is the Pearson relation \(2t\phi_0+\phi_0'=0\).
\end{example}

Weak normality need not be strong even in the finite range:
for \(r=0\), \(s=2\), \(\boldsymbol\eta=(1,1)\), \(p=1\),
\(m_1=1\), and \(\widehat a_*=\xi_1=0\), one has
\(C=\operatorname{diag}(1,-1)\), and the unique components of
\(\mathcal B^{\mathrm H}=\phi_0'\) are \((0,1)\).
The adjacent moment is
\(\int_{\mathbb R}t\phi_0'(t)\,\mathrm dt=-1\), but the first component
does not attain its prescribed degree.

For comparison with coinciding Gaussian centres, when \(r=0\) the
identity \(\phi_a^{(k)}=(-2)^kH_k(t-a)\phi_a\) relates the derivative
rows to \((\phi_a,t\phi_a,\ldots,t^{s-1}\phi_a)\) by an invertible
constant lower triangular matrix. The latter rows are obtained, up to
nonzero constants, by taking divided differences of
\(\phi_0(t)\e^{\zeta t}\) and letting its parameters \(\zeta\) tend to \(2a\).
The corresponding upper triangular change of components preserves the
ordered degree bounds \(\eta_1\ge\cdots\ge\eta_s\).
For \(r>0\), however, the transforms of the convolution rows have
finite singularities at \(-\rho_h\), whereas every finite span of
Gaussian rows and their centre derivatives has entire transforms.
Thus a constant change of row basis cannot identify these two families.

\section*{Funding}
This work was supported by the research project PID2024-155133NB-I00,
\emph{Ortogonalidad, aproximaci\'on e integrabilidad: aplicaciones en procesos
estoc\'asticos cl\'asicos y cu\'anticos}. The funder had no role in the design
of the study, the preparation of the manuscript, or the decision to submit
the work for publication.

\section*{Conflict of interest}
The author declares no conflict of interest.

\section*{Data availability}
No datasets were generated or analyzed in this study.

\printbibliography
	
\end{document}